\RequirePackage{fix-cm} % fix some latex issues see: http://texdoc.net/texmf-dist/doc/latex/base/fixltx2e.pdf
\documentclass[ oneside,titlepage,numbers=noenddot,headinclude,%1headlines,
                footinclude=true,abstract=false,
                BCOR=0mm,paper=letter,fontsize=12pt,% Stanford: US letter, 10-12pt
                american,%
                ]{scrreprt}

\usepackage{silence}
\usepackage[left=1.5in,right=1in,top=1in,bottom=1in,headheight=15.3pt,headsep=0.5in,footskip=0.5in]{geometry}
\PassOptionsToPackage{dvipsnames}{xcolor}
\usepackage{hyperref}
\usepackage{amsmath}
\usepackage{tikz}
\usetikzlibrary{calc,positioning, arrows.meta}
\usepackage{tabularx}
\usepackage{array}
\usepackage{nicefrac}
\usepackage{caption}
\usepackage{float}
\usepackage{siunitx}
\usepackage{mathtools}
\usepackage{wrapfig}
\usepackage{appendix}
\usepackage{mdframed} % for boxed
\usepackage{makecell}
\usepackage{booktabs}
\usepackage{enumitem} % for description list formatting

\usepackage{algorithm}
\usepackage{algorithmicx,tikz}
\usepackage{algcompatible}
\usepackage[noend]{algpseudocode}

\newcolumntype{Y}{>{\centering\arraybackslash$}X<{$}}

\PassOptionsToPackage{utf8}{inputenc}
	\usepackage{inputenc}

\PassOptionsToPackage{eulerchapternumbers,listings,%
					 pdfspacing,%floatperchapter,%linedheaders,%
					 subfig,beramono,eulermath,parts}{xgthesis}                                        
\newcommand{\myTitle}{Unifying Optimization and Dynamics to Parallelize Sequential Computation: A Guide to Parallel Newton Methods for Breaking Sequential Bottlenecks \xspace}
\newcommand{\myName}{Xavier Gonzalez\xspace}
\newcommand{\myUni}{Stanford University\xspace}

\newcounter{dummy} % necessary for correct hyperlinks (to index, bib, etc.)
\providecommand{\mLyX}{L\kern-.1667em\lower.25em\hbox{Y}\kern-.125emX\@}

	\usepackage{babel}

\usepackage[threshold=0]{csquotes}  % threshold=0 ensures blockquote always displays as block

\usepackage{cleveref}
\crefname{chapter}{chapter}{chapters}
\Crefname{chapter}{Chapter}{Chapters}
\crefname{section}{section}{sections}
\Crefname{section}{Section}{Sections}
\crefname{subsection}{subsection}{subsections}
\Crefname{subsection}{Subsection}{Subsections}
\crefname{figure}{figure}{figures}
\Crefname{figure}{Figure}{Figures}
\crefname{table}{table}{tables}
\Crefname{table}{Table}{Tables}
\crefname{equation}{equation}{equations}
\Crefname{equation}{Equation}{Equations}
\crefname{theorem}{theorem}{theorems}
\Crefname{theorem}{Theorem}{Theorems}
\crefname{lemma}{lemma}{lemmas}
\Crefname{lemma}{Lemma}{Lemmas}
\crefname{proposition}{proposition}{propositions}
\Crefname{proposition}{Proposition}{Propositions}
\crefname{corollary}{corollary}{corollaries}
\Crefname{corollary}{Corollary}{Corollaries}
\crefname{definition}{definition}{definitions}
\Crefname{definition}{Definition}{Definitions}
\crefname{appendix}{appendix}{appendices}
\Crefname{appendix}{Appendix}{Appendices}
\crefname{part}{part}{parts}
\Crefname{part}{Part}{Parts}
\PassOptionsToPackage{%
    backend=biber,       % <--- UNCOMMENT THIS (Modern engine)
    language=auto,%
    style=numeric-comp,%
    sorting=nyt, 
    maxbibnames=10, 
    natbib=true 
}{biblatex}
\usepackage{biblatex}
\PassOptionsToPackage{fleqn}{amsmath}       % math environments and more by the AMS
    \usepackage{amsmath}
\usepackage{amsthm}  % theorem environments
\usepackage{amssymb} % additional math symbols

\theoremstyle{plain}
\newtheorem{theorem}{Theorem}[chapter]
\newtheorem*{theorem*}{Theorem} 
\newtheorem{lemma}[theorem]{Lemma}

\usepackage{aliascnt} % Make sure this is in your preamble

\newaliascnt{example}{theorem}
\newaliascnt{proposition}{theorem}

\newtheorem{example}[example]{Example}
\newtheorem{proposition}[proposition]{Proposition}

\aliascntresetthe{example}
\aliascntresetthe{proposition}
\newaliascnt{definition}{theorem}
\theoremstyle{definition}
\newtheorem{definition}[definition]{Definition}
\aliascntresetthe{definition}
\crefname{definition}{definition}{definitions}
\Crefname{definition}{Definition}{Definitions}
\crefname{proposition}{proposition}{propositions}
\Crefname{proposition}{Proposition}{Propositions}
\crefname{example}{example}{examples}
\Crefname{example}{Example}{Examples}
\crefname{appendix}{appendix}{appendices}
\Crefname{appendix}{Appendix}{Appendices}

\newaliascnt{remark}{theorem}
\theoremstyle{remark}

\aliascntresetthe{remark}
\crefname{remark}{remark}{remarks}
\Crefname{remark}{Remark}{Remarks}

\PassOptionsToPackage{T1}{fontenc} % T2A for cyrillics
    \usepackage{fontenc}
\usepackage{textcomp} % fix warning with missing font shapes

\usepackage{mathptmx}  % Times New Roman for text and math (Stanford approved)
\usepackage{setspace}  % Line spacing control
\usepackage{scrhack} % fix warnings when using KOMA with listings package          
\usepackage{xspace} % to get the spacing after macros right  
\usepackage{mparhack} % get marginpar right
\PassOptionsToPackage{printonlyused,smaller}{acronym} 
    \usepackage{acronym} % nice macros for handling all acronyms in the thesis
\usepackage{tabularx} % better tables
\usepackage{caption}
\usepackage{subfig}  
\usepackage{listings} 
\PassOptionsToPackage{pdftex,hyperfootnotes=false,pdfpagelabels}{hyperref}
    \usepackage{hyperref}  % backref linktocpage pagebackref
\PassOptionsToPackage{pdftex}{graphicx}
    \usepackage{graphicx}

\hypersetup{%
    colorlinks=true, linktocpage=true, pdfstartpage=3, pdfstartview=FitV,%
    breaklinks=true, pdfpagemode=UseNone, pageanchor=true, pdfpagemode=UseOutlines,%
    plainpages=false, bookmarksnumbered, bookmarksopen=true, bookmarksopenlevel=1,%
    hypertexnames=true, pdfhighlight=/O,%nesting=true,%frenchlinks,%
    urlcolor=webbrown, linkcolor=RoyalBlue, citecolor=webgreen, %pagecolor=RoyalBlue,%
    pdftitle={\myTitle},%
    pdfauthor={\textcopyright\ \myName, \myUni},%
    pdfsubject={},%
    pdfkeywords={},%
    pdfcreator={pdfLaTeX},%
    pdfproducer={LaTeX}%
}

\makeatletter
\@ifpackageloaded{babel}%
    {%
       \addto\extrasamerican{%
                }%
       \addto\extrasngerman{% 
                }%  
    }{\relax}
\makeatother

\usepackage{xgthesis}
\newcommand{\StateSpaceModelDiagram}[1][0]{ % Default to 0 observed nodes
  \begin{tikzpicture}[
    node distance=1cm,
    observed/.style={
        circle, draw, fill=gray!30,
        minimum size=1cm,
        text width=1cm,
        align=center,
        inner sep=0pt,
        font=\footnotesize
    },
    latent/.style={
        circle, draw,
        minimum size=1cm,
        text width=1cm,
        align=center,
        inner sep=0pt,
        font=\footnotesize
    },
    arrow/.style={->, thick},
    shorten >=1pt,
    shorten <=1pt
  ]

  \node[\ifnum#1>0 observed\else latent\fi] (s0) {$s_0$};
  \node[right=of s0, \ifnum#1>1 observed\else latent\fi] (s1) {$s_1$};
  \node[right=of s1, \ifnum#1>2 observed\else latent\fi] (s2) {$s_2$};
  \node[right=1cm of s2, \ifnum#1>3 observed\else latent\fi] (sT1) {$s_{T-1}$};
  \node[right=of sT1, \ifnum#1>4 observed\else latent\fi] (sT) {$s_T$};

  \draw[arrow] (s0) -- (s1) node[midway, above] {$f_1$};
  \draw[arrow] (s1) -- (s2) node[midway, above] {$f_2$};
  \draw[arrow] (sT1) -- (sT) node[midway, above] {$f_T$};

  \node at ($(s2)!0.5!(sT1)$) {\large $\cdots$};

  \end{tikzpicture}
}

\newcommand{\A}{A} % used for dynamics Jacobian
\newcommand{\J}{\mathbf{J}} % used for residual Jacobian
\newcommand{\E}{\mathbb{E}}
\newcommand{\diag}{\mathrm{diag}}
\renewcommand{\b}[1]{\mathbf{#1}}
\newcommand{\osw}{\chi_w}

\begin{document}
\frenchspacing
\raggedbottom
\selectlanguage{american} % american ngerman
%\renewcommand*{\bibname}{new name}
%\setbibpreamble{}

%********************************************************************
% Frontmatter - Stanford Format
% Title page counts as "i" but is unnumbered
% Copyright (ii) and Signature (iii) pages are auto-inserted by Stanford
% Numbering starts at iv with Abstract
%*******************************************************
\pagestyle{empty}
%*******************************************************
% Stanford PhD Dissertation Title Page
% Format: Uppercase, centered, no bold, no page number
%*******************************************************
\thispagestyle{empty}
\begin{titlepage}
\begin{center}

\vspace*{\fill}

{\large
UNIFYING OPTIMIZATION AND DYNAMICS TO \\
PARALLELIZE SEQUENTIAL COMPUTATION: \\
\vspace{0.1cm}
A GUIDE TO PARALLEL NEWTON METHODS FOR \\
BREAKING SEQUENTIAL BOTTLENECKS
}

\vspace{2cm}

A DISSERTATION

\vspace{0.3cm}

SUBMITTED TO THE DEPARTMENT OF STATISTICS

\vspace{0.3cm}

AND THE COMMITTEE ON GRADUATE STUDIES

\vspace{0.3cm}

OF STANFORD UNIVERSITY

\vspace{0.3cm}

IN PARTIAL FULFILLMENT OF THE REQUIREMENTS

\vspace{0.3cm}

FOR THE DEGREE OF

\vspace{0.3cm}

DOCTOR OF PHILOSOPHY

\vspace{2cm}

Xavier Gonzalez

\vspace{0.3cm}

March 2026

\vspace*{\fill}

\end{center}
\end{titlepage}

% Start Roman numeral page numbering at iv (after auto-inserted copyright & signature)
\pagenumbering{roman}
\setcounter{page}{4}
\pagestyle{plain}

% Stanford order: Abstract, then optional preface/acknowledgments/dedication
%*******************************************************
% Abstract - Stanford Format
%*******************************************************
\pdfbookmark[1]{Abstract}{Abstract}

\begin{center}
{\large\bfseries Abstract}
\end{center}

\vspace{1em}

Recurrent neural networks (RNNs) were widely regarded as "inherently sequential" because each hidden state depends on the previous one. This sequential dependency creates a computational bottleneck: evaluating an RNN on a sequence of length $T$ seems to require $O(T)$ time steps, even with unlimited parallel processors. This dissertation challenges the conventional wisdom and develops methods that enable parallel evaluation of nonlinear RNNs with $O((\log T)^2)$ computational depth. Moreover, the methods I have developed and studied are very general, and can parallelize the broad class of computations falling under the heading of \emph{state space models} (SSMs). SSMs include not only nonlinear RNNs but also Markov chain Monte Carlo (MCMC), sampling from diffusion models, and explicit differential equation solvers, among many other applications.

In my PhD, I built on an approach \citep{deeppcr, deer2024} that reformulates RNN evaluation as a fixed-point problem and applies Newton's method to leverage the parallel scan algorithm. However, when I began work on this subject, the community's understanding of this \emph{parallel Newton method} was hindered by certain limitations. The Newton iterations suffered from a lack of scalability in the state dimension $D$, a lack of stability in certain applications, and a general lack of understanding of its convergence properties and rates. In this thesis, I address these limitations with methodological and theoretical contributions.

The methodological contributions of this thesis include developing scalable and stable parallel Newton methods, based on quasi-Newton and trust-region approaches.
The quasi-Newton methods further accelerate the training of RNNs and use a factor of $D$ less memory.
The trust-region approaches are parallelized over the sequence length using a parallel Kalman filter and are significantly more stable than their undamped counterparts.
These methods have inspired follow-up work both in nonlinear sequence modeling and in parallelizing MCMC. 

The theoretical contributions of this thesis include establishing the convergence rates of parallel Newton methods. 
Both the Newton and quasi-Newton methods enjoy global convergence in at most $T$ iterations. Moreover, we show that the conditioning of the optimization landscape, as quantified by its Polyak-Łojasiewicz (PL) constant, is determined by the stability of the dynamical system, as quantified by its Largest Lyapunov Exponent (LLE).
By doing so, we show that stable (i.e. $\text{LLE}<0$) dynamics enjoy convergence in $\mathcal{O}(\log T)$ iterations, while unstable dynamics converge too slowly for parallelization to work.

In sum, this thesis unlocks scalable and stable methods for parallelizing sequential computation, and provides a firm theoretical basis for when such techniques will and will not be effective.
This thesis also serves as a guide to parallel Newton methods for researchers who want to write the next chapter in this ongoing story.

\vfill

\clearpage%*******************************************************
% Acknowledgments - Stanford Format
%*******************************************************
\pdfbookmark[1]{Acknowledgments}{acknowledgments}

\begin{center}
{\large\bfseries Acknowledgments}
\end{center}

\vspace{1em}

Thank you to everyone who has made this PhD dissertation possible; and all who have taught and mentored me over the years.

Thank you to my advisor Scott Linderman for teaching me so much about research, collaboration, mentoring, software engineering, math, statistics, neuroscience, and more. You have created one of the kindest and happiest labs at Stanford, and you are the driving force behind this social and collaborative culture.

Speaking of the Linderman lab, thank you to all the amazing mentors and friends I've met in Scott's lab. I am grateful for the opportunity to work with and learn from such talented researchers.

Thank you to all my collaborators: I have learned so much and had so much fun working with you. Thank you in particular to the four postdocs who especially mentored me on these projects: Andy Warrington, Leo Kozachkov, Kelly Buchanan, and David Zoltowski. Each of you taught me so much in different ways, and I am grateful for your generosity of time and wisdom.

Thank you to my family, friends, and mentors for their support and encouragement. Thank you especially to my parents Javier and Natalya, my sister Natasha, and my girlfriend Tiffany. Your nurturing and support has made everything possible---and your love has brightened my days. 

Ultimately, thank you to God---you are the source of all good things.
Thank you for the many blessings of my PhD.

\vfill

% \clearpage\include{FrontBackmatter/Publications}
\clearpage%*******************************************************
% Table of Contents
%*******************************************************
%\phantomsection
\refstepcounter{dummy}
\pdfbookmark[1]{\contentsname}{tableofcontents}
\setcounter{tocdepth}{2} % <-- 2 includes up to subsections in the ToC
\setcounter{secnumdepth}{3} % <-- 3 numbers up to subsubsections
\manualmark
\markboth{\spacedlowsmallcaps{\contentsname}}{\spacedlowsmallcaps{\contentsname}}
\tableofcontents 
\automark[section]{chapter}
\renewcommand{\chaptermark}[1]{\markboth{\spacedlowsmallcaps{#1}}{\spacedlowsmallcaps{#1}}}
\renewcommand{\sectionmark}[1]{\markright{\thesection\enspace\spacedlowsmallcaps{#1}}}
\clearpage\pagenumbering{arabic}
\pagestyle{scrheadings}
%\setcounter{page}{90}
% use \clearpage here to avoid problems with pdfbookmark
\clearpage
\ctparttext{The first part of this thesis provides the motivation and background for parallelizing
dynamical systems. We introduce the fundamental problem of sequential
computation in deep learning and review the mathematical foundations that enable
parallel evaluation of these models.
\par
\begin{minipage}{\linewidth}
    \centering
    \captionsetup{hypcap=false}
    \includegraphics[width=0.8\linewidth]{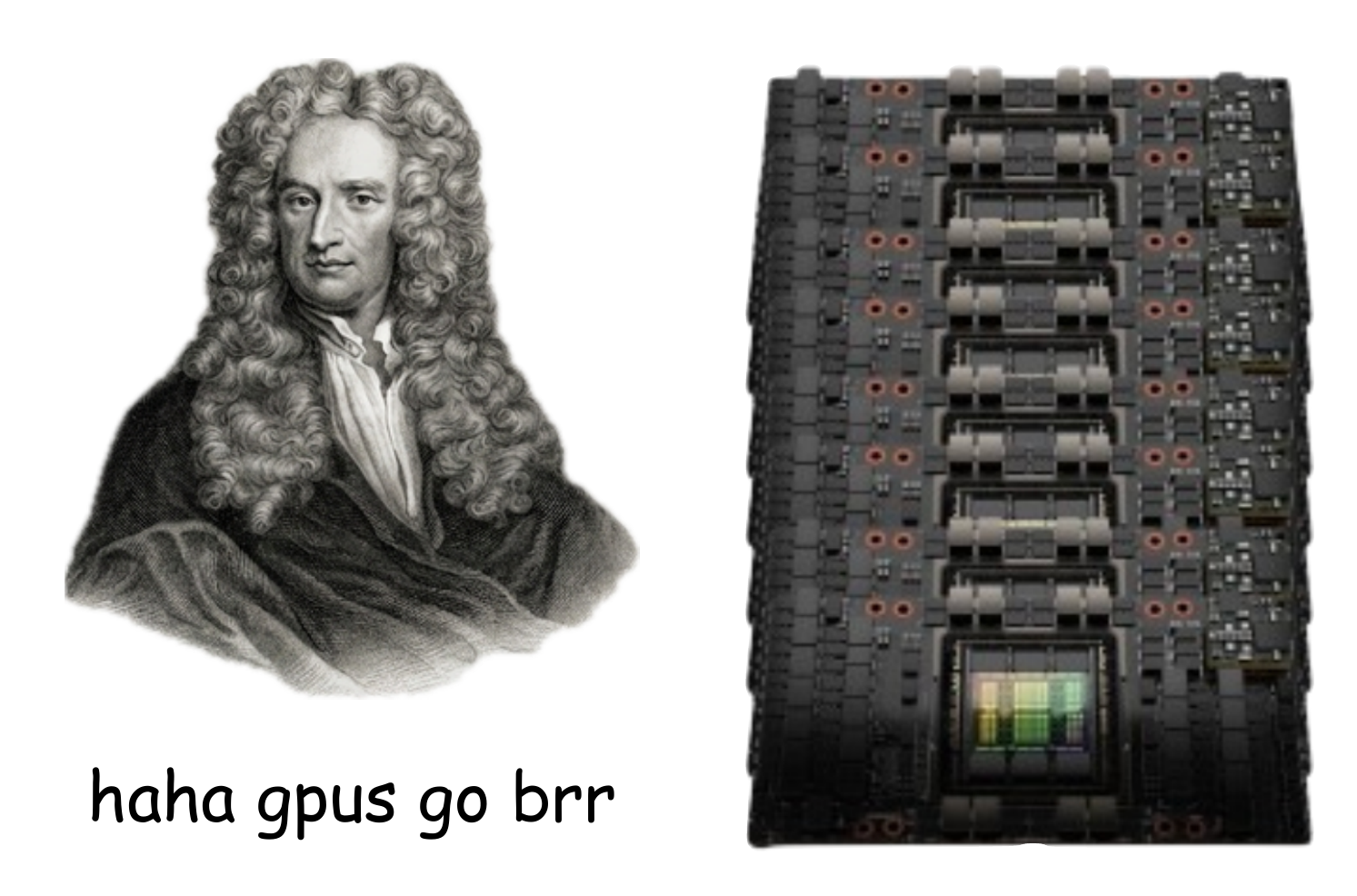}
    \captionof{figure}{\textbf{Parallel Newton methods.} With a clever connection between Newton's method and the parallel scan, we can use GPUs to parallelize and therefore accelerate dynamical systems.}
    \label{fig:newton_brr}
\end{minipage}
}
\part{Introduction and Background}
%************************************************
\chapter{Introduction}\label{ch:introduction}
%************************************************

Sequential processes are ubiquitous in statistics and machine learning.
Evaluating a recurrent neural network~\citep{goodfellow2016deep}, sampling from a diffusion model~\citep{sohl2015deep, ho2020denoising, song2021score} or with Markov chain Monte Carlo (MCMC)~\citep{Diaconis2009MCMC, geyer2011introduction}, generating from a deep state space model~\citep{gu2022s4, smith2023s5, orvieto-resurrecting, mamba}, and unrolling layers of a deep neural network~\citep{he2016deep,vaswani2017attention} all involve sequential computations.
Naively, these computations require time proportional to the length of the input or the depth of the architecture, and in some cases, they may not take full advantage of massively parallel modern hardware like %NVIDIA 
graphics processing units (GPUs). For example, if the computational graph is a very long chain where each individual step in the chain is not too computationally intensive, but the chain is very long, this computational graph will not fully utilize the approximately 10,000 cores of a modern GPU.

This mismatch between the requirements of sequential computation and the design of modern parallel hardware has led to sequential models losing the "hardware lottery" \citep{hooker2021hardware} and being replaced by more easily parallelized architectures.
This broad story is most clearly exemplified in the transition from recurrent neural networks (RNNs), the dominant sequence modeling architecture prior to 2018, towards attention and the transformer architecture, an embarrassingly parallel approach that powers most of modern AI, including the "generative pretrained transformers" behind ChatGPT and other modern large language models (LLMs) \citep{brown2020language}.
In fact, as \citet{vaswani2017attention} write in the introduction to their landmark paper that introduced the transformer (emphasis added):

\blockquote{\textbf{Recurrent neural networks have been firmly established as state of the art approaches} in sequence modeling and
transduction problems such as language modeling and machine translation. However, \textbf{their inherently
sequential nature precludes parallelization within training examples}, which becomes critical at longer
sequence lengths, as memory constraints limit batching across examples. Despite significant improvements in RNN computational efficiency, \textbf{the fundamental
constraint of sequential computation remains.}}

In the era of massively parallel hardware like GPUs, and ever longer sequences of data to process, the "inherently sequential" \citep{vaswani2017attention, martin2018parallelizing, linear_attention, smith_thesis, liu2025serial} nature of RNNs was viewed as a disqualifying disadvantage in many applications.

Incredibly, however, as introduced in the seminal papers of \citet{deeppcr} and \citet{deer2024}, nonlinear RNNs and many other types of "inherently sequential" computations \emph{can} be parallelized over the sequence length.
This parallelization is achieved by recasting the problem of sequential evaluation as a high-dimensional nonlinear equation that can be solved using Newton iterations that are parallelized over the sequence length.
However, when these \emph{parallel Newton methods} were first published, limitations blocked their wider use and adoption. These were standard limitations for Newton's method in general \citep{ortega1970iterative, NocedalWright, boyd2004convex}, namely
\begin{itemize}
    \item A lack of \emph{scalability} of the method, especially as the state size increased;
    \item A lack of \emph{stability} of the convergence of the method in certain applications; and 
    \item A lack of understanding of under what conditions the Newton iterations would \emph{converge}, and at what rates.
\end{itemize}
This dissertation helps to resolve these limitations.
Methodologically, we introduce quasi-Newton methods to provide scalable parallelization and trust-region methods to provide stable parallelization. Theoretically, we provide detailed convergence analyses of these methods, including proofs of global convergence and the identification of the stability of the underlying dynamical system as a critical decider of whether or not efficient parallelization is possible. In doing so, we have unlocked scalable parallelization of nonlinear RNNs and a robust theoretical understanding of under what conditions such parallelization is desirable. Moreover, these methods parallelize not only nonlinear RNNs \citep{deer2024, gonzalez2024scalable, farsang2025scaling, danieli2025pararnn} but can also parallelize a wide range of models called \emph{state space models} (SSMs).
In this thesis, an SSM is a discrete-time dynamical system with state $s_t \in \mathbb{R}^D$ that evolves over time by a transition function $s_t = f_t(s_{t-1})$ (see \Cref{sec:ssm}).
Examples of chain-like computational graphs involving SSMs include sampling from MCMC \citep{pmcmc, grazzi2025parallel} or diffusion models \citep{deeppcr, shih2023parallel, tang2024accelerating, selvam2024selfrefining, parasolver25, han2025chords}, solving differential equations with explicit methods \citep{iacob2025parallel}, and many other diverse applications \citep{gonzalez2025unifying}.

Taken together, this thesis lays the foundation for exciting future work in parallelizing a broad range of important primitives, while also more clearly delineating which processes are---and are not---"inherently sequential."
This thesis serves as an introduction to parallel Newton methods for researchers eager to contribute to this exciting, new field.

\section{Extended History}\label{sec:longer_lit}

The modern development of massively parallel hardware like GPUs and TPUs has created new urgency around the parallelization of sequential processes, contributing to the modern development of these parallel Newton methods \citep{deeppcr, deer2024}.
However, this thesis and the parallel Newton methods it surveys build on a long tradition of parallel-in-time computing \citep{gander201550}. In short, as long as there have been parallel computers and long sequences, there has been important work in parallelizing these sequential processes---and the modern massive increase in scale has led to a renaissance of these methods.

While there were of course many earlier efforts at parallel computers, the ILLIAC IV is widely credited as being the first massively parallel computer \citep{bell1971computer}. The ILLIAC IV was developed in the 1960s and 1970s at the University of Illinois, and was designed to have 256 processors that could carry out computation in parallel.
Almost immediately, this novel development in hardware led to novel developments in algorithms.
For example, in 1973, \citet{Stone1973} explicitly cited the ILLIAC as motivation for his development of a technique to solve tridiagonal systems of equations\footnote{Which is extremely similar to parallel Newton methods, which in their simplest form solve bidiagonal systems of equations: see \Cref{sec:deer}} in parallel. \citet{Stone1973} called this method \emph{recursive doubling}, and it is known today as \emph{parallel cyclic reduction} or the \emph{parallel associative scan}.
We provide more background on the parallel scan as a general and fundamental primitive in \Cref{sec:pscan}.
However, to give a specific example, a canonical application of the parallel scan is as a technique to use $T$ processors to multiply $T$ matrices together in $\mathcal{O}(\log T)$ computational depth---thus enabling exponential speedups on large parallel machines.

In addition to the development of parallel methodology, the development of the ILLIAC also spurred fundamental work in the theory of what computations could and could not be parallelized. 
For example, in 1975, \citet{hyafil1975bounds} and \citet{kung1976new} explicitly cited the ILLIAC as motivation for their study of which algorithms and models could be \emph{efficiently parallelized}.
They showed that linear recursions enjoy speedups from parallel processors while nonlinear recursions of rational functions with degree larger than one in general cannot. These prescient works set the stage for the more general findings of this thesis presented in \Cref{part:theory}, where we explicitly link the dynamical properties of the recursion to its parallelizability.

The desire to solve differential equations over long time windows also led to the development of parallel-in-time methods for continuous-time initial value problems (IVPs) \citep{nievergelt1964parallel, gear1988parallel}. While this dissertation primarily focuses on discrete-time SSMs, there are intimate links between discrete and continuous time, just as there are intimate links between difference and differential equations. In fact, the numerical solving of differential equations almost always eventually reduces to solving some discretization of the ordinary differential equation (ODE). Consequently, it is unsurprising that the ODE parallel-in-time, multiple shooting, and multigrid literature has many of the ingredients of modern parallel Newton methods. For example, in 1989 \citet{Bellen1989} suggested a quasi-Newton method for solving differential equations that has almost all the core components of parallel Newton methods---except for the parallel scan.

Indeed, the core ingredients of parallel Newton methods remained scattered throughout the parallel-in-time literature. For example, \citet{horton1995algorithm} proposed parallel-in-time solvers for differential equations using the parallel scan---but applied this technique to other fixed-point iterations like Gauss-Seidel and Jacobi, not Newton. 

This preference for Gauss-Seidel and Jacobi iterations persisted in many strands of the literature. For example, \citet{deshpande1995rigorous} provided a theoretical analysis of convergence rates for these parallel-in-time methods.
In discrete time, \citet{naumov2017parallel} showed how evaluating Markov chains could be cast as a system of nonlinear equations and discussed many techniques from numerical analysis for solving them, again focusing on Jacobi and Gauss-Seidel; \citet{song2021accelerating} extended this program with many deep learning experiments.
A possible explanation for this preference for Jacobi and Gauss-Seidel iterations is the heavier computational cost of a single Newton iteration, especially on less massively parallelized machines. On the other hand, when Newton iterations were suggested---as by \citet{GanderVandewalle2007} as an interpretation of parareal iterations \citep{Lions2001}---the link to parallelization via the parallel scan was omitted.

Therefore, to the best of my knowledge, the full marriage between parallel scans and Newton iterations had to wait until 2023 and the seminal papers of \citet{deeppcr} and \citet{deer2024}---even though all the necessary ingredients had existed in the literature for at least thirty years. A likely factor in the delay was the fracturing of knowledge and motivations across different communities in dynamics, parallel computation, numerical analysis, and machine learning.
\Cref{ch:background} brings together the necessary background from all of these disciplines to close this gap and facilitate communication between these different communities.

Undoubtedly the hardware and software ecosystems also played a role in the eventual development of parallel Newton's method. In software, the standardization of autodifferentiation \citep{maclaurin_thesis, paszke2019pytorch, jax2018github} made the computation of Jacobians less burdensome. In hardware, the development of GPUs with thousands of processors and gigabytes of on-device memory made the cost of Newton iterations far less burdensome than they were on the ILLIAC and other earlier parallel machines. The importance of these software and hardware lotteries \citep{hooker2021hardware} in the development of algorithms cannot be overstated. While this thesis provides an introduction to techniques that let us take "inherently sequential" processes and reduce their latency on parallel hardware, we must remain open to the possibility that further developments in hardware, software, and algorithm may lead to yet more radically different approaches in the future.

\section{Outline}

In this introduction, we provided a brief survey of the history of parallel computation and parallel-in-time algorithms. In particular, we discussed how the recent rise of massively parallel processors like GPUs has further stimulated the advancement of parallel algorithms for "inherently sequential" computation. Building on this work, my thesis has developed scalable methods for parallel-in-time computation and a firm theoretical understanding of under what conditions such parallel-in-time computation makes sense. The rest of this thesis is organized as follows:
\begin{itemize}
    \item \Cref{ch:background} provides fundamental background for understanding parallel Newton methods, tying together dynamics, parallel computing, and numerical analysis.
    \item \Cref{ch:scalable} introduces our first method, a quasi-Newton method for scalable parallelization.
    \item \Cref{ch:elk} introduces our second method, a trust-region method for stable parallelization.
    \item \Cref{ch:predictability} establishes our theoretical analysis of convergence rates for the Gauss-Newton optimization method for parallelizing SSMs.
    \item \Cref{ch:quasi_convergence} studies the convergence rates of a wide class of quasi-Newton methods for parallelizing SSMs.
    \item \Cref{ch:conclusion} concludes by summarizing the contributions of this thesis and discussing promising directions for future work.
\end{itemize}

The following publications form the basis of this dissertation:

\vspace{1em}

\noindent\textbf{Chapters 3 and 4} are based on:
\begin{quote}
\textbf{Xavier Gonzalez}, Andrew Warrington, Jimmy T.H. Smith, and Scott W. Linderman. "Towards Scalable and Stable Parallelization of Nonlinear RNNs." In \emph{Advances in Neural Information Processing Systems (NeurIPS)}, 2024.
\end{quote}

\noindent\textbf{Chapter~5} is based on:
\begin{quote}
\textbf{Xavier Gonzalez*}, Leo Kozachkov*, David M. Zoltowski, Kenneth L. Clarkson, and Scott W. Linderman. "Predictability Enables Parallelization of Nonlinear State Space Models." In \emph{Advances in Neural Information Processing Systems (NeurIPS)}, 2025.
\end{quote}

\newpage

\noindent\textbf{Chapter~6} is based on:
\begin{quote}
\textbf{Xavier Gonzalez*}, E. Kelly Buchanan*, Hyun Dong Lee, Jerry Weihong Liu, Ke Alexander Wang, David M. Zoltowski, Leo Kozachkov, Christopher R\'{e}, and Scott W. Linderman. "A Unifying Framework for Parallelizing Sequential Models with Linear Dynamical Systems." In \emph{Transactions on Machine Learning Research (TMLR)}, 2026.
\end{quote}

\noindent Throughout this thesis, we also include material from
\begin{quote}
David M. Zoltowski*, Skyler Wu*, \textbf{Xavier Gonzalez}, Leo Kozachkov, and Scott W. Linderman. "Parallelizing MCMC Across the Sequence Length." In \emph{Advances in Neural Information Processing Systems (NeurIPS)}, 2025.
\end{quote}
This last paper extends and develops quasi-Newton methods to parallelize Markov chain Monte Carlo across the sequence length.

\chapter{Background}\label{ch:background}

This thesis uses techniques from applied math to parallelize a class of sequential processes known as \emph{state space models (SSMs)}.
Therefore, in this chapter, we provide background on the three diverse foundational areas---dynamics, parallel computing, and optimization---so that we can bring them together as \emph{parallel Newton methods}. 

\section{Dynamics: State Space Models}\label{sec:ssm}

The first topic this thesis brings into play is dynamics. In particular, we study a class of sequential processes called state space models \cite{murphy2023ssm}. In this background section, we define state space models, survey their broad use in statistics and machine learning, and discuss why their evaluation was deemed to be ``inherently sequential.''

\subsection{State Space Models (SSMs)}

A state space model is a discrete-time dynamical system with a fixed state size. We denote this state by $s_t \in \mathbb{R}^D$, where the subscript $t$ denotes the time of the state, and the dimension $D$ denotes the state size.

The state evolves according to a \emph{dynamics} or \emph{transition} function as
\begin{equation}\label{eq:ssm}
    s_{t+1} = f_t(s_t).
\end{equation}
Importantly, state space models satisfy the \emph{Markov property}, where the state at time $t+1$ depends only on the state at time $t$, and not any of the previous states. Informally, the Markov property means that once we know the present, we can forget the past.

\textbf{Our primary consideration in this thesis is how to evaluate (equivalently ``simulate'', ``unroll'' or ``roll-out'') an SSM} from an initial condition $s_0$. We make this goal precise in the following problem statement:

\paragraph{Problem Statement (Unrolling an SSM):} Evaluate the sequence $\mathbf{s}_{1:T}=(s_1, s_2, \ldots, s_T)$ starting from $s_0$, where $s_t$ follows the SSM dynamics in \cref{eq:ssm}.
\begin{figure}
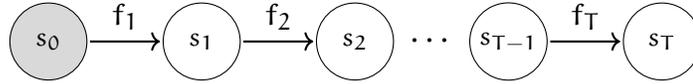

    \centering
    \StateSpaceModelDiagram[1]
    \caption{Unrolling an SSM. We shade the initial state $s_0$ to indicate that we know the initial condition.}
    \label{fig:ssm}
\end{figure}
\Cref{fig:ssm} indicates graphically how when we unroll the dynamics of an SSM from known initial condition $s_0$, we obtain a computational graph that is a chain of sequential dependencies. Throughout this thesis, we will use $T$ to denote the sequence length.

Often, state space models also take an \emph{input} $u_t \in \mathbb{R}^D$ at each time step. Thus, the dynamics become
\begin{equation}\label{eq:ssm_driven}
    s_{t+1} = f(s_t, u_t).
\end{equation}
However, as illustrated in \Cref{fig:ssm_curry}, we can always \emph{curry} the input into the dynamics function to obtain an equivalent SSM without inputs. Specifically, we define the curried dynamics functions as
\begin{equation*}
    f_{\textcolor{blue}{t}}(s_t) := f(s_t, \textcolor{blue}{u_t}).
\end{equation*}
Thus, we can rewrite the SSM with inputs as an SSM without inputs, as in \cref{eq:ssm}.
While almost all the SSMs we consider in this thesis take inputs, we will often omit them from the notation for simplicity, relying on the fact that we can always curry them into the dynamics functions.

\begin{figure}
    \centering
    \begin{tikzpicture}[
    node distance=1cm and 1.5cm, 
    observed/.style={
        circle, draw, fill=gray!30,
        minimum size=1cm,
        text width=1cm,
        align=center,
        inner sep=0pt,
        font=\footnotesize
    },
    latent/.style={
        circle, draw,
        minimum size=1cm,
        text width=1cm,
        align=center,
        inner sep=0pt,
        font=\footnotesize
    },
    input/.style={ 
        circle, draw, fill=gray!50, 
        minimum size=1cm,
        text width=1cm,
        align=center,
        inner sep=0pt,
        font=\footnotesize
    },
    arrow/.style={->, thick},
    shorten >=1pt,
    shorten <=1pt
  ]

  % First Layer Nodes (Latent Variables s)
  \node[observed] (s0_top) {$s_0$};
  \node[latent, right=of s0_top] (s1_top) {$s_1$};
  \node[latent, right=of s1_top] (s2_top) {$s_2$};
  \node[latent, right=1.5cm of s2_top] (sT1_top) {$s_{T-1}$};
  \node[latent, right=of sT1_top] (sT_top) {$s_T$};

  % Second Layer Nodes (Inputs u)
  \node[input, below=of s1_top] (u1) {$\textcolor{blue}{u_1}$};
  \node[input, below=of s2_top] (u2) {$\textcolor{blue}{u_2}$};
  \node[input, below=of sT1_top] (uT1) {$\textcolor{blue}{u_{T-1}}$};
  \node[input, below=of sT_top] (uT) {$\textcolor{blue}{u_T}$};

  % First Layer Arrows (s_0 -> s_1 -> s_2 -> ... -> s_T)
  \draw[arrow] (s0_top) -- (s1_top) node[midway, above] {$f$};
  \draw[arrow] (s1_top) -- (s2_top) node[midway, above] {$f$};
  \draw[arrow] (sT1_top) -- (sT_top) node[midway, above] {$f$};

  % Arrows from Inputs to Latent Variables
  \draw[arrow] (u1) -- (s1_top);
  \draw[arrow] (u2) -- (s2_top);
  \draw[arrow] (uT1) -- (sT1_top);
  \draw[arrow] (uT) -- (sT_top);

  % Ellipsis in the Middle
  \node at ($(s2_top)!0.5!(sT1_top)$) {\large $\cdots$};

\end{tikzpicture}

% Second TikZ picture: Double Arrow
\centering
\begin{tikzpicture}
    \node at (0,0) {\huge$\updownarrow$}; % Double-headed arrow
\end{tikzpicture}

% Third TikZ picture: Bottom diagram
\begin{tikzpicture}[
    node distance=1cm and 1.5cm,
    observed/.style={
        circle, draw, fill=gray!30,
        minimum size=1cm,
        text width=1cm,
        align=center,
        inner sep=0pt,
        font=\footnotesize
    },
    latent/.style={
        circle, draw,
        minimum size=1cm,
        text width=1cm,
        align=center,
        inner sep=0pt,
        font=\footnotesize
    },
    input/.style={ 
        circle, draw, fill=gray!50, 
        minimum size=1cm,
        text width=1cm,
        align=center,
        inner sep=0pt,
        font=\footnotesize
    },
    arrow/.style={->, thick},
    shorten >=1pt,
    shorten <=1pt
  ]

  % Bottom Layer Nodes (Corresponding Latent Variables s)
  \node[observed] (s0_bottom) {$s_0$};
  \node[latent, right=of s0_bottom] (s1_bottom) {$s_1$};
  \node[latent, right=of s1_bottom] (s2_bottom) {$s_2$};
  \node[latent, right=1.5cm of s2_bottom] (sT1_bottom) {$s_{T-1}$};
  \node[latent, right=of sT1_bottom] (sT_bottom) {$s_T$};

    % Bottom Layer Arrows (Same as top)
    \draw[arrow] (s0_bottom) -- (s1_bottom) node[midway, above] {$f_{\textcolor{blue}{1}}$};
    \draw[arrow] (s1_bottom) -- (s2_bottom) node[midway, above] {$f_{\textcolor{blue}{2}}$};
    \draw[arrow] (sT1_bottom) -- (sT_bottom) node[midway, above] {$f_{\textcolor{blue}{T}}$};

  % Ellipsis in the Bottom Layer
  \node at ($(s2_bottom)!0.5!(sT1_bottom)$) {\large $\cdots$};

\end{tikzpicture}
    \caption{Graphical diagram showing the equivalence (based on currying) between an SSM driven by inputs and an autonomous system with time-varying transition dynamics. We shade the inputs $u_t$ to indicate that they are known.}
    \label{fig:ssm_curry}
\end{figure}
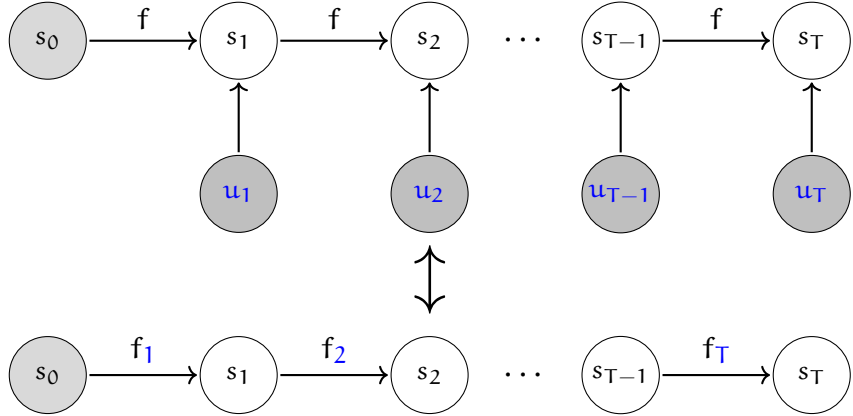

\subsection{Examples of SSMs}

The framework given in \cref{eq:ssm} is extremely general, and many well-known models in statistics and machine learning can be expressed as SSMs. 
We summarize some important examples in \Cref{tab:ssm_examples}.
\begin{table}[H]
\centering
\begin{tabularx}{\textwidth}{X|XXX} % Sets total width to the text margin
\toprule
\textbf{SSM} & \textbf{State ($s_t$)} & \textbf{Input ($u_t$)} & \textbf{Transition ($f$)} \\ 
\midrule
Linear Dynamical Systems (LDS) \citep{luenberger1979introduction} & State & Input & Linear \\ \hline 
Deep SSMs \citep{gu_thesis, goel_thesis, smith_thesis} & Stack of states & Input & Linear \\ \hline 
Recurrent Neural Networks (RNNs) \citep{jordan1986serial, elman1990finding, lstm, sutskever2013training, gru} & Hidden state & Input & RNN cell \\ \hline 
MCMC \citep{Diaconis2009MCMC, geyer2011introduction, pmcmc}  & Current sample & Noise & Transition kernel \\ \hline
Sampling from diffusion models \citep{song_dissertation, shih2023parallel, lai2025principles} & Noisy image & Noise & Denoising function \\ \hline
Explicit differential equation solvers \citep{gander201550, kidger_dissertation, iacob2025parallel} & Current state & N/A & Numerical integrator \\ \hline
"Recurrent depth" for transformers \citep{dehghani2018universal, schone2025implicit, geiping2025scaling, wang2025hierarchical,jolicoeur2025less} & Layer activations & Original input  & Transformer block \\ \hline
State of reinforcement learning (RL) agent \citep{sutton2018reinforcement, psenka2026parallel} & Environment state & Noise & Environment dynamics \\ \hline
Gradient descent \citep{boyd2004convex} & Parameter values & N/A & Gradient step \\ \hline
The human brain \citep{vyas2020computation} & neural activity & sensory input & synapses \\
\bottomrule
\end{tabularx}
\caption{Some illustrative examples of state space models (SSMs).}
\label{tab:ssm_examples}
\end{table}
The first example in this \Cref{tab:ssm_examples} is the \emph{linear dynamical system (LDS)}, which has linear transition dynamics, that is, the state evolves as
\begin{equation}\label{eq:lds}
    s_{t+1} = A_t s_t + B_t u_t.
\end{equation}
Thus, an LDS is a special case of an SSM (\cref{eq:ssm}), but where the dynamics function $f_t$ is affine.
These linear dynamical systems have enjoyed a resurgence in machine learning recently as linear RNNs \citep{martin2018parallelizing,orvieto-resurrecting}  or deep state space models \citep{gu2022s4,smith2023s5,mamba}. 
In these deep learning architectures, the temporal dynamics of each layer are linear, but the output of each layer is passed through a nonlinearity to become the input of the next layer.

\subsubsection{Bayesian inference for linear Gaussian SSMs: Kalman filtering and smoothing}\label{ssc:kalman}

We take a brief aside to discuss Bayesian inference in state space models, as the core primitives of Kalman filtering and smoothing are fundamental to our stable parallelization techniques developed in \Cref{ch:elk}.

We begin by noting that we can include many probabilistic models into the SSM framework by incorporating stochastic inputs into our SSM dynamics equation \cref{eq:ssm_driven}.
A fundamental probabilistic model is the \emph{linear Gaussian state space model} (LGSSM), where the latent variables $s_t$ follow linear dynamics with Gaussian noise, and emit observations $o_t$ with linear readouts with Gaussian noise \citep{murphy2023probabilistic, sarkka2023bayesian}. See \Cref{fig:lgssm}. In particular, note that the LGSSM is a simple way to make an LDS a probabilistic object: the latent variables $s_t$ are modeled as an LDS.

\begin{figure}[t]
    \centering
    \begin{tikzpicture}
        % Latent nodes (s)
        \node[circle, draw=black] (s0) at (0,0) {$s_0$};
        \node[circle, draw=black] (s1) at (2,0) {$s_1$};
        \node[circle, draw=black] (s2) at (4,0) {$s_2$};
        \node[circle, draw=black] (s3) at (6,0) {$s_3$};
        
        % Observed nodes (o) - moved to positive y-coordinate
        \node[circle, draw=black, fill=gray!20] (o1) at (2,1.5) {$o_1$};
        \node[circle, draw=black, fill=gray!20] (o2) at (4,1.5) {$o_2$};
        \node[circle, draw=black, fill=gray!20] (o3) at (6,1.5) {$o_3$};
        
        % Edges for latent chain
        \draw[-to] (s0) -- (s1);
        \draw[-to] (s1) -- (s2);
        \draw[-to] (s2) -- (s3);
        
        % Edges for emissions
        \draw[-to] (s1) -- (o1);
        \draw[-to] (s2) -- (o2);
        \draw[-to] (s3) -- (o3);
    \end{tikzpicture}

    \caption{\textbf{A linear Gaussian state space model (LGSSM):} The LGSSM consists of latent variables $s_t$ and observed variables $o_t$. The generative model of the LGSSM consists of dynamics $s_{t+1} \sim \mathcal{N}\left( A s_t, Q \right)$ and emissions $o_{t+1} \sim \mathcal{N}\left( C s_{t+1}, R \right)$.}
    \label{fig:lgssm}
\end{figure}
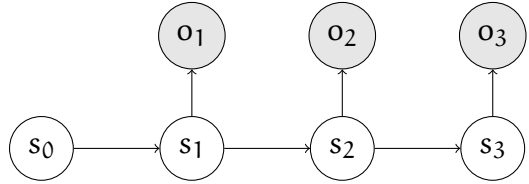

Two canonical inferential targets in the LGSSM are the filtering distributions, $p(s_t \mid o_{1:t})$, and the smoothing distributions, $p(s_t \mid o_{1:T})$.
    The Kalman filter \citep{kalman-filter} and Rauch-Tung-Striebel (RTS) smoother\footnote{Occasionally we will call the RTS smoother a ``Kalman'' smoother for simplicity.} \citep{rauch1965maximum} obtain the filtering and smoothing distributions (respectively) in an LGSSM. The Kalman filter makes a single pass forward in time to get the filtering distributions, while the RTS smoother then makes an additional pass backwards in time to get the smoothing distributions.
Thus, these canonical algorithms for Bayesian inference in LGSSMs would also at first glance seem to be inherently sequential.

\subsection{Limitation of SSMs: "Inherently Sequential"}

Indeed, despite the breadth of use of SSMs across statistics and machine learning, it was widely believed that SSMs were ``inherently sequential'' to evaluate \citep{vaswani2017attention, linear_attention, smith_thesis}.
With longer sequences, this sequential evaluation becomes a computational bottleneck, especially on modern hardware like GPUs and TPUs that thrive on parallelism.
As a result, in keeping with the ``hardware lottery'' \citep{hooker2021hardware}, many researchers began to turn away from SSMs in favor of more parallelizable approaches.

However, it turns out that there is a simple but effective way to parallelize our first and simplest example of SSMs in \Cref{tab:ssm_examples}: \emph{linear} dynamical systems.
Therefore, in our next background section on parallel computing, we review the \emph{parallel associative scan} that allows us to parallelize linear dynamical systems. 
Ultimately, as we will see in \Cref{sec:deer}, a clever use of the parallel scan allows us to parallelize SSMs \emph{in general}, despite their ``inherently sequential'' nature.

\section{Parallel Computing: The Parallel Associative Scan}\label{sec:pscan}

The \emph{parallel scan}~\citep{Stone1973, blelloch1990prefix}, also known as the \textit{associative} scan and, colloquially, \textit{pscan}, is a well-known primitive in the parallel computing literature \citep{hillis1986data, ladner1980parallel, lakshmivarahan1994parallel}. The core idea of the parallel scan is a divide-and-conquer algorithm. We illustrate this point in the simple example of multiplying a series of matrices together.

\subsection{The Parallel Scan: A Gentle Introduction} 

\paragraph{Simple example: multiplying a sequence of matrices}
Consider the following problem: given a series of square matrices $A_1, A_2, \hdots, A_{T-1}, A_T$, compute their product\footnote{Note that we have the matrices act via left-multiplication over the sequence length, because this is the most common way to write matrix-vector products.}, $A_T A_{T-1} \hdots A_2 A_1$. The simplest way to carry out the matrix multiplication is sequentially: first compute $A_1$, then compute $A_2 A_1$, then compute $A_3 A_2 A_1$, and so on. Such an approach takes $\mathcal{O}(T)$ time.

A core insight of the parallel scan is that matrix multiplication is \emph{closed}; that is, if $A_s \in \mathbb{R}^{D \times D}$ and $A_t \in \mathbb{R}^{D \times D}$, then $A_t A_s \in \mathbb{R}^{D \times D}$. Thus, matrix products can be computed recursively in pairs, as illustrated in \Cref{fig:divide_and_conquer}.

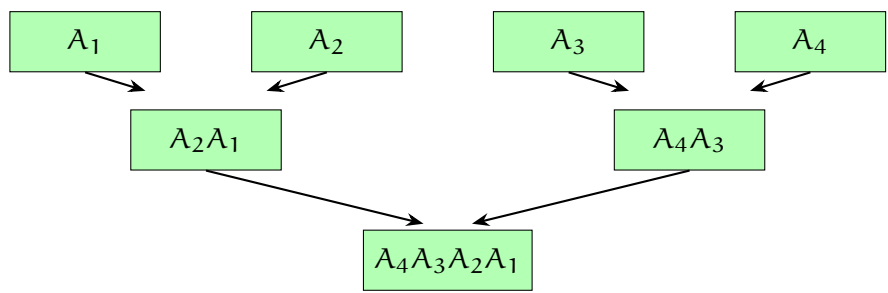
\begin{figure}[ht]
\centering
\begin{tikzpicture}[every node/.style={font=\small}, node distance=1.2cm and 1.2cm]

% Define styles (tikzset is the modern form of \tikzstyle)
\tikzset{
  block/.style = {rectangle, draw, minimum width=2cm, minimum height=0.8cm, fill=green!30},
  arrow/.style = {thick, ->, >=Stealth} % capital S works with arrows.meta
}

% First row
\node[block] (block1) {$A_1$};
\node[block, right=1.2cm of block1] (block2) {$A_2$};
\node[block, right=1.2cm of block2] (block3) {$A_3$};
\node[block, right=1.2cm of block3] (block4) {$A_4$};

% Midpoints (name the calc results so 'of' can use them)
\coordinate (mid12) at ($(block1)!0.5!(block2)$);
\coordinate (mid34) at ($(block3)!0.5!(block4)$);
\coordinate (mid56) at ($(mid12)!0.5!(mid34)$);

% Second row
\node[block, below=0.9cm of mid12] (block5) {$A_2 A_1$};
\node[block, below=0.9cm of mid34] (block6) {$A_4 A_3$};

% Third row
\node[block, below=2.5cm of mid56] (block7) {$A_4 A_3 A_2 A_1$};

% Arrows first->second
\draw[arrow] (block1.south) -- ($(block1.south)!0.5!(block5.north)$);
\draw[arrow] (block2.south) -- ($(block2.south)!0.5!(block5.north)$);
\draw[arrow] (block3.south) -- ($(block3.south)!0.5!(block6.north)$);
\draw[arrow] (block4.south) -- ($(block4.south)!0.5!(block6.north)$);

% Arrows second->third
\draw[arrow] (block5.south) -- ($(block5.south)!0.9!(block7.north)$);
\draw[arrow] (block6.south) -- ($(block6.south)!0.9!(block7.north)$);

\end{tikzpicture}
\caption{\textbf{Parallel Scan for Matrix Multiplication.} We illustrate a divide-and-conquer approach to compute the product $A_4 A_3 A_2 A_1$. Note that this divide-and-conquer approach naturally leads to $\mathcal{O}(\log T)$ depth. }\label{fig:divide_and_conquer}
\end{figure}

Because of the divide-and-conquer (binary-tree-like) nature of this approach to multiplying matrices, with $\mathcal{O}(T)$ processors, the time needed to get the matrix product is only $\mathcal{O}(\log T)$. 
This simple example illustrates the core intuition behind the parallel scan: a closed operation leading to a divide-and-conquer approach that parallelizes a computation so that it takes sublinear time.
However, there are two additional details of the parallel associative scan that we should address: arbitrary binary associative operators and closure; and getting intermediate products.

\paragraph{Detail \#1: Parallel scans for arbitrary binary associative operators}

Matrix multiplication is an associative operator, as $A_3 (A_2 A_1) = (A_3 A_2) A_1$. In general, consider a binary associative operator $\otimes$, which would satisfy $q_3 \otimes (q_2 \otimes q_1) = (q_3 \otimes q_2) \otimes q_1$. Now, let us further assume that this binary associative operator is closed:
\begin{definition}[Closure]\label{def:closure}
	A binary associative operator $\otimes$ is closed over a set $\mathcal{S}$ if it satisfies the property:
	\begin{equation}\label{eq:closure_def}
		q_1 \in \mathcal{S}, q_2 \in \mathcal{S} \Rightarrow q_2 \otimes q_1 \in \mathcal{S}.
	\end{equation}
\end{definition}
If $\otimes$ is closed, then we can again use a parallel scan to compute the cumulative product of the operands.

A wide range of binary associative operators are closed, and can thus be parallelized with the parallel scan. We have already seen that matrix multiplication is such a binary associative operator. An even simpler example of a binary associative operator amenable to the parallel scan is \textbf{scalar addition}. The fact that addition of scalars (and vectors) is closed allows cumulative sums to be computed with the parallel scan algorithm. When the binary associative operator is addition, it is also known as the \textit{prefix sum} algorithm. Clearly, addition is associative and closed, and so summing a series of scalars can be done with a divide-and-conquer approach.

\paragraph{Detail \#2: Obtaining the intermediate terms in the product}

The parallel scan is meant to be a parallelized implementation of the \textsc{Scan} primitive from functional programming \citep{bird1998introduction}.
However, \textsc{Scan} not only returns the final product  $A_T A_{T_1} \hdots A_1$, as we illustrated in \Cref{fig:divide_and_conquer}, but also all the intermediate terms $A_1$, $A_2 A_1$, $A_3 A_2 A_1$, etc.

In fact, the parallel scan provides all the intermediate terms as well. We again illustrate in our motivating example of matrix multiplication, in particular the setting where $T=8$. We will denote the individual matrices as $A_1, A_2, A_3, \hdots A_8$, and their products as $A_{s:t}$, i.e. $A_{5:6} = A_6 A_5$. 

The first phase of the parallel scan is the \emph{up-sweep}, and takes $\log(T)$ iterations and $\mathcal{O}(T)$ memory.
Crucially, note that we are using $\mathcal{O}(T)$ processors in parallel as well.
We start multiplying adjacent pairs of matrices together. Looking, for example, at Position 8 of \Cref{tab:upsweep}, we go from $A_8$ to $A_{7:8}$ to $A_{5:8}$ to $A_{1:8}$. 

Then, in the \emph{down-sweep}, we fill in the missing products to obtain all the cumulative products $A_{1:t}$ for $1 \leq t \leq T$. Intuitively, the down-sweep also takes $\mathcal{O}(\log T)$ iterations, for the same reason that any natural number $T$ can be represented using $1 + \log_2(T)$ digits in binary.  

\begin{table}[ht]
\centering
\begin{tabularx}{\textwidth}{|Y|Y|Y|Y|Y|Y|Y|Y|}
\toprule
\text{Pos. 1} & \text{Pos. 2} & \text{Pos. 3} & \text{Pos. 4} &
\text{Pos. 5} & \text{Pos. 6} & \text{Pos. 7} & \text{Pos. 8} \\
\midrule
{ \color{gray} A_1} & { \color{gray} A_2 }  & { \color{gray} A_3 } & { \color{gray} A_4 } & { \color{gray} A_5} & { \color{gray} A_6 }  & { \color{gray} A_7 } & { \color{gray} A_8 } \\
A_1 & { \color{blue} A_{1:2}} & A_3 & { \color{blue} A_{3:4}} & A_5 & { \color{blue} A_{5:6} } & A_7 & { \color{blue} A_{7:8} } \\
A_1 & A_{1:2} & A_3 & { \color{blue} A_{1:4} } & A_5 & A_{5:6} & A_7 & { \color{blue} A_{5:8} }  \\
A_1 & A_{1:2} & A_3 & A_{1:4} & A_5 & A_{5:6} & A_7 & { \color{blue} A_{1:8} }  \\
\bottomrule
\end{tabularx}
\caption{\textbf{Up-sweep} for multiplying $A_1, A_2, \hdots A_8$.}
\label{tab:upsweep}
\end{table}

\begin{table}[ht]
\centering
\begin{tabularx}{\textwidth}{|Y|Y|Y|Y|Y|Y|Y|Y|}
\toprule
\text{Pos. 1} & \text{Pos. 2} & \text{Pos. 3} & \text{Pos. 4} &
\text{Pos. 5} & \text{Pos. 6} & \text{Pos. 7} & \text{Pos. 8}  \\
\midrule
A_1 & A_{1:2} & A_3 & A_{1:4} & A_5 & { \color{blue} A_{1:6} } & A_7 &  A_{1:8}  \\
A_1 & A_{1:2} & { \color{blue} A_{1:3} }  & A_{1:4} & { \color{blue} A_{1:5} } & A_{1:6}  & { \color{blue} A_{1:7} } &  A_{1:8}  \\
\bottomrule
\end{tabularx}
\caption{\textbf{Down-sweep} for multiplying $A_1, A_2, \hdots A_8$.}
\label{tab:downsweep}
\end{table}

Thus, together, the up-sweep and the down-sweep of the parallel scan run in $\mathcal{O}(\log T)$ time on $\mathcal{O}(T)$ processors, and at the end of this algorithm, we get all of the intermediate products\footnote{See the last row of \Cref{tab:downsweep}.} (the ``prefix sums''). 

\subsection{Parallelizing Linear Dynamical Systems}\label{ssc:pscan_lds}

Having digested the fundamentals of the parallel scan, it becomes apparent that composition of affine functions is also a binary associative operator that is closed. Therefore, it is possible to parallelize over the sequence length the roll-out of an LDS evolving according to \cref{eq:lds}.

In more detail, consider the affine function $f_i(x) = A_i x + b_i$. Notice that the composition of affine functions is also affine, as $f_{j} (f_i(x)) = A_j A_i x + \left( b_j + A_j b_i  \right)$. Thus, if we represent the operands as ordered pairs $(A_i, b_i)$ and $(A_j, b_j)$, we can write the associative operator $\otimes$ for the composition of affine functions as
    \begin{equation}\label{eq:bin_op_affine}
    	(A_i, b_i) \otimes (A_j, b_j) = (A_j A_i, b_j + A_j b_i).
    \end{equation}
    Thus, we observe that in this setting, $\otimes$ is closed. We also should check that $\otimes$ is associative: we can do so with either elementary algebra, or by observing that function composition is associative.

This observation that composition of affine functions can be parallelized with the associative scan is what lets us parallelize LDSs. The insight that LDSs could be parallelized with parallel scans led to a revolution in the deep sequence modeling community based on transformer alternative architectures. Because LDSs could be parallelized, this led to the development of both linear RNNs \citep{martin2018parallelizing, orvieto-resurrecting} and deep SSMs \citep{smith2023s5, mamba}. These approaches boil down to sequence mixing layers that are LDSs (and therefore parallelizable with parallel scan), stacked nonlinearly in depth. As we will see throughout this thesis, decomposing nonlinear SSM dynamics into LDSs that can be parallelized with the parallel scan is our fundamental tool in parallelizing arbitrary SSMs.

\subsection{Parallelizing Kalman Filtering and Smoothing}

In the previous section, we showed how we can parallelize the evaluation of a linear dynamical system. In this section, we discuss how we can also parallelize Bayesian inference---Kalman filtering and smoothing---in probabilistic models based on LDSs, namely linear Gaussian SSMs (which we reviewed in \Cref{ssc:kalman}).

Both the Kalman filter and RTS smoother would seem to be inherently sequential algorithms, requiring $\mathcal{O}(T)$ time. However, \citet{parallel-kalman} demonstrated that the Kalman filter and RTS smoother can also be parallelized over the sequence length via the construction of custom binary associative operators and a parallel scan. While we leave the details of this construction to \citet{parallel-kalman}, we note that it is intuitively plausible to be able to parallelize filtering and smoothing in an LGSSM with a parallel scan because
    \begin{itemize}
        \item the dynamical backbone is an LDS, for which we have a parallel scan (cf. \cref{eq:bin_op_affine});
        \item since everything is linear and Gaussian, all distributions remain Gaussian, hinting at closure; and
        \item we can combine $p(s_{t'} | s_0, o_{1:t'})$ with $p(s_t | s_{t'}, o_{t'+1: t})$ to obtain $p(s_t | s_0, o_{1:t})$, suggesting a divide-and-conquer strategy.
    \end{itemize}
    These parallel filtering and smoothing algorithms are useful in machine learning, allowing for parallelization of structured variational autoencoders \citep{johnson2016composing, zhao2023revisiting}. Similar approaches also work for Hidden Markov Models \citep{hassan2021temporal} and for computing log-normalizing constants \citep{hu2025sing}. 

\subsection{The Difficulties of Parallelizing an SSM in general}\label{ssc:nonlin_pscan}

The astute reader might note that the composition of functions, i.e. $f_1 \circ f_2$, is \emph{always} a binary associative operator. So, why do we have all these special cases of parallel scans, and not simply one parallel scan for the composition $\circ$ of arbitrary functions $f_i$? 

The reason to have many different parallel scans is precisely the importance of having the binary associative operator be \emph{closed}. In all the previous examples, the binary associative operator $\otimes$ satisfies \Cref{def:closure}, letting us easily store combinations of operands $q_i \otimes q_j$ and so employ a divide-and-conquer technique.

While we could consider some gigantic function space $\mathcal{F}$, for which function composition would be closed, the practical question then becomes: how would we store the combinations of operands? If we do not have some compact representation for elements of $\mathcal{F}$, then we cannot use a parallel scan in practice, even though the parallel scan may seem applicable in theory.

Nonetheless, we still have the parallel scan for parallelizing LDSs. When one has a hammer (the parallel scan for LDSs), everything begins to look like a nail. Thus, one might attempt the seemingly hacky approach of taking nonlinear dynamical systems, and iteratively
\begin{itemize}
    \item linearizing the system
    \item evaluating the linearized system in parallel with the parallel scan.
\end{itemize}
Incredibly, this approach, which motivates this thesis, is not a hack but is rather an instantiation of Newton's method! Therefore, in the next section, we review Newton's method in optimization and numerical analysis generally.

\section{Numerical Analysis: Newton's method}\label{sec:math_background}

Newton's method is one of the most fundamental approaches in root-finding, optimization, and numerical analysis generally \citep{ortega1970iterative, NocedalWright, boyd2004convex, hubbard2015vector}. In this background section, we review the fundamentals of Newton's method and other related techniques in root-finding, optimization, and fixed-point methods. 

\subsection{Root-finding}\label{ssc:root_finding}

Consider a high-dimensional nonlinear function $\mathbf{r}(\mathbf{s}): \mathbb{R}^P \mapsto \mathbb{R}^P$.
A standard problem in numerical analysis is to find the root of such a function, i.e. find those $\mathbf{s}^\star$ for which $\mathbf{r}(\mathbf{s}^\star) = \mathbf{0}$. In high dimensions and for a complicated function, it is not immediately obvious how one might find such a zero in an efficient manner.

However, if our function is affine, i.e. $\mathbf{r}(\mathbf{s}) = \mathbf{M} \mathbf{s} + \mathbf{b}$, then provided $\mathbf{M}$ is invertible there is at least a straightforward way to find the root $\mathbf{s}^\star$, as $\mathbf{s}^\star = - \mathbf{M}^{-1} \mathbf{b}$.

Newton's method for root-finding for differentiable functions $\mathbf{r}$ is based on the idea of iteratively
\begin{itemize}
    \item linearizing $\mathbf{r}$ around our current guess $\mathbf{s}^{(i)}$ to form the affine function $\mathbf{\hat{r}}^{(i)}$; and then
    \item finding the root of $\mathbf{\hat{r}}^{(i)}$, and making it our new guess $\mathbf{s}^{(i+1)}$.
\end{itemize}
We show a graphical depiction of Newton's method for a one-dimensional function $\mathbf{r}(\cdot): \mathbb{R} \mapsto \mathbb{R}$ in \Cref{fig:pchord}.
\begin{figure}
    \centering
    \includegraphics[width=0.9\linewidth]{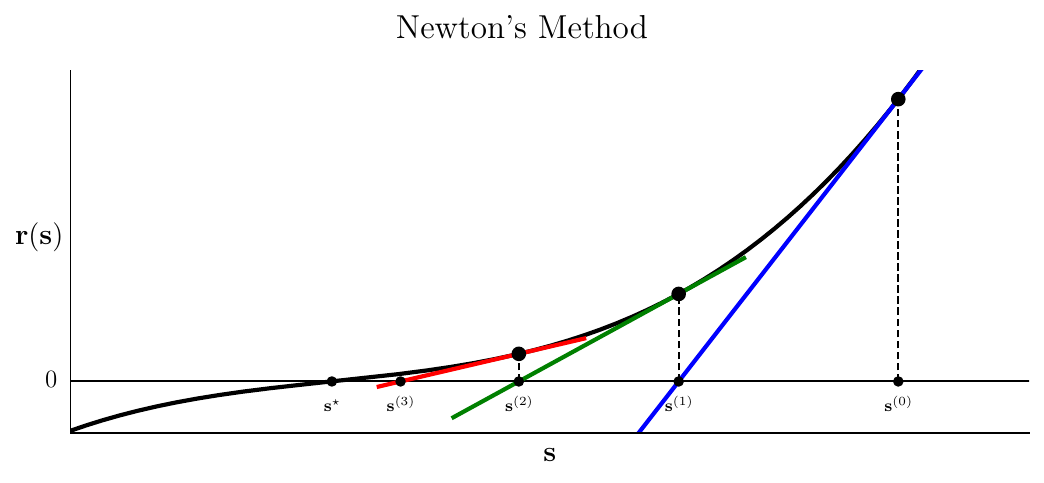}
    \caption{\textbf{Newton's method for root-finding.} Here we illustrate 3 iterations of Newton's method for root-finding on the one-dimensional cubic function $\mathbf{r}(\mathbf{s}) = (\mathbf{s} - 0.4)^3 + 0.45 (\mathbf{s} - 0.4)$. We observe that each iteration of Newton's method involves linearizing the function to obtain $\mathbf{\hat{r}}^{(i)}(\cdot)$ (shown in color) and then finding the zero of this linearization to obtain our next guess.}
    \label{fig:pchord}
\end{figure}
Let us define the notational shorthand
\begin{equation*}
    \mathbf{J}^{(i)} := \dfrac{\partial \mathbf{r}}{\partial \mathbf{s}}(\mathbf{s}^{(i)}),
\end{equation*}
where we choose $\mathbf{J}$ to stand for the Jacobian matrix (i.e. derivative) of $\mathbf{r}$.  With this notation, we see that the first step is given by a first-order Taylor expansion of $\mathbf{r}$ around $\mathbf{s}^{(i)}$, i.e.
\begin{equation*}
    \mathbf{\hat{r}}^{(i)}(\mathbf{s}) := \mathbf{r}(\mathbf{s}^{(i)}) + \mathbf{J}^{(i)} \left( \mathbf{s} - \mathbf{s}^{(i)} \right).
\end{equation*}
Thus, we see that every step of Newton's method for root-finding---where we aim to find the zero of $\mathbf{\hat{r}}^{(i)}(\mathbf{s})$---is given by
\begin{equation}\label{eq:newton_root}
    \mathbf{s}^{(i+1)} = \mathbf{s}^{(i)} - \left( \mathbf{J}^{(i)} \right)^{-1} \mathbf{r}(\mathbf{s}^{(i)}).
\end{equation}
Of course, for \cref{eq:newton_root} to be valid, we must have $\J^{(i)}$ invertible---but for the parallel Newton methods considered in this dissertation, it always will be (see \cref{eq:big_j}).

Another limitation of Newton's method immediately visible from \cref{eq:newton_root} is the need to store and invert $\J^{(i)} \in \mathbb{R}^{P \times P}$.
In particular, the matrix inversion requires $\mathcal{O}(P^3)$ floating point operations (FLOPs).
While the implementation of the numerical linear algebra can be optimized \citep{golub2013matrix}, the overall cost of Newton's method has inspired a broad literature on cheaper, approximate \emph{quasi-Newton} methods \citep{NocedalWright, nocedal1980updating, liu1989limited}. We build on and contribute to this quasi-Newton literature in \Cref{ch:scalable}. 

While \Cref{fig:pchord} shows intuitively why Newton's method can be a powerful technique for root-finding, let us discuss some of its convergence properties further and more formally.

\paragraph{Convergence of Newton's method}

Newton's method is known to enjoy \emph{quadratic convergence} \underline{within a basin} around the solution $\mathbf{s}^\star$ \citep{hubbard2015vector, ortega1970iterative, NocedalWright}.

One way to define quadratic convergence, following the presentation in \citet{NocedalWright}, is via the notion of \emph{Q-convergence}, which is short for \emph{quotient-convergence}:
\begin{definition}[Q-convergence]
    Consider a sequence of iterates $\{ \mathbf{s}^{(i)} \}$ which is converging to a limit $\mathbf{s}^\star$ as $i \to \infty$. Then this sequence Q-converges to $\mathbf{s}^\star$ with \emph{order $q$} and with \emph{rate of convergence $\gamma$} if, for all $(i)$ sufficiently large, 
    \begin{equation}\label{eq:Qconvergence}
        \| \mathbf{e}^{(i+1)} \|\leq \gamma \| \mathbf{e}^{(i)} \|^q,
    \end{equation}
    where the \emph{errors} are defined by
    \begin{equation}\label{eq:error}
        \mathbf{e}^{(i)} := \mathbf{s}^{(i)} - \mathbf{s}^\star,
    \end{equation}
    and $\| \cdot \|$ is any valid vector norm.
\end{definition}
If the order $q=1$, we say that the iterative method enjoys \emph{linear convergence}, while if $q=2$, it enjoys \emph{quadratic convergence}.
In linear convergence, the error
satisfies\footnote{Here, we slightly abuse notation to make $\mathbf{e}^{(0)}$ the first iteration where the inequality in \eqref{eq:Qconvergence} holds.} $\| \mathbf{e}^{(i)} \| \leq \gamma^i \| \mathbf{e}^{(0)} \|$, indicating that the norm of the error decays exponentially in the number of iterations, with base $\gamma$. We must have $\gamma < 1$ for linear convergence to converge to a limit. 
In quadratic convergence, the error satisfies $\| \mathbf{e}^{(i)} \| \leq \nicefrac{(\gamma \| \mathbf{e}_0 \| )^{2^{i}}}{\gamma}$, indicating the norm of the error decays \emph{doubly-exponentially} with base $\gamma \| \mathbf{e}_0 \|$. Again, however, to actually enjoy decrease with quadratic convergence, we must have $\gamma \| \mathbf{e}_0 \| < 1$, giving rise to the \emph{basin of quadratic convergence} $\mathcal{B}_Q$ given by
\begin{equation*}
    \mathcal{B}_Q := \left\{ \mathbf{s}^{(i)} : \| \mathbf{s}^{(i)} - \mathbf{s}^\star \| < \frac{1}{\gamma} \right\}.
\end{equation*}
With this definition, we now provide a simple proof\footnote{following e.g. the proof of Proposition 4 of \citet{parasolver25}.} that Newton's method enjoys quadratic rate in a basin around the solution $\mathbf{s}^\star$.
\begin{proposition}\label{prop:Newton_quadratic}
    Say we are trying to find a root of $\mathbf{r}(\mathbf{s}): \mathbb{R}^P \mapsto \mathbb{R}^P$ with Newton's method as defined in \cref{eq:newton_root}. If we assume that $\mathbf{J}(\mathbf{s})$ is $L$-Lipschitz and is always invertible with $\| \mathbf{J}(\mathbf{s})^{-1} \| \leq \beta$ for all $\mathbf{s}$, then Newton's method converges quadratically in the basin given by $\left\{ \mathbf{s}^{(i)} : \| \mathbf{s}^{(i)} - \mathbf{s}^\star \| \leq \frac{2}{L \beta} \right\}$.
\end{proposition}
\begin{proof}
    Subtract $\mathbf{s}^\star$ from both sides of \cref{eq:newton_root} to obtain
    \begin{equation*}
        \mathbf{e}^{(i+1)} = \mathbf{e}^{(i)} - \left( \J(\mathbf{s}^{(i)}) \right)^{-1} \mathbf{r}(\mathbf{s}^{(i)}).
    \end{equation*}
    Taylor expanding $\mathbf{r}(\cdot)$ around $\mathbf{s}^{(i)}$, we get the equality
    \begin{equation*}
        \mathbf{r}(\mathbf{s}^\star) = \mathbf{r}(\mathbf{s}^{(i)}) - \J(\mathbf{s}^{(i)}) \mathbf{e}^{(i)} + \mathbf{R}^{(i)},
    \end{equation*}
    where the remainder $\mathbf{R}^{(i)}$ satisfies $\| \mathbf{R} \| \leq \frac{L}{2} \| \mathbf{e}^{(i)} \|^2$. Since $\mathbf{r}(\mathbf{s}^\star) =0$, it follows that
    \begin{equation*}
        \mathbf{e}^{(i+1)} = \left( \J(\mathbf{s}^{(i)}) \right)^{-1} \mathbf{R}^{(i)}.
    \end{equation*}
    Taking norms on both sides and inputting the assumptions, it follows that
    \begin{equation*}
        \| \mathbf{e}^{(i+1)} \| \leq \frac{L \beta}{2} \| \mathbf{e}^{(i)} \|^2,
    \end{equation*}
    i.e. Newton's method enjoys quadratic convergence in the specified basin.
\end{proof}

However, this quadratic convergence of Newton's method only holds \emph{locally}, i.e. for initial guesses $\mathbf{s}^{(0)}$ that are close to the zero $\mathbf{s}^\star$.
It is stronger and more helpful to have guarantees for \emph{global} convergence, i.e. assurances that an iterative solver will converge (and with a specified rate) \emph{no matter} the initial guess $\mathbf{s}^{(0)}$.

Unfortunately, in general, Newton's method does \emph{not} enjoy global convergence guarantees \citep{NocedalWright}. We illustrate with a simple example.

\begin{figure}[ht]
    \centering
    \includegraphics[width=\linewidth]{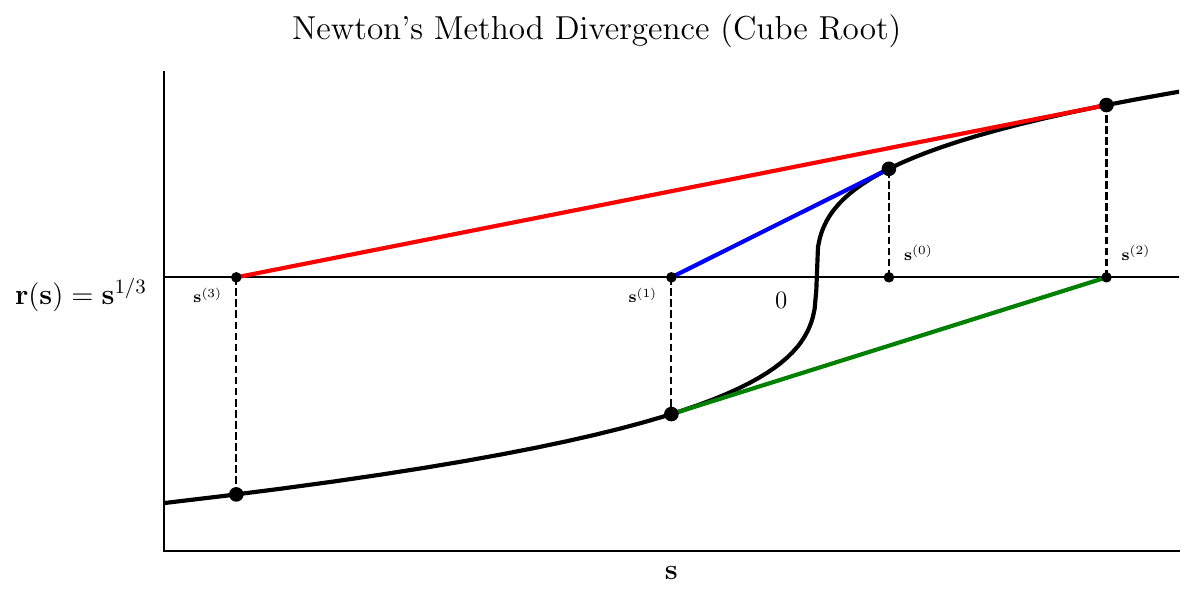}
    \caption{\textbf{Newton's method can globally diverge.} A graphical depiction showing how Newton's method for root-finding can globally \emph{diverge} for a simple function like $\mathbf{r}(\mathbf{s}) = \mathbf{s}^{1/3}$.}
    \label{fig:newton_diverges}
\end{figure}
\begin{example}[Newton's method can diverge: $\mathbf{r}(\mathbf{s}) = \mathbf{s}^{1/3}$.]\label{example:newton_diverge}
    Consider the standard cube root function $\mathbf{r}(\mathbf{s}) = \mathbf{s}^{1/3}$ defined on all of the real line.
    At all points $\mathbf{s} \in \mathbb{R}$, the derivative of this function is given by
    \begin{equation*}
        \mathbf{r}'(\mathbf{s}) = \frac{1}{3} \mathbf{s}^{-2/3}.
    \end{equation*}
    So, plugging into \cref{eq:newton_root}, it follows that, no matter the initial guess, the Newton iterates follow
    \begin{equation*}
        \mathbf{s}^{(i+1)} = -2 \mathbf{s}^{(i)}.
    \end{equation*}
    Consequently, we observe that for any initial guess not the unique solution $\mathbf{s}^\star = 0$, Newton's method will diverge for this function, as shown in \Cref{fig:newton_diverges}. Note that \Cref{prop:Newton_quadratic} does not even apply in this setting because the derivative is neither Lipschitz nor does its inverse have a uniform bound.
\end{example}
Studying the convergence rates of parallel Newton methods---as well as their possible instabilities---is a major theme of this thesis.
 
\subsection{Optimization}\label{ssc:opt}

While Newton's method is usually first presented in an introduction to calculus as a method for root-finding (\Cref{ssc:root_finding}), it is best known in machine learning in the context of optimization.
Say we have an objective function $F(\mathbf{s}): \mathbb{R}^P \to \mathbb{R}$ that is twice differentiable, and we wish to find its minimum, i.e.
\begin{equation*}
    \mathbf{s}^\star = \mathrm{argmin}_{\mathbf{s} \in \mathbb{R}^P} F(\mathbf{s}).
\end{equation*}
For large dimension $P$ and a complicated objective function $F(\mathbf{s})$,  optimization can be very difficult.
In fact, the problem of high-dimensional optimization is one of the central problems of machine learning \citep{adam, shampoo, jordan2024muon, vyas2024soap}.

However, if $F(\mathbf{s})$ is a convex quadratic function, i.e. we can write $F(\mathbf{s}) = \frac{1}{2} \mathbf{s}^{\top} \mathbf{M} \mathbf{s} + \mathbf{b}^{\top} \mathbf{s} + c$ for positive-definite matrix $\mathbf{M}$, then its unique minimizer is given by $\mathbf{s}^\star = -\mathbf{M}^{-1} b$.

Thus, Newton's method for optimization of a twice-differentiable function is directly analogous to Newton's method for root-finding for a differentiable function.
In Newton's method for root-finding, we built on the fact that we could solve invertible linear systems, and so for a nonlinear system $\mathbf{r}(\mathbf{s}) = \mathbf{0}$, we iteratively linearize $\mathbf{r}(\cdot)$ and solve.
In Newton's method for optimization, we build on the fact that we have a closed form solution for the minimum of a convex quadratic, and so for a twice-differentiable function $F(\cdot)$, we iteratively build and minimize the quadratic surrogate of $F(\cdot)$.

The quadratic surrogate for $F(\cdot)$ at our current guess $\mathbf{s}^{(i)}$ is given by
\begin{equation}\label{eq:quad_surr}
    \hat{F}_i(\mathbf{s}) =  F(\mathbf{s}^{(i)}) + \nabla_{\mathbf{s}} F(\mathbf{s}^{(i)})^{\top} \left( \mathbf{s} - \mathbf{s}^{(i)} \right) +  \frac{1}{2} \left( \mathbf{s} - \mathbf{s}^{(i)} \right)^{\top} \nabla^2_{\mathbf{s}}F(\mathbf{s}^{(i)}) \left( \mathbf{s} - \mathbf{s}^{(i)} \right),
\end{equation}
where $\nabla^2_{\mathbf{s}}F(\mathbf{s}^{(i)}) \in \mathbb{R}^{P \times P}$ is the Hessian of $F(\cdot)$ evaluated at $\mathbf{s}^{(i)}$. Therefore, if the Hessian is positive-definite, the minimizer of $\hat{F}_i(\mathbf{s})$, and therefore the formula for the next iteration in Newton's method for optimization, is 
\begin{align*}
    \mathbf{s}^{(i+1)} & := - (\nabla^2_{\mathbf{s}}F(\mathbf{s}^{(i)}))^{-1} \left( \nabla_{\mathbf{s}} F(\mathbf{s}^{(i)}) - \nabla^2_{\mathbf{s}}F(\mathbf{s}^{(i)}) \mathbf{s}^{(i)}  \right) \\
    & = \mathbf{s}^{(i)} - (\nabla^2_{\mathbf{s}}F(\mathbf{s}^{(i)}))^{-1} \nabla_{\mathbf{s}} F(\mathbf{s}^{(i)}).
\end{align*}
However, we recognize this update for Newton's method for optimization as the same as Newton's method for root-finding in \cref{eq:newton_root}, where the function we are finding the root of is $\nabla_{\mathbf{s}} F (\cdot): \mathbb{R}^P \to \mathbb{R}^P$.
Thus, we see that Newton's method for \emph{optimization} of a function $F(\cdot)$ is nothing more than Newton's method for \emph{root-finding} applied to the \emph{derivative} of $F(\cdot)$.

This connection is part of the rich interplay in numerical analysis between root-finding (finding the zero of a function) and optimization (finding the minima of a function) \cite{NocedalWright}. The fact that Newton's method for optimization of objective function $F(\cdot)$ is equivalent to Newton's method for root-finding applied to its derivative $\nabla F (\cdot)$ makes sense because for a differentiable function $F(\mathbf{s}): \mathbb{R}^P \to \mathbb{R}$, its minima lie among its stationary points (the set of points where its derivative $\nabla F(\mathbf{s}) = \mathbf{0}$).

\paragraph{Gauss-Newton method for optimization of sum-of-squares}

However, there are even more connections between root-finding and optimization.
If we return to the problem of finding a root of a residual function $\mathbf{r}(\mathbf{s}): \mathbb{R}^P \to \mathbb{R}^P$, we observe that we can form a \emph{merit function}
\begin{equation}\label{eq:merit}
    \mathcal{L}(\mathbf{s}) := \frac{1}{2} \| \mathbf{r}(\mathbf{s}) \|_2^2.
\end{equation}
Because $\mathcal{L}(\mathbf{s})$ is a sum-of-squares objective, it is greater than or equal to zero, and we observe that $\mathcal{L}(\mathbf{s}^\star) = 0$, meaning the root $\mathbf{s}^\star$ of $\mathbf{r}(\cdot)$ is also the minimizer of the merit function\footnote{While it is admittedly counterintuitive to desire to "minimize" a "merit function," we follow the naming convention set by the classic textbook \citet{NocedalWright}.} $\mathcal{L}(\cdot)$.

By basic calculus, we observe that the gradient and Hessian of $\mathcal{L}$ are given by
\begin{align*}
    \nabla \mathcal{L}(\mathbf{s}) & = \mathbf{J}^{\top} \mathbf{r} \\ 
    \nabla^2 \mathcal{L}(\mathbf{s}) & = \mathbf{J}^{\top} \mathbf{J} + \sum_{i = 1}^{P} r_i(\mathbf{s}) \nabla^2 r_i(\mathbf{s}),
\end{align*}
where 
\begin{equation*}
    \mathbf{J}(\mathbf{s}) := \dfrac{\partial \mathbf{r}}{ \partial \mathbf{s}}(\mathbf{s}).
\end{equation*}
While we could apply Newton's method for optimization to $\mathcal{L}$, we know from our previous discussion that this would be Newton's method for root-finding applied to the gradient $\nabla \mathcal{L}(\mathbf{s}) = \mathbf{J}^{\top}(\mathbf{s}) \mathbf{r}(\mathbf{s})$, and \emph{not} Newton's method for root-finding applied to the original residual function $\mathbf{r}(\mathbf{s})$.

However, a very simple modification called the \emph{Gauss-Newton} method restores the link between optimization of the sum-of-squares merit function $\mathcal{L}$ and root-finding of the residual $\mathbf{r}$. In the Gauss-Newton method, we apply Newton's method, but we approximate the Hessian by $\mathbf{J}^{\top} \mathbf{J}$. The Gauss-Newton method is thus a way to get the benefit of second-order methods while only taking one derivative. Moreover, its updates take the form
\begin{equation*}
    \mathbf{s}^{(i+1)} = \mathbf{s}^{(i)} - \left(\mathbf{J}^{\top} \mathbf{J} \right)^{-1} \mathbf{J}^{\top} \mathbf{r}.
\end{equation*}
If $\mathbf{J}$ is invertible, then the Gauss-Newton updates take the form
\begin{equation*}
    \mathbf{s}^{(i+1)} = \mathbf{s}^{(i)} - \mathbf{J}^{-1} \mathbf{r},
\end{equation*}
which we again recognize as \cref{eq:newton_root}, i.e. root-finding for $\mathbf{r}$.

Note, therefore, that if $\mathbf{J}$ is invertible, then Gauss-Newton as an optimization technique for $\mathcal{L}$ is mathematically equivalent to Newton's method for root finding applied to $\mathbf{r}$.
For this reason, another interpretation of the Gauss-Newton method is as \emph{linearizing the residual function} $\mathbf{r}(\mathbf{s})$: that is, each step of the Gauss-Newton method minimizes the quadratic loss
\begin{equation*}
    \hat{\mathcal{L}}_{\mathbf{s}^{(i)}}(\mathbf{s}) := \frac{1}{2} \left\| \mathbf{r}(\mathbf{s}^{(i)}) + \mathbf{J}(\mathbf{s}^{(i)}) \left( \mathbf{s} - \mathbf{s}^{(i)} \right) \right\|_2^2.
\end{equation*}

For small residuals, the Newton and Gauss-Newton methods have similar convergence properties (cf. \citep{NocedalWright}). Importantly, just like Newton's method for root-finding, they can also both diverge globally. For example, take \Cref{example:newton_diverge} and turn it into an optimization problem with objective function $F(\mathbf{s}) = \mathbf{s}^{4/3}$ to see that Newton's method will diverge or merit function $\mathcal{L}(\mathbf{s}) = \mathbf{s}^{2/3}$ to see that Gauss-Newton will diverge.

\subsection{Fixed-point methods}\label{ssc:fxd_pts}

We can also write each step of Newton's method as the action of an operator $\mathcal{A}_N(\mathbf{s}): \mathbb{R}^P \to \mathbb{R}^P$, i.e.
\begin{align*}
    \mathbf{s}^{(i+1)} & = \mathcal{A}_{N}(\mathbf{s}^{(i)}) \\
    \mathcal{A}_{N}(\mathbf{s}^{(i)}) & = \mathbf{s}^{(i)} - \left( \mathbf{J}^{(i)} \right)^{-1} \mathbf{r}(\mathbf{s}^{(i)}).
\end{align*}
Importantly note that if $\mathbf{s}^\star$ is a root of $\mathbf{r}(\mathbf{s})$, i.e. $\mathbf{r}(\mathbf{s}^\star) = \mathbf{0}$, then it follows that
\begin{equation*}
    \mathcal{A}_N(\mathbf{s}^\star) = \mathbf{s}^\star,
\end{equation*}
i.e. $\mathbf{s}^\star$ is a \emph{fixed-point} of the Newton's method operator $\mathcal{A}_N$.

In general, a \emph{fixed-point} problem aims to find $\mathbf{s}^\star$ satisfying
\begin{equation*}
    \mathbf{F}(\mathbf{s}^\star) = \mathbf{s}^\star,
\end{equation*}
for some function $\mathbf{F}(\mathbf{s}): \mathbb{R}^P \to \mathbb{R}^P$. Any fixed-point problem can be interpreted as a root-finding problem by defining $\mathbf{r}(\mathbf{s}) = \mathbf{s} - \mathbf{F}(\mathbf{s})$, and then asking to find $\mathbf{s}^\star$ such that $\mathbf{r}(\mathbf{s}^\star) = 0$.

Because of all of these connections, Newton's method is also a foundational concept in fixed-point methods and solvers \citep{ortega1970iterative}. 
As we will see in \Cref{ch:quasi_convergence}, many different fixed-point methods can be used to parallelize SSMs, including Picard and Jacobi iterations. We discuss in more detail in that section.

\section{Putting it all together: Parallel Newton methods}\label{sec:deer}

In the previous sections, we reviewed dynamics, parallel computation, and numerical analysis, with the goal of combining these three diverse fields to parallelize the unrolling of state space models. In this section, we combine these three ingredients to show how parallel Newton methods allow for the parallelization of such ``inherently sequential'' processes.

\subsection{Parallel Newton methods: DEER and DeepPCR}

Concurrently, \citet{deer2024} developed DEER and \citet{deeppcr} developed DeepPCR, both of which are the same \emph{parallel Newton} method for the parallelization of SSMs.
This section reviews their foundational work.
Throughout this thesis, we use the terms ``DeepPCR'', ``DEER'', and ``parallel Newton methods'' interchangeably.

The fundamental idea of parallelizing SSMs is to replace sequential evaluation with parallel \emph{iterative} evaluation. We compare these two approaches to evaluating an SSM in \Cref{fig:parr_v_seq}.
\begin{figure}
    \centering
    \includegraphics[width=0.95\linewidth]{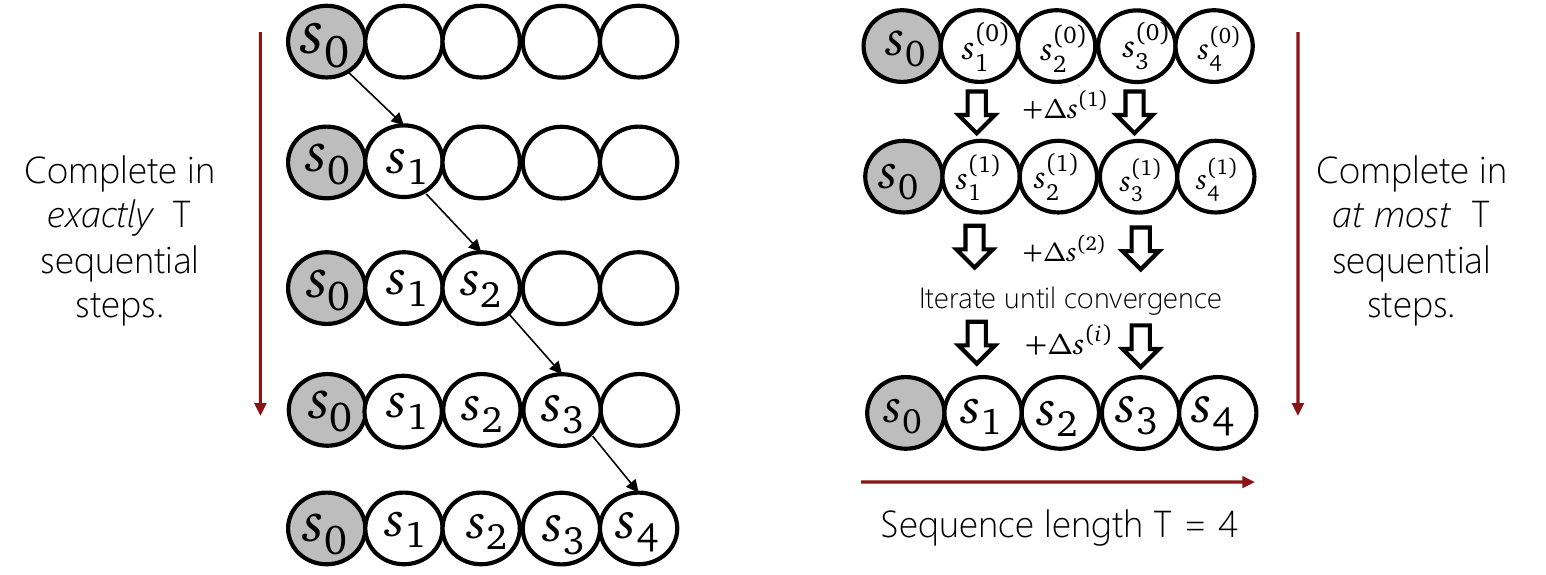}
    \caption{Comparison of standard sequential evaluation of an SSM \textbf{(left)} with \emph{parallel iterative evaluation} of an SSM \textbf{(right)}. In the parallel iterative paradigm, we make a guess over the \emph{entire} sequence, as indicated by the top right row labeled $s_1^{(0)}, s_2^{(0)}, \hdots$. Using parallel computation over the sequence length, we find an update $\Delta \mathbf{s}^{(i)}$ to go from our current guess $\mathbf{s}^{(i)}$ for the entire trajectory to our next guess $\mathbf{s}^{(i+1)}$. Adapted from Figure 1 of \citet{deer2024}.}
    \label{fig:parr_v_seq}
\end{figure}
Going forward, we will denote the true roll-out from the SSM over the entire trajectory of length $T$ as $\mathbf{s}^\star \in \mathbb{R}^{TD}$, i.e. $s_1^\star = f_1(s_0), s_2^\star = f_2(s_1^\star)$, and in general $s_t^\star = f_t(s_{t-1}^\star)$.\footnote{In our discussion of parallel Newton methods, and henceforth in this thesis, we will use bold script for variables of shape $TD$ or $TD \times TD$, and not bold for variables of shape $D$ or $D \times D$. So, bolding will be reserved for variables that extend over the sequence length, while variables that are just at a particular point in time will not be bolded. We follow this convention to distinguish between operations that occur across the sequence length vs. at a particular point.}
Note that at initialization, $\mathbf{s}^\star \neq \mathbf{s}^{(0)}$, i.e. we may be initializing in a way that is not faithful at all to the true SSM dynamics. Thus, we need to make updates $\mathbf{s}^{(i+1)} = \mathbf{s}^{(i)} + \Delta \mathbf{s}^{(i)}$ in a way that brings our guesses close to $\mathbf{s}^\star$ in a small number of iterations. This desideratum raises an important question about our updates $\Delta \mathbf{s}^{(i)}$:

\blockquote{\textbf{How can we compute a useful $\Delta \mathbf{s}^{(i)}$ in a way that uses \emph{parallel computation over the sequence length? }}}

While the addition $\mathbf{s}^{(i)} + \Delta \mathbf{s}^{(i)}$ is embarrassingly parallel over the sequence length, we will not achieve our goal of \emph{parallelizing} SSMs over the sequence length if computing the update $\Delta \mathbf{s}^{(i)}$ itself requires inherently sequential computation.

Parallel Newton methods offer an ingenious way to compute these updates $\Delta \mathbf{s}$ using parallel computation over the sequence length. The core insight is that even though our initial guess $\mathbf{s}^{(0)}$ may be completely wrong---and even though we do not use the true roll-out $\mathbf{s}^\star$ at any point in the computation---we can still use the SSM dynamics from \cref{eq:ssm} to measure \emph{how wrong our current guess is}.

We measure how wrong our initial guess is with its \emph{residual vector} $\mathbf{r}(\mathbf{s}^{(i)}) \in \mathbb{R}^{TD}$. Each entry $r_t$ of the residual vector is given by the one-step prediction error, i.e.
\begin{equation}\label{eq:r_def}
    r_t(\mathbf{s}^{(i)}) := s_t^{(i)} - f_t(s_{t-1}^{(i)}).
\end{equation}
Crucially, $\mathbf{r}(\mathbf{s}^\star) = \mathbf{0}$ because $\mathbf{s}^\star$ follows the SSM dynamics, and in fact $\mathbf{s}^\star$ is the unique zero of $\mathbf{r}(\cdot)$.

Thus, by defining the residual in \cref{eq:r_def}, we have recast the problem of SSM evaluation as a high-dimensional root-finding problem, i.e. starting from an initial guess $\mathbf{s}^{(0)}$, find $\mathbf{s}^\star \in \mathbb{R}^{TD}$ such that
\begin{equation}\label{eq:resid_to_zero}
    \mathbf{r}(\mathbf{s}^\star) = \mathbf{0}.
\end{equation}
We discussed exactly this type of problem in our background \cref{ssc:root_finding} on root-finding!

As their name suggests, parallel Newton methods solve this high-dimensional nonlinear \cref{eq:resid_to_zero} using Newton's method. Moreover, in the specific case of evaluating SSMs, where the residual at each time step is given by \cref{eq:r_def}, each Newton update is given by a linear dynamical system and so can be evaluated using a parallel scan.

That each step of Newton's method for finding the zero of the residual defined in \cref{eq:r_def} is an LDS comes from the fact that at each step, Newton's method \emph{linearizes} the residual. To review, at each step of Newton's method for root-finding, we find the root of the linearized residual $\hat{\mathbf{r}}^{(i)}(\mathbf{s})$, where each entry of $\hat{\mathbf{r}}^{(i)}(\mathbf{s})$ is given by
\begin{equation}\label{eq:deer_rhat}
    \hat{\mathbf{r}}_t^{(i)}({ \color{blue} \mathbf{s}}) = {\color{blue} s_t} - \underbrace{\left( f_t(s_{t-1}^{(i)}) + \A_t^{(i)} ({ \color{blue} s_{t-1}} - s_{t-1}^{(i)})  \right)}_{\text{linearization of dynamics function $f_t$ at $s_{t-1}^{(i)}$}},
\end{equation}
where throughout this thesis we use the shorthand
\begin{equation}\label{eq:lil_j}
    \A_t^{(i)} := \dfrac{\partial f_t}{\partial s_{t-1}}(s_{t-1}^{(i)}) \in \mathbb{R}^{D \times D}.
\end{equation}
Since each step of Newton's method involves finding $\mathbf{s}^{(i+1)}$ such that $\hat{\mathbf{r}}^{(i)}(\mathbf{s}^{(i+1)}) = \mathbf{0}$, we see from \cref{eq:deer_rhat} that setting each component of $\hat{\mathbf{r}}^{(i)}$ to zero gives rise to the LDS
\begin{equation}\label{eq:deer_lds}
    \boxed{{ \color{blue} s_t^{(i+1)}} = \A_t^{(i)} { \color{blue} s_{t-1}^{(i+1)}} + \underbrace{\left( f_t(s_{t-1}^{(i)}) - \A_t^{(i)} s_{t-1}^{(i)}\right)}_{b_t^{(i)}}.}
\end{equation}
But as discussed in \cref{ssc:pscan_lds}, with $\mathcal{O}(T)$ processors we can evaluate any LDS in $\mathcal{O}(\log T)$ computational depth. Thus, we have shown that on a massively parallel machine like a GPU, we can evaluate each iteration of a parallel Newton method in $\mathcal{O}(\log T)$ time. If we can converge in fewer than $\mathcal{O}(\nicefrac{T}{\log T})$ iterations, then for sufficiently long sequence lengths and powerful parallel processors, we would expect to see wallclock speedups from parallelizing SSMs.

We summarize the parallel Newton methods in \Cref{alg:deer}, and provide a more detailed derivation in the next section.
\begin{algorithm}[ht]
    \caption{Parallel Newton methods for evaluating nonlinear SSMs}\label{alg:deer}
  \begin{algorithmic}
  \Procedure{ParallelNewton}{$f$, $s_0$, \text{initial guess} $\mathbf{s}_{1:T}^{(0)}$, \text{tolerance} $\epsilon$}
    \For{$i=0,1,\ldots, T$}
        \State $A_{1:T}, b_{1:T}  \gets \textsc{LinearizeDynamics}(f, s_0, \mathbf{s}_{1:T}^{(i)})$ \Comment{For all $t$ in parallel}
        \State $\mathbf{s}_{1:T}^{(i+1)} \gets \textsc{EvaluateLDS}( A_{1:T}, b_{1:T}, s_0, \mathbf{s}_{1:T}^{(i)})$
        \Comment{pscan has $\mathcal{O}(\log T)$ depth}
        \If{$\textsc{ComputeError}(f, \mathbf{s}_{1:T}^{(i+1)}) < \epsilon$}
            \State \textbf{break}
        \EndIf
    \EndFor
    \State \Return $\mathbf{s}_{1:T}^{(i+1)}$
    \EndProcedure
  \end{algorithmic}
\end{algorithm}

\subsection{More in depth derivation}

We provide an alternative derivation of the parallel Newton update in \cref{eq:deer_lds} to highlight important notions.

To apply Newton's method for root-finding to the residual used in DEER/DeepPCR (defined coordinate-wise in \cref{eq:r_def}), the update given in \cref{eq:newton_root} is
\begin{equation}\label{eq:deer_full_update}
    \mathbf{s}^{(i+1)} = \mathbf{s}^{(i)} - \underbrace{\mathbf{J}(\mathbf{s}^{(i)})^{-1} \mathbf{r}(\mathbf{s}^{(i)})}_{\Delta \mathbf{s}^{(i)}},
\end{equation}
where the Jacobian matrix $\mathbf{J} := \dfrac{\partial \mathbf{r}}{\partial \mathbf{s}}( \mathbf{s}) \in \mathbb{R}^{TD \times TD}$ is a block bidiagonal matrix of the form
\begin{equation}\label{eq:big_j}
    \mathbf{J} = \begin{pmatrix} I_D & 0 & \hdots & 0 & 0 \\
    \\
    - \A_2 & I_D & \hdots & 0 & 0 \\
    \vdots & \vdots  & \ddots & \vdots & \vdots \\ 
    0 & 0  & \hdots & I_D & 0 \\
    0 & 0  & \hdots & -\A_{T} & I_D \\
    \end{pmatrix},
\end{equation}
where $\A_t \in \mathbb{R}^{D \times D}$ are defined as in \eqref{eq:lil_j}. Importantly, the Jacobian $\mathbf{J}$ in \cref{eq:big_j} is always invertible with all eigenvalues equal to one. 
Storing and naively inverting the Jacobian is infeasible for large state size $D$ or sequence length $T$.
However, since $\mathbf{J}(\mathbf{s})$ is block bidiagonal, we can solve for $\Delta \mathbf{s}$ (i.e. $\mathbf{J}(\mathbf{s}^{(i)})^{-1} \mathbf{r}(\mathbf{s}^{(i)})$)  by forward substitution.
This reduces to a linear recursion with the initial condition $\Delta s_1^{(i+1)} = -r_1(s^{(i)})$, and for $t > 1$,
\begin{align}\label{eq:lin_rr}
    \Delta s_t^{(i+1)} = 
    \A_t^{(i)} \Delta s_{t-1}^{(i+1)} - r_t(s^{(i)}).
\end{align}
Plugging \cref{eq:lin_rr} into \cref{eq:newton_root} and simplifying, we again obtain the DEER/DeepPCR update \cref{eq:deer_lds}.

We emphasize that the Newton update $\Delta \mathbf{s} \in \mathbb{R}^{TD}$ is given by
\begin{equation*}
    \Delta \mathbf{s}^{(i)} = \mathbf{J}(\mathbf{s}^{(i)})^{-1} \mathbf{r}(\mathbf{s}^{(i)}).
\end{equation*}
The beauty of parallel Newton updates for SSMs is that we can exploit the particular block bidiagonal structure of $\mathbf{J}$ (shown in \cref{eq:big_j}) to invert $\mathbf{J}$ using a parallel scan. However, it is also worth examining $\mathbf{J}^{-1}$, which is itself a structured matrix. Using the example of $T=4$ to demonstrate, we see that $\mathbf{J}^{-1}$ takes the form
\begin{equation}\label{eq:J_inv}
    \mathbf{J}^{-1} = \begin{pmatrix}
    I_D & 0 & 0 & 0 \\
    \A_2 & I_D & 0 & 0 \\
    \A_3 \A_2 & \A_3 & I_D & 0 \\
    \A_4 \A_3 \A_2 & \A_4 \A_3 & \A_4 & I_D
    \end{pmatrix}.
\end{equation}
What \cref{eq:J_inv} is meant to demonstrate is that, in general, $\mathbf{J}^{-1}$ is itself lower triangular, with each block term being a product of a sequence of Jacobian matrices. This particular form for $\mathbf{J}^{-1}$ makes sense because we know from \cref{eq:lin_rr} that we can invert $\mathbf{J}$ with an LDS, and so simply applying $\mathbf{J}^{-1}$ should be equivalent to applying the convolution that is this LDS. Studying the properties and conditioning of $\mathbf{J}^{-1}$ will be crucial to proving convergence rates of parallel Newton methods, which we do in \Cref{part:methods}.

We note that, as discussed in \Cref{ssc:opt}, all of the above can be interpreted as applying the Gauss-Newton method for optimization to a merit function $\mathcal{L}(\mathbf{s}) = \frac{1}{2} \| \mathbf{r}(\mathbf{s}) \|_2^2$, in addition to the provided interpretation as Newton's method for root finding on $\mathbf{r}(\mathbf{s})$. This optimization perspective on parallel Newton methods is foundational for this dissertation, as we use it to develop scalable and stable methods (\Cref{part:methods}), as well as prove convergence rates (\Cref{part:theory}). 

Finally, we can also view each each iteration of \cref{eq:deer_full_update} as a fixed-point iteration \citep{movahedi2025fixed}.
In this way, another perspective on DEER is that it recasts RNNs and nSSMs in general in the framework of \emph{deep equilibrium models} (DEQs) \citep{bai2019deep, bai2022equilibrium}.
Fixed-point methods, including Newton iterations \citep{kennedy2010parallel}, are often commonly used in the field of \emph{multidisciplinary optimization} (MDO) in aeronautical engineering \citep{martins2013multidisciplinary, kochenderfer2026algorithms}.
All of these fields have deep connections to DEER, and interesting future work could involve exploring them.

\subsection{Limitations of Newton's method}\label{ssc:limitations}

\Cref{eq:deer_lds} is the fundamental update behind Newton's method for parallelizing an SSM. However, it also contains the ingredients behind some critical limitations of "plain vanilla" Newton's method: scalability and stability.

\paragraph{Methodological limitations: scalability and stability}

The difficulty in \emph{scaling} \cref{eq:deer_lds} comes from the need to instantiate $T$ Jacobian matrices, each of which are $\mathbb{R}^{D \times D}$. Because the parallel scan must instantiate all of these matrices simultaneously, doing so requires $\mathcal{O}(T {\color{red} D^2})$ memory, which can be prohibitive for large state size or long sequence length. Moreover, because the parallel scan involves dense matrix-matrix multiplies, the total computational work is $\mathcal{O}(T {\color{red} D^3})$.
While the factor of $T$ in the work is divided across parallel processors\footnote{If we have $\mathcal{O}(T)$ parallel processors, the $\mathcal{O}(T)$ work is done in $\mathcal{O}(\log T)$ computational depth.}, the cubic cost in state size can also make the method prohibitively slow for large state size. For these reasons, using the update \cref{eq:deer_lds} in parallel Newton methods is difficult to use in practice at scale, often running out of memory or running too slowly.

The difficulties in \emph{stability} for \cref{eq:deer_lds} also come from studying its behavior as a linear dynamical system. In particular, the spectral norm of any matrix $J_t$ measures the maximum amount by which it may increase the size of a vector to which it is applied. So, intuitively, if the spectral norms of too many Jacobian matrices in \cref{eq:deer_lds} are larger than one, the update \cref{eq:deer_lds} may be highly unstable, resulting in numerical overflow and slow convergence. These difficulties with stability are common in Newton methods in general, see \Cref{example:newton_diverge}.

\paragraph{Gaps in theoretical understanding: convergence properties}

Finally, both foundational works of \citet{deeppcr} and \citet{deer2024} explicitly left open the question of the global convergence of the parallel Newton method, i.e., will the method converge regardless of our initial guess $\mathbf{s}^{(0)}$.
In general, Newton's method does not enjoy such properties, as we showed in \Cref{example:newton_diverge}. But confidence that the method will robustly and globally converge is important for broad deployment of the method. Moreover, while it is broadly known that Newton's method enjoys quadratic convergence in a basin around its solution \cite{NocedalWright, boyd2004convex, deer2024}, it was unclear if anything more could be said specifically about the rates of convergence of parallel Newton methods. In particular, it was unclear if we could generally expect speed-ups from parallelization in arbitrary SSMs, or if there were certain SSMs that benefit from parallelization and other SSMs that are more efficient to evaluate sequentially.

Resolving these scaling and stability limitations of the parallel Newton method (\Cref{part:methods}) and providing general theory about its convergence properties (\Cref{part:theory}) are the contributions of the rest of this thesis.
\clearpage
\ctparttext{The second part of this thesis presents its methodological contributions.
We develop methods for scalable and stable parallelization of nonlinear SSMs. We achieve scalability using a quasi-Newton method we develop and call \emph{quasi-DEER}. We achieve stability using a trust region method we develop and call \emph{ELK}: \textbf{E}valuating \textbf{L}evenberg-Marquardt with \textbf{K}alman. 
\par
\begin{minipage}{\linewidth}
    \captionsetup{hypcap=false}
    \includegraphics[width=\linewidth]{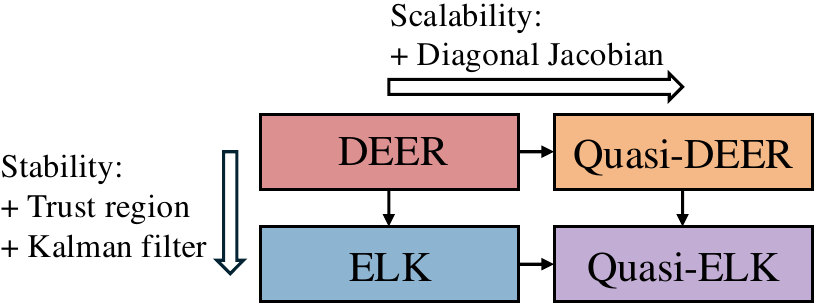}
    \captionof{figure}{\textbf{The ungulates.} This methods part of this thesis introduces scalable and stable variants of DEER. Broadly, we call these methods "parallel Newton methods." More colloquially, we call these methods "the ungulates," which are large hoofed mammals like deer and elk. The experiments in this part are based on the code available at: \url{https://github.com/lindermanlab/elk}}
    \label{fig:ungulates}
\end{minipage}
}
\part{Methods: Scalable and Stable Parallelization}\label{part:methods}
%************************************************
\chapter{Scalable Parallelization: Quasi-Newton Methods}\label{ch:scalable}
%************************************************

As we discussed in the Introduction and in the Background (\Cref{sec:deer}), the parallel Newton methods of \citet{deeppcr} and \citet{deer2024} provide a novel approach to parallelize nonlinear state space models (nSSMs), even though evaluating nSSMs had long been believed to be "inherently sequential."

However, it is well known in numerical analysis that Newton's method---while an extremely powerful and fundamental method---has many limitations (see \Cref{sec:math_background} as well as a textbook treatment in \citep{NocedalWright}). The common thread throughout this thesis is how we can leverage the vast literature on numerical analysis to extend, improve, and understand parallel Newton methods.

In this chapter, we focus in particular on the limitation of Newton's method with respect to \emph{scalability}. 
In general, for trying to find the root $\mathbf{s}^*$ of a high-dimensional function $\mathbf{r}(\cdot): \mathbb{R}^P \to \mathbb{R}^P$, Newton's method has updates of the form
\begin{equation*}
    \mathbf{s}^{(i+1)} = \mathbf{s}^{(i)} - \mathbf{J}^{-1} \mathbf{r}.
\end{equation*}
In general, this Newton update is prohibitive for large dimension $P$ because it involves
\begin{itemize}
    \item computing the derivative $\mathbf{J}$; 
    \item storing the $P \times P$ matrix $\mathbf{J}$; and
    \item inverting this matrix. 
\end{itemize}
All three of the steps are expensive in either compute or memory.

For a parallel Newton method, the dimension of $\mathbf{s}$ is $TD$, where $T$ is the sequence length and $D$ is the state size. Thus, forming a $TD \times TD$ matrix is in general intractable. Parallel Newton methods avoid forming $\mathbf{J}$ explicitly, instead using the structure (\cref{eq:r_def}) of 
the one-step prediction error $\mathbf{r}(\cdot)$ to cast each step of Newton's method as a linear dynamical system (LDS) (\cref{eq:deer_lds}):
\begin{equation*}
    s_t^{(i+1)} = {\color{red} \A_t^{(i)}} s_{t-1}^{(i+1)} + \left( f_t(s_{t-1}^{(i)}) - \A_t^{(i)} s_{t-1}^{(i)}\right).
\end{equation*}
However, each $\A_t \in \mathbb{R}^{D \times D}$, and so parallelizing this LDS using a parallel scan results in work that scales as $\mathcal{O}(T {\color{red} D^3})$ and memory requirement that scales as $\mathcal{O}(T {\color{red} D^2})$. For large state sizes and sequence lengths, these costs soon become prohibitive.

Fortunately, there exists a wide literature \citep{NocedalWright} on \emph{quasi-Newton} methods that use some approximation $\Tilde{\mathbf{J}}$ for $\mathbf{J}$. In this chapter, we explore ways to scale parallel Newton methods by introducing quasi-Newton methods that are amenable to a parallel scan.

\section{Quasi-DEER: A diagonal approximation}

We propose a very simple quasi-Newton approximation we call \emph{quasi-DEER}\footnote{\citet{deer2024} names parallel Newton methods DEER, for "\textbf{D}ifferential \textbf{E}quations as fixed-point it\textbf{ER}ation."}, where we use the diagonal of the Jacobians, i.e. we use updates of the form
\begin{equation}\label{eq:quasi_lds}
    s_t^{(i+1)} = {\color{orange} \mathrm{diag}[\A_t^{(i)}]} s_{t-1}^{(i+1)} + \left( f_t(s_{t-1}^{(i)}) - {\color{orange} \mathrm{diag}[\A_t^{(i)}]} s_{t-1}^{(i)}\right).
\end{equation}
We developed this diagonal approximation because of its compatibility with the parallel scan and because of its lower computational and memory cost (vs dense matrix multiplication).

To be compatible with the parallel scan, a the operands of the chosen binary operator crucially must remains \emph{closed} (\Cref{def:closure}). Fortunately, the product of two diagonal matrices is again a diagonal matrix.

Moreover, using diagonal matrices is clearly more memory and compute efficient than using dense matrices. Both the memory cost of storing and the computational work of multiplying these diagonal matrices now scales only as $\mathcal{O}(T{\color{orange} D})$, i.e. linearly with the state size. 

However, this quasi-DEER method based on a diagonal approximation of the Jacobian of the dynamics is very different from anything in the standard quasi-Newton literature \cite{NocedalWright}. Some immediate and natural questions are:
\begin{enumerate}
    \item will this approach even converge?
    \item if this approach does converge, will it converge in few enough iterations actually to be useful?
\end{enumerate}
One response to this question is to note that while a diagonal approximation serves as a type of matrix that enjoys an efficient parallel scan, in general any form of approximation $\Tilde{\A}_t$ to $\A_t$ would work as a quasi-DEER method\footnote{i.e., an efficient iterative step that makes use of the parallel scan} if the class of matrices used for $\Tilde{\A}_t$ are closed under composition and have memory and compute costs that scales linearly in $D$. We discuss in \Cref{ch:quasi_convergence} how many foundational fixed-point methods can be interpreted as different forms of quasi-DEER for different approximations $\Tilde{\A}_t$.

Incredibly, however, \emph{all} such quasi-DEER methods (including the full Newton method DEER and the diagonal approximation) enjoy \emph{global convergence}. Note that Newton's method in general may fail to converge. This global convergence of parallel methods and \emph{all} quasi-versions of the form proposed is a special feature of the particular problem of parallelizing nonlinear SSMs (cf. \cref{eq:r_def}). Thus, we can answer our first question: \textbf{yes}, this diagonal approximation in fact converges globally.

Moreover, the diagonal approximation also performs well empirically. We showcase its performance for evaluating and training nonlinear RNNs, including on a benchmark dataset from computational neuroscience.

Finally, beyond global convergence, we can provide a bound on how slowly quasi-DEER can converge, which is effectively based on the quality of the approximation for different dynamical systems.

In the rest of this chapter, we will discuss the global convergence of quasi-DEER and all of its variants, as this result is foundational for this thesis and line of work. We will also showcase the experiments showing the empirical usefulness of the method. We defer a discussion of quasi-DEER convergence rates to \Cref{ch:quasi_convergence} in the "Theory" part of this thesis. We instead conclude this chapter with discussions of further extensions that have been made to quasi-DEER, as well as promising directions for future work.

\section{Global convergence}\label{sec:global_convergence}

In general, Newton's method is not guaranteed to converge (\Cref{example:newton_diverge}). This general risk of failing to converge led both \citet{deeppcr} and \citet{deer2024} to flag the question of convergence in parallel Newton methods as an important open question, though neither answered this question.

In fact, this question of DEER's convergence was answered in 1989 by \citet[Remark 2.1]{Bellen1989}, which we rediscovered in \citet[Proposition 1]{gonzalez2024scalable}.
Not only is DEER globally convergent, but so are a wide variety of quasi-DEER methods, including the use of the diagonal approximation.
\begin{proposition}\label{prop:global_convergence}
    Consider the problem of finding $\mathbf{s}^{\star}_{1:T}$ which satisfy $s^{\star}_t = f_t(s^{\star}_{t-1})$ and $s_1^{\star} = f_1(s_0)$, for known dynamics functions $\{ f_t \}_{t=1}^T$ and initial condition $s_0$. Also consider an iterative method $\mathcal{A}(\cdot)$ of the form $\mathbf{s}_{1:T}^{(i+1)} = \mathcal{A}(\mathbf{s}_{1:T}^{(i)})$ where the action of the operator $\mathcal{A}(\cdot)$ can be written as a linear dynamical system over the sequence length, i.e. each application of $\mathcal{A}$ takes the form
    \begin{equation}\label{eq:gen_qdeer}
        s_t^{(i+1)} = \Tilde{\A}_t s_{t-1}^{(i+1)} + \left( f_t(s_{t-1}^{(i)}) - \Tilde{\A}_t s_{t-1}^{(i)}\right),
    \end{equation}
    for \emph{arbitrary} matrices $\{ \Tilde{\A}_t \}_{t=1}^T$.

    Then updates based on $\mathcal{A}(\cdot)$ will converge to $\mathbf{s}^{\star}_{1:T}$ in at most $T$ iterations, \emph{regardless} of the initial guess $\mathbf{s}_{1:T}^{(0)}$.
\end{proposition}

\begin{proof}
    The intuition for this proof is that the initial condition for $s_0$ is fixed and known, and that each iteration of $\mathcal{A}(\cdot)$ as given by \cref{eq:gen_qdeer} gives at least one more correct term in the sequence length, while not disturbing any previously correct terms.

    Formally, we prove this theorem by induction.

    \textbf{Base case}: we know the initial condition $s_0$, as it is fixed and given by assumption.

    \textbf{Induction hypothesis}: assume at iteration $(i)$ that $\mathbf{s}_{1:t_i}^{(i)} = \mathbf{s}^{\star}_{1:t_i}$, i.e. the first $t_i$ terms are correct.

    \textbf{Induction step}: we need to show that $\mathbf{s}_{1:t_i+1}^{(i+1)} = \mathbf{s}^{\star}_{1:t_i+1}$, i.e. that none of the previously correct terms become wrong, and that at least one more term becomes correct. Rewriting \cref{eq:gen_qdeer} as
    \begin{equation}\label{eq:qdeer_lds}
        s_t^{(i+1)} =  f_t(s_{t-1}^{(i)}) + \Tilde{\A}_t \left( s_{t-1}^{(i+1)}  - s_{t-1}^{(i)} \right),
    \end{equation}
    we see that if $s_{t-1}$ is correct at both iterations $(i)$ and $(i+1)$, i.e. $s_{t-1}^{(i)} = s_{t-1}^{(i+1)} = s_{t-1}^{\star}$, then it must be the case that $s_t^{(i+1)} = f_t(s_{t-1}^{\star}) = s_t^{\star}$. Of course, $s_0^{(i+1)} = s_0^{(i)} = s_0^{\star}$ because $s_0$ is a fixed and known initial condition. So, by the above logic, it follows that if $\mathbf{s}_{1:t_i}^{(i)} = \mathbf{s}^{\star}_{1:t_i}$, then $\mathbf{s}_{1:t_i+1}^{(i+1)} = \mathbf{s}^{\star}_{1:t_i+1}$. 

    Since we have shown in the induction step that one more correct term always accrues with each application of $\mathcal{A}(\cdot)$, and because of our base case that $s_0^{(0)} = s_0^{\star}$, the result follows from induction.
\end{proof}

\Cref{prop:global_convergence} is significant and interesting for a number of reasons.

First, \Cref{prop:global_convergence} answers the question posed by both \citet{deeppcr} and \citet{deer2024}: does DEER converge globally? In general, Newton's method does not enjoy global convergence\footnote{In fact, as \citet{hubbard2015vector} write in their classic textbook: "no one knows anything about the global behavior of Newton's method."}, but we show that not only DEER but in fact a wide family of quasi-DEER methods all enjoy global convergence. This special behavior is a result of the special structure of our residual $\mathbf{r}(\cdot)$ that arises from parallelizing SSMs (see \cref{eq:r_def}). 

\Cref{prop:global_convergence}, as stated as Proposition 1 in \citet{gonzalez2024scalable}, was the first of its kind for global convergence in the context of parallelizing nonlinear RNNs with Newton iterations. 
While on the one hand this result was surprising since Newton's method can in general diverge (\Cref{fig:newton_diverges}), this exact result was known in the parallel-in-time literature: see \citet[Remark 2.1]{Bellen1989} and \citet[Remark 4.7]{GanderVandewalle2007}.
These results were also rediscovered in the context of parallelizing sampling from diffusion models (another nonlinear SSM). Notably, \citet{shih2023parallel} proved a special case of \Cref{prop:global_convergence} for $\Tilde{\A}_t = I_D$ using the same proof by induction mechanism. \citet{tang2024accelerating} then proved an even stronger result that includes our \Cref{prop:global_convergence} as a special case. We include an extended discussion of Theorem 3.6 of \cite{tang2024accelerating} in \Cref{appendix:global_convergence}. All in all, this core result has been rediscovered many times in different communities.

Second, \Cref{prop:global_convergence} ensures that \emph{arbitrary approximations} can be used in the computation of the Jacobians $\A_t = \dfrac{\partial f_t}{\partial s_{t-1}}$ without damaging the global convergence (though the convergence rate may slow). This guarantee of global convergence extends not only to the diagonal approximation proposed in \citet{gonzalez2024scalable}, but also to a stochastic version proposed in \citet{pmcmc} which we will discuss further in \Cref{ssc:efficiency}. In fact, as \citet{pmcmc} demonstrates, \Cref{prop:global_convergence} ensures global convergence \emph{even when the dynamics $f_t$ are not differentiable}, as in the Metropolis-Hastings algorithm \citep{metropolis1953equation, hastings1970monte, chib1995understanding} for Markov chain Monte Carlo (MCMC). \citet{pmcmc} shows empirically that using updates based on \cref{eq:gen_qdeer} works well even for non-differentiable dynamics $f_t$ by using an intelligent choice of surrogate gradient $\Tilde{\A}_t$. 

Finally, the proof by induction of \Cref{prop:global_convergence} highlights how parallel Newton methods converge in a "causal" manner, i.e. from the start at the initial condition $s_0$ to the end at $s_T$. This arrow of causality has important implications both for the design of parallel Newton variants, as we will see in \Cref{ch:elk}, as well as for an interpretation of what parallel Newton methods are doing, as we will see in \Cref{ch:predictability}. Furthermore, this "causal convergence" also results in a useful heuristic when parallelizing systems that are unstable or at the edge of stability: if intermediate computations in the parallel Newton method should ever overflow numerically, they can always be reset to an arbitrary value without damaging global convergence (though of course slowing the rate of convergence). We make great use of this "reset heuristic" in \Cref{ch:elk}. Finally, this left-to-right convergence also justifies implementing parallel Newton methods with a \emph{sliding window} \citep{shih2023parallel, pmcmc}, where only $t_c$ states have \cref{eq:gen_qdeer} applied to them at a time. While using $t_c < T$ will increase the number of iterations needed to converge, the memory layout and other architectural features of GPUs can lead the choice of certain $t_c < T$ resulting in wallclock speedups compared to naively applying \cref{eq:gen_qdeer} over the entire sequence length \citep{shih2023parallel, pmcmc}. Using a sliding window to implement parallel Newton methods is best practice and should always be used.

Having discussed the important implications of the theoretical convergence of parallel Newton methods, we now let the rubber hit the road and ask the question: but does the diagonal approximation in \cref{eq:quasi_lds} work in practice?

\section{Experiments and performance of quasi-DEER}

In this section, we showcase a variety of settings where quasi-DEER performs well in the parallel evaluation and training of nonlinear RNNs, specifically using the Gated Recurrent Unit (GRU) \cite{gru} as a simple and expressive RNN cell. 

\subsection{Quasi-DEER for Evaluation}
\label{ssc:gru_bench}
To benchmark the speed and memory usage of sequential evaluation, DEER, and quasi-DEER on forward passes of RNNs, we use an experimental design from \citet{deer2024}.
The task is to evaluate an untrained GRU across a range of hidden state sizes ($D$) and sequence lengths ($T$) on a 16GB V100 GPU; the inputs to the RNN also have dimension $D$.
We evaluate these RNNs using three approaches: sequential evaluation, DEER, and quasi-DEER. For DEER and quasi-DEER, we end the Newton iterations when $\| \mathbf{s}^{(i)} - \mathbf{s}^{(i-1)} \|_\infty < \mathrm{tol}$, for some specified tolerance $\mathrm{tol}$. In these experiments, we use a tolerance of $\mathrm{tol} = \num{1e-4}$. In \Cref{fig:quasi_acc}, we show qualitatively that both DEER and quasi-DEER converge with great accuracy to the true sequential rollouts. 
\begin{figure}[H]
  \centering
  \includegraphics[width=\textwidth]{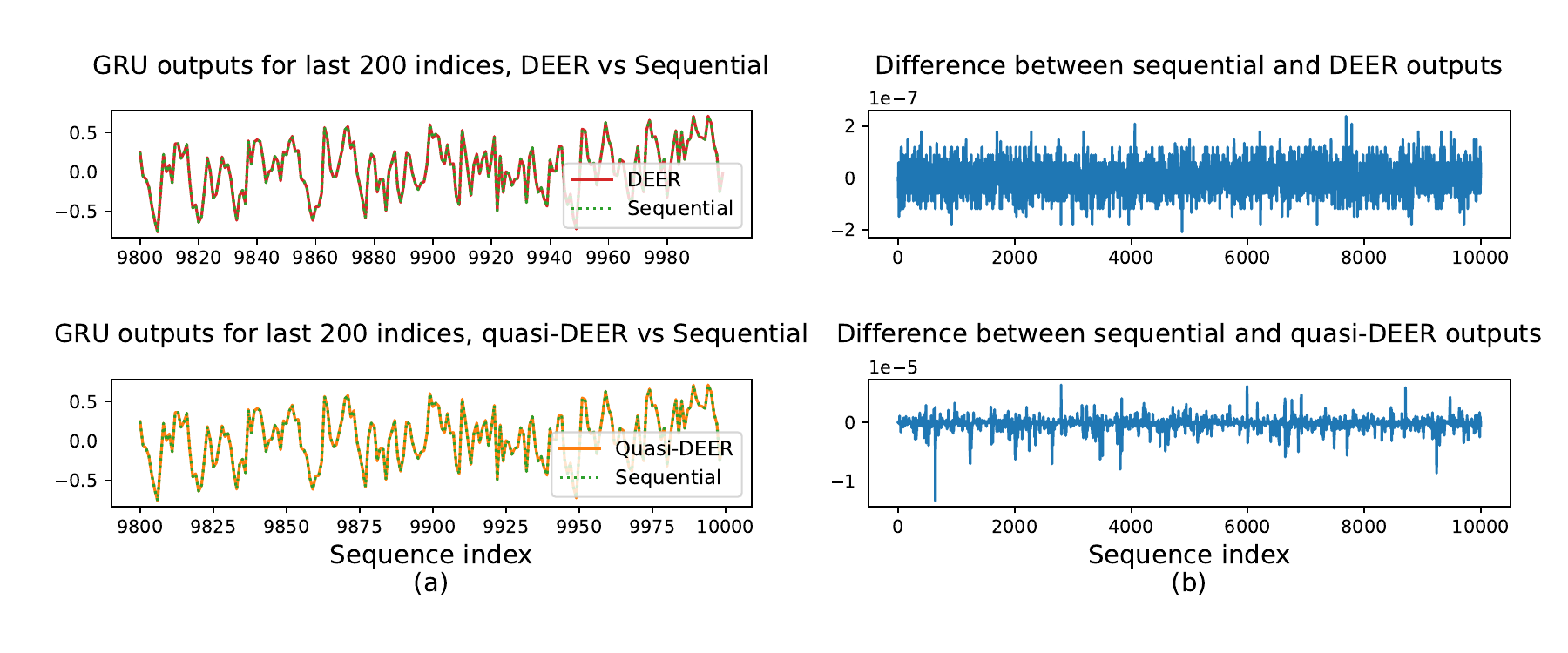}
  \caption{The accuracy of evaluating with parallelized methods (DEER and quasi-DEER) as opposed to sequential evaluation.
  The parallelized methods converge to the correct trace within numerical precision.
  The hidden state size is $D=4$ and the sequence length is $T=10,000$.
}\label{fig:quasi_acc}
\end{figure}

Having confirmed the accuracy of the parallel Newton methods, we now compare the wall-clock time and memory usage of sequential evaluation, DEER, and quasi-DEER.
Results shown in \Cref{fig:untrained_gru}.
\begin{figure}
  \centering
  \includegraphics[width=0.95\textwidth]{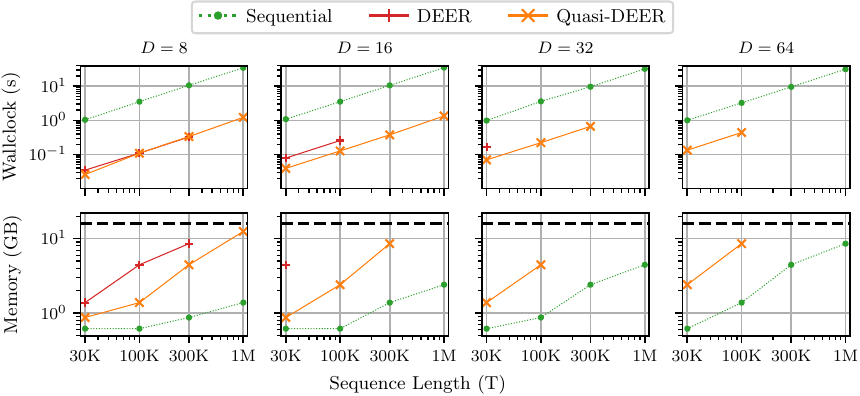}
  \vspace*{-0.0cm}
  \caption{\textbf{Evaluating an untrained GRU.} Relative performance of sequential, DEER and quasi-DEER for evaluating a randomly initialized (and untrained) GRU on (\textbf{Top Row}) wall-clock time, averaged over 20 random seeds and (\textbf{Bottom Row}) memory, averaged over 3 random seeds.
  All experiments use a 16GB V100 SMX2 (memory capacity indicated by the black dashed line) and Newton methods were run to convergence.
  Missing points in each series indicate the GPU ran out of memory. 
  In these settings, quasi-DEER has a runtime commensurate with DEER, but with lower memory consumption. Therefore, quasi-DEER can work at scales where DEER cannot.
  \label{fig:untrained_gru}
}
\end{figure}
Both DEER and quasi-DEER are up to twenty times faster than sequential evaluation.  
The runtimes are similar between DEER and quasi-DEER for small networks, because although quasi-DEER steps are faster, quasi-DEER takes more iterations to converge.  
For larger networks, the difference in runtime is more pronounced.
We also see that quasi-DEER requires as much as an order of magnitude less memory than DEER, thus allowing the application to architectural regimes previously infeasible with DEER.  

In Figure~\ref{fig:larger_regime}, we run the timing benchmarks of Section~\ref{ssc:gru_bench} on a wider range of sequence lengths $T$ and hidden state sizes $D$, on a larger GPU (a V100 with 32 GB) and with a smaller batch size of 1.
In doing so, we highlight the parallel nature of DEER and quasi-DEER, as their wall-clock time scales sublinearly in the sequence length $T$ in smaller ($D$, $T$) regimes.
However, we note that in the larger regimes considered in our main text and in \citet{deer2024}, we often observe linear scaling in the sequence length $T$ for the wall-clock time of DEER and quasi-DEER, even though these algorithms are still faster than sequential evaluation. 
Figure~\ref{fig:larger_regime} shows good evidence that these parallel algorithms are suffering from saturation of the GPU, and would benefit from even more optimized implementations.

\begin{figure}
  \centering
  \includegraphics[width=\textwidth]{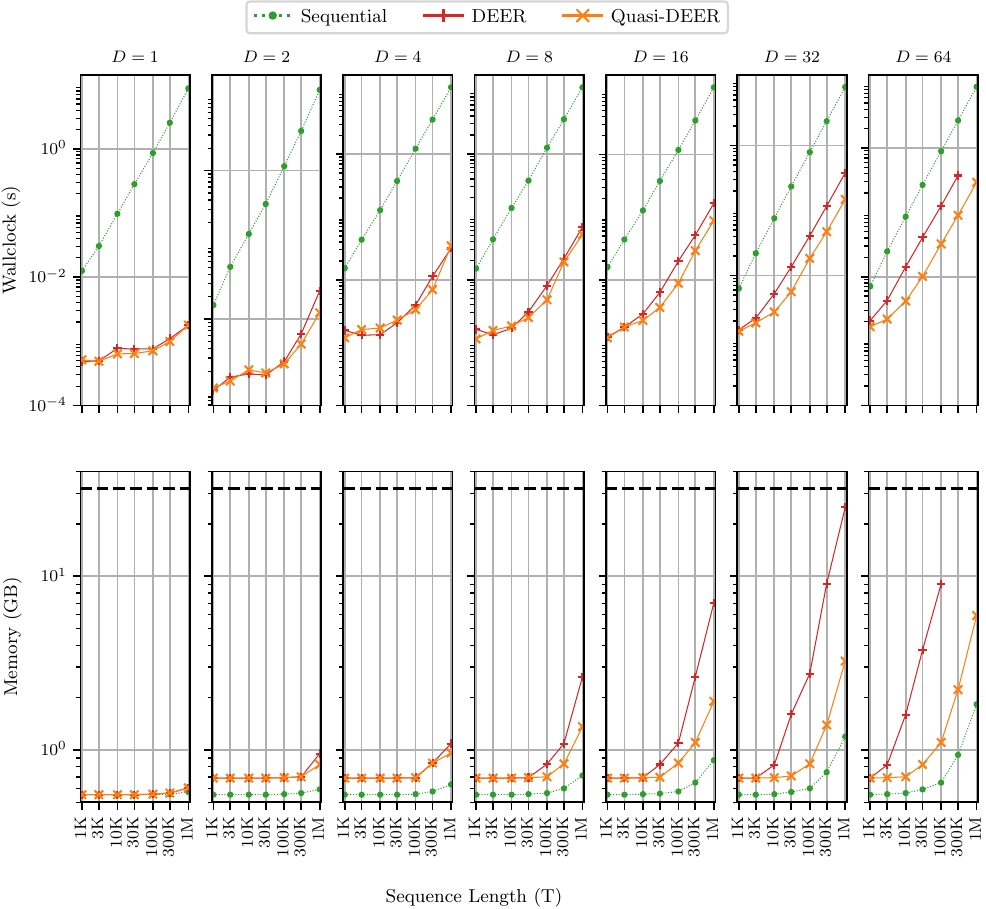}
  \caption{\textbf{Evaluating an untrained GRU.} Sublinear and linear timing regimes for parallelized algorithms. The above experiments were run on a 32 GB V100 with a batch size of 1.
  As in Figure~\ref{fig:untrained_gru}, we use 20 seeds for timing, 3 seeds for memory, and the dashed black line indicates the memory capacity of the GPU (32 GB).
  We observe that in smaller regimes in $D$ and $T$ that the wall-clock time shows sublinear scaling indicative of the use of parallel algorithms.
  However, when the GPU becomes saturated, the benefits of parallelization are reduced and we begin to see linear scaling in wall-clock time with $T$.}
  \label{fig:larger_regime}
\end{figure} 

The parallel scan, given sufficiently many processors, scales as $O(\log T)$. As we show in Figure~\ref{fig:larger_regime}, we see this speedup at low model sizes and sequence lengths. Once the processors are saturated, we see a linear increase in the runtime (since the amount of work done is linear), but it is making much more effective use of the GPU, resulting in a constant factor speedup over sequential application at larger model sizes/sequence lengths. 

Together, these experiments confirm that quasi-DEER can replicate the performance of DEER, but with a smaller memory footprint.  

\subsection{Quasi-DEER for Training}
\label{ssc:worms}
We verify that quasi-DEER expedites training nonlinear RNN models.
We replicate the third experiment from \citet{deer2024}, where a GRU is trained to classify \emph{C. elegans} phenotypes from the time series of principal components of the worms' body posture~\citep{eigenworms}.  
This task is colloquially known as the "eigenworms" task.
With a sequence length of $T=17,984$, it is the longest task on the UEA Multivariate Time Series Classification archive, a standard benchmark set for assessing the performance of sequence models on long sequences \citep{bagnall2018uea}.

\begin{figure}
    \centering
    \includegraphics[width=\textwidth]{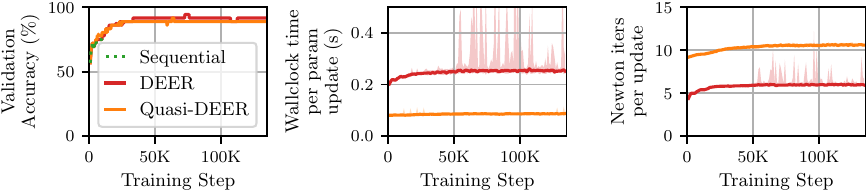}
    \caption{\textbf{Training a GRU with DEER.}
    Comparison of DEER and quasi-DEER during GRU training for the \emph{C. elegans} time-series classification task (Section~\ref{ssc:worms}).  Each time series has length $T=17,984$.  We show the median, and 5-95\% interval across a rolling window of 20 training steps.  \textbf{(Left)}  DEER and quasi-DEER have the similar validation accuracy trajectories, indicating similar training dynamics.  The sequential trace shown is for 24 hours of training (compared to 11 and 4 hours for the whole DEER and quasi-DEER traces).  \textbf{(Center)}  Each quasi training iteration is 2.5 times faster than each DEER training iteration.  Sequential training steps took more than 6 seconds each (not pictured).   \textbf{(Right)}  Each quasi training iteration requires (approximately) 2 times more Newton iterations to converge, indicating that each quasi Newton step is approximately 5 times faster than the corresponding DEER Newton step.}
    \label{fig:eigenworms}
\end{figure}
We show results in Figure~\ref{fig:eigenworms}.
We see that the training dynamics under quasi-DEER leads to similar validation accuracy trajectories.
However, every quasi-DEER training step is faster by a factor of $2.5$, despite performing around $2$ times more Newton iterations per training step.
This finding highlights how quasi-DEER can improve DEER when training nonlinear RNNs.
In our experiment, we use the quasi-DEER approximation for the backward pass as well, leading to gradients that are different from DEER in this setting. In this particular experiment, we found that there was very little degradation in performance (Figure~\ref{fig:eigenworms}, left). Nonetheless, in general we recommend modifications to quasi-DEER that also allow for an exact backwards pass: see the discussion in \Cref{ssc:bwds}.

The RNN used in this experiment is a 5 layer GRU. 
When we evaluate this architecture in parallel, we evaluate each layer in parallel using (quasi)-DEER. In Figure \ref{fig:eigenworms} (right), we report the number of (quasi)-DEER iterations averaged over all layers and batches.

\section{Further development and directions for future work}\label{sec:quasi_future}

Since publication of this diagonal quasi-Newton method at NeurIPS in 2024, there have been many extensions. 
There are also many interesting avenues for future work. This section highlights additional important ideas and future directions for quasi-Newton methods for parallelizing nSSMs.

\subsection{Efficiently Estimating the Diagonal of the Jacobian}\label{ssc:efficiency}

\paragraph{Efficiently Estimating the Diagonal of the Dynamics Jacobian}

The diagonal approximation presented in \cref{eq:quasi_lds} uses the $\{ \mathrm{diag}(\A_t) \}$, the diagonals of the dynamics Jacobian, to be significantly more memory and work efficient. However, an important question is: how does one acquire these diagonals?

The simplest approach is to compute the $\{ \A_t \}$ with autodifferentiation, and then take their diagonals.
This simple approach still decreases the required work during the parallel scan by a factor of $D^2$. We use this approach in the "eigenworms" experiment in \Cref{fig:eigenworms}, where we show empirically in this setting that this simple approach can still yield substantial speedups. However, the price of this simplicity is that we do not unlock all of the benefits of quasi-DEER. For example, this approach offers no savings on peak memory utilization. Furthermore, during autodifferentiation, we still require $D$ function calls.

An approach to unlock the full benefits of quasi-DEER is to compute $\{ \mathrm{diag}(\A_t) \}$ analytically and implement its closed form directly. We follow this approach in \Cref{fig:untrained_gru} and \Cref{fig:larger_regime}, demonstrating substantial memory savings.

Nonetheless, computing derivatives by hand has many drawbacks, and for sufficiently complex dynamics functions $\mathrm{diag}(\A_t)$ may not even have an implementable closed form.
For this reason, \citet{pmcmc} takes a different approach: provide a \emph{stochastic estimator} of $\mathrm{diag}(\A_t)$ that requires only $\mathcal{O}(D)$ memory and one function call. This approach leverages the \emph{Hutchinson estimator} for the diagonal of a matrix \citep{hutchinson1989stochastic, bekas2007estimator, yao2021adahessian}.

Consider a matrix $A$. The Hutchinson estimator $\hat{A}$ for $\mathrm{diag}(A)$ is
\begin{equation}\label{eq:hutch}
    \hat{A} = v \odot A v,
\end{equation}
where each entry of $v$ is an iid draw from a Rademacher random variable, i.e. $v=1$ with probability $\nicefrac{1}{2}$ and $v=-1$ otherwise; and where $\odot$ represents elementwise multiplication of two vectors. $\hat{A}$ is an unbiased estimator for $\mathrm{diag}(A)$ as $\E[\hat{A}]~=~\mathrm{diag}(A)$.

As presented in \cref{eq:hutch}, the Hutchinson estimator $\hat{A}$ seems a bit silly: we already knew $A$, and so could have just taken $\mathrm{diag}(A)$ directly. However, say we want to instead find the diagonal of $\dfrac{\partial f}{\partial s}(s)$---which is exactly what we need to run quasi-DEER---\emph{without ever instantiating the $D \times D$ matrix} (which requires wasteful memory and compute costs). After sampling the Rademacher variable $v$, we can compute the matrix-vector product  $\dfrac{\partial f}{\partial s}(s) v$ with a single \emph{Jacobian vector product} (JVP), which requires only a single pass through $f$.

A JVP is a standard primitive in automatic differentiation\footnote{ See \citet{baydin2018automatic} and \citet{maclaurin_thesis} for more details on automatic differentiation, more commonly and colloquially known as "autodiff"} libraries like JAX \citep{jax2018github} and PyTorch \citep{paszke2019pytorch}. A JVP takes in a simple function $f$ and a tangent vector $v$, and returns the product of the Jacobian of $f$ with the tangent vector $v$. 
By virtue of the chain rule, if JVPs for a suitable basis of functions are defined in an autodifferentiation library, one can evaluate derivatives for wide-classes of functions. In fact, the Jacobian $\nicefrac{\partial f}{\partial s}(s)$ can be obtained from $D$ number of JVPs, one for each basis vector.

Consequently, the Hutchinson estimator obtains an unbiased estimator for the diagonal of the dynamics Jacobian. However, the Hutchinson estimator never needs to instantiate the $D \times D$ matrix $\mathrm{diag}(\nicefrac{\partial f}{\partial s}(s))$, and requires only a single function call. Thus, the Hutchinson estimator shares the same cost in memory and compute as analytically implementing $\mathrm{diag}(\nicefrac{\partial f}{\partial s}(s))$---but even when the closed form of $\mathrm{diag}(\nicefrac{\partial f}{\partial s}(s))$ is difficult to obtain, the Hutchinson estimator can still be computed easily using standard autodifferentiation libraries. The variance of the Hutchinson estimator can be reduced by using more Rademacher random variables. Moreover, if the Jacobian $\dfrac{\partial f}{\partial s}$ truly is diagonal, then the Hutchinson estimator is exact. In any case, because of \Cref{prop:global_convergence}, we know that substituting an approximate Jacobian based on the Hutchinson estimator will \emph{still} converge globally. Finally, \citet{pmcmc} provides a variety of empirical demonstrations showing the strong performance of the Hutchinson estimator for parallelizing the sampling of complicated, high-dimensional distributions via Markov chain Monte Carlo.

In conclusion, if the desired diagonal is tractable analytically and performance is paramount, implementing the derivative directly may yield the most efficient performance. However, if computing the diagonal is intractable or unwieldy for prototyping many functions $f$, the Hutchinson estimator introduced in \citet{pmcmc} allows for the use of autodifferentiation---with comparable memory and compute costs---to obtain a practically useful estimate.

\subsection{Generalizing quasi-DEER to other approximate Jacobians}\label{ssc:approx}

As we will formalize in \Cref{ch:quasi_convergence}, the closer our approximate dynamics matrices $\Tilde{\A}_t$ are to the true Jacobians $\A_t$, the faster the rate of convergence. Moreover, we know from \Cref{prop:global_convergence} that any approximate Jacobian will still result in global convergence. Thus, a major direction of future research in quasi-Newton methods is finding other structured matrices that improve expressivity while retaining efficiency.

\paragraph{Reparameterizing the dynamics to be diagonal}

Clearly, if the dynamics are axes-aligned, then the Jacobian is a diagonal matrix and the diagonal approximation is exact.
If the dynamics are not axes-aligned, but there exist some coordinate transform on the $s_t$ to make the dynamics axes-aligned, then we could also run quasi-DEER on these reparameterized dynamics to enjoy the efficiency of quasi-DEER with the convergence speed of full DEER.
However, even if each matrix individually is diagonalizable, it is not always possible to find a basis in which a set of matrices $\{ \A_t \}$ are mutually diagonalizable.
Nonetheless, even if we only \emph{approximately} diagonalize the $\{ \A_t \}$, we know from \Cref{prop:global_convergence} that the resulting quasi-DEER will still globally converge, and may still be much faster than just taking the diagonal approximation. Ways to obtain such an approximate joint diagonalization include taking some representative matrix, such as the first Jacobian $\A_1$ or an average of all the Jacobians, and finding its eigenbasis. 
In general, such a resulting eigenbasis is complex-valued. For reasons still not fully understood, such a complex-valued reparameterization struggles with convergence, especially on GPUs, likely indicating an issue with numerical precision.

However, an elegant approach for reparameterization taken by \citet{pmcmc} is to use a \emph{real} eigenbasis. This real eigenbasis is obtained by symmetrizing the representative matrix before finding its eigenbasis. This approach is particularly well-suited to the context of parallelizing MCMC because the dynamics Jacobian in Langevin dynamics \citep{Langevin1908Eng}---a common sampling approach that is a backbone of the MALA MCMC algorithm \citep{besag1994comment}---is already a real symmetric matrix because it is the Hessian of the log probability of the target distribution $p$. \citet{pmcmc} demonstrate across a wide-range of experiments that this reparameterization using a real-valued eigenbasis is a robust, efficient, and effective method for parallelizing MCMC over the sequence length. 

A final consideration around reparameterization is its computational cost. Another advantage of reparameterizing in MCMC is that the cost of the eigendecomposition is a fixed, one-time cost for a particular kernel---whereas in the context of parallelizing RNNs, one may have to frequently rediagonalize to account for the change in dynamics across gradient updates.

\paragraph{Using other structured matrices in the parallel scan}

In quasi-DEER as presented in \cref{eq:quasi_lds}, we used $\mathrm{diag}(\A_t)$ for our approximate dynamics matrix $\Tilde{A}_t$. We chose to use the diagonal because the composition of diagonal matrices is closed---multiplying two diagonal matrices together yields another diagonal matrix---and closure of the operation is required to use the parallel scan.

Nonetheless, \Cref{prop:global_convergence} shows that \emph{any} approximate matrix $\Tilde{A}_t$ will still result in global convergence. For example, in \Cref{ch:quasi_convergence}, we will show that many common fixed-point methods---including Jacobi and Picard iterations---can also be interpreted as versions of quasi-DEER with different types of approximation techniques used to form $\Tilde{A}_t$.

Therefore, it is natural to ask what other types of structured matrices $\Tilde{A}_t$ can be easily computed and are closed under composition. For example, in parallelizing Hamiltonian Monte Carlo (HMC)\citep{neal2011mcmc, betancourt2017conceptual}, which includes both position and momenta variables, \citet{pmcmc} demonstrated that "diagonal-block" matrices\footnote{i.e. a block matrix where every block is a diagonal matrix} satisfy these desiderata. Under permutation of the coordinates, these diagonal-block matrices are equivalent to block diagonal matrices. A benefit of block diagonal matrices is that they can better utilize the tensor cores of GPUs.

Other possibilities for future work include developing quasi-DEER methods based on parallel scan for other structured matrices, such as low-rank matrices. For example, \citet{terzic2025permutation} developed an efficient parallel scan for permutation matrices, which could be an intriguing option for quasi-DEER in certain settings. Moreover, other matrices such as Householder matrices \citep{bischof1985wy} are not well-suited to parallel scans, but admit a chunkwise parallel form that has achieved great success for language modeling in the DeltaNet architecture \citep{deltanet, yang2024parallelizing}. In general, there are many varieties \citep{sindhwani2015structured, dao2022monarch} of structured matrices that all merit further exploration for use in the context of parallelizing nonlinear SSMs, whether using parallel scans, chunkwise parallel approaches, or other as yet unimagined schemes.

\paragraph{Foregoing autodiff and using Broyden type methods}

A unique aspect of the quasi-DEER methods discussed in this chapter when compared with the broader quasi-Newton literature (cf. \citep{dennis1996numerical, NocedalWright}) is the manner in which the approximate derivative $\Tilde{\J}$ is constructed.
In all of the instantiations of quasi-DEER discussed above, we in some way differentiate the residual $\mathbf{r}(\cdot)$ at \emph{every iteration}; and then we use an approximation of this derivative to reduce the memory and compute requirements of the parallel scan we use to evaluate the resulting LDS.

However, much of the quasi-Newton literature---especially the widely-used \emph{Broyden methods}
\citep{broyden1970convergence, liu1989limited, dennis1996numerical, fang2009multisecant}---is motivated by trying to avoid the computational cost of differentiating $\mathbf{r}(\cdot)$ itself.\footnote{In contrast, the quasi-DEER methods have accepted the cost of differentiating $\mathbf{r}(\cdot)$, and instead focus on reducing the cost of the next step, which is the parallel scan.} In Broyden methods, an approximation to either $\J$ or $\J^{-1}$ is built up over the optimization trajectory using only information gleaned from the trajectory itself (primarily the values $\{\mathbf{s}^{(i)}\}$ and $\{\mathbf{r}(\mathbf{s}^{(i)})\}$). 

As discussed in \citet{dennis1996numerical}, building up an approximation $\Tilde{\J}$ for $\J$ is called \emph{Broyden's first method} or \emph{Broyden's good update}; building up an approximation $\mathbf{G}$ for $\J^{-1}$ is called \emph{Broyden's second method} or \emph{Broyden's bad update}.
A seeming advantage of the so-called "bad update" is that by approximating $\J^{-1}$ directly, one does not have to bear the cost of the matrix inversion.
However, the reason for this colorful nomenclature is the robust observation across practitioners that Broyden's good update tended to outperform Broyden's bad update in application (cf. \citep{dennis1996numerical}); \citet{lin2021explicit} provides theoretical analysis suggesting that the good update is more robust in a wider range of initializations.

\citet{tang2024accelerating} used Broyden's bad update to parallelize the evaluation of nonlinear SSMs---in their chosen setting, sampling from diffusion models.
Building on this work, future work that leans more deeply into the rich literature of Broyden methods---especially Broyden's good update---could have important implications for parallelizing nonlinear SSMs.

\subsection{Training and the backwards pass}\label{ssc:bwds}

To train an RNN, we need both the forward pass (which fills in the state trajectory $\mathbf{s}_{1:T}$) and the backward pass (which computes the gradient of some loss function with respect to the RNN parameters $\theta \in \mathbb{R}^P$, and updates those parameters accordingly).

In this chapter, we have primarily focused on how DEER and quasi-DEER let us parallelize the \emph{forward} pass of an RNN, up to numerical precision. However, we should also discuss how to parallelize the backwards pass as well, specifically the fact that \textbf{DEER also has an exact backward pass that is parallelized across the sequence length.}

To show this, let us consider an RNN cell parameterized by $\theta$, i.e. $s_t = f_\theta(s_{t-1})$, where here $s_t$ represents the RNN hidden state. Assume we want to train our RNN to minimize some supervised scalar loss $L(s_T)$ that is explicitly a function of the final RNN hidden state, but of course depends recurrently on all of the RNN hidden states $\mathbf{s}_{1:T}$ and the RNN cell parameters $\theta$.

In modern deep learning, optimization is always done using updates to the parameters $\theta$ based on some function of the gradient of the loss $L$ with respect to the parameters. This derivative is computed during the "backwards pass," i.e. backpropagation or the chain rule. In the context of RNNs, this approach to finding the derivative is also called \emph{backpropagation through time (BPTT)}. This name emphasizes that we are applying the chain rule over dependencies over the sequence length (which, especially in neuroscience applications, can be thought of as time).

Therefore, using the chain rule to compute $\nicefrac{dL}{d\theta}$, it follows that
\begin{align}
    \dfrac{d L}{d \theta}(s_T) & = \dfrac{\partial L}{\partial s_T}(s_T) \dfrac{d s_T}{d \theta}(s_{T-1}) \label{eq:bptt} \\
    \dfrac{d s_t}{d \theta} & =  \underbrace{\dfrac{\partial s_t}{\partial s_{t-1}}}_{\A_t} \dfrac{d s_{t-1}}{d \theta} + \dfrac{\partial f}{\partial \theta}(s_{t-1}), \label{eq:bptt_lds}
\end{align}
where the $\A_t$ are exactly the dynamics function Jacobians that DEER uses to parallelize the forward pass, and we can compute $\dfrac{\partial f}{\partial \theta}$ over all $\mathbf{s}_{1:T}$ in an embarrassingly parallel manner using a \textsc{Map}.
Moreover, \cref{eq:bptt,eq:bptt_lds} indicate that BPTT is an LDS, which we know how to parallelize using a parallel scan. 
In more detail, unrolling this recursion in \cref{eq:bptt,eq:bptt_lds}, it follows that
\begin{equation}\label{eq:unrolled_bwds}
    \underbrace{\dfrac{d L}{d \theta}(s_T)}_{\in \mathbb{R}^{1 \times P}} = \sum_{t=1}^T \underbrace{\dfrac{\partial L}{\partial s_T}(s_T)}_{\in \mathbb{R}^{1 \times D}} \cdot \underbrace{\prod_{\tau=T}^{t+1}\A_{\tau}}_{\in\mathbb{R}^{D \times D}} \cdot \underbrace{\dfrac{\partial f}{\partial \theta}(s_{t-1})}_{\in \mathbb{R}^{D \times P}}.
\end{equation}
We observe that we can obtain all of the products $\prod_{\tau=T}^{t+1}\A_{\tau}$ with a parallel scan, showing how we can also parallelize the backwards pass.

To summarize, while DEER may need multiple LDSs to achieve the exact forward pass, when it has converged it has the exact matrices $\{ A_t \}$ needed for the backward pass, which can then also be parallelized with a single parallel scan.

However, this single parallel scan backwards would incur all of memory and compute costs that quasi-DEER sought to avoid. For this reason, in the eigenworms experiment shown in \Cref{fig:eigenworms}, we also use the $ \{ \tilde{A}_t = \mathrm{diag}(\A_t) \}$ for the \emph{backwards} pass shown in \cref{eq:unrolled_bwds} as well. As we run only one parallel scan in the eigenworms experiment, we are using an \emph{approximate} gradient that is not equal to the true gradient that would arise from either DEER or sequential evaluation.
This approach using approximate gradients worked empirically for training in the eigenworms experiment shown in \Cref{fig:eigenworms}.
Furthermore, \citet{caillon2025fast} used an even more approximate gradient--choosing $\tilde{A}_t$ be a \emph{random} diagonal matrix---though their language modeling experiments show degraded performance relative to full BPTT.
Because of the alteration of the training dynamics, we do \textbf{not} recommend training with approximate gradients in general. We instead recommend the following two alternatives to obtain exact backwards passes in a computationally feasible manner.

First, \citet{farsang2025scaling}, \citet{danieli2025pararnn}, and \citet{zattra2025context} take the approach of adjusting the RNN architectures \emph{to have diagonal\footnote{or block-diagonal.} Jacobians}, with \citet{danieli2025pararnn} demonstrating that this approach scales to strong language modeling performance with 7 billion parameter models. In such architectures, the quasi-DEER approximate Jacobian is actually \emph{exact}, i.e. the memory and compute costs of running full DEER are reduced to be linear in $D$. Furthermore, the backwards pass is now exact as well. While restricting the architecture in this way would seem intuitively to reduce its expressivity, precisely and rigorously investigating this intuition is an important avenue for future work. 
Moreover, just as using richer structured matrices could increase the convergence speed of quasi-DEER (see our discussion in \Cref{ssc:approx}), so too could they improve the expressivity of RNN architectures. 

Second, just as both DEER and quasi-DEER use multiple parallel scans to obtain an exact forward pass, we could also use multiple parallel scans to obtain an exact \emph{backwards} pass for quasi-DEER. In particular, observing \cref{eq:unrolled_bwds}, we can treat $\tilde{s}_T := \nicefrac{\partial L}{\partial s_T}(s_T) \in \mathbb{R}^D$ as an initial state\footnote{Note the reversal of time for the backwards pass.} in an LDS with transition function being multiplication by $\A_t^{\top}$. 
We can then implement the resulting quasi-DEER \emph{backwards} update along the lines of the form shown in \cref{eq:quasi_lds} where we can use $\diag[\A_t]$ for our transition matrix and use vector-Jacobian products (VJPs) of $f_t$ to efficiently compute $A_t^{\top} \tilde{s}_t^{(i)}$ for the transition function.\footnote{We can use a VJP of $f_t$ to compute $A_t^{\top} \tilde{s}_t^{(i)}$ because $\A_t := \nicefrac{\partial f_t}{\partial s_{t-1}}(s_{t-1}^{\star})$.} While such an approach provably must converge, assessing its efficacy empirically is an interesting avenue for future work. However, we note that a highly related idea called Highway backpropagation has already been shown to accelerate the training of GRUs for character level language modeling \citep{fagnou2025accelerated}. 

In conclusion, DEER enjoys exact forward and backwards passes, using multiple parallel scans to achieve the forward pass and a single parallel scan for the backwards pass. Quasi-DEER as implemented in \Cref{fig:eigenworms} enjoys an exact forward pass, but only an approximate backwards pass, as it uses only a single parallel scan for the backwards pass. Using an exact backward pass is important in general, and can be achieved by restricting the architecture to make the $\tilde{A_t}$ exact, or by using multiple parallel scans in the backward pass as well. Of course, there are other ways beyond BPTT to train RNNs, including e-prop \citep{bellec2020solution}, forward-mode optimization \citep{sofo}, evolutionary methods \citep{eggroll}, and zeroth order methods \citep{chaubard2025scaling}.

\subsection{Initializing the guess for the state trajectory}

An important consideration for parallel Newton methods is how to initialize the initial guess for the state trajectory $\mathbf{s}_{1:T}^{(0)}$. 
As we saw in \Cref{prop:Newton_quadratic}, Newton's method enjoys quadratic convergence if it is initialized close to the true solution $\mathbf{s}^{\star}$.
However, in general, picking an initial guess for the trajectory that is close to the true trajectory $\mathbf{s}_{1:T}^{\star}$ can be as difficult as finding the true trajectory $\mathbf{s}_{1:T}^{\star}$ itself. An exception can be if an approximate trajectory is already known, which could arise if we were training an RNN on a single sequence (so the state trajectory does not change too much with each training step), or conducting sensitivity analysis in Markov chain Monte Carlo (so each chain is close).
    
In this chapter, the parallel Newton methods were initialized from all zeros. Consequently, the initial dynamics matrices $\A_t^{(0)}$ are all the same, which can exacerbate instability. A better approach used in \Cref{part:theory} is to initialize at random.

Probably the best approach presented in the literature comes from \citet{danieli2025pararnn}, which uses one Jacobi iteration (starting from all zeros) to initialize the states, i.e. $s_t^{(0)} = f(0, u_t)$, where $u_t$ are the inputs to the RNN. This Jacobi iteration is embarrassingly parallel over the sequence length, and can provide a good initialization for the parallel Newton methods.
Further research in even better initialization may prove fruitful.
%************************************************
\chapter{Stable Parallelization: ELK and Trust Region Methods}\label{ch:elk}
%************************************************

Another well-known failure mode of Newton's methods is instability---the failure mode where Newton's method diverges, with the solutions growing in magnitude instead of converging to the solution $\mathbf{s}^{\star}$ (cf. \Cref{fig:newton_diverges}).
This exact failure mode of never converging does not directly apply to parallel Newton methods, which are guaranteed by \Cref{prop:global_convergence} to converge in at most $T$ iterations.
However, if parallel Newton methods take too many iterations to converge, they will be slower than sequential evaluation, defeating the goal of parallelization in the first place.
As we discuss in this chapter, a failure mode that can slow down the convergence of Newton methods, especially in finite precision, is when intermediate iterates $\mathbf{s}^{(i)}$ explode in value.
To overcome this slowed convergence caused by instability, we introduce a parallelized trust-region optimizer called ELK: \textbf{E}valuating \textbf{L}evenberg-\textbf{M}arquardt with \textbf{K}alman.

\section{Levenberg-Marquardt and Trust-Region Methods}

For the purpose of stabilizing parallel Newton methods, we take the optimization perspective discussed in \Cref{ssc:opt}, focusing on the merit function $\mathcal{L}(\mathbf{s})$ introduced in \cref{eq:merit}. However, instead of optimizing this merit function with the Gauss-Newton algorithm (GN) (i.e. DEER), we will instead use the \emph{Levenberg-Marquardt} (LM) algorithm \citep{levenberg1944method, marquardt1963algorithm}, one of the most standard \emph{trust-region} approaches.
\begin{figure}
    \centering
    \includegraphics[scale=0.1]{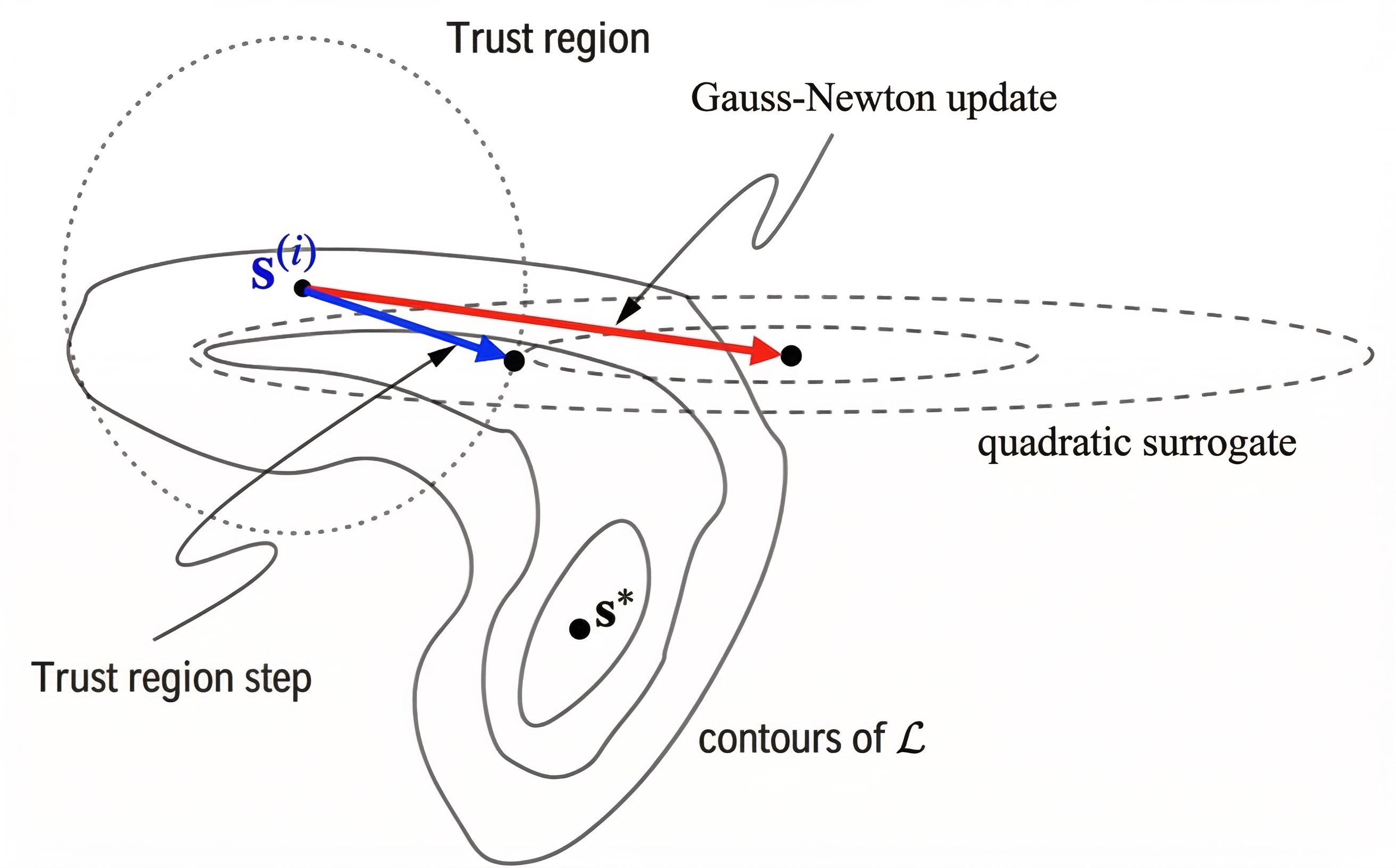}
    \caption{\textbf{Graphical Depiction of Trust-Region Methods.} We show both an undamped Gauss-Newton step (red) and a stabilized trust-region step (blue). The solid lines indicate the contours of the merit function $L$ we want to minimize. The dashed lines indicate the contours of the quadratic surrogate that Gauss-Newton is minimizing on this iteration. The dotted lines indicate the \emph{trust region} around $\mathbf{s}^{(i)}$; trust-region methods restrict the update to this ball, resulting in this case in an update that reduces the objective. Figure adapted from \citet[Figure 4.1]{NocedalWright}.}
    \label{fig:trust_region}
\end{figure}
The idea of a trust region is simple, and is depicted in \Cref{fig:trust_region}, which is adapted from Figure 4.1 of \citet{NocedalWright}.
The core idea of trust-region methods is that the quadratic surrogate being minimized by the Gauss-Newton method may only be accurate or helpful in a neighborhood of our current guess $\mathbf{s}^{(i)}$. 
Thus, trust-region methods require that the next iterate $\mathbf{s}^{(i+1)}$ minimize the merit function $\mathcal{L}(\mathbf{s})$, subject to being in some neighborhood of the current guess $\mathbf{s}^{(i)}$.
Trust regions are often used in conjunction with Newton's method to improve numerical stability and convergence.
Each Gauss-Newton step solves an \emph{unconstrained} optimization problem, while each trust-region step solves a \emph{constrained} optimization problem.

The Levenberg-Marquardt algorithm in particular is a canonical trust-region method.
Let us define the quadratic surrogate that Levenberg-Marquardt is minimizing at each iteration $(i)$ as a function of the step $\Delta \mathbf{s}$ it takes, i.e.
\begin{align} \label{eq:linmerit}
    \widetilde{\mathcal{L}}_{\mathbf{s}^{(i)}}(\Delta \mathbf{s}) = \frac{1}{2} \left\| \mathbf{r}(\mathbf{s}^{(i)}) + \J(\mathbf{s}^{(i)}) \Delta \mathbf{s} \right\|_2^2.
\end{align}
Then, LM uses updates that solve the constrained optimization problem
\begin{align}\label{eq:constr_opt}
    \min_{\Delta \mathbf{s}} \widetilde{\mathcal{L}}_{\mathbf{s}^{(i)}}(\Delta \mathbf{s}) \quad \text{ subject to } \left\| \Delta \mathbf{s} \right\|_2 \leq D_{i+1},
\end{align}
where $D_{i+1}$ is an upper bound on the step size, thus defining our trust region.

Note that both the objective $\widetilde{\mathcal{L}}_{\mathbf{s}^{(i)}}$ and the constraint $g(\Delta \mathbf{s}) := \| \Delta \mathbf{s} \|_2^2 - D_{i+1}^2$ are convex in $\Delta s$ (in fact, both are quadratic).
Therefore, by the method of Lagrange multipliers, minimizing the constrained optimization problem in \cref{eq:constr_opt} is equivalent to minimizing the Lagrangian
\begin{align}\label{eq:lagrangian}
     \widehat{\mathcal{L}}_{\mathbf{s}^{(i)}}(\Delta \mathbf{s}) &= 
     \widetilde{\mathcal{L}}_{\mathbf{s}^{(i)}}(\Delta \mathbf{s}) + \frac{\lambda_{i+1}}{2} \left\| \Delta \mathbf{s} \right\|_2^2
\end{align}
over $\Delta \mathbf{s}$ for some fixed $\lambda_{i+1} \geq 0$.
Note that if $\lambda_{i+1} = 0$, then the \emph{unconstrained} minimizer of $\widetilde{\mathcal{L}}_{\mathbf{s}^{(i)}}(\Delta \mathbf{s})$
is inside of the trust region.

Since \cref{eq:lagrangian} is quadratic in $\Delta \mathbf{s}$, it follows that
\begin{equation*}
    \Delta \mathbf{s}_{\mathrm{LM}} = - \left( \J^{\top} \J + \lambda \mathbf{I} \right)^{-1} \J^{\top} \mathbf{r}.
\end{equation*}
Therefore, we observe that if $\lambda = 0$, then we recover the typical GN step.
On the other hand, if $\lambda$ is large, we see that the LM update approaches the gradient descent (GD) update with step size $1 / \lambda$.

Intuitively, LM is "regularizing" the update by adding a non-negative term to the diagonal of $\J^{\top} \J$.
This regularization can help to stabilize the update when $\J^{\top} \J$ has small eigenvalues, resulting in an update with smaller and more manageable magnitude.
In fact, this stabilization technique used by LM is exactly analogous to the regularization technique of \emph{ridge regression} or \emph{$\ell_2$-regularization} used in statistical machine learning \citep{tikhonov1963solution, hoerl1970ridge, krogh1991simple, hastie2009elements, hastie2020ridge, wang2025test}.

However, while it is intuitive why LM can help stabilize the GN updates, it is not immediately obvious how we can parallelize the LM update over the sequence length to help us achieve our goal of parallelizing the evaluation of nonlinear SSMs.
In the next section, we will show how we can parallelize LM updates in our setting via a connection with Kalman smoothing.

\section{ELK: Evaluating Levenberg-Marquardt with Kalman}

There is a rich literature connecting optimization techniques with the problem of filtering and smoothing\footnote{For background on filtering and smoothing, see our introduction to Bayesian filtering and smoothing in \Cref{ssc:kalman}, and \citet{sarkka2023bayesian} for the standard textbook introduction.}.
In particular, \citet{bell-filter} and \citet{bell-smoother} draw connections between the Gauss-Newton method and the iterated extended Kalman filter and smoother~\citep{Sorenson1966, sarkka2023bayesian}.
Because Gauss-Newton is unstable, it is natural to use Levenberg-Marquardt~\citep{levenberg1944method, marquardt1963algorithm} to stabilize the filtering/smoothing problem~\citep{Sarkka-lm, Chen2013, mandel2016hybrid}.

This connection between optimization and Kalman smoothing hinges on the following point noted by \citet{Sarkka-lm}: the minimizer of this Lagrangian in \cref{eq:lagrangian} can be obtained by a Kalman smoother. 
We emphasize this connection in the following proposition.
\begin{proposition}\label{prop:ELK}
    Solving for the Levenberg-Marquardt update that minimizes \eqref{eq:lagrangian} with fixed $\lambda_{i+1}$ is equivalent to finding the \emph{maximum a posteriori} (MAP) estimate of $\mathbf{s}_{1:T}$ in a linear Gaussian state space model, which can be done in $\mathcal{O}(\log T)$ time on a sufficiently large parallel machine.
\end{proposition}

\begin{proof}[Proof]
Expanding the residual and Jacobian functions in \eqref{eq:linmerit}, we see that up to an additive constant, the negative Lagrangian can be rewritten as,
\vspace{-1em}
\begin{multline} \label{eq:lgssm}
    -\widehat{\mathcal{L}}(\Delta \mathbf{s}, \lambda_{i+1})
    \stackrel{\cdot}{=} \log \mathcal{N}\left(\mathbf{s}_1 \mid f(\mathbf{s}_0), I_D \right) + \sum_{t=1}^T \log \mathcal{N}\left(\mathbf{s}_t^{(i)} \, \Big|\, \mathbf{s}_t, \frac{1}{\lambda_{i+1}} I_D \right) \\
    + \sum_{t=2}^T \log \mathcal{N}\left(\mathbf{s}_t \, \Big|\,  f(\mathbf{s}_{t-1}^{(i)}) + \left[ \frac{\partial f}{\partial \mathbf{s}}(\mathbf{s}_{t-1}^{(i)}) \right] (\mathbf{s}_{t-1} - \mathbf{s}_{t-1}^{(i)}), I_D \right),
\end{multline}
where $\mathcal{N}(\mathbf{x} \mid \mu, \Sigma)$ denotes the probability density function of the multivariate normal distribution.

We recognize \eqref{eq:lgssm} as the log joint probability of a linear Gaussian state space model (LGSSM)~\citep{sarkka2023bayesian} on $(\mathbf{s}_1, \ldots, \mathbf{s}_T)$. Consequently, the dynamics distributions are given by the linearization of $f$, and the emissions are the previous iteration's states, $\mathbf{s}^{(i)}$. The parameter $\lambda_{i+1}$ sets the precision of the emissions, governing how far the posterior mode deviates from the previous states. 
We show the graphical diagram for this LGSSM for $T=3$ in \Cref{fig:elk_kalman}. 

The minimizer of \eqref{eq:lagrangian} is the posterior mode of the LGSSM \eqref{eq:lgssm}, and can be obtained by Kalman smoothing~\citep{sarkka2023bayesian}. As with the linear recursions in DEER, the Kalman smoother can be implemented as a parallel scan that scales as $\mathcal{O}(\log T)$ in time on a machine with $\mathcal{O}(T)$ processors~\citep{parallel-kalman,linderman2025dynamax}.
\end{proof}
\begin{figure}
    \centering
    \begin{tikzpicture}
        % Define a uniform node style for circles with minimum size
        [node distance=2cm, minimum size=1cm, circle, draw=black, align=center]

        % Define nodes
        \node[circle, draw=black, fill = gray, minimum size=1cm] (x0) at (0,0) {\textcolor{black}{$s_0$}};
        \node[circle, draw=black, minimum size=1cm] (x1) at (2,0) {\textcolor{red}{$s_1$}};
        \node[circle, draw=black, fill=gray, minimum size=1cm] (y1) at (2,-1.5) {\textcolor{blue}{$s_{\textcolor{blue}{1}}^{(i)}$}};
        \node[circle, draw=black, minimum size=1cm] (x2) at (4,0) {\textcolor{red}{$s_2$}};
        \node[circle, draw=black, fill=gray, minimum size=1cm] (y2) at (4,-1.5) {\textcolor{blue}{$s_{\textcolor{blue}{2}}^{(i)}$}};
        \node[circle, draw=black, minimum size=1cm] (x3) at (6,0) {\textcolor{red}{$s_3$}};
        \node[circle, draw=black, fill=gray, minimum size=1cm] (y3) at (6,-1.5) {\textcolor{blue}{$s_{\textcolor{blue}{3}}^{(i)}$}};

        % Draw arrows with labels above
        \draw[-to] (x0) -- (x1) node[midway, above] {\textcolor{red}{$A_{1}, b_{1}$}};
        \draw[-to] (x1) -- (y1) node[midway, right] {\textcolor{blue}{$1/\lambda$}};
        \draw[-to] (x1) -- (x2) node[midway, above] {\textcolor{red}{$A_{2}, b_{2}$}};
        \draw[-to] (x2) -- (y2) node[midway, right] {\textcolor{blue}{$1/\lambda$}};
        \draw[-to] (x2) -- (x3) node[midway, above] {\textcolor{red}{$A_{3}, b_{3}$}};
        \draw[-to] (x3) -- (y3) node[midway, right] {\textcolor{blue}{$1/\lambda$}};
    \end{tikzpicture}
    \caption{\textbf{Graphical Diagram of the ELK LGSSM.} We provide a graphical diagram illustrating how the LM update in the context of parallelizing nSSMs is the MAP solution to posterior inference in an appropriately constructed LGSSM. Without any observations (i.e. $\lambda=0$, or equivalently observations with infinite variance), we simply recover the DEER update. However, by using our previous state $\mathbf{s}_{1:T}^{(i)}$ as our observations, we restrict the dynamics to a trust region.}
    \label{fig:elk_kalman}
\end{figure}
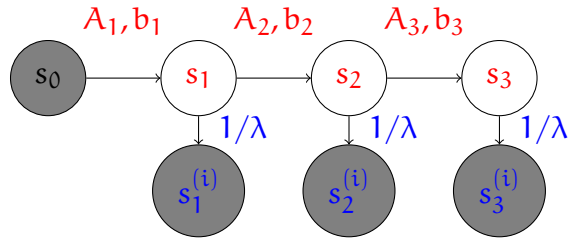

Therefore, we can evaluate an RNN by minimizing the merit function with the Levenberg-Marquardt algorithm. Since each step of LM can be performed by parallel Kalman smoothing, we call this approach \emph{Evaluating Levenberg-Marquardt with Kalman} (ELK).
Note that DEER is a special case of ELK, where $\lambda=0$, which can be seen as minimizing the unpenalized linearized objective \eqref{eq:linmerit}, or, alternatively, taking a Newton step with an infinitely large trust region.
Moreover, under certain conditions, ELK also enjoys global convergence guarantees~\citep[][Thms. 11.7, 11.8]{NocedalWright}.

\paragraph{Quasi-ELK: Scalability and Stability}
As with DEER, we can substitute an approximate Jacobian into the Lagrangian to obtain the \emph{quasi-ELK} algorithm. 
Quasi-ELK enjoys the compute and memory scaling of quasi-DEER, as well as stability from the trust region damping from ELK.
We show empirically in \Cref{sec:elk_exps} that while quasi-ELK takes more iterates to converge than ELK, each quasi-ELK iterate is faster, giving overall runtime speedups. 

\paragraph{Implementation Details} 
The convergence rate of (quasi-)ELK depends on the trust region radius $D_{i}$ (or alternatively $\lambda_{i}$).
Although there exist methods to analytically set $\lambda_i$~\citep[Algorithm 4.3]{NocedalWright}, these approaches require factorizing $\nicefrac{\partial \mathbf{r}}{\partial \mathbf{s}}$, which is intractable at scale.
Therefore, in practice, we treat $\lambda$ as a hyperparameter set by a sweep over log-spaced values.

We also use Kalman filtering instead of smoothing.
We do so for two main reasons: filtering requires less work and memory; and we also found it to converge in fewer Newton iterations than smoothing. 
We hypothesize that this faster convergence is related to \Cref{prop:global_convergence}, whose proof shows that the early part of the trace converges first. The traces in parallel Newton iterations converge causally, propagating information from the ground truth initial condition $s_0$ to the end of the sequence. Therefore, it makes intuitive sense that a Kalman filter, which is also causal, would have better empirical performance than a Kalman smoother.

Using the Kalman filter also provides an intuitive explanation---based on dynamics instead of optimization---as to how ELK calms instabilities that can arise in DEER. We discuss this connection in the next section. 

\section{Dynamics perspective on ELK}\label{sec:elk_dyn}

A complementary perspective on how ELK results in more stable evaluation of nonlinear RNNs is to see how the Kalman filter damps the spectral norms of the Jacobian matrices $\{A_t\}$ of the transition dynamics. 
The spectral norm of a matrix gives the maximum factor by which it can scale an input vector, and so is intuitively related to the stability of a linear dynamical system. 
We first provide a high-level overview, and then provide a more detailed derivation.

\paragraph{Overview} Let $\{\A_t\}$ be the Jacobians $\{ \nicefrac{\partial f_t}{\partial s_{t-1}} \}$ used in the linear recurrence relations and $b_t$ be the offsets. 
Then the prediction step of the Kalman filter (ELK) is the same as DEER. However, after applying the update step in ELK (which imposes the trust region), we obtain a second linear recurrence relation where the linear operator is given by $\Gamma_t \A_t$. 
Note that $\Gamma_t$ is a symmetric positive definite matrix with eigenvalues bounded above by $\nicefrac{1}{1 + \lambda}$. Thus, by the Spectral Theorem, it follows that the spectral norm of $\Gamma_t \A_t$ is bounded above by $\nicefrac{\| A_t \|}{1+\lambda}$.
Note that larger $\lambda$ corresponds to more regularization/smaller trust region; and therefore correspondingly results in smaller effective spectral norm.
We recover DEER exactly if $\lambda=0$. Thus, while large spectral norms of $A_t$ are a cause of the instability of DEER when evaluating unstable dynamical systems, ELK directly attenuates these spectral norms, providing an explanation for why the intermediate iterations using ELK remain stable.

\paragraph{Derivation}
We define our dynamics used in Newton iteration $(i+1)$ as
\begin{align*}
    \A_t & = \dfrac{ \partial f_{t}}{ \partial s_{t-1}}(s_{t-1}^{(i)}) \\
    b_t & = f_{t}(s_{t-1}^{(i)}) - \dfrac{ \partial f_{t}}{ \partial s_{t-1}}(s_{t-1}^{(i)}) s_{t-1}^{(i)}.
\end{align*}
Now $\A_t \in \mathbb{R}^{D \times D}$ and $b_t \in \mathbb{R}^D$.

In line with considering the system as the LDS in \eqref{eq:lgssm}, we set the process noise to be $I_D$, and with the emissions governed by
\begin{equation*}
    s_t^{(i+1)} \sim \mathcal{N}(s_t^{(i)}, \sigma^2 I_D),
\end{equation*}
where $\sigma^2$ controls the size of our trust region, since $\lambda = 1/\sigma^2$.

In the notation of \citet{murphy}, we see that the \textbf{predict step} is
\begin{align*}
    \mu_{(t+1)|t} & = \A_{t+1} \mu_{t|t} + b_{t+1} \\
    \Sigma_{(t+1)|t} & = \A_{t+1} \Sigma_{t|t} \A_{t+1}^{\top} + I_D.
\end{align*}

Meanwhile, the \textbf{update step} is
\begin{align*}
    \mu_{(t+1)|(t+1)} & = \mu_{(t+1)|t} +  \Sigma_{(t+1)|t} (\Sigma_{(t+1)|t} + \sigma^2 I_D)^{-1}(s_{t+1}^{(i)} - \mu_{(t+1)|t}) \\
    \Sigma_{(t+1)|(t+1)} & = \Sigma_{(t+1)|t} - \Sigma_{(t+1)|t} (\Sigma_{(t+1)|t} + \sigma^2 I_D)^{-1} \Sigma_{(t+1)|t}^{\top}.
\end{align*}

To unpack this further, we first define the \emph{attenuation matrix}
\begin{equation*}
    \Gamma_{t+1} \coloneq \sigma^2 \Big(\A_{t+1} \Sigma_{t|t} \A_{t+1}^{\top} + (\sigma^2 + 1) I_D\Big)^{-1}.
\end{equation*}

Because $\Sigma_{t|t}$ is a covariance matrix, it is symmetric positive semidefinite, and so $\A_{t+1} \Sigma_{t|t} \A_{t+1}^{\top}$ is also symmetric positive semidefinite, and so all of its eigenvalues are nonnegative. Therefore, all the eigenvalues of  $\A_{t+1} \Sigma_{t|t} \A_{t+1}^{\top} + (\sigma^2 + 1) I_D$ are greater than or equal to $\sigma^2 + 1$.

Consequently, $\Gamma_{t+1}$ is symmetric positive definite. Thus, by the Spectral Theorem, all eigenvalues of $\Gamma_{t+1}$ are positive. By the above argument, the eigenvalues of $\Gamma_{t+1}$ are all less than or equal to $\frac{\sigma^2}{1 + \sigma^2} < 1$.
Moreover, since $\Gamma_{t+1}$ is symmetric positive definite, its eigenvalues are equal to its singular values, and so $\| \Gamma_{t+1} \|_2 \leq \nicefrac{1}{1 + \lambda}$.

Thus, we observe that the resulting filtering is given by the recurrence relation
\begin{align*}
    \mu_{(t+1)|(t+1)} & = \overbrace{\Gamma_{t+1} \A_{t+1}}^{\text{linear dynamics}}  \mu_{t | t} \\
    & +
    \overbrace{\Gamma_{t+1} b_{t+1} + (\A_{t+1} \Sigma_{t|t} \A_{t+1}^{\top} + I_D) \Big(\A_{t+1} \Sigma_{t|t} \A_{t+1}^{\top} + (\sigma^2 + 1) I_D\Big)^{-1} s_{t+1}^{(i)}}^{\text{bias term}} \\
    \Sigma_{(t+1)|(t+1)} & = \Gamma_{t+1} (\A_{t+1} \Sigma_{t|t} \A_{t+1}^{\top} + I_D).
\end{align*}
If we are given the $\left\{ \Sigma_{t|t} \right\}$, we see that the filtered means (the updates for ELK) come from a linear recurrence relation with linear term $\Gamma_{t+1} \A_{t+1}$.

Finally, by the submultiplicativity of norms and our results above, it follows that
\begin{align*}
    \| \Gamma_{t+1} A_{t+1} \|_2 & \leq \| \Gamma_{t+1} A_{t+1} \|_2 \\
    & \leq \frac{1}{1 + \lambda} \| A_{t+1} \|_2. \tag*{$\square$}
\end{align*}

\section{Experiments and performance of ELK}\label{sec:elk_exps}

Having derived the ELK algorithm and studied its theoretical properties, we now empirically assess its performance in parallelizing dynamical systems at the edge of stability.
We examine two dynamical systems: a sine wave and the Lorenz-96 dynamical system \citep{lorenz1996predictability}. All the experiments in this section were run on a single NVIDIA A100 GPU with 80 GB onboard memory.

\subsection{Edge of stability: Parallelizing a sine wave}

First, we pretrain an RNN to recapitulate a sine wave. For our architecture, we use a GRU with hidden states $\mathbf{h}_t \in \mathbb{R}^3$ and scalar inputs $x_t \in \mathbb{R}$.
However, at every point in the sequence $t,$ we readout the hidden state $h_t \in \mathbb{R}^3$ and use it to parameterize a mean $\mu_{t+1} \in \mathbb{R}$ and a variance $\sigma^2_{t+1} \in \mathbb{R}_+$.
We then sample $x_{t+1}$ according to $x_{t+1} \sim \mathcal{N}(\mu_{t+1}, \sigma_{t+1}^2)$; this \emph{output} $x_{t+1}$ is then fed back in as the \emph{input} to the autoregressive GRU at time step $t+1$ to make the new hidden step $\mathbf{h}_{t+1}$. 
Crucially, when parallelizing this architecture, the Markovian state $\mathbf{s}_t$ must be expanded to include the current sampled output value, as well as the current GRU state.  

We pretrain this GRU using standard sequential evaluation and backpropagation-through-time to produce a noisy sine wave of length 10,000.
We train the GRU on 1024 traces $\mathbf{x}_{1:T}$ generated from a sine wave with amplitude 10 and white noise applied to each time step, and the training objective is to minimize the negative log probability of the $\mathbf{x}_{1:T}$.

We note that such a system is Markovian with state dimension $D = \dim(\mathbf{h}) + \dim(x)$, as together the hidden state $\mathbf{h}_t$ and output $x_{t+1}$ determine the next hidden state $\mathbf{h}_{t+1}$ and output $x_{t+2}$. 
Thus, in the notation of \cref{eq:ssm}, a hidden state $\mathbf{s}_t$ of the Markovian state space model is $\mathbf{s}_t = (x_{t+1}, \mathbf{h}_t)$.
Therefore, we can apply parallel Newton methods to try to find the correct trace $\mathbf{s}^*$ in a parallelized manner instead of autoregressively.

\paragraph{Initialized AR GRU}  We first repeat the analysis in Section~\ref{ssc:gru_bench} for evaluating a randomly initialized autoregressive GRU. 
We see in the top left panel of Figure~\ref{fig:argru} that all four parallel Newton methods converge rapidly and stably to the correct trace, indicated by a low mean absolute discrepancy (MAD) between the true trace and the generated trace.  

\paragraph{Trained AR GRU}  We then study a pre-trained GRU that generates a noisy sine wave (see Figure~\ref{fig:argru}, bottom). 
The linear recurrence relation \eqref{eq:lin_rr} was numerically unstable in DEER and quasi-DEER.  
To remedy these instabilities, we take the approach described earlier of setting the unstable parts of the trace to a fixed value (here zero).
Doing so ensures convergence, but at the cost of ``resetting'' the optimization for large swathes of the trace (Figure~\ref{fig:argru}, bottom) and slowing convergence (see Figure~\ref{fig:argru}, top right).
This finding highlights how the instabilities of DEER --- which are inherited from both pathologies of Newton's method and the parallel recurrence --- can be crippling in even very simple scenarios. 
While resetting allows for convergence, the resulting convergence is very slow.

\begin{figure}[t]
    \centering
    \includegraphics[width=0.95\textwidth]{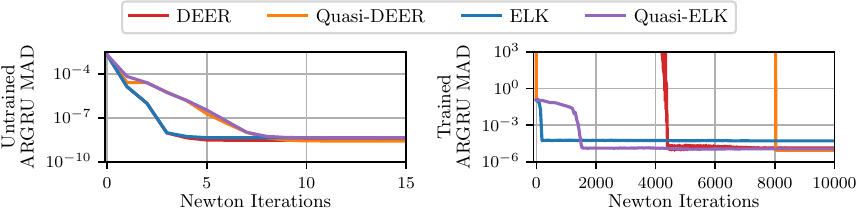}
    \includegraphics[width=0.95\textwidth]{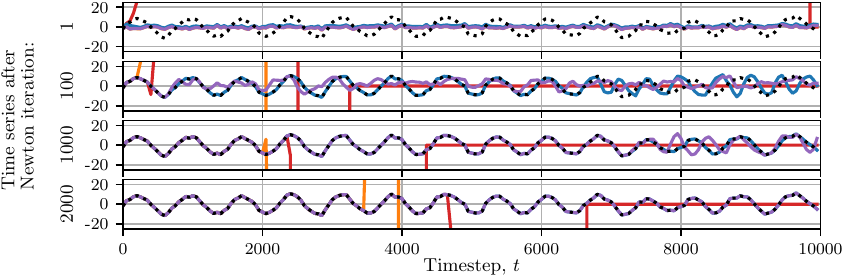}
    \vspace*{-0.0cm}
    \caption{ELK stabilizes parallel evaluation of an AR GRU.
    \textbf{(Top Left)} The mean absolute difference (MAD) evaluated on the outputs converges rapidly for all four methods on a sequence generated by an \emph{untrained} AR GRU.
    \textbf{(Top Right)} The MAD for evaluating a trained AR GRU. Undamped DEER variants are unstable and converge slowly (using the reset heuristic). ELK stabilizes and accelerates convergence.
    \textbf{(Bottom)} The output after 1, 100, 1000, and 2000 Newton iterations. The black dotted line is the true trace.  ELK and quasi-ELK converge rapidly, but DEER and quasi-DEER are unstable.  The lines where DEER and quasi-DEER are zero depict the zeroing heuristic. 
    }
    \label{fig:argru}
\end{figure} 

We then apply ELK and quasi-ELK.
We show the results in the top right and bottom panels of Figure~\ref{fig:argru}.
We select the trust region size with a one-dimensional search over log-spaced values between  $10^0$ and $10^7$.
We see ELK has stabilized convergence, with the evaluation never incurring numerical instabilities or requiring heuristics.
Crucially, by taking more stable steps (and not needing stabilizing heuristics) ELK and quasi-ELK converge faster than DEER and quasi-DEER.
ELK can stabilize and expedite the convergence of DEER, with quasi-ELK faster still (by wall-clock time).

However, when run on an A100 GPU with 80 GB onboard memory, all parallel Newton methods (including DEER) are slower than sequential generation, as shown in \Cref{tab:timing}.
Quasi-ELK is the fastest parallel method, taking 221 milliseconds, compared to sequential evaluation, taking 96 milliseconds.
For comparison, DEER took 1,225 milliseconds. 
Quasi-ELK therefore still represents a large improvement in runtime over previous parallel methods.
\begin{table}[ht]
  \caption{Time to evaluate a length $T=10,000$ trained AR GRU using sequential vs parallelized methods. We note the \texttt{dynamax} package~\citep{linderman2025dynamax} we used for the parallel Kalman filter implementation in ELK is not optimized for speed, and hence these run times could be further improved.}
  \label{tab:timing}
  \centering
  {\small
  \begin{tabular}{lp{2cm}p{2cm}p{2cm}}
    \toprule
    Algorithm & Time per Newton step (ms, mean $\pm$ std) & Newton steps to convergence & Total time to convergence (ms) \\
    \midrule
    \multicolumn{4}{l}{\textbf{Sequential Evaluation}} \\
    \midrule
    Sequential & \multicolumn{1}{r}{N/A} & \multicolumn{1}{r}{N/A} & \multicolumn{1}{r}{$96$} \\
    \midrule
    \multicolumn{4}{l}{\textbf{Parallelized Methods}} \\
    \midrule
    DEER & \multicolumn{1}{r}{$0.282 \pm 0.0005$} & \multicolumn{1}{r}{$4449$} & \multicolumn{1}{r}{$1255$} \\
    Quasi-DEER & \multicolumn{1}{r}{$0.087 \pm 0.0002$} & \multicolumn{1}{r}{$7383$} & \multicolumn{1}{r}{$642$} \\
    ELK & \multicolumn{1}{r}{$3.600 \pm 0.0670$} & \multicolumn{1}{r}{$172$} & \multicolumn{1}{r}{$619$} \\
    Quasi-ELK & \multicolumn{1}{r}{$0.141 \pm 0.0004$} & \multicolumn{1}{r}{$1566$} & \multicolumn{1}{r}{$221$} \\
    \bottomrule
  \end{tabular}
  }
\end{table}

These timing results are illustrative of multiple themes of our paper.
We see that the undamped Newton steps are individually faster because they are carrying out fewer computations. The undamped Newton steps are just computing a linear recurrence relation, while the trust-region methods are computing a filtering pass.
However, because the undamped Newton methods are numerically unstable, they take dramatically more Newton steps to converge.

Similarly, we see that the quasi methods are dramatically faster than their dense counterparts as they replace $\mathcal{O}(D^3)$ matrix-matrix multiplication with $\mathcal{O}(D)$ diagonal matrix multiplication.
The $\mathcal{O}(D^3)$ work required by a parallel scan on a dense linear recurrence likely saturates the GPU). We see in Table~\ref{tab:timing} that individual steps in the dense DEER/ELK are (approximately) a factor of between 3.5 and 30 times slower \emph{per step} than their quasi  (diagonal) variants. However, they take a factor of between 2 and 10 fewer iterations. 

\paragraph{Further details on setting $\lambda$}

\begin{figure}[ht]
    \centering
    \includegraphics[width=\textwidth]{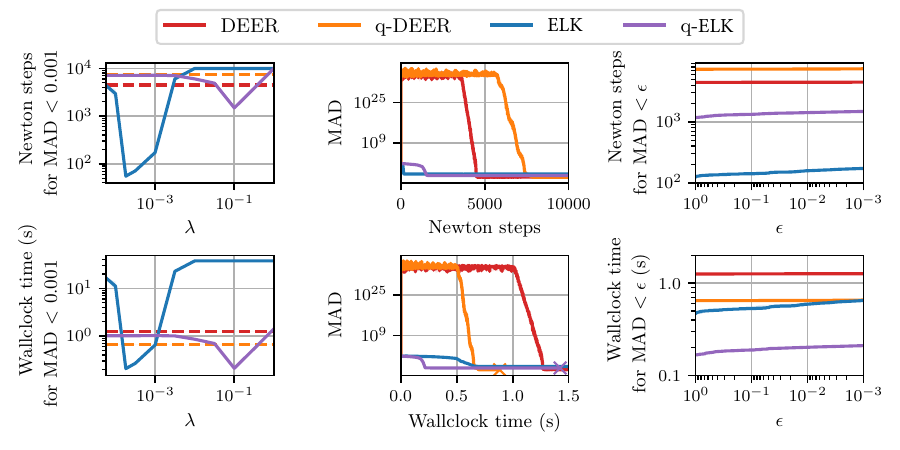}
    \caption{Experiment to show how to set the hyperparameters for (quasi)-ELK on the AR GRU pre-trained to generate a noisy sine wave (Figure~\ref{fig:argru} in the main text). Top row plots Newton steps; bottom row plots wall-clock time.  Lower is better for all plots.  (\textbf{Left})  median steps/time to convergence over $\lambda$ over $15$ sequences.
    Quartiles are shaded but are very small.
    DEER methods are independent of $\lambda$.  (\textbf{Center}) Updated version of Figure~\ref{fig:argru} instead plotting MAD as a function of wall-clock time.  (\textbf{Right}) Time to convergence is robust as a function of convergence threshold $\epsilon$.  Median and quartiles across 15 sequences are shown. DEER methods are nearly constant at the thresholds considered (very slight positive slope).  Note we plot for \emph{increasing} $\lambda$ corresponding to a smaller trust region, and \emph{reducing} $\epsilon$ corresponding to a tighter convergence threshold.
    }
    \label{fig:supp1}
\end{figure}

We provide more details on how to set the hyperparameters for ELK in  \Cref{fig:supp1}.
We sweep over the hyperparameter for 15 different input sequences, and plot the median and quartiles of the cost to convergence in terms of Newton iterates and runtime (left column of Figure~\ref{fig:supp1}).
We see a U-shaped curve: large $\lambda$ takes needlessly small steps, slowing progress; small $\lambda$ results in many resets, slowing convergence. Crucially, we see there is little variance across individual sequences. These results show that there is a well-behaved  dependence that can be optimized on a validation set with a simple 1-d grid search.

We also chart the approximation error against cost for the AR GRU (center and right column of Figure~\ref{fig:supp1}). We see that the approximation error reduces in fewer Newton steps with full DEER as opposed to quasi-DEER, but, crucially, the wall-clock time (the more important of the two metrics) is notably lower across all accuracies for quasi-DEER. This indicates that our more efficient – but approximate – quasi-DEER is broadly preferable to the more expensive – but exact – DEER updates. Furthermore, the stabilized ELK and quasi-ELK are better still. We also show the steps/time to convergence for a range of accuracy thresholds, and see that our methods outperform DEER across the full range of thresholds and metrics.

\subsection{Chaotic system: Parallelizing the Lorenz96 System}

\begin{figure}[ht]
    \centering
    \includegraphics[width=\textwidth]{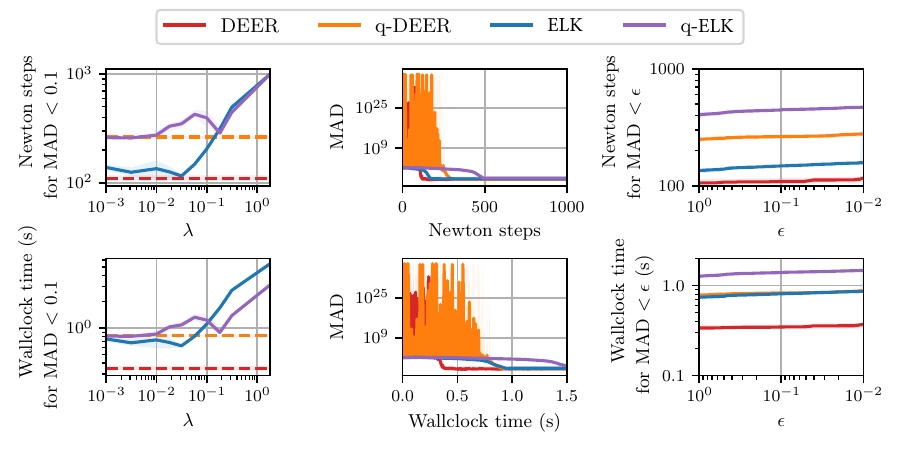}
    \includegraphics[width=\textwidth]{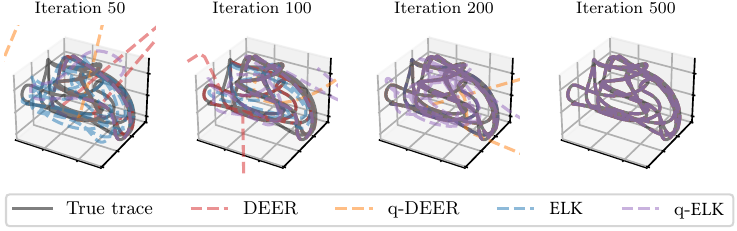}
    \caption{Evaluating the Lorenz96 system in parallel. \textbf{(Top two rows)}: Same format as Figure~\ref{fig:supp1}. \textbf{(Bottom row)}: Plot of Lorenz96 trajectory during optimization.  DEER methods are noticeably more unstable than ELK methods.
    }
    \label{fig:supp2}
\end{figure}

Having investigated the parallel Newton methods on the edge of stability (a sinusoidal oscillation), we now investigate their performance on a chaotic system. We tackle the parallel evaluation of the classic non-linear 5-dimensional Lorenz-96 system, with $F=8$ which results in chaotic dynamics. We seek to evaluate this system (for $T=1000$ timesteps) using (quasi)-DEER and (quasi)-ELK.
We directly use the Lorenz-96 dynamics as our nonlinear dynamics function $f$, i.e. the architecture/time evolution \emph{is} the Lorenz-96 ODE system, evaluated using the Dormand-Prince solver \citep{dormand1980family}. The state is the five-dimensional Lorenz system state. The input is therefore the initial condition of the ODE; and the outputs are the $T \times 5$ subsequent system states. Of course, ODE solvers are also examples of SSMs (see \Cref{tab:ssm_examples}).

We demonstrate that all the parallelized methods converge to the correct trace, but that (quasi)-ELK is dramatically more stable at intermediate Newton iterations prior to convergence.
We see that DEER and ELK methods converge in a comparable number of steps (this makes sense as DEER is a special case of ELK for $\lambda \to 0$). DEER is faster (in terms of wall-clock time) because of the extra work done per ELK iteration. However, ELK has stabilized convergence, whereas DEER relies heavily on resetting. Interestingly we see that quasi is slower by all metrics, suggesting that the chaotic dynamics may require the more accurate updates. Quasi methods can be implemented to consume notably lower memory, however, and so may be preferable in certain circumstances.

In Figure~\ref{fig:supp2}, we report mean absolute deviation (MAD) of the time series at Newton iteration $(i)$ against the true state sequence. “Iteration” then refers to the number of Newton iterations, i.e. the number of updates applied to the entire state sequence. We set hyperparameters using 10 different evaluations of the Lorenz96 (i.e. starting from 10 different initial points).

\section{Further extensions: scale- and clip-ELK}

Since running the experiments for ELK published in our NeurIPS 2024 paper \citep{gonzalez2024scalable}, we developed simpler and more lightweight damping techniques to achieve many of the stabilization benefits of ELK. \citet{pmcmc} often uses these damping techniques to parallelize MCMC chains. We discuss two of these extensions, \emph{scale-ELK} and \emph{clip-ELK}, below.

\subsection{Scale-ELK}

Motivated by our demonstration in \Cref{sec:elk_dyn} which shows that ELK reduces the spectral norms of the Jacobian matrices in the transition dynamics, we recommend a more lightweight version of ELK which we call \emph{scale-ELK}.

Scale-ELK uses a hyperparameter $k \in [0,1]$ (as opposed to $\lambda \in [0, \infty)$ used by ELK). Scale-ELK uses a linear dynamical system just like DEER, with the dynamics defined as
\begin{align*}
    \A_t & = (1-k) \dfrac{ \partial f_{t}}{ \partial s_{t-1}}(s_{t-1}^{(i)}) \\
    b_t & = f_{t}(s_{t-1}^{(i)}) - (1-k) \dfrac{ \partial f_{t}}{ \partial s_{t-1}}(s_{t-1}^{(i)}) s_{t-1}^{(i)}.
\end{align*}
Thus, setting $k=0$ recovers DEER, while setting $k=1$ recovers a (computationally expensive form of) sequential evaluation.
Ideally, $k$ is chosen to keep the spectral norms of $\{ A_t\}_{t=1}^T$ below 1.
Note that $k_t$ can also be chosen on a timestep dependent basis.
By \Cref{prop:global_convergence}, scale-ELK also enjoys global convergence.

Scale-ELK enjoys two primary benefits over ELK. First, an evaluation of scale-ELK uses fewer FLOPs than ELK, as scale-ELK is just parallelizing an LDS while ELK uses a parallelized Kalman filter. Second, the Kalman filter involves inverses which run the risk of introducing numerical instability, while scale-ELK avoids these complications.

\subsection{Clip-ELK}

However, scale-ELK still has a hyperparameter $k$. Although this hyperparameter can be set using techniques as shown in \Cref{fig:supp1}, it would be desirable to have a hyperparameter-free method.

Therefore, we propose \emph{clip-ELK}, which is a hyperparameter free approach to achieve the same goal of a stable LDS. Clip-ELK applies to the "quasi" diagonal approximation only, and simply clips each element of $A_t$ (which in this setting is a diagonal matrix) to be between $[-1,1]$. Clip-ELK also converges globally by \Cref{prop:global_convergence}. Moreover, by design, it ensures that each iteration of clip-ELK is a stable LDS.
Clipping can also be done to some hyperparameter, e.g. $[-\underline{\rho}, \bar{\rho}]$, for $\underline{\rho}, \bar{\rho} \leq 1$.

\section{Conclusion}

ELK presents a beautiful connection between dynamics---Kalman filtering and smoothing---and optimization---the Levenberg-Marquardt, or trust-region methods---to parallelize dynamical systems in a stable way. In the experiments we provide in \Cref{sec:elk_exps}, we show that the intermediate ELK iterates are much more stable than the DEER iterates. Interestingly, at early (around 100) iterations, even though ELK has not recovered the exact trace $\mathbf{s}_{1:T}^{\star}$, \Cref{fig:argru,fig:supp2} show that it qualitatively appears to have the right "manifold" of the dynamics. For this reason, ELK could prove very useful in the "early stopping" of these parallel Newton methods in the context of parallelizing MCMC chains, or in general when the desired output of a procedure is a distribution instead of an exact trajectory.

Nevertheless, as shown in \Cref{tab:timing}, even at the edge of stability, all of the parallel Newton methods struggle to achieve parity with---let alone beat---the speed of sequential evaluation for obtaining an exact trajectory.
These difficulties raise the important question: are there certain dynamical systems that cannot be parallelized efficiently? 
We answer this question in the next part of this thesis, which provides a thorough account of the convergence rates of these parallel Newton methods.
\ctparttext{The third part of this thesis presents its theoretical contributions. We present the first detailed analysis of the convergence rates of these parallel Newton methods. In particular, we show how the predictability of the dynamics is the primary determining factor of the convergence rate of the method. Furthermore, we show how a wide-range of fixed-point methods in use for parallelizing sequential computation can be unified in the quasi-DEER framework. We show how the quality of the quasi-DEER approximation in this framework affects the convergence rates of different fixed-point methods in different problems.
\par
\begin{minipage}{\linewidth}
    \centering
    \captionsetup{hypcap=false}
    \includegraphics[width=\linewidth]{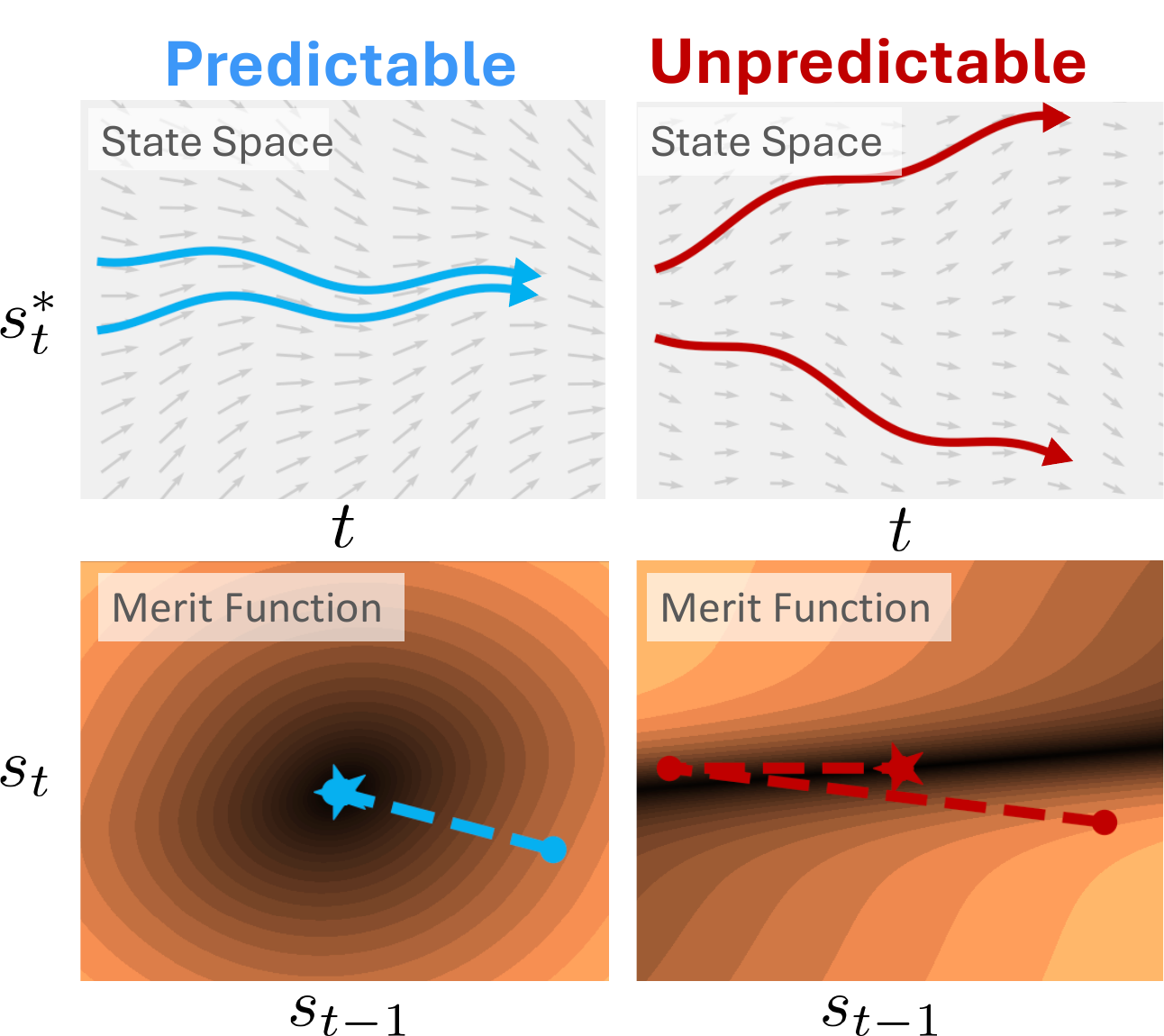}
    \captionof{figure}{\textbf{Predictability enables parallelization.} Predictable dynamics yield well-conditioned merit functions, enabling rapid convergence.
Unpredictable dynamics produce flat or ill-conditioned merit landscapes, resulting in slow convergence or numerical failure.}
    \label{fig:pred}
\end{minipage}
}
\part{Theory: Convergence Rates}\label{part:theory}
%************************************************
\chapter{Convergence Rates of Gauss-Newton for Parallelizing nonlinear SSMs}\label{ch:predictability}
%************************************************

The previous chapters developed practical algorithms for parallelizing nonlinear SSMs. A natural question arises: \emph{which systems admit efficient parallelization?} 

This chapter establishes a fundamental connection between the \emph{dynamics} of a system and the difficulty of the resulting \emph{optimization} problem (of the merit function defined in \cref{eq:merit}).
Our central result is that \textbf{predictability enables parallelization}: systems whose future states can be reliably predicted from past states admit efficient parallel evaluation, while chaotic systems do not.

In particular, we establish a precise relationship between a system's dynamics and the conditioning of its corresponding optimization problem, as measured by its Polyak-Łojasiewicz (PL) constant. We show that the predictability of a system, defined as the degree to which small perturbations in state influence future behavior and quantified by the largest Lyapunov exponent (LLE), impacts the number of optimization steps required for evaluation.
    
For predictable systems, the state trajectory can be computed in at worst $\mathcal{O}((\log T)^2)$ time, where $T$ is the sequence length: a major improvement over the conventional sequential approach. One factor of $\log(T)$ comes from the computational cost of each Gauss-Newton step, which uses a parallel scan. The other factor of $\log(T)$ comes from the number of Gauss-Newton steps needed to converge, which yields the interpretation that a predictable nonlinear SSM can be thought of as a stack of $\mathcal{O}(\log T)$ LDSs.
In contrast, chaotic or unpredictable systems exhibit poor conditioning,
with the consequence that parallel evaluation converges too slowly to be useful. 

Importantly, our theoretical analysis shows that predictable systems always yield well-conditioned optimization problems, whereas unpredictable systems lead to severe conditioning degradation. We validate our claims through extensive experiments, providing practical guidance on when nonlinear dynamical systems can be efficiently parallelized. We highlight predictability as a key design principle for parallelizable models.

\section{Predictability and the Largest Lyapunov Exponent}\label{sec:lle}

Predictability is usually defined through its antonym: \textit{un}predictability \citep{lighthill1986recently,strogatz2018nonlinear}.

\textbf{Unpredictable systems} are dynamical systems
whose future behavior is highly sensitive to small perturbations.
The system's intrinsic sensitivity amplifies small perturbations and leads to massive divergence of trajectories.
A common example is a chaotic system, like the weather: a butterfly flapping its wings in Tokyo today can lead to a thunderstorm in Manhattan next month \citep{lighthill1986recently,strogatz2018nonlinear}.
 Given a snapshot of the current atmospheric state, weather models can provide accurate forecasts over short time horizons—typically a few days. However, predictions degrade rapidly beyond that, as the system’s intrinsic sensitivity amplifies small uncertainties in the initial snapshot \citep{lorenz1963deterministic}.
 
By contrast, \textbf{predictable systems} \citep{thompson1957uncertainty, lorenz1996predictability} are those in which small perturbations are forgotten. Small perturbations are diminished over time, rather than amplified. A familiar example is aviation: a patch of choppy air rarely makes an airplane land at the wrong airport.

The notion of (un)predictability can be formalized through various routes such as chaos theory \citep{gleick2008chaos,schuster2006deterministic} and contraction analysis \citep{lohmiller1998contraction,FB-CTDS}.
We provide a definition of predictability in terms of the Largest Lyapunov Exponent (LLE)
\citep{pikovsky2016lyapunov,strogatz2018nonlinear}:

\begin{mdframed}[linewidth=1pt]
\begin{definition}[\textbf{Predictability and Unpredictability}]\label{def:predictable_systems}
Consider a sequence of Jacobians $\A_1, \A_2, \cdots, \A_T$. We define the associated Largest Lyapunov Exponent (LLE) to be
\begin{equation}\label{eq:LLE_definition}
\text{LLE} \ := \ \lim_{T \to \infty} \frac{1}{T} \log\left( \left\| \A_T \A_{T-1} \cdots \A_1 \right\| \right) \ = \ \lambda,
\end{equation}
where $\| \cdot \|$ is an induced operator norm. If $\lambda < 0$, we say that the nonlinear state space model is \emph{predictable} at $s_0$. Otherwise, we say it is \emph{unpredictable}.
\end{definition}
\end{mdframed}
Suppose we wish to evaluate a nonlinear SSM \eqref{eq:ssm} from an initial condition \( s_0 \), but we only have access to an approximate measurement \( s_0' \) that differs slightly from the true initial state. If the system is unpredictable ($\lambda > 0$), then the distance between nearby trajectories grows as
\begin{equation}\label{eq:chaotic_seperation}
 \|s_t - s_t'\| \ \sim \ e^{\lambda t} \|s_0 - s_0'\|.     
\end{equation}
Letting \( \Delta \) denote the maximum acceptable deviation beyond which we consider the prediction to have failed, the time horizon over which the prediction remains reliable scales as  
\begin{equation}\label{eq:chaotic_time_horizon}
\text{Time to degrade to $\Delta$ prediction error} \ \sim \ \frac{1}{\lambda} \log\left( \frac{\Delta}{\|s_0 - s_0'\|} \right).
\end{equation}  
This relationship highlights a key limitation in unpredictable systems: even significant improvements in the accuracy of the initial state estimate yield only logarithmic gains in prediction time. The system's inherent sensitivity to initial conditions overwhelms any such improvements. Predictable systems, such as contracting systems, have the opposite property: trajectories initially separated by some distance will eventually converge towards one another (\Cref{fig:pred}), \textit{improving} prediction accuracy over time.

The sign of $\lambda$ determines the system's qualitative behavior:
\begin{itemize}
    \item $\lambda < 0$ (\textbf{predictable}): Perturbations decay exponentially. Small errors in initial conditions have diminishing effects on future states. Examples include stable linear systems and contractive nonlinear maps.

    \item $\lambda = 0$ (\textbf{marginal}): Perturbations neither grow nor decay on average. This is the boundary between predictable and chaotic dynamics.

    \item $\lambda > 0$ (\textbf{chaotic}): Perturbations grow exponentially. The system exhibits sensitive dependence on initial conditions---the hallmark of chaos. Small errors rapidly amplify, making long-term prediction impossible.
\end{itemize}

We will show that the predictability of the dynamics directly governs the conditioning of the corresponding merit function 
\begin{equation}\label{eq:merit_ssm}
    \mathcal{L}(\mathbf{s}_{1:T}) := \frac{1}{2} \| \mathbf{r}(\mathbf{s}_{1:T}) \|_2^2.
\end{equation}
To show this rigorously, in the next section we introduce the Polyak-Łojasiewicz (PL) constant $\mu$ to quantify the conditioning (flatness) of $\mathcal{L}$. 

\section{Polyak-Łojasiewicz and Merit Landscape Conditioning}\label{subsec:Merit_PL}

\citet{ChewiStromme2025BallisticPL_AIHP} state that
\begin{quote}
    The Polyak-Łojasiewicz (PL) condition forms the cornerstone of modern non-convex optimization.
\end{quote}
Also known as gradient dominance, the PL condition \citep{polyak1963gradient, NesterovPolyak2006, karimi2016linear, fazel2018global} is simple:
A function $\mathcal{L} (\mathbf{s})$ is $\mu$-PL if it satisfies, for $\mu > 0$, 
\begin{equation}\label{eq:PL}
\frac{1}{2} ||\nabla \mathcal{L} (\mathbf{s})||^2 \ \geq \ \mu \, \left(\mathcal{L}(\mathbf{s}) - \mathcal{L}(\mathbf{s}^{\star}) \right)  
\end{equation}
for all $\mathbf{s}$. The largest $\mu$ for which \cref{eq:PL} holds for all $\mathbf{s}$ is called the PL constant of $\mathcal{L}(\mathbf{s}).$

In general, it can be difficult to use the PL condition if the minimum $\mathcal{L}(\mathbf{s}^{\star})$ is not known in advance. However, $\mathcal{L}(\mathbf{s}^{\star}) = 0$ in all applications of parallel Newton methods in this thesis, allowing for further simplification of \cref{eq:PL}.
PL is a form of gradient dominance because \cref{eq:PL} requires that if we are away from the true minimum (i.e. $\mathcal{L}(\mathbf{s}) - \mathcal{L}(\mathbf{s}^{\star})$ is large), then the gradient must be large as well. Therefore, the PL constant $\mu$ can be thought of as a measure of the "flatness" of the merit function: as $\mu \to 0$, the magnitude of the gradient becomes smaller and smaller as the merit function landscape becomes flatter and flatter, as shown in \Cref{fig:PL_flat}.
\begin{figure}
    \centering
    \includegraphics[width=0.95\linewidth]{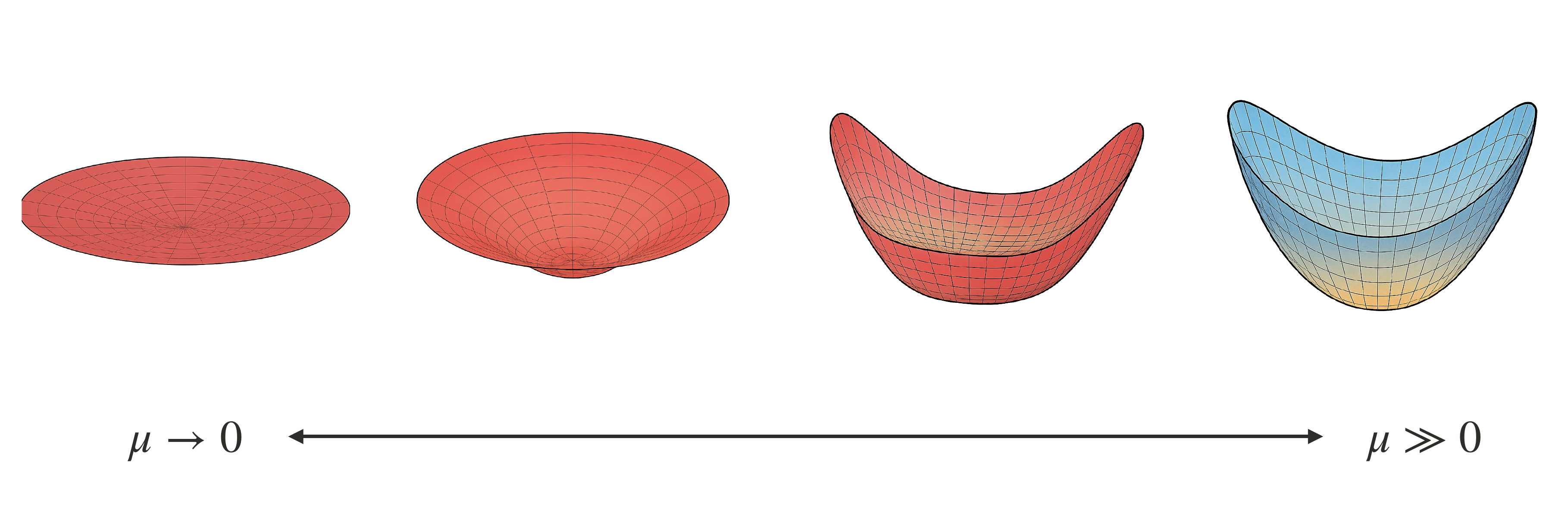}
    \caption{\textbf{PL constant $\mu$ captures the flatness of the merit function landscape.} We provide a schematic illustrating how a smaller PL constant $\mu$ results in flatter merit function landscapes.}
    \label{fig:PL_flat}
\end{figure}

All of the intuition and results about PL functions applies to parallel Newton methods because the merit function defined in \cref{eq:merit} satisfies \cref{eq:PL}. In fact, this result is known in the literature for general sum-of-squares functions \citep{NesterovPolyak2006}:
\begin{proposition}\label{prop:DEER_is_PL}
The merit function $\mathcal{L}(\mathbf{s})$  defined in \cref{eq:merit} satisfies \cref{eq:PL} for
\begin{equation}\label{eq:mu_def}
    \mu := \inf_{\mathbf{s}} \sigma_{\mathrm{min}}^2(\mathbf{J}(\mathbf{s})).
\end{equation}
\end{proposition}
\begin{proof}
Observe that 
\[ \nabla \mathcal{L}(\mathbf{s}) = \J(\mathbf{s})^\top \mathbf{r}(\mathbf{s}) \quad \text{and} \quad \mathcal{L}(\mathbf{s}^*) = 0.\]
Substituting these expressions into the PL inequality in \cref{eq:PL} we obtain
\[ \mathbf{r}^\top \, \J(\mathbf{s}) \, \J(\mathbf{s})^\top \, \mathbf{r} \ \geq \ \mu \, \mathbf{r}^\top \mathbf{r}.\]
Therefore, if $\J$ is full rank, then the merit function $\mathcal{L}$ is $\mu$-PL, where
\begin{align*}
    \mu & = \inf_{\mathbf{s}} \lambda_{\min} \left(\mathbf{J}(\mathbf{s}) \mathbf{J}(\mathbf{s})^{\top} \right) \\ 
    & = \inf_{\mathbf{s} } \sigma^2_{\min} \left(\mathbf{J}(\mathbf{s}) \right) \\ 
\end{align*}
\end{proof}
Consequently, the merit function in \cref{eq:merit_ssm} that is minimized by parallel Newton methods satisfies a number of desirable properties. For example, the merit function is \emph{invex}, meaning that all stationary points are global minima. In other words, no optimizer of the merit function in \cref{eq:merit_ssm} can be stuck in a local minimum or saddle point, because there are none: there is only the global minimizer $\mathbf{s}^{\star}$. The PL condition implies invexity, but we can also see the invexity of $\mathcal{L}(\mathbf{s}_{1:T})$ more directly: its gradient is $\nabla \mathcal{L}(\mathbf{s}) = \J(\mathbf{s})^\top \mathbf{r}(\mathbf{s})$, and $\J$ (defined in \cref{eq:big_j}) is always invertible. Therefore, the gradient can only be zero (a stationary point) when the residual is also zero, which occurs only at the true sequential rollout $\mathbf{s}_{1:T}^{\star}$.

Another reason why the PL condition is so important is that it is morally designed to be equivalent to linear rate for gradient descent. To provide this intuition, consider \emph{gradient flow} on a loss function $\mathcal{L}(\mathbf{s})$, i.e. the time evolution of $\mathcal{L}$ subject to $\mathbf{s}$ evolving according to
\begin{equation*}
    \dot{\mathbf{s}} = -\nabla \mathcal{L}(\mathbf{s}).
\end{equation*}
Then, if $\mathcal{L}$ is $\mu$-PL with $\mathcal{L}(\mathbf{s}^{\star}) = 0$, then
\begin{align*}
    \dot{\mathcal{L}} & = \nabla \mathcal{L} \cdot \dot{\mathbf{s}} && \text{chain rule} \\
    & = - \| \nabla \mathcal{L} \|^2 && \text{def. of grad. flow} \\
    & \leq -2 \mu \mathcal{L} && \text{PL condition}
\end{align*}
Therefore, it follows that $\mathcal{L}(t) \leq \mathcal{L}(0) \exp(-2 \mu t)$, which is linear rate for a continuous time system (i.e., the loss decays exponentially with the number of steps taken).
Note that the size of $\mu$ determines the precise convergence rate, with smaller $\mu$ (flatter landscapes) converging more slowly.
Converting this argument from gradient flow (continuous time) to gradient descent (discrete steps) is done in Theorem 1 of \citet{karimi2016linear} and requires only an additional Lipschitzness assumption to account for the discrete step sizes.
And, of course, by working backwards from the key desideratum of linear rate---i.e. that $\mathcal{L}(t) \leq \mathcal{L}(0) \exp(\gamma t)$ for some $\gamma$---we can also derive the PL condition.
Therefore, by showing that the merit function minimized by parallel Newton methods is PL, we show that it morally should achieve linear rate with gradient descent---albeit with a rate controlled by the flatness of the landscape $\mu$.

Having introduced the key ingredients---dynamical predictability as quantified by the LLE, and merit function conditioning as quantified by the PL constant---we now combine them in the next section to show how dynamical properties impact properties of $\J$ and $\mathcal{L}$.

\section{Conditioning depends on dynamical properties}\label{sec:cond}

In this section, we provide two results showing how the key quantities of the parallel Newton problem---chiefly the Jacobian $\J := \nicefrac{\partial \mathbf{r}}{\mathbf{s}_{1:T}}$ and the merit function $\mathcal{L}$---are determined by properties of the underlying dynamical system.
In particular, in \Cref{theorem:PL-LLE}, we show that the conditioning of $\J$ and $\mathcal{L}$ are determined by the predictability of the dynamics, while in \Cref{theorem:Lipschitz} we show that the Lipschitzness of $\J$ is also controlled by the Lipschitzness of the dynamical Jacobians $\{ \A_t \}$. These two results facilitate the proof and interpretation of convergence rates for various parallel Newton methods.

\subsection{Merit Function PL Constant is Controlled by the Largest Lyapunov Exponent of Dynamics}\label{ssc:pl_from_lyapunov}
As stated earlier, the Largest Lyapunov Exponent is a commonly used way to define the (un)predictability of a nonlinear state space model. In order to proceed, we need to control more carefully how the product of Jacobian matrices in \eqref{eq:LLE_definition} behaves for finite-time products. We will assume that there exists a "burn-in" period where the norm of Jacobian products can transiently differ from the LLE. In particular, we assume that 
\begin{equation}\label{eq:LLE_regularity}
\forall t > 1, \ \forall k \geq 0, \ \forall \mathbf{s}, \qquad b \ e^{\lambda k } \le \ \left\| \A_{t+k -1} \A_{t+k-2} \cdots \A_t \right\| \ \le \ \ a \ e^{\lambda k },  
\end{equation}
where $a \geq 1$ and $b \leq 1$. The constant $a$ quantifies the potential for transient growth—or overshoot—in the norm of Jacobian products before their long-term behavior emerges, while $b$ quantifies the potential for undershoot. 

\begin{theorem}\label{theorem:PL-LLE}
Assume that the LLE regularity condition \eqref{eq:LLE_regularity} holds. Then the PL constant $\mu$ satisfies
\begin{equation}\label{eq:merit_jacobian_inverse_LLE}
     \dfrac{1}{a} \cdot \dfrac{e^{\lambda} - 1}{e^{\lambda T} - 1} \ \leq \ \sqrt{\mu} \ \leq \ \min\left(\frac{1}{b} \cdot \frac{1}{e^{\lambda (T-1)}}, 1\right).
\end{equation}
\end{theorem}
\begin{proof}
See \Cref{appendix:pl_from_lyapunov_appendix} for the full proof and discussion. We provide a brief sketch.
Because $\sigma_{\mathrm{min}}(\mathbf{J}) = \nicefrac{1}{\sigma_{\mathrm{max}}(\mathbf{J}^{-1})}$, it suffices to control $\| \mathbf{J}^{-1} \|_2$. We can write $\mathbf{J} = \mathbf{I} - \mathbf{N}$ where $\mathbf{N}$ is a nilpotent matrix. 
Thus, it follows that $\mathbf{J}^{-1} = \sum_{k=0}^{T-1} \mathbf{N}^k$. As we discuss further in \Cref{appendix:pl_from_lyapunov_appendix}, the matrix powers $\mathbf{N}^k$ are intimately related to the dynamics of the system. The upper bound on $\|\mathbf{J}^{-1}\|_2$ follows after applying the triangle inequality and the formula for a geometric sum. The lower bound follows from considering $\|\mathbf{N}^{T-1}\|_2$.
\end{proof}
\vspace{-1.5mm}
Theorem \ref{theorem:PL-LLE} is the main result of this chapter, offering a novel connection between the predictability $\lambda$ of a nonlinear state space model and the conditioning $\mu$ of the corresponding merit function, which affects whether the system can be effectively parallelized.
If the underlying dynamics are unpredictable ($\lambda > 0$), then the merit function quickly becomes poorly conditioned with increasing $T$, because the denominators of both the lower and upper bounds explode due to the exponentially growing factor. Predictable dynamics $\lambda < 0$ lead to good conditioning of the optimization problem, and parallel methods based on merit function minimization can be expected to perform well in these cases. 
Indeed, when $\lambda < 0$, the conditioning of the merit function becomes asymptotically \textit{independent} of the sequence length $T$, due to the exponentially shrinking factor.

The proof mechanism we have sketched upper and lower bounds $\| \mathbf{J}^{-1} \|_2$ in terms of norms of Jacobian products. We only use the assumption in \cref{eq:LLE_regularity} to express those bounds in terms of $\lambda$. As we discuss at length in \Cref{appendix:pl_from_lyapunov_appendix}, we can use different assumptions from \cref{eq:LLE_regularity} to get similar results. \Cref{theorem:PL-LLE} and its proof should be thought of as a framework, where different assumptions (which may be more or less relevant in different settings) can be plugged in to yield specific results.

\paragraph{Why Unpredictable Systems have Excessively Flat Merit Functions
}
Theorem \ref{theorem:PL-LLE} demonstrates that the merit function becomes extremely flat for unpredictable systems and long trajectories. This flatness poses a fundamental challenge for \textit{any} method that seeks to compute state trajectories by minimizing the merit function. We now provide further intuition to explain why unpredictability in the system naturally leads to a flat merit landscape.

Suppose that we use an optimizer to minimize the merit function \eqref{eq:merit_ssm} for an unpredictable system until it halts with some precision. Let us further assume that the first state of the output of this optimizer following the initial condition is $\epsilon$-close to the true first state,~$\| s_1 - s^*_1 \| = \epsilon$.
Suppose also that the residuals for all times greater than one are precisely zero---in other words, the optimizer starts with a "true" trajectory starting from initial condition $s_1$. Then the overall residual norm is at most $\epsilon$, 
\begin{align*}
\| \mathbf{r}(\mathbf{s}) \|^2 
= \| s_1 - f(s_0) \|^2 
\leq \left( \| s_1 - s_1^* \| + \| s_1^* - f(s_0) \| \right)^2 
= \| s_1 - s_1^* \|^2 
= \epsilon^2.
\end{align*}
However, since $s_{t}$ and $ s_{t}^*$ are by construction both trajectories of an unpredictable system starting from slightly different initial conditions $s_1$ and $s_1^*$, the distance between them will grow exponentially as a consequence of \cref{eq:chaotic_time_horizon}. By contrast, predictable systems will have errors that shrink exponentially. This shows that changing the initial state $s_1$ by a small amount can lead to a massive change in the trajectory of an unpredictable system, but a \textit{tiny change} in the merit function. Geometrically, this corresponds to the merit function landscape for unpredictable systems having excessive flatness around the true solution (Figure \ref{fig:pred}, bottom right panel). Predictable systems do not exhibit such flatness, since small residuals imply small errors. Theorem \ref{theorem:PL-LLE} formalizes this idea.

\subsection{Residual function Jacobian Inherits the Lipschitzness of the Nonlinear State Space Model}\label{ssc:L_inherit}
In addition to the parameter $\mu$, which measures the conditioning of the merit function, the difficulty of minimizing the merit function is also influenced by the Lipschitz continuity of its Jacobian $\J$. The following theorem establishes how the Lipschitz continuity of the underlying sequence model induces Lipschitz continuity in $\J$.
\begin{theorem}\label{theorem:Lipschitz}
If the dynamics of the underlying nonlinear state space model have $L$-Lipschitz Jacobians, i.e., 
\[
\forall \, t > 1, \ \ \ s, s' \in \mathbb{R}^D: \quad \|\A_t(s) - \A_t(s')\| \leq L \|s - s'\|,
\]
then the residual function Jacobian $\J$ is also $L$-Lipschitz, with the same $L$. 
\end{theorem}
\begin{proof}
By assumption, for each $t$, 
\[ \forall s,s' \in \mathbb{R}^D: \quad 
\|\A_t(s_t) - \A_t(s_t')\|_2 \;\le\; L\,\|s_t - s_t'\|_2.
\]
Define $D_t := \A_t(s_t') - \A_t(s_t)$ and
\[
\mathbf{D} \;:=\; \mathbf{J}(\mathbf{s}') - \mathbf{J}(\mathbf{s}).
\]
Since $\mathbf{D}$ places the blocks $D_t$ along one subdiagonal, we have
\[
\|\mathbf{D}\|_2 \;=\; \max_{t}\,\|D_t\|_2.
\]
But each block $D_t$ satisfies the Lipschitz bound
\[
\|D_t\|_2 \;\le\; L\,\|s_t' - s_t\|_2,
\]
so
\[
\|\mathbf{D}\|_2 
\;=\; 
\max_{t}\|D_t\|_2
\;\le\;
L \,\max_{t}\|s_t' - s_t\|_2
\;\le\;
L\,\|\mathbf{s}' - \mathbf{s}\|_2.
\]
Hence, it follows that
\[
\|\mathbf{J}(\mathbf{s}') - \mathbf{J}(\mathbf{s})\|_2 
\;=\; 
\| \mathbf{D} \|_2
\;\le\;
L\,\|\mathbf{s}' - \mathbf{s}\|_2.
\]
Thus $\mathbf{J}$ is $L$-Lipschitz.
\end{proof}
Theorem \ref{theorem:Lipschitz} will be important for the analysis in Section \ref{sec:rates_of_DEER_convergence}, where we consider convergence rates.
Because Gauss-Newton methods rely on iteratively linearizing the dynamics (or equivalently the residual), they converge in a single step for linear dynamics $L=0$, and converge more quickly if the system is close to linear ($L$ is closer to $0$). 

\section{Rates of Convergence for Optimizing the Merit Function}\label{sec:rates_of_DEER_convergence}
In~\Cref{sec:cond}, we established that the predictability of the nonlinear state space model directly influences the conditioning of the merit function. This insight is critical for analyzing \textit{any} optimization method used to compute trajectories via minimization of the merit function.

In this section, we apply those results to study the convergence behavior of the Gauss-Newton (DEER) algorithm for the merit function defined in \cref{eq:merit_ssm}. We derive worst-case bounds on the number of optimization steps required for convergence. In addition, we present an average-case analysis of DEER that is less conservative than the worst-case bounds and more consistent with empirical observations.

\subsection{DEER Always Converges Globally at a Linear Rate}
Although DEER is based on the Gauss-Newton method, which generally lacks global convergence guarantees, we prove that DEER always converges globally at a linear rate. This result relies on the problem’s specific hierarchical structure, which ensures that both the residual function Jacobian $\b J$ and its inverse are lower block-triangular. In particular, we prove the following theorem:
\begin{theorem}\label{theorem:deer_converges_globally}
Let the DEER (Gauss--Newton) updates be given by \cref{eq:deer_full_update}, and let $\mathbf{s}^{(i)}$ denote the $i$-th iterate. Let ${\mathbf{e}^{(i)} :=  \mathbf{s}^{(i)} - \mathbf{s}^* }$ denote the error at iteration $i$, and assume the regularity condition in \cref{eq:LLE_regularity}. Then the error converges to zero at a linear rate:
\[
\| \mathbf{e}^{(i)} \|_2 \leq \osw \, \beta^i \| \mathbf{e}^{(0)} \|_2,
\]
for some constant $\osw \geq 1$ independent of $i$, and a convergence rate $0 < \beta < 1$.
\end{theorem}
\begin{proof}
Our general strategy for deriving DEER convergence bounds will be to fix some weighted norm
$\| \mathbf{e} \|_W \coloneq \| \b W^{1/2} \mathbf{e} \|_2$, for a symmetric positive definite matrix $\mathbf{W}$.
Doing so induces the operator norm
$\| \mathbf{J} \|_W \coloneq \| \b W^{1/2} \mathbf{J} \b W^{-1/2}\|_2$ such that each DEER step is a contraction in this norm, with contraction factor $\beta \in [0,1)$. This will imply that the DEER error iterates decay to zero with linear rate, as
\begin{equation}\label{eq:lin_rate_W}
    \|\mathbf{e}^{(i)} \|_W \leq \beta^i \| \mathbf{e}^{(0)} \|_W,   
\end{equation}
i.e.
\begin{equation*}
    \| \mathbf{W}^{1/2} \mathbf{e}^{(i)} \|_2 \leq \beta^i \| \mathbf{W}^{1/2} \mathbf{e}^{(0)} \|_2.
\end{equation*}
Using the above equation and properties of singular values, it follows that
\begin{equation*}
    \sqrt{\lambda_{\min}(\mathbf{W})} \| \mathbf{e}^{(i)} \|_2 \leq \beta^i \sqrt{\lambda_{\max}(\mathbf{W})} \| \mathbf{e}^{(0)} \|_2.
\end{equation*}
Therefore, to convert the linear rate in \cref{eq:lin_rate_W} back to standard Euclidean space, we incur an additional multiplicative factor that depends on the conditioning of $\b W^{1/2}$:
\begin{equation}\label{eq:error_bound_condition_number}
\| \mathbf{e}^{(i)} \|_2 \leq \osw \, \beta^i \| \mathbf{e}^{(0)} \|_2, \qquad \text{where} \qquad \osw \coloneq \sqrt{\frac{\lambda_{\max}(\b W)}{\lambda_{\min}(\b W)}}.    
\end{equation}

\paragraph{DEER as a Contraction Mapping}

Recall that the DEER (Gauss-Newton) updates are given by 
\[\mathbf{s}^{(i+1)} = \mathbf{s}^{(i)} - \mathbf{J}^{-1}(\mathbf{s}^{(i)}) \mathbf{r}(\mathbf{s}^{(i)})\]
Recalling that $\b r(\b s^*) = \b 0$ and subtracting the fixed point $\mathbf{s}^{*}$ from both sides, we have that 
\[\mathbf{e}^{(i+1)} \ = \ \mathbf{e}^{(i)} - \mathbf{J}^{-1}(\mathbf{s}^{(i)}) \mathbf{r}^{(i)} + \mathbf{J}^{-1}(\mathbf{s}^{(i)}) \, \b r(\b s^*) \ = \ \mathbf{e}^{(i)} - \mathbf{J}^{-1}(\mathbf{s}^{(i)}) \, \bigg(\b r( \mathbf{s}^{(i)}) - \b r(\b s^*) \bigg). \]
This equation can be written using the mean value theorem as 
\[\mathbf{e}^{(i+1)} = \bigg( \b I - \mathbf{J}^{-1}(\mathbf{s}^{(i)}) \mathbf{B}^{(i)} \bigg) \mathbf{e}^{(i)}  \qquad \text{where} \qquad \mathbf{B}^{(i)}  \coloneq \int_0^1 \mathbf{J}(\mathbf{s}^* + \tau \mathbf{e}^{(i)}) \, d\tau\]
From this, we can conclude that the DEER iterates will converge (i.e., the error shrinks to zero) if 
\begin{equation}\label{eq:DEER_contraction_1}
\| \mathbf{I} -  \mathbf{J}^{-1}(\mathbf{s}^{(i)}) \mathbf{B}^{(i)} \|_W \ = \ \| \mathbf{J}^{-1}(\mathbf{s}^{(i)}) \left(\mathbf{J}(\mathbf{s}^{(i)}) - \mathbf{B}^{(i)} \right) \|_W \leq \beta <  1.   
\end{equation}

\paragraph{Constructing the Weighted Norm}
We will choose a diagonal weighted norm, given by 
\begin{equation}
\mathbf W \;\coloneqq\; \operatorname{Diag} \bigl(I_D,\;w^{2}I_D,\;\dots,\;w^{2(T-1)}I_D\bigr)
      \;\in\;\mathbb R^{TD\times TD},
      \qquad w>0 .
\label{eq:W}
\end{equation}
Under the norm induced by \eqref{eq:W} we have
\begin{align}
\|\mathbf J(\mathbf{s}^{(i)})-\mathbf B^{(i)}\|_{W} &\;\le\; 2w\rho ,               && \label{eq:JminusB}\\[4pt]
\|\mathbf J^{-1}(\mathbf{s}^{(i)})\|_{W} &\;\le\; a\frac{1-(we^{\lambda})^{T}}{1-we^{\lambda}}, && \label{eq:Jinv}
\end{align}
where $\rho$ upper bounds $\| J \|_2$ over all states in the DEER optimization trajectory.

% Combined bound
Multiplying \eqref{eq:JminusB} and \eqref{eq:Jinv} yields
\begin{equation}
\|\mathbf J^{-1}(\mathbf{s}^{(i)})\|_{W}\,\|\mathbf J(\mathbf{s}^{(i)})-\mathbf B^{(i)}\|_{W}
      \;\le\;
      2a w\rho\,
      \frac{1-(we^{\lambda})^{T}}{1-we^{\lambda}} .
\label{eq:ProductBound}
\end{equation}

% Choice of w
To ensure the right‑hand side of \eqref{eq:ProductBound} does not exceed a prescribed
$\beta\in[0,1)$, choose
\begin{equation}
w \;=\; \frac{\beta}{2\rho a+\beta e^{\lambda}} .
\label{eq:wChoice}
\end{equation}

% Consistency checks
With this choice, 
\begin{equation}
we^{\lambda}<1,
\qquad\text{and}\qquad
\dfrac{2a w\rho}{1-we^{\lambda}} \;=\; \beta,
\label{eq:Consistency}
\end{equation}
so the geometric series in \eqref{eq:Jinv} is convergent and the bound in
\eqref{eq:ProductBound} holds for all $T$, because
\[ \|\mathbf J^{-1}(\mathbf{s}^{(i)})\|_{W}\,\|\mathbf J(\mathbf{s}^{(i)})-\mathbf B^{(i)}\|_{W}
      \;\le\;
      2a w\rho\,
      \frac{1-(we^{\lambda})^{T}}{1-we^{\lambda}}  \ = \ \beta \left(1- (w e^{\lambda })^T\right) \ \leq \ \beta.\]

This shows that we can always pick a weighted norm so that DEER converges with linear rate \textit{in that norm}. Converting back into the standard Euclidean norm using \eqref{eq:error_bound_condition_number} and substituting in the condition number of $\b W^{1/2}$ one finds that
\begin{equation}\label{eq:general_DEER_convergence}
 \| \mathbf{e}^{(i)} \|_2 \leq \ \left(\frac{2\rho a+\beta e^{\lambda}}{\beta }\right)^T  \, \beta^i \, \| \mathbf{e}^{(0)} \|_2.    
\end{equation}
Thus, the DEER error converges with linear rate towards zero.
\end{proof}
Theorem \ref{theorem:deer_converges_globally} is unexpected since, in general, Gauss-Newton methods do not enjoy global convergence.
The key caveat of this theorem is the multiplicative factor $\osw$, which can grow exponentially with the sequence length $T$. This factor governs the extent of transient error growth before the decay term $\beta^i$ eventually dominates.

Theorem \ref{theorem:deer_converges_globally} has several useful, practical consequences. First, when the nonlinear state space model is sufficiently contracting ($\lambda$ is sufficiently negative), then $\osw$ in  Theorem \ref{theorem:deer_converges_globally} can be made small, implying that in this case DEER converges with little-to-no overshoot.
% (Appendix \ref{sec:DEER_Global_Convergence_Contraction}).

Theorem \ref{theorem:deer_converges_globally} also lets us establish key worst-case and average-case bounds on the number of steps needed for Gauss-Newton
% and gradient descent
to converge to within a given distance of the solution. In particular, when $\osw$ does not depend on the sequence length $T$, then \Cref{theorem:deer_converges_globally} implies Gauss-Newton will only require $\mathcal{O}\left( (\log T)^2 \right)$ total computational time, with one $\log$ factor coming from the parallel scan at each optimization step and the other coming from the total number of optimization steps needed.
% We elaborate on these points in Appendix \ref{section:worst_case_complexity}.

\subsection{Size of DEER Basin of Quadratic Convergence}\label{ssc:quad_basin}
It is natural that DEER depends on the Lipschitzness of $\mathbf{J}$ since Gauss-Newton converges \emph{in one step} for linear problems, where $L = 0$.
In \Cref{sec:cond}, we showed that the conditioning of the merit function, as measured by the PL-constant $\mu$, depends on the stability, or predictability, of the nonlinear dynamics. Thus, the performance of DEER depends on the ratio of the nonlinearity and stability of the underlying nonlinear state space model. Note that once $\b s$ is inside the basin of quadratic convergence, it takes $O(\log \log (1/\epsilon))$ steps to reach $\epsilon$ residual (effectively a constant number of steps). 

Because DEER converges so quickly within its basin of quadratic convergence, it is important to understand the size of this basin in terms of the properties of the underlying SSM we are trying to parallelize.
We provide such a bound in \Cref{theorem:basin_size}.
We make no claim about the originality of lower bounding the size of the basin of quadratic convergence in Gauss-Newton.
In fact, our proof of \Cref{theorem:basin_size} closely follows the convergence analysis of \emph{Newton's} method in Section 9.5.3 of \citet{boyd2004convex}.
Our contribution is we highlight the elegant way the predictability $\lambda$ and nonlinearity $L$ of a dynamical system influence an important feature of its merit function's landscape.

\begin{theorem}\label{theorem:basin_size}
Let $\mu$ denote the PL-constant of the merit function, which Theorem \ref{theorem:PL-LLE} relates to the LLE $\lambda$. Let $L$ denote the Lipschitz constant of the Jacobian of the dynamics function $\A(s)$. Then, $\nicefrac{2 \mu}{L}$ lower bounds the radius of the basin of quadratic convergence of DEER; that is, if
\begin{equation}\label{eq:basin_of_q_convergence_size}
||\b r(\b s^{(i)}) ||_2 <  \ \dfrac{2 \mu}{L},
\end{equation}
then $\b s^{(i)}$ is inside the basin of quadratic convergence. In terms of the LLE $\lambda$, it follows that if 
\begin{equation*}
    ||\b r(\b s^{(i)}) ||_2 < \ \dfrac{2}{a^2 L} \cdot \left( \dfrac{e^{\lambda}-1}{e^{\lambda T}-1} \right)^2,
\end{equation*}
then $\b s^{(i)}$ is inside the basin of quadratic convergence.
\end{theorem}
\begin{proof}
Suppose we are at a point $\b s^{(i)} \in \mathbb{R}^{TD}$ (i.e. DEER iterate $i$), and we want to get to $\b s^{(i+1)}$.
The change in the trajectory is,
\begin{equation*}
    \Delta \mathbf{s}^{(i)} := -\b J(\b s^{(i)})^{-1} \b r(\b s^{(i)})
\end{equation*}
(where the iteration number will hopefully be clear from context).
The merit function is $\mathcal{L}(\mathbf{s}) = \frac{1}{2} \| \b r(\mathbf{s}) \|_2^2$, so if we can get some control over $\| \mathbf{r}(\mathbf{s}^{(i)}) \|_2$, we will be well on our way to proving a quadratic rate of convergence.

First, leveraging the form of the Gauss-Newton update, we can simply "add zero" to write
\begin{align*}
    \b r(\mathbf{s}^{(i+1)}) & = \mathbf{r}(\mathbf{s}^{(i)} + \Delta \b s^{(i)}) \\
    & = \mathbf{r}(\mathbf{s}^{(i)} + \Delta \b s^{(i)}) - \mathbf{r}(\mathbf{s}^{(i)}) - \b J(\mathbf{s}^{(i)}) \Delta \b s^{(i)}
\end{align*}
Next, we can write the difference $\mathbf{r}(\mathbf{s}^{(i)} + \Delta \b s^{(i)}) - \mathbf{r}(\mathbf{s}^{(i)})$ as the integral of the Jacobian, i.e.
\begin{equation*}
    \mathbf{r}(\mathbf{s}^{(i)} + \Delta \b s^{(i)}) - \mathbf{r}(\mathbf{s}^{(i)}) = \int_0^1 \b J\left(\b s^{(i)} + \tau \Delta \b s^{(i)}\right) \Delta \b s^{(i)} \, d\tau.
\end{equation*}
Therefore,
\begin{equation*}
    \b r(\mathbf{s}^{(i+1)}) = \int_0^1 \left( \b J\left(\b s^{(i)} + \tau \Delta \b s^{(i)}\right) - \b J( \mathbf{s}^{(i)}) \right) \Delta \b s^{(i)} \, d\tau
\end{equation*}
Taking $\ell_2$-norms and using the triangle inequality, it follows that
\begin{equation*}
    \| \b r(\mathbf{s}^{(i+1)}) \|_2 \leq \int_0^1 \left\| \left( \b J\left(\mathbf{s}^{(i)} + \tau \Delta \b s^{(i)}\right) - \b J(\mathbf{s}^{(i)}) \right) \Delta \b s^{(i)} \right\|_2 d\tau.
\end{equation*}
Now, if we assume that $\b J$ is $L$-Lipschitz and use the definition of spectral norm, it follows that
\begin{equation*}
    \left\| \left( \b J\left(\mathbf{s}^{(i)} + \tau \Delta \b s^{(i)}\right) - \b J(\mathbf{s}^{(i)}) \right) \Delta \b s^{(i)} \right\|_2 \leq \tau L \| \Delta \b s^{(i)} \|_2^2,
\end{equation*}
and so taking the integral we obtain
\begin{align*}
    \| \b r(\mathbf{s}^{(i+1)}) \|_2 & \leq \frac{L}{2} \| \Delta \b s^{(i)} \|_2^2 \\
    & = \frac{L}{2} \mathbf{r}(\mathbf{s}^{(i)})^{\top} \b J(\mathbf{s}^{(i)})^{-\top} \b J(\mathbf{s}^{(i)})^{-1} \mathbf{r}(\mathbf{s}^{(i)}).
\end{align*}
By definition, $\sqrt{\mu}$ is a lower bound on all singular values of $\mathbf{J}(\mathbf{s}(i))$, for all $i$.
Therefore, $ \| \b J(\mathbf{s}^{(i)})^{-1} \|_2 \leq \nicefrac{1}{\sqrt{\mu}}$ for all $i$, and it follows that
\begin{equation}\label{eq:GN_residual_shrink}
    \| \mathbf{r}(\mathbf{s}^{(i+1)}) \|_2 \leq \frac{L}{2 \mu} \| \mathbf{r}(\mathbf{s}^{(i)}) \|_2^2,
\end{equation}
which is the direct analogy of \citet[9.33]{boyd2004convex}. To reiterate, here $L$ is the Lipschitz constant of $\b J$, while $\mu:=\inf_{i \in \mathbb{N}} \sigma^2_{\mathrm{min}}\left( \mathbf{J}(\mathbf{s}^{(i)}) \right)$. 

While this is a quadratic convergence result for GN, this result is not useful unless $\| \mathbf{r}(\mathbf{s}^{(i+1)}) \|_2 \leq \| \mathbf{r}(\mathbf{s}^{(i)}) \|_2$ (i.e. would backtracking line search accept this update).
However, if we have $\| \mathbf{r}(\mathbf{s}^{(i)}) \|_2 < \frac{2 \mu}{L}$, then every step guarantees a reduction in $\mathbf{r}$ because in this case
\begin{equation*}
    \| \mathbf{r}(\mathbf{s}^{(i+1)}) \|_2 < \| \mathbf{r}(\mathbf{s}^{(i)}) \|_2.
\end{equation*}
Therefore, we have $\| \mathbf{r}(\mathbf{s}^{(j)}) \|_2 < \frac{2 \mu}{L}$ for all $j > i$.
Thus, we have related the size of the basin of quadratic convergence of GN on the DEER objective to the properties of $\b J$.
Note that with linear dynamics, each $\A_t$ is constant in $s$, and so each $\A_t$ is $0$-Lipschitz. Thus, the basin of quadratic convergence becomes infinite. Intuitively, if $\A_t$ doesn't change too quickly with $s$, then DEER becomes a more and more potent method.
\end{proof}

\section{Experiments}\label{sec:experiments}
We conduct experiments to support the theory developed above, demonstrating that predictability enables parallelization of nonlinear SSMs. To illustrate this point, we use Gauss-Newton optimization (aka DEER).
Our code is at~\url{https://github.com/lindermanlab/predictability_enables_parallelization}

\subsection{The Convergence Rate Exhibits a Threshold between Predictable and Chaotic Dynamics}\label{ssc:exp_thresh}

\begin{figure}[ht]
    \centering
    \includegraphics[width=1.0\textwidth]{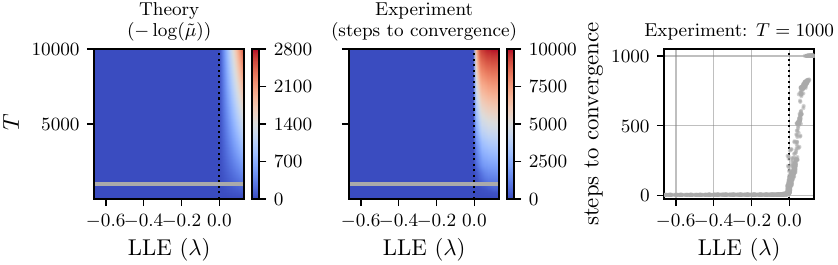}
    \caption{\textbf{Threshold phenomenon in DEER convergence based on system predictability.} In a family of RNNs, DEER has fast convergence for predictable systems and prohibitively slow convergence for chaotic systems. \textbf{Left (Theory):} We depict Theorem \ref{theorem:PL-LLE}, illustrating how the conditioning of the optimization problem degrades as $T$ and the LLE ($\lambda$) increase. \textbf{Center (Experiment):} We vary $\lambda$ across the family of RNNs, and observe a striking concordance in the number of DEER optimization steps empirically needed for convergence with our theoretical characterization of the conditioning of the optimization problem. \textbf{Right:} For 20 seeds, each with 50 different values of $\lambda$, we plot the relationship between $\lambda$ and the number of DEER steps needed for convergence for the sequence length $T=1000$ (gray line in left and center panels). We observe a sharp increase in the number of optimization steps at precisely the transition between predictability and unpredictability.}
    \label{fig:thresh}
\end{figure}

Theorem \ref{theorem:PL-LLE} predicts a sharp phase transition in the conditioning of the merit function at $\lambda = 0$, which should be reflected in the number of optimization steps required for convergence. To empirically validate this prediction, we vary both the LLE and sequence length $T$ within a parametric family of recurrent neural networks (RNNs), and measure the number of steps DEER takes to converge. We generate mean-field RNNs following \citet{engelken2023lyapunov}, scaling standard normal weight matrices by a single parameter that controls their variance and therefore the expected LLE. 

In more detail, we rolled out trajectories from a mean-field RNN with step size $1$ for $20$ different random seeds. The dynamics equations follow the form
\begin{equation*}
    s_{t+1} = W \mathrm{tanh}(s_t) + u_t,
\end{equation*}
for mild sinusoidal inputs $u_t$. We have $s_t \in \mathbb{R}^D$, where in our experiments $D=100$. Note that because of the placement of the saturating nonlinearity, here $s_t$ represents current, not voltage. We draw each entry $W_{ij} \stackrel{\mathrm{iid}}{\sim} \mathcal{N}(0,\nicefrac{g^2}{D})$, where $g$ is a scalar parameter. We then set $W_{ii} = 0$ for all $i$ (no self-coupling of the neurons). A key point of \citet{engelken2023lyapunov} is that by scaling the single parameter $g$, the resulting RNN goes from predictable to chaotic behavior. While \citet{engelken2023lyapunov} computes the full Lyapunov spectrum in the limit $D \to \infty$, for finite $D$ we can compute a very accurate numerical approximation to the LLE. In particular, we use \Cref{alg:LLE_estimate} to compute the LLE in a numerically stable way. Note that the algorithm nominally depends on the initial unit vector $u_0$. For this reason, we choose 3 different unit vectors (initialized at random on the unit sphere) and average over the 3 stochastic estimates. However, in practice we observe that the estimate is very stable with respect to choice $u_0$, and agrees with systems for which the true LLE is known, such as the Henon and logistic maps.

\begin{algorithm}
\caption{Numerically Stable Computation of Largest Lyapunov Exponent (LLE)}
\begin{algorithmic}[1]
    \State \textbf{Input:} Initial unit vector $u_0$, total iterations $T$
    \State \textbf{Initialize:} $\text{LLE} \gets 0$
    \For{$t = 1$ to $T$}
        \State Compute evolved vector: $u_t \gets J_t u_{t-1}$
        \State Compute stretch factor: $\lambda_t \gets \|u_t\|$
        \State Normalize vector: $u_t \gets u_t / \lambda_t$
        \State Accumulate logarithmic stretch: $\text{LLE} \gets \text{LLE} + \log \lambda_t$
    \EndFor
    \State \textbf{Output:} Estimated LLE $\lambda \gets \text{LLE} / T$
\end{algorithmic}
\label{alg:LLE_estimate}
\end{algorithm}

In Figure \ref{fig:g_vs_lle_thresh}, we verify numerically that there is a monotonic relationship between $g$ and the LLE of the resulting system, and that the min-max range for 20 seeds is small.
\begin{figure}
    \centering
    \includegraphics[width=0.5\linewidth]{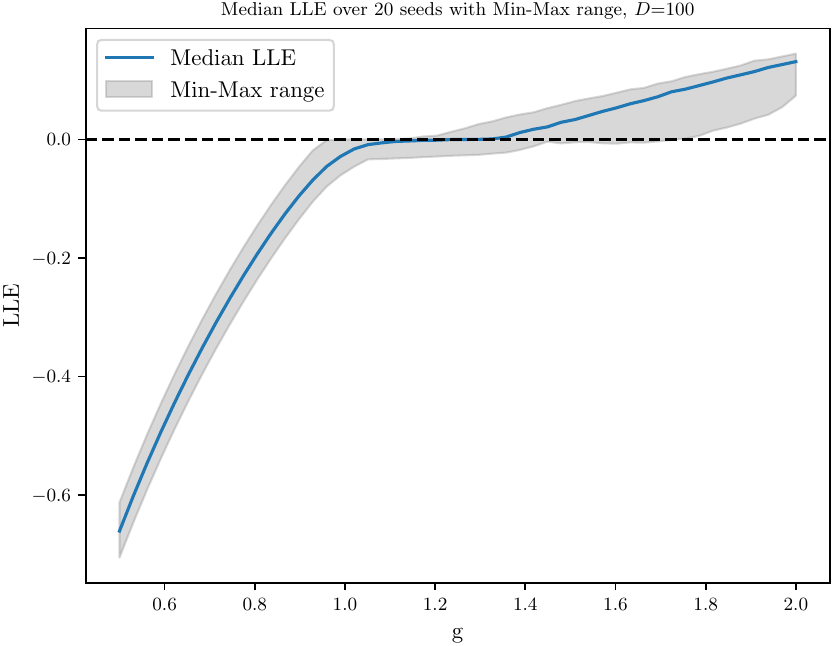}
    \caption{\textbf{Robust relationship in mean field RNN between variance parameter $g$ and LLE of the system.} For 20 seeds, we observe a robust and non-decreasing relationship between the scalar parameter $g$ and the LLE of the resulting mean-field RNN. The plot above is made for $50$ different values of $g$ from $0.5$ to $2.0$ (linearly spaced). We estimate the LLE over a sequence length of $T=9999$.}
    \label{fig:g_vs_lle_thresh}
\end{figure}
Accordingly, when making Figure \ref{fig:thresh} (Center), we use the monotonic relationship between $g$ and the LLE from Figure \ref{fig:g_vs_lle_thresh} to map the average number of DEER steps (over 20 different seeds) needed for convergence for different values of $g$ to the appropriate value of the LLE. We use 50 values of $T$ from 9 to 9999 (log spaced) to make Figure \ref{fig:thresh} (Center). We highlight $T=1000$ in Figure \ref{fig:thresh} (Right).

Overall, in \Cref{fig:thresh}, we observe a striking correspondence between the conditioning of the optimization problem (represented by~$-\log \tilde{\mu}$, where $\tilde{\mu}$ is the lower bound for $\mu$ from \Cref{theorem:PL-LLE}) and the number of steps DEER takes to converge.
This relationship holds across the range of LLEs, $\lambda$, and sequence lengths, $T$.
There is a rapid threshold phenomenon around $\lambda=0$, which divides predictable from unpredictable dynamics, precisely as expected from \Cref{theorem:PL-LLE}.
The correspondence between~$-\log \tilde{\mu}$ and the number of optimization steps needed for convergence can be explained by DEER iterates approaching the basin of quadratic convergence with linear rate.

\paragraph{Wallclock time and other optimizers}

Our findings about the conditioning of the merit landscape apply to any solver. To show the generality of \Cref{prop:global_convergence}, we parallelize the sequential rollout of the mean field RNN with other optimizers like quasi-Newton and gradient descent, and observe that the number of steps these optimizers take to converge also scales with the LLE. We also record wallclock times on an H100, and observe that DEER is faster than sequential by an order of magnitude in predictable settings, but slower by an order of magnitude in unpredictable settings. We summarize this experiment in \Cref{fig:gd_and_wallclock}.

\begin{figure}[ht]
    \centering
    \includegraphics[width=1.0\textwidth]{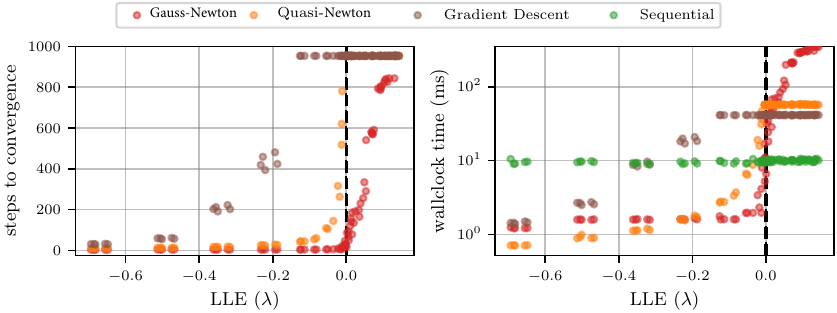}
    \caption{\textbf{Convergence rates and wallclock time for many optimizers.} We supplement the mean-field RNN experiment by also considering quasi-Newton and gradient descent methods \textbf{(top)}, and recording wallclock time, including for sequential evaluation \textbf{(bottom)}}
    \label{fig:gd_and_wallclock}
\end{figure}

 In the top panel of \Cref{fig:gd_and_wallclock}, we observe that the number of steps for gradient descent and quasi-DEER to converge also scales monotonically with the LLE, as we expect from \Cref{theorem:PL-LLE}. DEER (Gauss-Newton) converges in a small number of steps all the way up to the threshold between predictability and unpredictability ($\lambda=0$). Intuitively, the performance of the other optimizers degrades more quickly as unpredictability increases because quasi-Newton and gradient descent use less information about the curvature of the loss landscape.

Even though gradient descent was slower to converge in this setting, we only tried gradient descent with a fixed step size. An advantage of a first-order method like gradient descent over a second-order method like Gauss-Newton (DEER) is that the first-order method is embarrassingly parallel (and so with sufficient parallel processors, the update runs in constant time), while DEER and quasi-DEER use parallel scans (and so the update runs in $O(\log T)$ time). Exploring accelerated first-order methods like Adam \citep{adam}, or particularly Shampoo \citep{shampoo} or SOAP \citep{vyas2024soap} (which are often preferred in recurrent settings like \cref{eq:ssm})---or in general trying to remove the parallel scan---are therefore very interesting directions for future work.

Sequential evaluation of \cref{eq:ssm} can also be thought of as block coordinate descent on the merit function $\mathcal{L}(\mathbf{s})$, where the block $s_t \in \mathbb{R}^D$ is optimized at optimization step $(t)$. The optimization of each block is a convex problem: simply minimize $\| s_t - f(s_{t-1}^*) \|_2^2$, or equivalently set $s_t = f(s_{t-1}^*)$. As sequential evaluation will always take $T$ steps to converge, we do not include it in the top panel of \Cref{fig:gd_and_wallclock}.

In the bottom panel of \Cref{fig:gd_and_wallclock}, we also report the wallclock times for these algorithms to run (our experiments are run on an H100 with 80 GB onboard memory). We observe that the run time of sequential evaluation (green) is effectively constant with respect to $\lambda$. We observe that in the predictable setting, DEER is an order of magnitude faster than sequential evaluation, while in the unpredictable regime, DEER is 1-2 orders of magnitude slower than sequential evaluation. This importance of using parallel evaluation only in predictable settings is a core practical takeaway from our theoretical contributions.

We run the experiment in \Cref{fig:gd_and_wallclock} on a smaller scale than the experiment in \Cref{fig:thresh} (Right). In \Cref{fig:gd_and_wallclock}, we consider 5 random seeds for 16 values of $g$ equispaced between $0.5$ and $2.0$. Each wallclock time reported is the average of 5 runs for the same seed. We use a batch size of 1. While DEER (Gauss-Newton) and quasi-DEER effectively do not have a step size (they use a step size of $1$ always). For each value of $g$, we ran gradient descent with the following set of step sizes $\alpha$: $0.01, 0.1, 0.25, 0.5, 0.6, 0.7, 0.8, 0.9, \text{and } 1.0$. For each value of $g$, we then pick the step size $\alpha$ that results in the fastest convergence of gradient descent. For the smallest value of $g=0.5$, we use $\alpha=0.6$; for $g=0.6$, we use $\alpha=0.5$; and for all other values of $g$, we use $\alpha=0.25$. Future work may investigate more adaptive ways to tune the step size $\alpha$, or to use a learning rate schedule.

We use a larger tolerance of $\nicefrac{\mathcal{L}(\mathbf{s})}{T} \leq 10^{-4}$ to declare convergence than in the rest of the paper (where we use a tolerance of $10^{-10}$) because gradient descent often did not converge to the same degree of numerical precision as sequential, quasi-DEER, or DEER. However, this is a per time-step average error on the order of $10^{-4}$, in a system where $D=100$ and each state has current on the order of $1$. Nonetheless, it is an interesting direction for future work to investigate how to get gradient descent to converge to greater degrees of numerical precision in these settings; and, in general, how to improve the performance of all of these parallel sequence evaluators in lower numerical precision.

\subsection{DEER can converge quickly for predictable trajectories passing through unpredictable regions}\label{ssc:2well}

\begin{figure}[ht]
    \centering
    \includegraphics[width=1.0\textwidth]{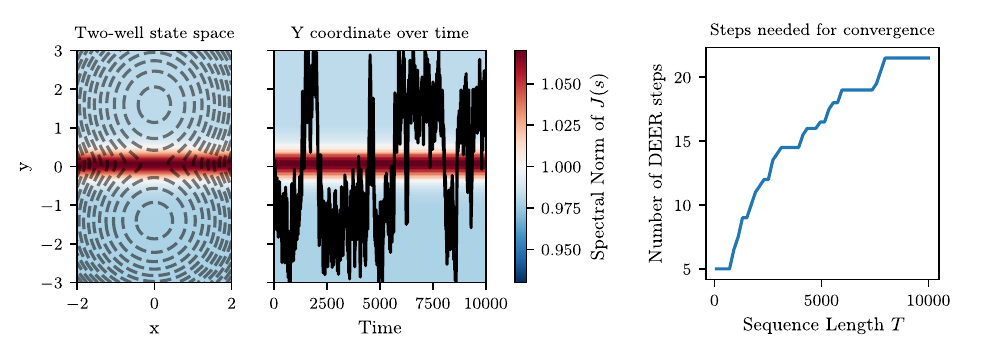}
    \caption{DEER converges quickly for Langevin dynamics in a two-well potential. \textbf{(Left)} An illustration of the two-well potential state space in $D=2$. We superimpose a contour plot of the potential on a color scheme showing the spectral norm of the dynamics Jacobian (blue indicates stability, red instability). \textbf{(Center)} A trace plot for the $y$-coordinate. The LLE of the system is $-0.0145$. \textbf{(Right)} We observe that this system, which has negative LLE, enjoys sublinear scaling in the sequence length $T$ in the number of DEER iterations needed to converge. We plot the median number of DEER steps to convergence over 20 seeds.}
    \label{fig:two_well}
\end{figure}

DEER may still converge quickly even if the system is unpredictable in certain regions.
As long as the system is predictable on average, as indicated by a negative LLE, DEER can still converge quickly.
This phenomenon is why we framed \Cref{theorem:PL-LLE} in terms of the LLE $\lambda$ and burn-in constants $a$, as opposed to a weaker result that assumes the system Jacobians have singular values less than one over the entire state space.

To illustrate, we apply DEER to Langevin dynamics in a two-well potential (visualized in~\Cref{fig:two_well} for $D = 2$). The dynamics are stable within each well but unstable in the region between them. Despite this local instability, the system’s overall behavior is governed by time spent in the wells, resulting in a negative LLE and sublinear growth in DEER’s convergence steps with sequence length~$T$ (Figure~\ref{fig:two_well}, right subplot). 

We form the two-well potential for our experiment in Section \ref{sec:experiments} as a sum of two quadratic potentials.
Concretely, we define the potential $\phi$ as the negative log probability of the mixture of two Gaussians, where one is centered at $(0,-1.4)$ and the other is centered at $(0,1.6)$, and they both have diagonal covariance.
In Langevin dynamics \citep{Langevin1908Eng, Friedman2022Langevin} for a potential $\phi$, the state $s_t$ evolves according to
\begin{equation}\label{eq:langevin}
    s_{t+1} = s_t - \epsilon \nabla \phi(s_t) + \sqrt{2 \epsilon} w_t,
\end{equation}
where $\epsilon$ is the step size and $w_t \stackrel{\mathrm{iid}}{\sim} \mathcal{N}(0,I_D)$. In our experiments, we use $\epsilon=0.01$.
\footnote{Notice that this is a discretization (with time step $\epsilon$) of the Langevin Diffusion SDE $ds(t) = - \nabla \phi(s(t)) dt + \sqrt{2} dw(t)$, where $w(t)$ is Brownian motion \citep{Higham2001AlgorithmicSDE}.}
Accordingly, the Jacobians of the dynamics (those used in DEER) take the form
\begin{equation*}
    \A_t = I_D - \epsilon \nabla^2 \phi (s_t).
\end{equation*}
As a result, the dynamics are contracting in regions where $\phi$ has positive curvature (inside of the wells, where the dynamics are robustly oriented towards one of the two basins) and unstable in regions where $\phi$ has negative curvature (in the region between the two wells, where the stochastic inputs can strongly influence which basin the trajectory heads towards). We observe that even though there are regions in state space where the dynamics are not contracting, the resulting trajectories have negative LLE. Accordingly, in \Cref{fig:two_well} (Right), we observe that the number of DEER iterations needed for convergence scales sublinearly, as the LLE of all the intermediate DEER trajectories after initialization are negative. These results demonstrate that if the DEER optimization path remains in contractive regions on average, we can still attain fast convergence rates as the sequence length grows.

Moreover, a further added benefit of our theory is demonstrated by our choice of initialization of DEER. Both \citep{deer2024} and \citep{gonzalez2024scalable} exclusively initialized all entries of $\mathbf{s}^{(0)}$ to zero.
However, such an initialization can be extremely pathological if the region of state space containing $\mathbf{0}$ is unstable, as is the case for the particular two well potential we consider.
For this reason, we initialize $\mathbf{s}^{(0)}$ at random (as iid standard normals).

An important consequence of this experiment is that it shows that there are systems that are not globally contracting that nonetheless enjoy fast rates of convergence with DEER. This fact is important because a globally contractive neural network may not be so interesting/useful for classification, while a locally contracting network could be.

Furthermore, in this experiment we show empirically that Langevin dynamics can have negative LLE (cf. \Cref{fig:two_well}). This result suggest that the Metropolis-adjusted Langevin algorithm (MALA), a workhorse of MCMC, may also be predictable in settings of interest, including multimodal distributions.
\citet{pmcmc} provides even stronger empirical evidence that MALA may be predictable for many target distributions of interest.

\begin{figure}
    \centering
    \includegraphics[width=0.95\linewidth]{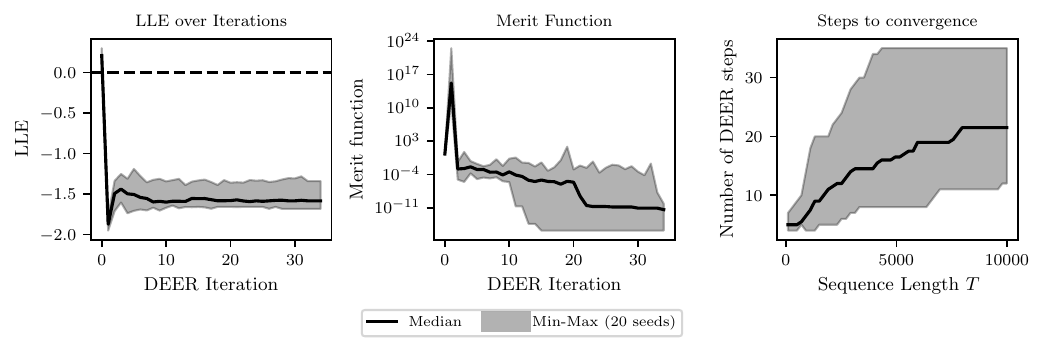}
    \caption{In this plot, we provide additional information about the behavior of DEER when rolling out Langevin dynamics on a two-well potential. \textbf{(Left)} We observe that across 20 random seeds (including different Langevin dynamics trajectories), the LLE for intermediate DEER iterations becomes negative after the first iteration. Consequently, we observe that the merit function \textbf{(Center)} experiences a spike on the very first DEER iteration (following initialization, which was the only trajectory with positive LLE), before trending towards convergence. As the system spends most of its time in contracting regions, we observe \textbf{(Right)} that the number of DEER iterations needed for convergence scales sublinearly with the sequence length $T$. We plot the min-max range for 20 seeds, and observe that even out of 20 seeds, the maximum number of DEER iterations needed to converge on a sequence length of $T=10,000$ is around $35$.}
    \label{fig:2well_app}
\end{figure}

\subsection{Application: Chaotic Observers}

\begin{table}[ht]
\centering
\footnotesize
\caption{Comparison of system and observer LLEs and number of DEER steps for $T=30,000$ and Euler discretization step size $\Delta t = 0.01$. }
\vspace{.5em}
\begin{tabular}{lcccc}
\toprule
\textbf{System} & \makecell{\textbf{LLE} \\ \textbf{(System)}} & \makecell{\textbf{LLE} \\ \textbf{(Observer)}} & \makecell{\textbf{DEER Steps} \\ \textbf{(System)}} & \makecell{\textbf{DEER Steps} \\ \textbf{(Observer)}} \\
\midrule
ABC & 0.16 & -0.08 & 4243 & 3 \\
Chua's Circuit & 0.02 & -1.37 & 697 & 14  \\
Kawczynski-Strizhak & 0.01 & -3.08 & 29396 & 2  \\
Lorenz & 1.02 & -6.28 & 30000 & 3 \\
Nosé--Hoover Thermostat & 0.02 & -0.13 & 29765 & 3 \\
Rössler & 0.01 & -0.07 & 29288 & 7  \\
SprottB & 0.20 & -0.39 & 29486 & 2 \\
Thomas & 0.01 & -3.07 & 12747 & 7  \\
Vallis El Niño & 0.58 & -2.48 & 30000 & 3  \\
\bottomrule
\end{tabular}
\label{tab:lle_deer_comparison}
\end{table}

Finally, we demonstrate a practical application of our theory in the efficient parallelization of chaotic observers. 
Observers are commonly used to reconstruct the full state of a system from partial measurements \citep{luenberger1979introduction,simon2006optimal}.
On nine chaotic flows from the \texttt{dysts} benchmark dataset \citep{dysts}, Table~\ref{tab:lle_deer_comparison} shows that while DEER converges prohibitively slowly on chaotic systems, it converges rapidly on stable observers of these systems, in accordance with our theory that predictability implies parallelizability.

We design observers for these systems using two standard approaches: (1) by directly substituting the observation into the observer dynamics, following \citet{pecora1990synchronization}, or (2) by incorporating the observation as feedback through a gain matrix, as in \citet{zemouche2006observer}. We then apply DEER to compute the trajectories of both the original chaotic systems and their corresponding stable observers. As anticipated by \Cref{theorem:PL-LLE}, the chaotic systems exhibit slow convergence—often requiring the full sequence length—whereas the stable observers converge rapidly.

As with the two-well experiment, we initialize our guess for $s_t^{(0)}$ as iid standard normals.

\section{Discussion}\label{sec:conclusion}
In this chapter, we provide the first precise characterization of the inherent difficulty of the optimization problem solved by parallel Newton methods.
The conditioning of the merit landscape determines if parallelization will
be faster in practice than sequential evaluation.
We show that the conditioning of the optimization problem is governed by the predictability of the underlying dynamics. We translate this insight into worst-case performance guarantees for specific optimizers, including Gauss–Newton (DEER). Our main takeaway is:
\emph{Predictable dynamics yield well-conditioned merit functions, enabling rapid convergence.
Unpredictable dynamics produce flat or ill-conditioned merit landscapes, resulting in slow convergence or numerical failure.}

\subsection{Related Work}

While \citet{deer2024} and \citet{deeppcr} introduced parallel Newton methods, they did not prove their global convergence. \Cref{prop:global_convergence} proves global convergence, though only with worst-case bounds of $T$ optimization steps. These prior works did not address the relationship between system dynamics and conditioning, or establish global linear convergence rates.

Global convergence rates for Gauss-Newton are rare, despite the breadth of optimization literature \citep{NocedalWright, boyd2004convex, zhao2024GN, Nesterov2018}. Theorem~\ref{theorem:deer_converges_globally} establishes global convergence with linear rate for Gauss-Newton by leveraging our specific problem structure, though similar results have existed for \emph{local} linear convergence \citep{ortega1970iterative}, most famously the Newton-Kantorovich theorem \citep{Kantorovich1948}. 

As discussed in \Cref{sec:longer_lit}, parallel-in-time methods, including multigrid methods, have a long history. Of particular relevance to this work, \citet{danieli2021multigrid} and \citet{de2025multigrid} study the CFL number for determining the usefulness of multigrid systems. More closely connecting the theory and practice of multigrid method with parallel Newton methods is a very interesting direction for future work. For example, \citet{jiang2026layer} uses multigrid methods to parallelize the evaluation and training of transformers over their layers.
More recently, several works have parallelized diffusion models via fixed-point iteration, including worst-case guarantees of $T$ steps \citep{shih2023parallel, tang2024accelerating, selvam2024selfrefining} as well as polylogarithmic rates in $T$ \citep{anari2024fast, chen2025accelerating}.
Crucially, prior work has not focused on the merit function, which we can define for any discrete-time dynamical system and optimizer.

To our knowledge, no prior work connects the LLE of a dynamical system to the conditioning of the corresponding optimization landscape, as established in \Cref{theorem:PL-LLE}. In particular, we showed that systems with high unpredictability yield poorly conditioned (i.e., flat) merit functions, linking dynamical instability to optimization difficulty in a geometrically appealing way.

The centrality of parallel sequence modeling architectures like transformers \citep{vaswani2017attention}, deep SSMs \citep{gu2022s4, smith2023s5, mamba}, and linear RNNs \citep{yang2024parallelizing} in modern machine learning underscores the need for our theoretical work. \citet{merrill2024illusion} explored the question of parallelizability through the lens of circuit complexity, analyzing when deep learning models can solve structured tasks in constant depth. Their focus complements ours, and suggests an opportunity for synthesis in future work \citep{liu2025serial}.

\subsection{Implications} 
Our work unlocks three key~implications for nonlinear state space models:
\begin{itemize}
    \item \emph{identifying} predictable systems as excellent candidates for parallelization;
    \item \emph{designing} sequence modeling architectures to be predictable if we want to parallelize them; and
    \item \emph{interpreting} predictable SSMs as an $\mathcal{O}(\log T)$ stack of LDSs, coupled nonlinearly in "depth".
\end{itemize}

\paragraph{Identifying predictable systems for parallelization} This chapter provides a principled way to determine, \textit{a priori}, whether optimization-based parallelization of a given model is practical. In many robotic or control systems, particularly ones that are strongly dissipative, this insight can enable orders-of-magnitude speed-ups on GPUs~\citep{kolter2019learning,beik2024neural,jaffe2024learning,fan2022learning,sindhwani2018learning,sun2021learning,tsukamoto2021contraction,revay2023recurrent, davydov2024perspectives}. 

For example, \citet{pmcmc} develops and leverages quasi-Newton methods to parallelize Markov Chain Monte Carlo over the sequence length, attaining order-of-magnitude speed-ups. These speed-ups occurred because the quasi-Newton methods converged quickly in the settings considered. Suggestively, MCMC chains are contractive in many settings \citep{BouRabeeEberleZimmer2020, MangoubiSmith2021, DiaconisFreedman1999IteratedRandomFunctions}. 
A precise characterization of what makes an MCMC algorithm and target distribution predictable would provide useful guidance for when one should aim to parallelize MCMC over the sequence length.
Providing precise theoretical justification for parallelizing MCMC over the sequence length is an exciting avenue for future work.

\paragraph{Designing predictable sequence mixers} Our results impact architecture design. When constructing nonlinear dynamical systems in machine learning—such as novel RNNs—parallelization benefits are maximized when the system is
made predictable. Given the large body of work on training stable RNNs \citep{hochreiter1991, lstm,sutskever2013training,miller2019stable,erichson2020lipschitz,kozachkov2022rnns,sashimi2022, krotov2023new,engelken2023gradient,orvieto2023resurrecting,zucchet2024stability, farsang2025scaling}, many effective techniques already exist for enforcing stability or predictability during training. A common approach is to \textit{parameterize} the model’s weights so
that the model is always stable.

For example, \citet{farsang2025scaling} and \citet{danieli2025pararnn} develop nonlinear SSMs and train them with DEER, with \citet{danieli2025pararnn} scaling to very strong performance as a 7B parameter language model. Both  highlight the fast convergence of DEER, which is a result of the contractivity of their architectures: \citet{farsang2025scaling} parameterizes their LrcSSM to be contractive, while \citet{danieli2025pararnn} clip the norms of their weight matrices. Ensuring a negative largest Lyapunov exponent through parameterization guarantees parallelizability for the entire training process, enabling faster and more scalable learning. Our contribution provides a theoretical foundation for why stability is essential in designing efficiently parallelizable nonlinear SSMs.

\paragraph{Interpreting SSMs as logarithmic-depth stacks of LDSs} Finally, our results have implications for the \emph{interpretation} of stable nSSMs.
Because each Gauss-Newton step in DEER is a linear dynamical system (LDS), and because we prove in \Cref{theorem:deer_converges_globally} that DEER converges in $\mathcal{O}(\log T)$ steps for a stable nSSM, we can interpret a stable nSSM as being equivalent to a "stack" of $\mathcal{O}(\log T)$ LDSs coupled by nonlinearities.
For example, if we have a nonlinear RNN as a sequence mixing layer, we can interpret this single layer with nonlinear dynamics as a hierarchical composition of linear state-space layers (SSMs), or equivalently, linear dynamical system (LDS) layers. Each layer can be evaluated in $\mathcal{O}(\log T)$ time with a parallel scan, and the total number of layers required scales as $\mathcal{O}(\log T)$. This perspective shows that nonlinear temporal dependencies can be captured through a logarithmic-depth stacking of linear dynamics. Figure \ref{fig:nssm_dssm_equivalence} provides a schematic illustration of this equivalence. 

More explicitly, each iteration of DEER is given by the LDS in \cref{eq:deer_lds}.
Therefore, we can interpret each "iteration" $(i)$ of DEER as a sequence-mixing "layer" $(i)$, where the sequence-mixing layer is an input-dependent switching linear dynamical system, like in Mamba \cite{mamba}. The input to "layer" $(i+1)$ is the state trajectory of the immediately preceding "iteration" or "layer" $(i)$.
Because we prove that DEER converges linearly in \Cref{theorem:deer_converges_globally}, it follows that a contractive nSSM can be simulated in $\mathcal{O}(\log T)$ LDS layers of the form shown in \cref{eq:deer_lds}, assuming the initial error grows polynomially in the sequence length.

\begin{figure}
    \centering
    \includegraphics[width=1.0\linewidth]{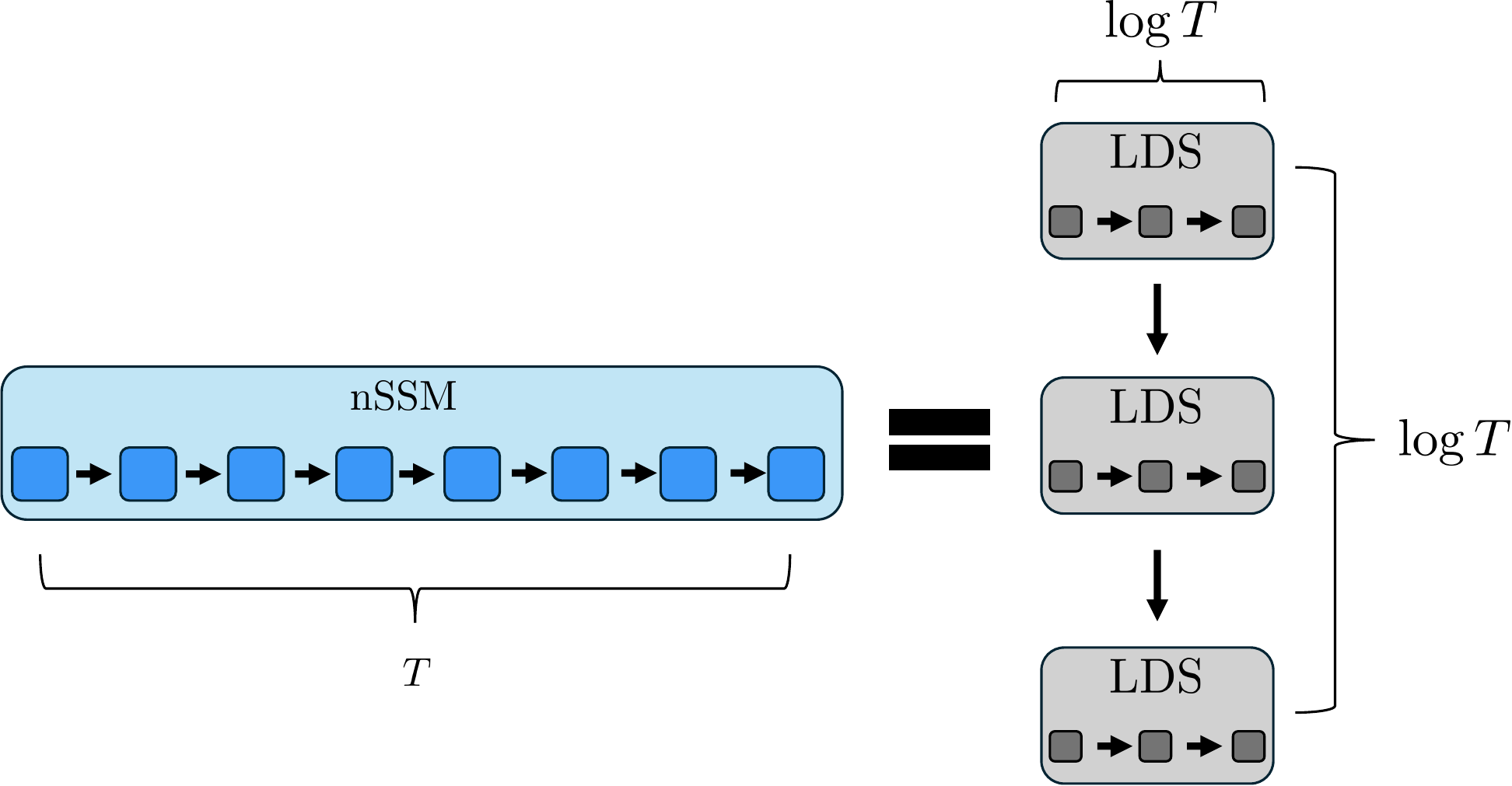}
    \caption{\textbf{Equivalence between a contractive nSSM and an $\mathcal{O}(\log T)$ stack of linear state-space models.} Contractivity implies that nonlinear dynamics can be decomposed into a hierarchy of $\mathcal{O}(\log T)$ layers of \emph{linear} SSMs, each of which can be evaluated in $\mathcal{O}(\log T)$ time by a parallel scan.}
    \label{fig:nssm_dssm_equivalence}
\end{figure}

\section{Extensions}

In this chapter, we focused primarily on the convergence rates of DEER, showing how the predictability of the dynamics affects the conditioning of $\J$.
However, as we discussed in \Cref{ssc:approx}, we can in general use any quasi- method that uses an approximate form of $\tilde{\A}_t$ to approximate the dynamics Jacobians $\A_t$.
A natural question is: how do such quasi approximations affect the convergence rates of these methods? 
Empirically, in the results presented in this thesis so far, such quasi approximations appear to slow convergence rates, but can we provide a quantitative and rigorous understanding of the quasi-convergence rates?

In the next chapter, we do just that---provide an analysis of the convergence rates of quasi-method---based on a combination of our work in \Cref{sec:cond} characterizing the conditioning of $\J$ and a convergence rate analysis presented in Proposition 4 of \citet{parasolver25}.
\chapter{Convergence Rates of quasi-Newton methods for Parallelizing SSMs}\label{ch:quasi_convergence}

In this last main chapter of this thesis, we tie up two loose ends relating to:
\begin{itemize}
    \item what do other members of the ungulate (quasi-Newton) family for parallelizing nSSMs look like and;
    \item what are their convergence rates?
\end{itemize}

In more detail, in \Cref{ssc:approx} we discussed how in principle any approximate Jacobians $\Tilde{A}_t$ could be substituted in for the dynamics Jacobians $\A_t$ in the LDS that comprises each DEER iteration (cf. \cref{eq:qdeer_lds}). Any such approximation still converges globally by \Cref{prop:global_convergence}, and forms a rich family of quasi-DEER methods. A natural question is: what updates do various Jacobian approximations $\Tilde{A}_t$ give rise to? 

We answer this question in \Cref{sec:lds_parr} by formulating a unifying framework of quasi-DEER updates, showing in particular that common fixed-point iterations like Jacobi \citep{song2021accelerating} and Picard \citep{shih2023parallel} arise from simple approximations to $\Tilde{A}_t$. While the general connections between Picard and Newton iterations and their convergence rates for solving nonlinear equations have long been known by the applied mathematics community~\citep{ortega1970iterative}, our contribution is to make these connections explicit in the setting of parallelizing nSSMs, a problem of central importance in machine learning.  This perspective clarifies the properties of each method and delineates their applicability across different problem regimes.

In \Cref{sec:conv_rates}, we further show the utility of this unifying framework by leveraging it to highlight the core properties controlling the convergence rates of these different methods. We do so by building on a nice decomposition of the convergence rates of Picard iterations proposed in Proposition 4 of \citet{parasolver25}. Our unifying framework shows that this result generalizes immediately over our ungulate family. Furthermore, we build on our work from \Cref{ch:predictability}, to show how the dynamical properties of the underlying nSSMs and the quasi-DEER approximation we use allow for further bounds and deeper analysis of convergence rates of the different fixed-point methods.

\section{Unifying fixed-point iterations as quasi-DEER methods}\label{sec:lds_parr}

In this section, we propose a unifying framework for parallelizing the evaluation of nonlinear SSMs (\cref{eq:ssm}) using linear dynamical systems (LDSs). 
In \Cref{tab:fxd_pt_sum} we show how both the parallel Newton and quasi-Newton methods we have discussed in this paper, as well as Picard and Jacobi iterations, all fit into this unifying framework. 

\begin{table}[t]
\caption{\textbf{Summary of fixed-point iteration schemes as linear dynamical systems}. We list the methods by the order of their approximation.
While higher order methods may converge in fewer iterations, each iteration may be more costly. For example, the prefix sum and parallel scan have $\mathcal{O}(\log T)$ depth, while a single Jacobi iteration has constant depth. For all the methods, each iteration is an LDS, i.e. they can be written in the form of \cref{eq:quasi_lds} where $\tilde{A}_{t+1}$ is the transition matrix.
By \Cref{prop:global_convergence}, these methods are guaranteed to converge in at most $T$ iterations. "Order" refers to the highest number of derivatives taken: Newton and quasi-Newton methods use first derivatives, while Picard and Jacobi methods do not use derivatives of $f_t$. 
}
\label{tab:fxd_pt_sum}
\centering
\small
\renewcommand{\arraystretch}{1.8}
\resizebox{\textwidth}{!}{
\begin{tabular}{@{} l c c c @ {}}
\toprule
\textbf{ Fixed-point method} & \textbf{Order} & \textbf{Transition matrix} $\boldsymbol{\tilde{A}_{t+1}}$ & \textbf{Parallelization} \\
\midrule
Newton & first-order & $\dfrac{\partial f_{t+1}}{\partial s_t} (s_t^{(i)})$ & \makecell{Parallel Scan \\ \small{(dense matrix multiplication)}} \\
Quasi-Newton & quasi first-order & $\operatorname{diag}\!\left[\dfrac{\partial f_{t+1}}{\partial s_t} (s_t^{(i)}) \right ]$ & \makecell{Parallel Scan \\ \small{(elementwise vector multiplication)}} \\
Picard & zeroth-order & $I_D$ & \makecell{Prefix Sum \\ \small{(vector addition)}} \\
Jacobi & zeroth-order & $0$ & \makecell{Map\\ \small{(embarrassingly parallel)}} \\
\bottomrule
\end{tabular}
}
\end{table}

Having discussed Newton and quasi-Newton methods at length in \Cref{sec:deer} and \Cref{ch:scalable}, we will introduce Picard and Jacobi iteration in this section.
Throughout, we will use the fixed-point operator notation $\mathcal{A}(\cdot): \mathbb{R}^{TD} \mapsto \mathbb{R}^{TD}$ introduced in \Cref{ssc:fxd_pts}.

\subsection{Picard iterations}

\citet{shih2023parallel} uses Picard iteration to parallelize sampling in diffusion models.
In fact, Picard iterations are often used in the context of evaluating differential equations, where
\begin{equation}\label{eq:ode}
    \dot{s} = g(s, t).
\end{equation}
After Euler discretization with step size $\Delta$, the continuous time \cref{eq:ode} becomes the discrete-time recursion,
\begin{equation}\label{eq:picard_setup}
    s_{t+1} = s_t + g(s_t, t) \cdot \Delta.
\end{equation}
The Picard fixed-point iteration, $\mathbf{s}_{1:T}^{(i+1)} = \mathcal{A}_{P}(\mathbf{s}_{1:T}^{(i)})$, is then given by,
\begin{equation}\label{eq:picard_shih}
    s_{t+1}^{(i+1)} = s_0 +  \sum_{\tau=0}^t g(s_\tau^{(i)}, \tau) \cdot \Delta.
\end{equation}
Because Picard iterations do not use any derivatives of the discrete-time recursion, we call them \emph{zeroth-order} fixed-point iterations.

\citet{shih2023parallel} proves by induction that for any dynamical system given by \cref{eq:picard_setup}, the fixed-point iterations given by \cref{eq:picard_shih} will converge to the true trajectory in at most $T$ iterations. The similarity of that proof and its techniques to \Cref{prop:global_convergence} begged the question as to how Picard and parallel Newton iterations related to each other. Our first result shows that Picard iterations are in fact a special case of quasi-DEER, where we approximate the Jacobian of the dynamics function by the identity matrix.

\begin{proposition}\label{prop:picard}
    The Picard iteration operator $\mathcal{A}_P$ given by \cref{eq:picard_shih} is a special case of an LDS, \cref{eq:qdeer_lds}, where the transition matrix is the identity,
    \begin{equation*}
        \tilde{A}_{t} = I_D.
    \end{equation*}
\end{proposition}

\begin{proof}
    Define $f_{t+1}(s_t) := s_t + g(s_t, t) \cdot \Delta$.
    Then, from \cref{eq:picard_shih} it follows that
    \begin{align*}
        s_{t+1}^{(i+1)} & = s_t^{(i+1)} + g(s_t^{(i)}, t) \cdot \Delta\\
        & = s_t^{(i+1)} - s_t^{(i)} + s_t^{(i)} + g(s_t^{(i)}, t) \cdot \Delta \\
        & = f_{t+1}(s_t^{(i)}) + (s_t^{(i+1)} - s_t^{(i)}).
    \end{align*}
    This is exactly of the form of the generic linear recursion shown in \cref{eq:quasi_lds}, with $\tilde{A}_{t} = I_D$.
\end{proof}

An important consequence of \Cref{prop:picard} is that like Newton iterations and quasi-Newton iterations, Picard iterations can also be cast as an LDS. In Newton iterations, the full Jacobian $\nicefrac{\partial f_{t}}{\partial s_{t-1}}$ is used in the LDS; in quasi-Newton iterations, the diagonal approximation $\mathrm{diag}[\nicefrac{\partial f_{t}}{\partial s_{t-1}}]$ is used; and in Picard iterations, the identity $I_D$ is used. The Picard iteration is more compute and memory efficient than even quasi-Newton, but is also generally a less faithful approximation and takes more iterations to converge, unless the Jacobian is well-approximated by the identity.

\subsection{Jacobi iterations}

Yet another seemingly different fixed-point method are Jacobi iterations \citep{ortega1970iterative}, which were used by \citet{song2021accelerating} to accelerate computation in a variety of settings in machine learning, such as feedforward networks with skip connections.
Jacobi iterations are also a zeroth-order fixed-point method, and are commonly used to solve systems of multivariate nonlinear equations of the form,
\begin{equation*}
    h_t(\mathbf{s}_{1:T}) = 0 \quad \forall t \in \{1, \ldots, T\}.
\end{equation*}
Instead, the Jacobi fixed-point operator,~$\mathbf{s}_{1:T}^{(i+1)}=\mathcal{A}_J(\mathbf{s}_{1:T}^{(i)})$, solves the following system of $T$ \emph{univariate} equations \emph{in parallel} to obtain $\mathbf{s}_{1:T}^{(i+1)}$,
\begin{align}
    h_t^{(i)} \big(x_1^{(i)}, \ldots, x_{t-1}^{(i)}, \, x_t, \, x_{t+1}^{(i)}, \ldots, x_T^{(i)}\big) = 0 
    \quad \forall t \in \{1, \ldots, T\}
    \label{eq:jacobi}
\end{align}

\citet{song2021accelerating} considers in particular the problem of solving recurrence relations of the form ${s_{t+1} = f_{t+1}(\mathbf{s}_{1:t})}$, and proves that, for such a system,
Jacobi iterations converge in at most $T$ iterations.
This result is directly analogous to \Cref{prop:global_convergence}. 
In fact, in the context of iteratively applying LDSs to parallelize Markovian state space models, we prove that Jacobi iterations are a type of degenerate quasi-Newton iterations, where we "approximate" the Jacobian of the dynamics function by zero.

\begin{proposition}\label{prop:jacobi}
    When applied to the Markovian state space model in \cref{eq:ssm}, the Jacobi iteration operator $\mathcal{A}_J$ specified by \cref{eq:jacobi} is a special case of the common form, \cref{eq:qdeer_lds}, where,
    \begin{equation*}
        \tilde{A}_{t+1} = 0.
    \end{equation*}
\end{proposition}

\begin{proof}
    In a Markovian state space model, the recurrence relation always takes the form specified in \cref{eq:ssm}, i.e. $s_{t+1} = f_{t+1}(s_{t})$. Thus, Jacobi iterations take the simple form
    \begin{equation*}
        s_{t+1}^{(i+1)} = f_{t+1}(s_t^{(i)}).
    \end{equation*}
    Because $s_{t+1}^{(i+1)}$ does not depend on $s_t^{(i+1)}$, we see that the transition matrix is zero.
\end{proof}

\subsection{Summary}

We have shown how important parallel fixed-point iterations---Newton, quasi-Newton, Picard, and Jacobi iterations---can all be cast as LDSs when deployed for evaluating nonlinear recursions, as summarized in \Cref{tab:fxd_pt_sum}.
The regimes where these different methods excel are therefore dictated by the form of the Jacobians of their dynamics functions: if each $f_{t+1}$ is close to an identity update (as is the case in sampling from a diffusion model with small discretization parameter), then Picard will excel; if the dynamics are nearly uncoupled across state dimensions, then quasi-Newton using a diagonal approximation will excel; and if the dynamics have multiple dependencies across coordinates and the dimension $D$ is not too large, then Newton iterations will excel. 
Jacobi iterations are most useful if the dynamics are heavily contracting or predictable, i.e. their largest Lyapunov exponent is close to zero (\Cref{sec:lle}). Another interpretation of very contracting dynamics is dynamics that are primarily input driven, i.e. $\nicefrac{\partial f_t}{\partial s_{t-1}} \approx 0$.

An important corollary is that because all of these fixed-point iterations can be cast as LDSs, they are all guaranteed to converge in all problem settings in at most $T$ iterations by \Cref{prop:global_convergence}. However, as we noted above, the precise convergence rates of the different fixed-point methods will be problem dependent.
In our next section, we provide theoretical analysis showing how the difference between the approximate Jacobian $\tilde{\A}_{t}$ of a fixed-point method and the true dynamics Jacobian $\nicefrac{\partial f_{t}}{\partial s_{t-1}}$ impacts the rate of convergence of different methods in different problems.

\section{Convergence rates for quasi-DEER}\label{sec:conv_rates}

In this section, we analyze the convergence properties of the fixed-point methods introduced in \Cref{sec:lds_parr}. We show that the convergence rate of these fixed-point methods can be understood in terms of how well the transition matrix $\tilde{\A}_{t}$ approximates the true dynamics Jacobian $\A_t \coloneq \nicefrac{\partial f_{t}}{\partial s_{t-1}}$ (cf. \Cref{tab:fxd_pt_sum}) and the stability of the LDS the fixed-point method gives rise to (cf. \cref{eq:qdeer_lds}).

To begin, we can substitute in our approximations $\tilde{\A}_t$ for $\A_t$ in the definition of $\J$ in \cref{eq:big_j} to obtain an approximate residual Jacobian $\tilde{\J}(\mathbf{s}_{1:T}) \in \mathbb{R}^{TD \times TD}$ given by
\begin{equation}\label{eq:bold_A}
    \tilde{\mathbf{J}}(\mathbf{s}_{1:T}) := \begin{pmatrix}
        I_D & 0 & 0 &  \hdots & 0 & 0\\
        -\tilde{A}_2(s_1) & I_D & 0 & \hdots & 0 & 0\\
        0 & -\tilde{A}_3(s_2) & I_D & \hdots & 0 & 0\\
        \vdots & \vdots & \vdots & \ddots & \vdots & \vdots \\
        0 & 0 & 0 & \hdots & I_D & 0 \\
        0 & 0 & 0 & \hdots & -\tilde{A}_T(s_{T-1}) & I_D \\
    \end{pmatrix},
\end{equation}
The corresponding fixed point iteration $\mathcal{A}$ takes the form
\begin{equation}\label{eq:gen_newton}
    \mathcal{A}(\mathbf{s}_{1:T}^{(i)}) := \mathbf{s}_{1:T}^{(i)} - \tilde{\J}(\mathbf{s}_{1:T}^{(i)})^{-1} \mathbf{r}(\mathbf{s}_{1:T}^{(i)}).
\end{equation}
For example, for Jacobi iterations, $\tilde{\J}_J(\mathbf{s}_{1:T})$ is always the identity matrix $I_{TD}$.
For Picard iterations, $\tilde{\J}_P(\mathbf{s}_{1:T})$ takes the form
\begin{equation}\label{eq:A_P}
    \tilde{\J}_P(\mathbf{s}_{1:T}) = \begin{pmatrix}
        I_D & 0 & 0 &  \hdots & 0 & 0\\
        -I_D & I_D & 0 & \hdots & 0 & 0\\
        0 & -I_D & I_D & \hdots & 0 & 0\\
        \vdots & \vdots & \vdots & \ddots & \vdots & \vdots \\
        0 & 0 & 0 & \hdots & I_D & 0 \\
        0 & 0 & 0 & \hdots & -I_D & I_D \\
    \end{pmatrix}.
\end{equation}
Different fixed-point methods $\mathcal{A}$ give rise to different matrices $\tilde{\J}$, which impacts their convergence rates.

\subsection{Convergence rates of fixed-point iterations}\label{ssc:conv_rates}

In this section, we closely follow the proof of Proposition 4 of \citet{parasolver25} to derive convergence rates for all fixed-point operators discussed in this paper. \citet{parasolver25} focused on the special case of Picard iterations, but our unifying framework allows us to see that their analysis generalizes immediately.

For any of the fixed-point methods discussed in this paper, we can bound the convergence rate of the error $\mathbf{e}^{(i)}$ defined in \cref{eq:error}.
\begin{proposition}[Proposition 4 of \citet{parasolver25}]\label{prop:ConvRates}
    Consider a fixed-point solver with updates given by \cref{eq:gen_newton} for some matrix $\tilde{\J}(\mathbf{s}_{1:T}^{(i)})$ with form specified by \cref{eq:bold_A}. 
    Let $L$ be the maximum of the Lipschitz constants of $\nicefrac{\partial f_t}{\partial s_{t-1}}$.
    Then $\| \mathbf{e}^{(i+1)} \|_2$ satisfies
    \begin{equation}\label{eq:prop3}
        \| \mathbf{e}^{(i+1)} \|_2 \leq \left\| \tilde{\J}(\mathbf{s}_{1:T}^{(i)})^{-1} \right\|_2 \cdot \left( \left\| \tilde{\J}(\mathbf{s}_{1:T}) - \mathbf{J}(\mathbf{s}_{1:T})  \right\|_2  \| \mathbf{e}^{(i)} \|_2 +  \frac{L}{2} (\| \mathbf{e}^{(i)} \|_2^2) \right),
    \end{equation}
    where $\| \cdot \|_2$ denotes the spectral norm of a matrix and the $\ell_2$ norm of a vector.
\end{proposition}
\begin{proof}
    Starting from \cref{eq:gen_newton}, we subtract $\mathbf{s}_{1:T}^{\star}$ from both sides to obtain
    \begin{equation*}
        \mathbf{e}^{(i+1)} = \mathbf{e}^{(i)} - \tilde{\J}(\mathbf{s}_{1:T}^{(i)})^{-1} \mathbf{r}(\mathbf{s}_{1:T}^{(i)}).
    \end{equation*}
    Next, we Taylor expand $\mathbf{r}(\cdot)$ around $\mathbf{s}_{1:T}^{(i)}$ to obtain
    \begin{equation*}
        \mathbf{r}(\mathbf{s}_{1:T}^{\star}) = \mathbf{r}(\mathbf{s}_{1:T}^{(i)}) - \mathbf{J}(\mathbf{s}_{1:T}^{(i)}) \mathbf{e}^{(i)} + \mathbf{R}(\mathbf{e}^{(i)}),
    \end{equation*}
    where $\mathbf{R}(\mathbf{e}^{(i)})$ is the second-order remainder function and has norm bounded by $\nicefrac{\| \mathbf{e}^{(i)} \|_2^2}{2}$ times the Lipschitz constant of $\mathbf{J}(\mathbf{s}_{1:T}^{(i)})$, which Theorem 3 of \citet{gonzalez2025predictability} shows is bounded by $L$.
    Since $\mathbf{r}(\mathbf{s}_{1:T}^{\star}) = \mathbf{0}$, it follows that
      \begin{equation}\label{eq:error_decomp}
        \mathbf{e}^{(i+1)} = \tilde{\J}(\mathbf{s}_{1:T}^{(i)})^{-1} \left( \underbrace{\left( \tilde{\J}(\mathbf{s}_{1:T}^{(i)}) - \mathbf{J}(\mathbf{s}_{1:T}^{(i)}) \right)}_{\text{Jacobian mismatch}} \mathbf{e}^{(i)} + \underbrace{\mathbf{R}(\mathbf{e}^{(i)})}_{\text{higher-order Taylor remainder}} \right).
    \end{equation}
    The result follows by taking norms on both sides and using the triangle inequality.
\end{proof}

\subsection{Limitations of this convergence analysis}\label{app:lims}
\Cref{prop:ConvRates} only guarantees a decrease in the error when the iterate $\mathbf{s}_{1:T}^{(i)}$ is already in a basin of decrease $\mathcal{B}_D$ given by
\begin{equation*}
    \mathcal{B}_D := \left\{ \mathbf{s}_{1:T} : \| \mathbf{e}(\mathbf{s}_{1:T}) \|_2 \leq 2 \cdot \frac{1 - \left\| \tilde{\J}(\mathbf{s}_{1:T})^{-1} \right\|_2 \left\| \tilde{\J}(\mathbf{s}_{1:T}) - \mathbf{J}(\mathbf{s}_{1:T})  \right\|_2}{L \left\| \tilde{\J}(\mathbf{s}_{1:T})^{-1} \right\|_2 }   \right\}.
\end{equation*}
However, since we know from Proposition 1 of \citet{gonzalez2024scalable} that all the fixed-point algorithms considered in this paper must eventually converge, we know that the iterates $\mathbf{s}_{1:T}^{(i)}$ must all eventually enter this basin of decrease $\mathcal{B}_D$ if $\mathcal{B}_D \neq \emptyset$. For this reason, \Cref{prop:ConvRates} provides helpful intuition about which fixed-point algorithms are useful for which dynamical systems.

For example, let us define the basin of linear rate $\mathcal{B}_L$ to comprise those $\mathbf{s}_{1:T}$ where $\left\| \tilde{\J}(\mathbf{s}_{1:T}) - \mathbf{J}(\mathbf{s}_{1:T})  \right\|_2  \| \mathbf{e}^{(i)} \|_2 > \frac{L}{2} (\| \mathbf{e}^{(i)} \|_2^2)$, i.e. the expression linear in $\| \mathbf{e}^{(i)} \|_2$ on the right side of \eqref{eq:prop3} dominates the expression quadratic in $\| \mathbf{e}^{(i)} \|_2$. It follows that $\mathcal{B}_L$ is given by
\begin{equation*}
    \mathcal{B}_L := \left\{ \mathbf{s}_{1:T} : \| \mathbf{e}(\mathbf{s}_{1:T}) \|_2 \leq \frac{2 \left\| \tilde{\J}(\mathbf{s}_{1:T}) - \mathbf{J}(\mathbf{s}_{1:T})  \right\|_2}{L}\right\}.
\end{equation*}
Therefore, when $\mathbf{s}_{1:T}^{(i)} \in \mathcal{B}_D \cap \mathcal{B}_L$, it follows that the norm of the error is guaranteed to decrease by a factor of $2 \left\| \tilde{\J}(\mathbf{s}_{1:T}^{(i)})^{-1} \right\|_2  \left\| \tilde{\J}(\mathbf{s}_{1:T}^{(i)}) - \mathbf{J}(\mathbf{s}_{1:T}^{(i)})  \right\|_2$. Moreover, as $\| \mathbf{e}^{(i)} \|_2$ approaches zero, the guaranteed factor of decrease approaches the value given by \cref{eq:lin_rate}.

Nonetheless, we can still extract very interesting intuitions about the convergence rates of different quasi-DEER approximations from \Cref{prop:ConvRates}, as we discuss in the next section.

\subsection{Intuitions about rates of convergence}\label{ssc:intuitions}

\Cref{eq:error_decomp} shows the error decomposes into two contributions.
The first term measures the discrepancy between the chosen linear operator $\tilde{\J}$ and the true Jacobian $\mathbf{J}$ of the residual.
The second term captures the effect of higher-order nonlinearities.

Moreover, from \cref{eq:prop3}, we see that as $\| \mathbf{e}^{(i)} \|_2$ approaches zero, the contribution from the first term, which is linear in $\| \mathbf{e}^{(i)} \|_2$, must eventually\footnote{under strong enough continuity assumptions.} dominate the contribution from the second term, which is quadratic in $\| \mathbf{e}^{(i)} \|_2$.
Typically, we would say the rate of decrease in $\| \mathbf{e}^{(i)} \|_2$ approaches an asymptotic linear rate $\gamma$ given by
\begin{equation}\label{eq:lin_rate}
    \gamma := \left\| \tilde{\J}(\mathbf{s}_{1:T}^{\star})^{-1} \right\|_2  \left\| \tilde{\J}(\mathbf{s}_{1:T}^{\star}) - \mathbf{J}(\mathbf{s}_{1:T}^{\star})  \right\|_2.
\end{equation}
Discussions of asymptotic linear rate are subtle in our setting, where all fixed-point methods are guaranteed to converge in $T$ iterations: see our discussion in \Cref{app:parallel_chord}.
Nonetheless, the functional form of $\gamma$ provides useful intuition about the convergence rates of different fixed-point methods.

In particular, we can study the two factors that make up the functional form of the asymptotic linear rate: $\left\| \tilde{\J}(\mathbf{s}_{1:T}^{\star}) - \mathbf{J}(\mathbf{s}_{1:T}^{\star})  \right\|_2$ and $\left\| \tilde{\J}(\mathbf{s}_{1:T}^{\star})^{-1} \right\|_2$.

\subsubsection{Intuitions from $\left\| \tilde{\J}(\mathbf{s}_{1:T}) - \mathbf{J}(\mathbf{s}_{1:T})  \right\|_2$}
We can control this quantity in terms of the spectral norms of the differences between the approximate and true dynamics Jacobians:
\begin{lemma}\label{lem:diff}
    If $\tilde{\J}(\mathbf{s}_{1:T})$ is given by \cref{eq:bold_A} and $\mathbf{J}(\mathbf{s}_{1:T})$ is given by \cref{eq:big_j}, then
    \begin{equation*}
        \left\| \tilde{\J}(\mathbf{s}_{1:T}) - \mathbf{J}(\mathbf{s}_{1:T})  \right\|_2 = \max_{2 \leq t \leq T} \left\| \tilde{A}_{t}(s_t) -\A_t(s_t) \right\|_2.
    \end{equation*}
\end{lemma}
\begin{proof}
    Plugging in the functional forms of $\tilde{\J}(\cdot)$ and $\mathbf{J}(\cdot)$, if we define $E_{t} := A_{t}(s_{t-1}) - \tilde{\A}_{t}(s_{t-1})$, then
\begin{equation*}
    \tilde{\J}(\mathbf{s}_{1:T}) - \mathbf{J}(\mathbf{s}_{1:T}) = \begin{pmatrix}
        0 & 0 & 0 &  \hdots & 0 & 0\\
        E_2 & 0 & 0 & \hdots & 0 & 0\\
        0 & E_3 & 0 & \hdots & 0 & 0\\
        \vdots & \vdots & \vdots & \ddots & \vdots & \vdots \\
        0 & 0 & 0 & \hdots & 0 & 0 \\
        0 & 0 & 0 & \hdots & E_T & 0 \\
    \end{pmatrix}.
\end{equation*}
The spectral norm of a matrix $M$ is equal to the square root of the largest eigenvalue of $M^{\top} M$. Defining $M := \tilde{\J}(\mathbf{s}_{1:T}) - \mathbf{J}(\mathbf{s}_{1:T})$, we see that
\begin{equation*}
    M^{\top} M = \begin{pmatrix}
        0 & 0 & 0 &  \hdots & 0 & 0\\
        0 & E_2^{\top} E_2 & 0 & \hdots & 0 & 0\\
        0 & 0 & E_3^{\top} E_3 & \hdots & 0 & 0\\
        \vdots & \vdots & \vdots & \ddots & \vdots & \vdots \\
        0 & 0 & 0 & \hdots & E_{T-1}^{\top} E_{T-1} & 0 \\
        0 & 0 & 0 & \hdots & 0 & E_T^{\top} E_T \\
    \end{pmatrix}.
\end{equation*}
Since $M^{\top} M$ is a block-diagonal matrix, its eigenvalues are equal to the union of the eigenvalues of each of the blocks $E_t^{\top} E_t$. Thus, it follows that the maximum eigenvalue of $M^{\top} M$ is equal to the maximum of all the eigenvalues of all the matrices $E_t^{\top} E_t$, and so the maximum singular value of $\tilde{\J}(\mathbf{s}_{1:T}) - \mathbf{J}(\mathbf{s}_{1:T})$ is given by $\max_{2 \leq t \leq T} \left\| \tilde{A}_{t}(s_{t-1}) -\A_t(s_{t-1}) \right\|_2$.
\end{proof}
A resulting intuition is that, for the fixed-point methods considered in this paper, \emph{their rate of convergence will be faster if their approximate Jacobian $\tilde{\A}_t$ is closer to the true dynamics Jacobian $\A_t$ in spectral norm.} 

For the purposes of showing the utility of this intuition in experiment, we will use the notation $\mathrm{diff}(\mathcal{A})$ to indicate this Jacobian approximation error, i.e.
\begin{equation}\label{eq:diff}
    \mathrm{diff}(\mathcal{A}) := \max_{2 \leq t \leq T} \left\| \tilde{\A}_t(s_{t-1}) - \A_t(s_{t-1})\right\|_2,
\end{equation}
where the sequence length $T$ and dynamical system $f$ should be evident from context.

\subsubsection{Intuitions from $\left\| \tilde{\J}(\mathbf{s}_{1:T})^{-1} \right\|_2$}
Because $\tilde{\J}(\mathbf{s}_{1:T})$ as defined in \cref{eq:bold_A} is a block bidiagonal matrix, it has a block lower triangular structure of the form 
\begin{equation}\label{eq:A_inv}
    \tilde{\J}(\mathbf{s}_{1:4})^{-1} = \begin{pmatrix}
    I_D & 0 & 0 & 0 \\
    \tilde{A}_2 & I_D & 0 & 0 \\
    \tilde{A}_3 \tilde{A}_2 & \tilde{A}_3 & I_D & 0 \\
    \tilde{A}_4 \tilde{A}_3 \tilde{A}_2 & \tilde{A}_4 \tilde{A}_3 & \tilde{A}_4 & I_D
    \end{pmatrix},
\end{equation}
shown above for $T=4$.

From \cref{eq:A_inv}, we see that the blocks of $\left\| \tilde{\J}(\mathbf{s}_{1:T})^{-1} \right\|_2$ are products of the transition matrices $\tilde{A}_{t}$ from the chosen fixed-point method (cf. \Cref{tab:fxd_pt_sum}). In particular, if the chosen fixed point method results in an \emph{unstable} LDS with $\| A_{t+1}\|_2 > 1$ at many points $t$ in the sequence, we see that $\left\| \tilde{\J}(\mathbf{s}_{1:T})^{-1} \right\|_2$ can be much larger than one.
In fact, as we saw in \Cref{sec:cond}, under suitable assumptions, the norm of $\tilde{\J}^{-1}$ is related to the dynamical stability of the linear-time varying system with transition matrices $\tilde{\A}_t$ arising from the fixed-point iterations.
The larger the LLE $\tilde{\lambda}$ of the LDS arising from the fixed-point iteration, the larger the norm of $\tilde{\J}^{-1}$ will be. 

More precisely, if we apply the regularity conditions in \cref{eq:LLE_regularity} to the LDS arising from the fixed-point iteration, then by the techniques used to prove \Cref{theorem:PL-LLE} it follows that
\begin{equation*}
    \max\left(1, b e^{\tilde{\lambda} (T-1) }\right) \leq \| \tilde{\J}^{-1} \|_2 \leq a \frac{e^{\tilde{\lambda} T} - 1}{e^{\tilde{\lambda}} - 1}.
\end{equation*}
Therefore, we observe that the presence of this term $\left\| \tilde{\J}(\mathbf{s}_{1:T})^{-1} \right\|_2$ in $\gamma$ gives rise to the intuition that fixed-point methods resulting in unstable LDSs should have slower rates of convergence. 
One way to grasp this intuition is that unstable LDSs suffer from numerical blowup, especially for large $T$. 

Moreover, in the special cases of Jacobi and Picard iterations, we can compute $\left\| \tilde{\J}(\mathbf{s}_{1:T}^{(i)})^{-1} \right\|_2$ analytically. For Jacobi iterations, $\left\| \tilde{\J}_J^{-1} \right\|_2 = 1$. For Picard iterations, the expression for $\left\| \tilde{\J}_P^{-1} \right\|_2$ is more complicated, but it scales as $O(T)$:
\begin{lemma}\label{lem:PicardNorm}
    Let $\tilde{\J}_P$ be as in \cref{eq:A_P}. Then
    \begin{equation*}
        \| \tilde{\J}_P^{-1} \|_2 = \frac{1}{2 \sin\left(\frac{\pi}{2 ( 2T + 1)}\right)}
    \end{equation*}
    By the small-angle approximation for sine, $\| \tilde{\J}_P^{-1} \|_2$ scales as $\mathcal{O}(T)$.
\end{lemma}

\begin{proof}
    Consider
    \begin{equation*}
        K := \tilde{\J}_P^{-\top} \tilde{\J}_P^{-1} = \begin{pmatrix}
            I_D & I_D & I_D & \hdots & I_D \\
            I_D & 2 I_D & 2 I_D & \hdots & 2 I_D  \\
            I_D & 2I_D & 3 I_D & \hdots & 3 I_D \\
            \vdots & \vdots & \vdots & \ddots & \vdots \\
            I_D & 2 I_D & 3 I_D & \hdots & T I_D 
        \end{pmatrix}.
    \end{equation*}
    We know that $\lambda_{\max}(K)^{1/2} = \| \tilde{\J}_P^{-1} \|_2$. Since $K$ is a Kronecker product $M \otimes I_D$, where $M_{i,j} = \min(i,j)$, the spectrum of $K$ is equivalent to the spectrum of $M$ (just with all eigenvalues having multiplicity $D$).
    Therefore, we seek to find the spectrum of $M \in \mathbb{R}^{T \times T}$.

    However, the spectrum of $M$ is known in the literature. 
    For example, Theorem 2.1 of \citet{daFonseca2007} shows that if $T \geq 3$, then the eigenvalues $\{ \lambda_k \}_{k=0}^{T-1}$ of $M$ are given by
    \begin{align*}
        \lambda_k & = \frac{1}{2} \left( 1 - \cos\left( \frac{2k+1}{2T + 1} \pi \right) \right)^{-1} \\
        & = \frac{1}{4} \left( \sin\left( \frac{2k+1}{2 (2T+1)} \pi \right)  \right)^{-2}
    \end{align*}
    where the second equality comes from the half-angle formula. We observe that the largest eigenvalue is therefore $\lambda_0$, and so the result follows after we take a square root.
\end{proof}

Because $\| \tilde{\J}_P^{-1} \|_2 > \| \tilde{\J}_J^{-1} \|$ for large $T$, the formula for $\gamma$ given by \cref{eq:lin_rate} yields the following expectation:
\begin{quote}
    In settings where the $\tilde{\A}_t$ from Picard vs. Jacobi iterations approximates the true dynamics Jacobian $\A_t$ equally well, we expect Jacobi iterations to converge more quickly because $\| \tilde{\J}_J^{-1} \| < \| \tilde{\J}_P^{-1} \|_2$.
\end{quote}
We now test this hypothesis with a simple simulation designed to show how \Cref{prop:ConvRates} provides helpful intuition about the convergence rates of different fixed-point methods.

\subsubsection{How fixed-point stability distinguishes between Jacobi and Picard iterations}

We demonstrate the helpfulness of the intuitions stemming from \Cref{prop:ConvRates} in a simple simulation.
We consider the LDS $s_{t+1} = \alpha s_t$, for $s_t \in \mathbb{R}^2$.
Because this is an LDS with diagonal dynamics, both the Newton and quasi-Newton iterations considered in this paper converge in one iteration. However, this simulation is useful for comparing Jacobi versus Picard iterations. This comparison is particularly fruitful in light of the formula for $\gamma$ given by \cref{eq:lin_rate} and \Cref{lem:diff} because, in this setting,
\begin{align*}
    \| \tilde{\J}_J - \mathbf{J} \|_2 & = \alpha \\
    \| \tilde{\J}_P  - \mathbf{J }\|_2 & = 1 - \alpha.
\end{align*}
However, $\| \tilde{\J}_J^{-1} \|_2 = 1$, while $\| \tilde{\J}_P^{-1} \|_2$ scales linearly with $T$.
Therefore, when comparing the number of Jacobi iterations needed to converge when the dynamics are multiplication by $\alpha$ to the number of Picard iterations needed to converge when the dynamics are multiplication by $1 - \alpha$, we expect fewer Jacobi iterations should be needed than Picard iterations, as $\gamma_J < \gamma_P$.

For $\alpha=0.5$, when $\| \widetilde{\mathbf{J}}_J - \mathbf{J} \|_2 = \| \widetilde{\mathbf{J}}_P  - \mathbf{J }\|_2$ , we see that Jacobi iterations converge in far fewer iterations than Picard iterations. Moreover, when comparing the behavior of Jacobi for simulating $f_{t+1}(x_t) = \alpha x_t$ with Picard for simulating $f_{t+1}(x_t) = (1 - \alpha) x_t$, we observe that Jacobi iterations always converge faster. However, when comparing for the same value of $\alpha$, we see that Picard can be faster than Jacobi when $\alpha$ is closer to one. This behavior makes sense, because in those settings the true Jacobian $\nicefrac{\partial f_{t+1}}{\partial x_t}$ is closer to $I_D$ than to $\mathbf{0}$.

\begin{figure}
    \centering
     \includegraphics[width=\linewidth,height=0.42\textheight,keepaspectratio]
    {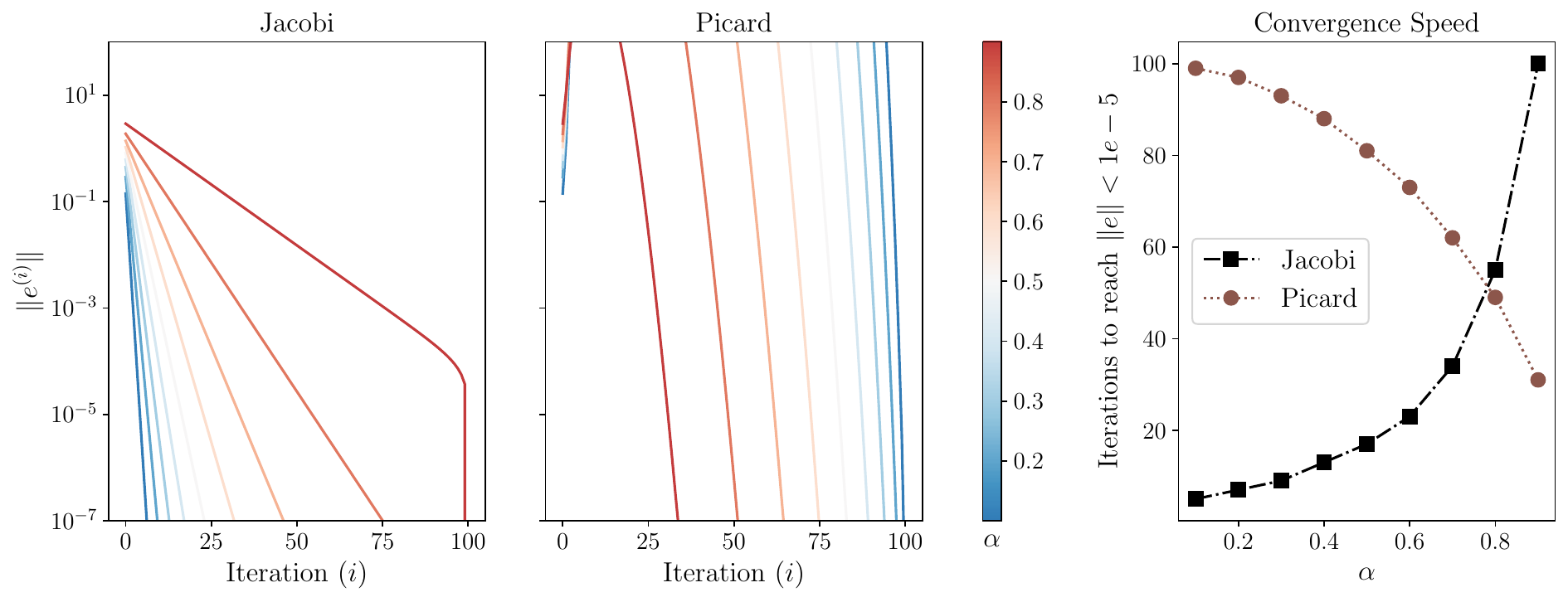}
        \caption{\textbf{Comparing Picard and Jacobi iterations on a diagonal LDS.} For the underlying dynamical system $s_{t+1} = \alpha s_t$, we plot the norm of the error $\mathbf{e}^{(i)}$ for Jacobi and Picard iterations. We denote the empirical slope of $\| \mathbf{e}^{(i)} \|_2$ for Jacobi iterations by $\gamma_J$.}
    \label{fig:JacobiPicard}
\end{figure}

Moreover, we observe that in this setting, the error $\mathbf{e}^{(i)}_{1:T}$ for Jacobi iteration shows a clear linear convergence rate, as predicted by \Cref{prop:ConvRates}. The slope of the norm of the errors of the Jacobi iterates should be $\log_{10}(\alpha)$ by \cref{eq:lin_rate} and \Cref{lem:diff}, and in fact those values are exactly the slopes of the lines in \Cref{fig:JacobiPicard} (Left panel).

\subsection{Summary of Convergence Analysis}

In \Cref{prop:ConvRates} we present an upper bound on the norm of the error of each fixed-point iterate.
As an upper bound, this result cannot always fully predict the precise trajectory of the norm of the error.
Nevertheless, we can extract pleasing intuitions from \Cref{prop:ConvRates}.
Furthermore, in our following section, we show how the resulting intuitions reflect the empirical behavior of these fixed-point methods in different settings. Most importantly, we show that the difference in spectral norm between the approximate Jacobians $\tilde{A}_{t}$ and the true dynamics Jacobians $\A_t$ provides a helpful perspective on where the fixed-point method will excel.

\section{Performance of the different fixed-point methods}\label{sec:tasks}

In this section, we consider three empirical case studies that illustrate how the unifying framework and convergence analysis presented in this paper provides guidance about which fixed-point schemes will excel in which settings.
This concordance is based on the structure of the Jacobian of $f_{t+1}$ and the relative computational cost of different fixed-point methods.

In a nutshell, we pay homage to Einstein and advise:
\begin{quote}
    \textit{Use the simplest approximate Jacobian as possible, but no simpler.}
\end{quote} 
To elaborate: simpler approximate Jacobians are less computationally expensive, meaning that each fixed-point iteration is more efficient. So, if the lower-order fixed-point method still converges in a small number of fixed-point iterations, it achieves the sequential roll-out $\mathbf{s}^{\star}$ in faster wall-clock times on GPUs than higher-order fixed point methods. However, if the higher-order fixed-point method (e.g. Newton or quasi-Newton) converges in far fewer iterations than the lower-order fixed-point method, then the increased computation of the higher-order fixed-point method is worthwhile. 
As supported by the theoretical analysis in \Cref{sec:conv_rates}, the number of iterations needed for a fixed-point method to converge is related to 
the difference in spectral norm between $\tilde{A}_t$ and $A_t := \nicefrac{\partial f_{t+1}}{\partial s_t}$.
We support this intuition with the following case studies.
All the experiments in this section were run on a single H100 with 80GB onboard memory, and the code is available at \url{https://github.com/lindermanlab/parallelizing_with_lds}

\subsection{Case study \#1: Solving the group word problem with Newton iterations}\label{sec:GrpWordProb}
Newton iterations should outperform quasi-Newton and Picard iterations in settings where the Jacobian of the recursion, $f_{t+1}$, is not well approximated by its diagonal, the identity matrix, or the zero matrix.
One example of such a recursion is the \textit{group word problem}, which has been used to theoretically and empirically assess the limits of sequential modeling architectures for state-tracking tasks~\citep{kim2023entity, liu2023transformers, merrill2024illusion, grazzi2024unlocking, schone2025implicit}. In the sequence-modeling community, the term "group word problem" is defined as follows.

\begin{definition}[Group Word Problem] Let $G$ be a finite group and let $g_1, g_2, \ldots, g_T$ be a sequence of group elements. The group word problem is to evaluate the product $g_{1} \cdot g_{2} \cdots g_{T}$. Since each $g_{t} \in G$, the product of these group elements belongs to $G$ as well.
\end{definition}

\citet{merrill2024illusion} emphasizes that nonlinear RNNs in both theory and practice are able to learn the group word problem in arbitrary groups to high accuracy with only a single layer, whereas compositions of popular linear RNNs linked by nonlinearities, such as S4 \citep{gu2022s4} and Mamba\footnote{Mamba allows input-dependent dynamics matrices but they must be diagonal, which prevents a single Mamba layer from implementing the particular LDS in~\Cref{prop:cayley}, which uses permutation matrices. \citet{merrill2024illusion} also demonstrate that a linear time-varying system with a dense transition matrix can learn the group word problem. }~\citep{mamba}, require a number of layers that grows with $T$.
\citet{merrill2024illusion} emphasizes that recurrent architectures with nonlinear transitions are well-suited for solving the group word problem, because in theory and practice, such architectures can learn the group word problem to high accuracy with a single layer. Other literature has explored the value of matrix-valued states \citep{beck2024xlstmextendedlongshortterm, grazzi2024unlocking}. However, in \Cref{prop:cayley} below, we show that neither nonlinearity nor matrix-valued states are needed to understand or solve the group word problem. Instead, the problem can be formulated as an LDS with vector-valued states and input-dependent transition matrices.

\begin{proposition}\label{prop:cayley}
    Let $G$ be a finite group. Then there exists some $D \leq |G|$ for which we can represent the group word problem as a time-varying LDS, $f_{t+1}(s_t) = A_{t+1} s_t$, with states $s_t \in \mathbb{R}^{D}$ denoting the running product of group elements and transition matrices $A_{t+1} \in \mathbb{R}^{D \times D}$ that depend on the input $g_{t+1}$.
\end{proposition}

\begin{proof}
    By Cayley's theorem, any finite group $G$ can be embedded in a symmetric group~$S_D$, for some $D \leq |G|$. Therefore, by choosing the initial state $s_0 \in \mathbb{R}^D$ to have $D$ distinct entries (a "vocabulary" of size $D$), we can use the tabular representation of permutations \citep[][eq. 1.5.2] {artin2011abstract} to represent an element of $S_D$ as $s_t$ (by a permutation of the elements of $s_0$).
    We can also choose $A_{t+1} \in \mathbb{R}^{D \times D}$ to be the permutation matrix corresponding to the embedding of $g_{t+1}$ in $S_D$, since any element of $S_D$ can be represented as a $D \times D$ permutation matrix (e.g., see~\Cref{fig:s5_results}B). Consequently $s_t = A_t A_{t-1} \hdots A_2 A_1 s_0$ is an embedding of an element of $G$ in $S_D$ in the tabular representation. In fact, $s_t \in \mathbb{R}^D$ represents the running product $g_1 g_2 \hdots g_{t-1} g_t$, which is precisely the goal of the group word problem.
\end{proof}

\begin{figure}[ht]
    \centering
\includegraphics[width=\linewidth,height=0.42\textheight,keepaspectratio, trim=1mm 2mm 1mm 3mm,clip]
    {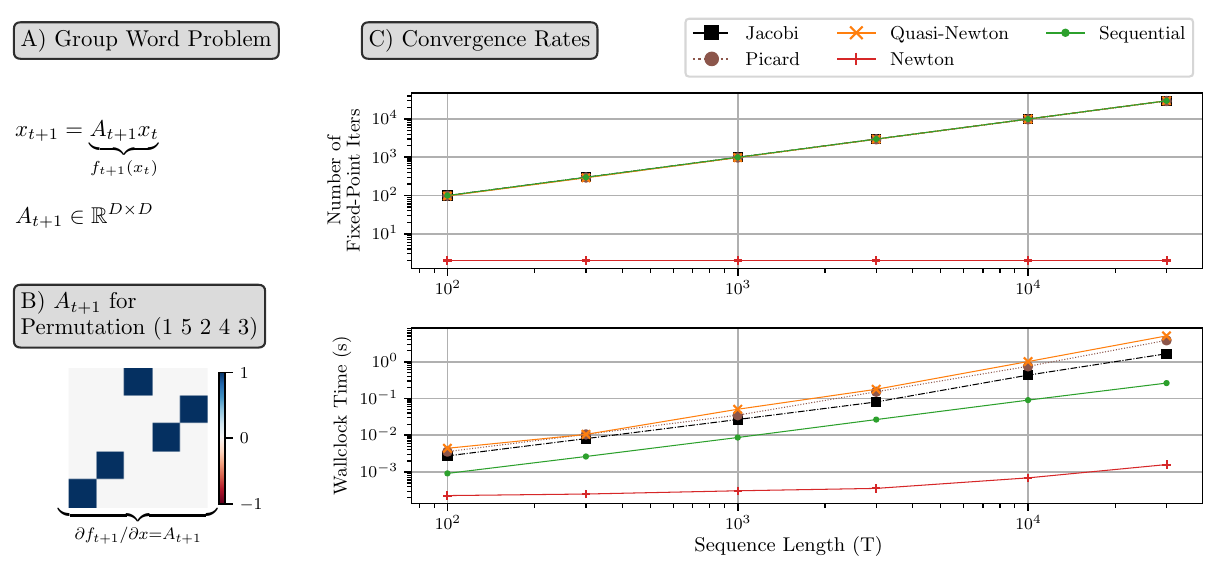}
    \caption{\textbf{A single Newton iteration solves the $S_5$ group word problem, whereas the number of iterations required for the other methods increases with sequence length.} We consider the task of evaluating the product of $S_5$ group elements. \textbf{A:} The group word problem can be expressed as an LDS with input-dependent state-transition matrices. \textbf{B:} An example input-dependent transition matrix $A_t$ for permutation $(1\ 5\ 2\ 4\ 3)$, in cycle notation. \textbf{C:} For each fixed-point method and a range of sequence lengths, $T$, we compute the median (over ten random seeds) number of fixed-point iterations to converge (top) and the median wall-clock time (bottom). While a single Newton iteration is sufficient to solve the $S_5$ problem, the number of iterations required for the other methods increases with the sequence length.}
    \label{fig:s5_results}
\end{figure}

Though we have cast the group word problem as a time-varying LDS with $f_{t+1}(s_t) = A_{t+1} s_t$, we can still evaluate this recursion with any of the fixed-point methods described above. 
Since the dynamics are linear, the Newton iteration corresponds to evaluating the LDS with a parallel scan, and it converges in one iteration.
While other methods would require more iterations to converge, they could still be more efficient in wall-clock time, since they use less memory and compute per iteration. 

However, we can use the Jacobian approximation error $\mathrm{diff}(\cdot)$ (defined in \cref{eq:diff}) of the different fixed-point methods to get a sense if the other fixed-point methods are likely to excel in this setting.

The state transition matrices of the group word problem are permutation matrices with spectral norm one, and so $\mathrm{diff}(\mathcal{A_J}) = 1$. Furthermore, since with high probability there will be a state transition matrix with diagonal all zero, it follows that $\mathrm{diff}(\mathcal{A_{QN}}) = 1$ while $\mathrm{diff}(\mathcal{A_{P}}) = 2$. Since we would expect to need $\mathrm{diff}(\mathcal{A}) < 1$ for a fixed-point method $\mathcal{A}$ to be effective, our theoretical analysis in \Cref{sec:conv_rates} suggests that none of the fixed-point methods other than Newton will be effective on the group word problem.
  
We test this hypothesis with a simple experiment simulating the $S_5$ word problem, a standard problem in the sequence modeling literature \citep{merrill2024illusion, grazzi2024unlocking}. In this setting, \Cref{fig:s5_results} shows that quasi-Newton, Picard, and Jacobi iterations require nearly $T$ iterations to converge.
On the other hand, we see that Newton's method solves the $S_5$ word problem with just one fixed-point iteration, as expected since the true dynamics are linear.
The speed-up is also apparent in the wall-clock time comparison, where we see that Newton is faster than other methods, regardless of $T$.

\subsection{Case Study \#2: Picard iterations struggle to parallelize RNNs}\label{sec:GruExp}
\begin{figure}[ht]
    \centering
    \includegraphics[width=\linewidth,height=0.42\textheight,keepaspectratio]
    {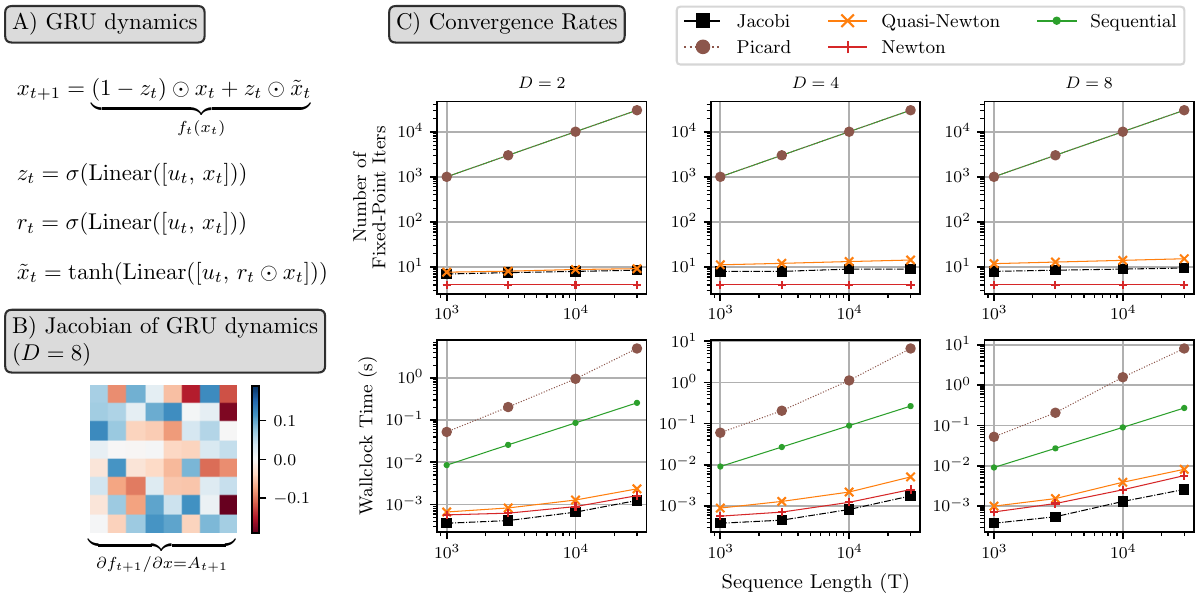}
    \caption{\textbf{Picard iterations struggle to parallelize RNNs.} We evaluate GRUs with random parameter initialization for different sequence lengths $T$ and hidden state sizes $D$. \textbf{A:} The nonlinear dynamics of a GRU, following \citet{Feng2024}, where $x_t$ is the hidden state, $u_t$ is the input, and the notation $\mathrm{Linear}[\cdot, \cdot]$ indicates a linear readout from the concatenation of two vectors. \textbf{B:} A representative Jacobian matrix $\nicefrac{\partial f_{t}}{\partial x}$ from a GRU trajectory, which is not well approximated by the identity matrix. \textbf{C:}  For each fixed-point method and a range of sequence lengths, $T$, and state sizes, $D$, we compute the median (over ten random seeds) number of fixed-point iterations to converge (top row) and the median wall-clock time (bottom row). Picard iterations take nearly $T$ iterations to converge, while the other fixed point methods yield order-of-magnitude speed-ups over sequential evaluation}
    \label{fig:gru_init}
\end{figure}

We next consider a task where Picard iterations struggle, while the other fixed-point methods excel. This task is parallelizing recurrent neural networks (RNNs), like the Gated Recurrent Unit or GRU \citep{gru}.

We show the results of a simple experiment in \Cref{fig:gru_init}. We evaluate GRUs with random parameter initialization for different hidden dimension sizes $D$ and sequence lengths $T$ using sequential evaluation as well as fixed-point iterations. This is the same experimental set up as that shown in \Cref{fig:untrained_gru}, except this time we are using H100s.
 
As we observe in Panel B of \Cref{fig:gru_init}, at initialization the Jacobian of the GRU has entries that are fairly small in value (on the order of $0.1$).
Therefore, it is intuitively plausible that $\mathrm{diff}(\mathcal{A}_J)$ and $\mathrm{diff}(\mathcal{A}_{QN})$ would both be less than one, while $\mathrm{diff}(\mathcal{A}_P)$ would be greater than one. 

To demonstrate the different values of the $\mathrm{diff}(\cdot)$ operator for quasi-Newton, Jacobi, and Picard iterations in this setting, we consider the setting $D=8$ and $T=1000$. For $10$ random seeds, we plot a variety of quantities relevant for $\gamma$ (cf. \cref{eq:lin_rate}) in \Cref{fig:gru_init}. We observe that lower values of $\gamma$ (i.e., faster rates of asymptotic linear convergence) coincide with fewer fixed-point iterations needed in \Cref{fig:gru_init}.
\begin{figure}[H]
    \centering
    \includegraphics[width=\linewidth]{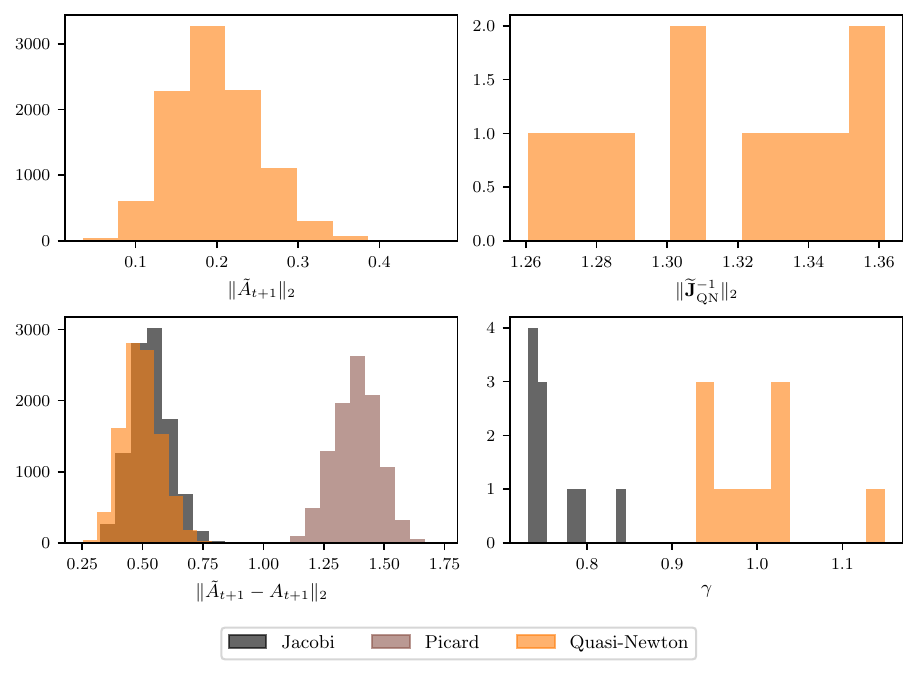}
    \caption{\textbf{Understanding the convergence rates in \Cref{fig:gru_init}.} In the setting of the GRU experiment for $D=8$ and $T=1000$, we plot relevant quantities for understanding the convergence rates of different methods over $10$ random seeds. \textbf{(Top left.)} We plot the spectral norm of the approximate Jacobian for the quasi-Newton iterations we consider in this paper, i.e. $\mathrm{diag}[\A_t(s_{t-1}^{\star})]$. \textbf{(Top right.)} For each of the $10$ random seeds, we plot $\| \tilde{\J}_{\mathrm{QN}}(\mathbf{s}_{1:T}^{\star})^{-1} \mathbf \|_2$. We observe that they are always larger than one. \textbf{(Bottom left.)} We plot the difference between approximate Jacobians and true dynamics Jacobians over all time steps and seeds for quasi-Newton, Jacobi, and Picard iterations. We observe that this difference for Picard iterations is always larger than one, and so we would intuitively expect Picard iteration to be very slow for parallelizing GRUs. This behavior is precisely what we see in \Cref{fig:gru_init}. \textbf{(Bottom right.)} Across the $10$ random seeds, we plot the value of $\gamma$ for Jacobi and quasi-Newton iterations (Picard would be $O(T)$ and so is not shown). Because $\| \tilde{\J}_{\mathrm{J}}(\mathbf{s}_{1:T}^{\star})^{-1} \mathbf \|_2 = 1$, the 10 $\gamma_J$'s are equivalent to the maximum values from the differences in (bottom left) over the 10 random seeds. However, since (top right) shows that $\| \tilde{\J}_{\mathrm{QN}}(\mathbf{s}_{1:T}^{\star})^{-1} \mathbf \|_2 > 1$, we observe that the values of $\gamma_{QN}$ are larger than in (bottom left). In summary, because the values of $\gamma_J$ are smaller than the values of $\gamma_{\mathrm{QN}}$, we would intuitively expect Jacobi to converge in fewer fixed-point iterations, which is exactly what we observe in \Cref{fig:gru_init}.   }
    \label{fig:diff}
\end{figure}
We observe that $\mathrm{diff}(\mathcal{A}_J)$ and $\mathrm{diff}(\mathcal{A}_{QN})$ are both below one always, which corresponds to their fast rates of convergence demonstrated in \Cref{fig:gru_init}. In contrast, $\mathrm{diff}(\mathcal{A}_P)$ is always greater than one, which corresponds to the slow rates of convergence of Picard iteration in the experiment depicted in \Cref{fig:gru_init}.

In conclusion, we expect quasi-Newton and Jacobi iterations to join Newton iterations in excelling in this setting, while we would expect Picard iterations to converge prohibitively slowly. This behavior is exactly what we observe in \Cref{fig:gru_init}. 

\subsection{Case Study \#3: Jacobi iterations struggle to parallelize discretized Langevin diffusion}\label{sec:WellExp}

\begin{figure}[H]
    \centering
    \includegraphics[width=0.9\linewidth]
    {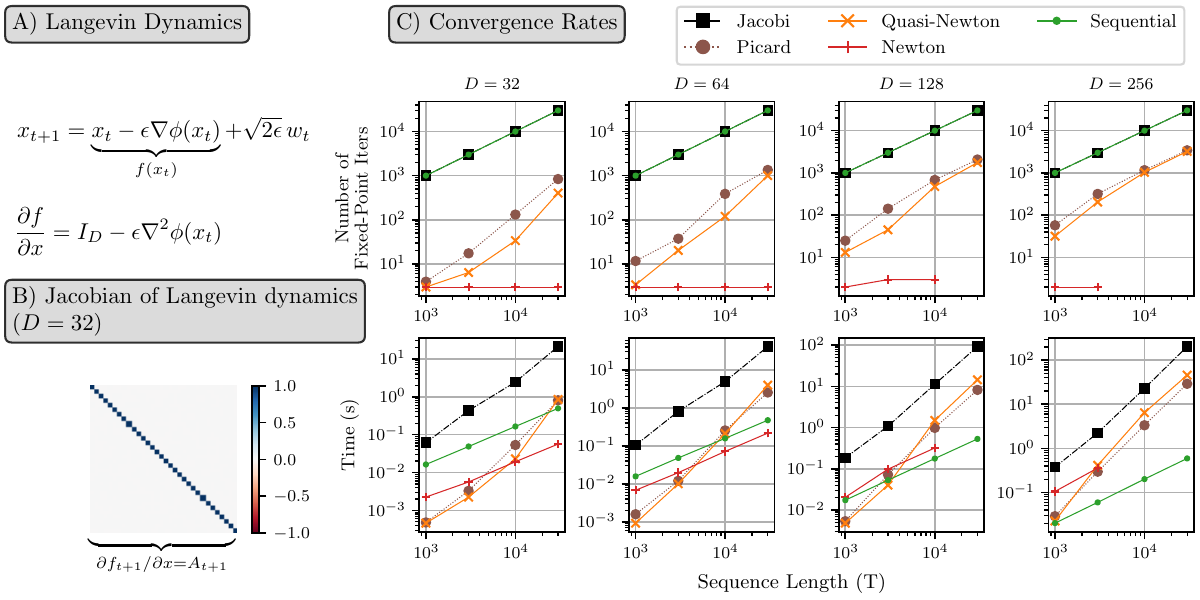}
    \caption{\textbf{Jacobi iterations struggle when the dynamics Jacobian is close to the identity.} We evaluate Langevin dynamics for a potential $\phi$.
    \textbf{A:} The nonlinear dynamics of Langevin dynamics for a potential $\phi$ and step size $\epsilon$, where $x_t$ is the state and $w_t$ is Gaussian noise.
    \textbf{B:} The Jacobian for Langevin dynamics is well-approximated by the identity matrix, especially for small step size $\epsilon=\num{1e-5}$.
    \textbf{C:} We evaluate Langevin dynamics for larger dimensions, plotting the median of 10 random seeds. Jacobi iteration consistently take $T$ steps and are always slower than sequential, while the other fixed-point methods converge in fewer $T$ steps and can be faster than sequential. The missing Newton iteration points indicate the GPU ran out of memory.
    }
    \label{fig:2well}
\end{figure}

Based on the theoretical analysis presented in \Cref{prop:ConvRates}, we expect that if the Jacobian of the dynamics function is well-approximated by the identity matrix, then Picard should converge relatively quickly and at considerably lower cost, especially when compared to the other zeroth-order method of Jacobi iterations. A canonical example of such a system where the dynamics are close to identity comes from a discretization of Langevin dynamics \citep{Langevin1908Eng, Friedman2022Langevin}. Langevin dynamics are a workhorse for MCMC \citep{besag1994comment} and motivated the development of score-matching methods \citep{song2019generative}, which are closely related to diffusion models~\citep{sohl2015deep,ho2020denoising,song2021score}.
As we discussed in \Cref{ssc:2well}, Langevin dynamics follows \cref{eq:langevin}, and consequently have a dynamics Jacobian that is well-approximated by the identity matrix for small step sizes $\epsilon$.
More generally, the identity approximation tends to be well-suited to problems where a differential equation is discretized with small step sizes, such as when sampling from diffusion models \citep{flowsanddiffusions2025}.
 
In fact, simply by observing the structure of the Jacobian in Panel B of \Cref{fig:2well}, we observe that the $\mathrm{diff}(\cdot)$ operator for Newton, quasi-Newton, and Picard iterations in this setting will be close to zero, while $\mathrm{diff}(\mathcal{A}_J)$ will be close to one. Therefore, based on our analysis in \Cref{prop:ConvRates}, we hypothesize that the other fixed-point methods should dramatically outperform Jacobi iterations in this setting.

We test this hypothesis with a simple experiment shown in \Cref{fig:2well}. We simulate Langevin dynamics on a potential $\phi$ given by the negative log probability of the mixture of two anisotropic Gaussians.
In this setting, Picard iterations take far fewer than $T$ iterations to converge and can be faster than sequential evaluation. We note that quasi-Newton iterations, which include information only about the diagonal of the Jacobian of the dynamics, appear to have comparable wall-clock time, by virtue of taking fewer iterations to converge (though each fixed-point iteration involves more work). 

Whether fixed-point iterations are faster than sequential evaluation also depends on memory utilization. For example, \citet{shih2023parallel} and \citet{parasolver25} demonstrated wall-clock speed-ups when using Picard iterations for sampling from a diffusion model using a "sliding window" to only evaluate chunks of the sequence length where the parallel scan algorithm can fit in memory. As we discuss in \Cref{sec:global_convergence}, using the sliding window is best practice for parallel Newton methods and should be used in all future work.

\section{Related Work}\label{sec:related_work}

In this chapter we unify prominent fixed-point methods for the parallel evaluation of sequences in the language of linear dynamical systems. While many papers have employed different fixed-point iterations for different problems in machine learning --- \citet{deer2024},  \citet{deeppcr}, and \citet{danieli2025pararnn} using Newton iterations, \citet{tang2024accelerating} and \citet{gonzalez2024scalable} using quasi-Newton iterations, \citet{shih2023parallel} using Picard iterations, and \citet{song2021accelerating} using Jacobi iterations, among other works --- to the best of our knowledge no one has explicitly unified these different methods in the language of linear dynamical systems. 

\paragraph{General unification of fixed-point methods: parallel-chord methods}

While connections between Newton's method and Picard iterations have been made before outside of the machine learning literature, our contribution is the tight coupling of these methods to LDSs in the context of parallel evaluation of nonlinear sequences. \citet[Ch. 7]{ortega1970iterative} considered the problem of solving a nonlinear equation $\mathbf{F}(\mathbf{s}) = \mathbf{0}$.  They showed that Newton and Picard iterations are special cases of general iterative methods where each iterate is given by
\begin{equation}\label{eq:or_pchord}
    \mathbf{s}^{(i+1)} = \mathbf{s}^{(i)} - \tilde{\J}(\mathbf{s}^{(i)})^{-1} \mathbf{F}(\mathbf{s}^{(i)}),
\end{equation}
for some matrix $\tilde{\J}(\mathbf{s}^{(i)})$. 
We discuss the relationship between the unifying frameworks put forward in \citet{ortega1970iterative} and in our paper at greater length in \Cref{app:parallel_chord}. The primary difference is that by focusing on the setting of nonlinear sequence evaluation, we bring into greater focus the role of the Jacobian of the dynamics function. Moreover, by unifying fixed-point iterations in the language of LDSs, we emphasize their parallelizability over the sequence length using the parallel scan \citep{blelloch1990prefix}.

\paragraph{Convergence rates of fixed-point methods} 
In the context of analysis of fixed-point methods in general, there is a broad literature \citep{ortega1970iterative, young2014iterative} on the convergence rates of different fixed-point methods. For example, \citet{ortega1970iterative} also proved convergence rates for iterative methods of the form in \cref{eq:or_pchord}. Though their methods have much in common with the proof techniques used to prove \Cref{prop:ConvRates} of this paper, their provided results are actually trivial in the setting considered in this paper. Part of the reason\footnote{We elaborate in \Cref{app:parallel_chord}.} for the inapplicability of the convergence results from \citet{ortega1970iterative} to our paper is that \citet{ortega1970iterative} consider the asymptotic setting, while it has been firmly established that in the particular setting considered in this paper, Jacobi, Picard, quasi-Newton, and Newton iterations all globally converge in at most $T$ iterations \citep{shih2023parallel, tang2024accelerating, gonzalez2024scalable}. For moving beyond this worst-case analysis, in \Cref{ch:predictability} we show that the difficulty of parallelizing a dynamical system is directly related to the stability of the system, which can be thought of as the "average" spectral norm of $\nicefrac{\partial f_{t+1}}{x_t}$. Proposition 4 of \citet{parasolver25} develops the foundations of the convergence analysis we present in \Cref{prop:ConvRates}. We extend their work by applying it to a wider variety of fixed-point methods, explicitly bounding many quantities of interest, and demonstrating its relevance in simulation.
  
\paragraph{Other fixed-point methods: mixing sequential with parallel}

In this chapter, we focus on Jacobi, Picard, and Newton iterations because of their prominence \citep{song2019mintnet, song2021accelerating, shih2023parallel, deeppcr, deer2024, gonzalez2024scalable, grazzi2025parallel, pmcmc, farsang2025scaling, danieli2025pararnn, iacob2025parallel} and their relationship to LDSs, as listed in \Cref{tab:fxd_pt_sum}. However, there is a wide literature on iterative solvers \citep{ortega1970iterative, young2014iterative}. Many of these other methods can also be parallelized over the sequence length, or provide a mixture of parallel and sequential computation.
For example, as we discussed in \Cref{sec:longer_lit}, \citet{naumov2017parallel} and \citet{song2021accelerating} consider using Gauss-Seidel iterations to accelerate computations in deep learning.
% For example, Naumov (2017) shows how evaluating Markov processes
% can be cast as a system of nonlinear equations and discussed many techniques from numerical analysis for
% solving them (though did not explicitly discuss Picard or Newton iterations).
% In a similar vein, \citet{song2021accelerating} considers Gauss-Seidel iterations. 
Although Gauss-Seidel iterations reduce to sequential evaluation when applied to Markovian processes, \citet{song2021accelerating} also emphasize how the structure of the problem and hardware considerations dictate the optimal mixture of parallel and sequential computation.
Parareal iterations mix parallel and sequential computation by applying parallelization at multiple length scales, and have also been used to parallelize diffusion models \citep{selvam2024selfrefining}. 
\citet{tang2024accelerating} also parallelized diffusion models using both a generalization of Jacobi iterations, as well as Anderson acceleration \citep{anderson1965iterative, walker2011anderson}, which they modify to be a form of quasi-Newton.

\section{Discussion}
\label{sec:discussion}

This work unified a variety of approaches for parallelizing recursions via fixed-point iterations---including zeroth-order methods like Jacobi and Picard iterations as well as first-order methods like Newton and quasi-Newton iterations---under a common framework.
In each case, the iterates reduce to evaluating an appropriately constructed linear dynamical system, which approximates the nonlinear recursion of interest. 
Moreover, we have demonstrated how this unifying framework provides insight into which different problems in machine learning are likely to benefit from which types of fixed-point iterations. In particular, we demonstrate that the structure of the Jacobian matrix of the dynamics function plays a key role in determining which fixed-point method to use.

For this reason, understanding the structure of the Jacobian of the dynamics function is important for using our framework. Fortunately, there are many problems where the structure of the Jacobian matrix is known in advance. As we showed in \Cref{sec:GrpWordProb}, the group word problem can always be simulated with permutation matrices for its dynamics. As we showed in \Cref{sec:WellExp}, discretized roll-outs from differential equations, used in sampling from diffusion models and rolling out neural ODEs, have $\nicefrac{\partial f}{\partial s}$ equal to the identity matrix plus a correction term equal to the discretization step-size. Moreover, as shown in \citet{pmcmc}, the dynamics of position and momenta variables in Hamiltonian Monte Carlo (HMC) results in banded matrices. Furthermore, in sequence modeling, one can \emph{design} a recurrent neural network to have Jacobians with desired structure, as we discussed in \Cref{ssc:bwds}. Finally, if there is truly no analytic information about the Jacobian in advance, its structure could be probed with finite-difference methods.

\paragraph{Future directions}

Clarifying the relationships and properties of these approaches through the lens of linear dynamical systems also suggests promising areas for future study. One clear direction of future work is to explore additional approaches for exploiting problem-specific structure, using our unifying framework to develop new fixed-point iterations. For example, an intermediate between Picard and quasi-Newton methods is a scaled identity approximation, $\tilde{A}_{t} = a_{t} I_D$. If we had prior knowledge on the appropriate scaling factors, $a_{t} \in \mathbb{R}$, we could avoid computing any Jacobian-vector product evaluations. More generally, there exist other groups of structured matrices with compact representations that are closed under composition such that a parallel evaluation of the LDS would be computationally efficient. Examples include permutation matrices, block-diagonal matrices, and block matrices where each sub-block is diagonal, among others. Future work should enumerate these use cases and investigate problem-specific applications where they are appropriate. One example application is for more efficient parallelization of the group word problem using a compact representation of permutation matrices, as was done by \citet{terzic2025permutation}. 

In conclusion, understanding the shared backbone of these fixed-point methods can also give practitioners guidance about which methods to use for which problems. As parallel evaluation of seemingly sequential processes becomes increasingly important in machine learning, these insights may provide valuable guidance to the field. 

\clearpage
\ctparttext{We conclude with a synthesis of our contributions and discuss
future research directions in the parallelization of sequential models.
\par
\begin{minipage}{\linewidth}
    \centering
    \captionsetup{hypcap=false}
    \includegraphics[width=0.8\linewidth]{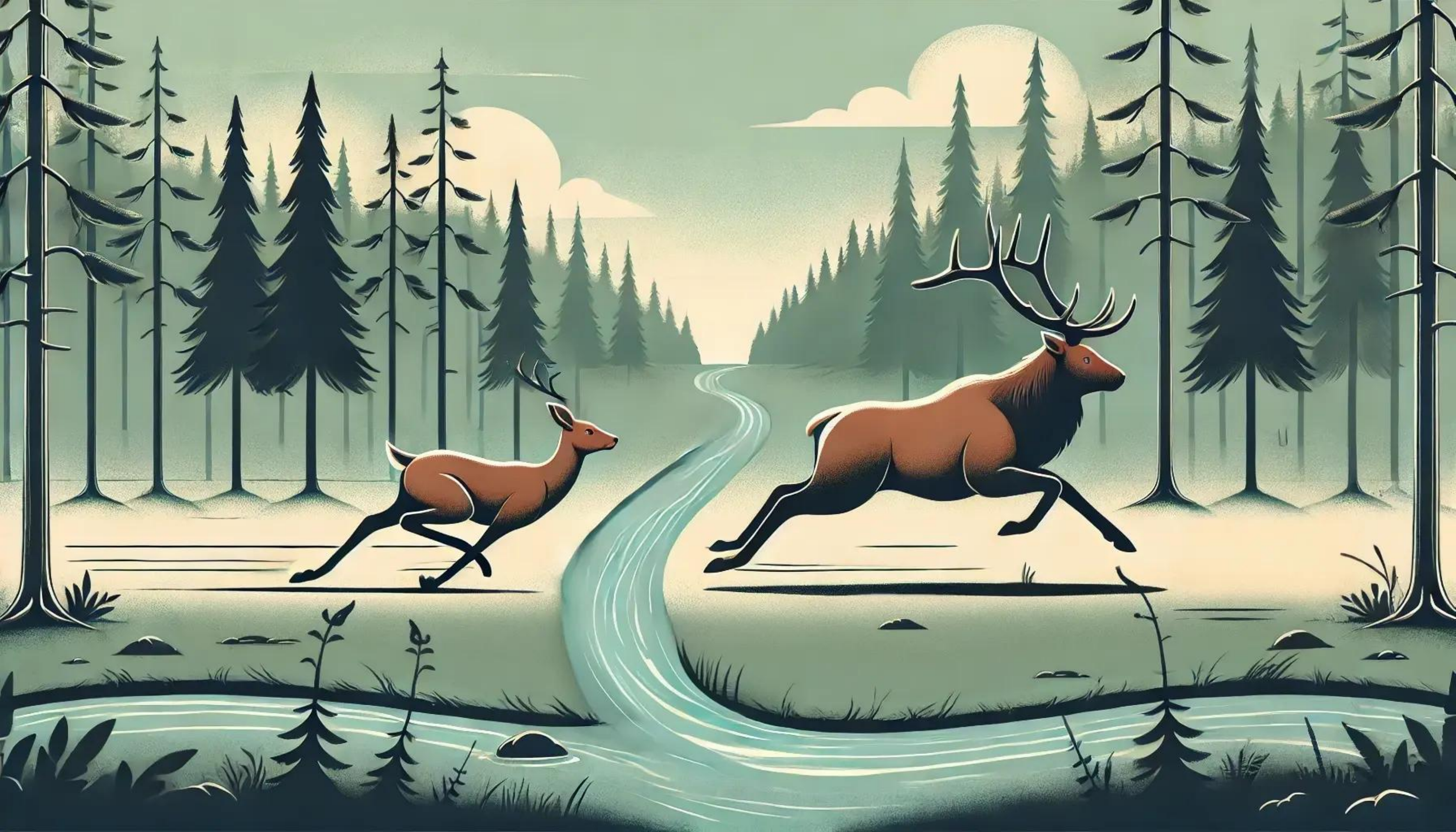}
    \captionof{figure}{What unexplored, verdant pastures await for the ungulate (parallel Newton) methods?}
    \label{fig:ungulates_art}
\end{minipage}
}
\part{Conclusion}
%************************************************
\chapter{Conclusion and Future Directions}\label{ch:conclusion}
%************************************************

This dissertation has challenged the conventional wisdom that recurrent neural networks and other state space models are "inherently sequential." Through a combination of algorithmic innovation and theoretical analysis, we have demonstrated that predictable state space models can be evaluated efficiently on parallel hardware, with computational depth scaling as $O((\log T)^2)$ rather than $O(T)$.

Parallel Newton methods are powerful tools to accelerate computation previously believed to be "inherently sequential."
This parallelization has the direct benefit of accelerating established methods like nonlinear RNNs \citep{deer2024, gonzalez2024scalable,farsang2025scaling, danieli2025pararnn}, Markov chain Monte Carlo \citep{pmcmc}, and the vast range of important applications of state space models in machine learning broadly (see \Cref{tab:ssm_examples}). 
Perhaps even more importantly, using parallel Newton methods allows researchers to explore alternative approaches using state space models more quickly, which may enable even more fundamental breakthroughs in the future.

In this conclusion, we briefly recapitulate the main contributions of this thesis, and highlight important directions for future work on parallel Newton methods.

\section{Summary of Contributions}

This thesis contributes to both the methodology and theoretical understanding of parallel Newton methods.

\Cref{part:methods} presents our methodological contributions. We extend parallel Newton methods by making connections to other canonical techniques from numerical analysis. In particular, we
\begin{itemize}
    \item improve the \emph{scalability} of parallel Newton methods by making connections to the quasi-Newton literature (\Cref{ch:scalable}); and 
    \item improve the \emph{stability} of parallel Newton methods by making connections to the trust-region literature (\Cref{ch:elk}).
\end{itemize}

\Cref{part:theory} presents our theoretical contributions. Driven by a desire to understand the limits of parallelizability, we conduct an in-depth analysis of the convergence rates of parallel Newton methods. In particular, we 
\begin{itemize}
    \item establish a novel connection between the \emph{predictability} of the SSM dynamics and the \emph{conditioning} of the merit function minimized by parallel Newton methods (\Cref{ch:predictability}). This connection allows us to derive convergence rates for DEER (the Gauss-Newton method for parallelizing nSSMs), and leads to the conclusion we can parallelize predictable dynamics, but should evaluate unpredictable dynamics sequentially. We also
    \item crystallize a \emph{unifying framework} that shows how other popular fixed-point methods, like Picard and Jacobi iterations, are also parallel Newton methods with different approaches to approximating the Jacobian (\Cref{ch:quasi_convergence}). This unifying framework allows for a general study of the convergence rates of many fixed-point methods, and highlights the settings where different methods excel.
\end{itemize}

These methodological and theoretical contributions provide a strong foundation for the deployment of parallel Newton methods. 
However, this research program is just beginning, and so we highlight exciting future directions in the next section.

\section{Future Directions}

We highlight two important directions for future work:
\begin{itemize}
    \item improving the methodology and implementation of parallel Newton methods; and
    \item finding the best application of parallel Newton methods across the wide range of state space models (\Cref{tab:ssm_examples}).
\end{itemize}

\subsection{Improving parallel Newton methods}

The growing excitement around parallel computation in machine learning has led to recent development of these parallel Newton methods across many different fields, including in the context of parallelizing nonlinear RNNs \citep{deer2024, gonzalez2024scalable,farsang2025scaling, danieli2025pararnn}, sampling from diffusion models \citep{deeppcr, shih2023parallel, tang2024accelerating, selvam2024selfrefining, parasolver25}, sampling from MCMC chains \citep{grazzi2025parallel, pmcmc}, and solving differential equations \citep{iacob2025parallel}. 
However, as all of these developments are recent and are scattered across different subfields, there is still much work to be done in optimizing and improving these methods, in terms of algorithmic innovation as well as efficient implementation.

\subsubsection{Broadening our use of numerical analysis}

A key contribution of \cref{part:methods} of this thesis was extending parallel Newton methods by drawing on the vast literature of numerical analysis, in our case quasi-Newton and trust-region methods. 
However, we have only scratched the surface of numerical analysis \citep{NocedalWright, ortega1970iterative, boyd2004convex, dennis1996numerical}.
We hope we can begin a wide-ranging research program to import useful techniques from numerical analysis to further improve parallel Newton methods. We discussed many extensions in \Cref{sec:quasi_future}.

Another example is broadening the range of targets for our parallel Newton methods. In this dissertation, we apply parallel Newton methods only to the goal of rolling out the dynamics in \cref{eq:ssm} from a fixed initial condition $s_0$. 
However, instead of considering only initial conditions, we could also consider \emph{boundary value} problems, where we may know the desired state at both the start ($t=0$) and the end ($t=T$) \citep{keller1968numerical, ascher1995numerical}.
Such a boundary problem arises, for example, in the E-step of a predictive coding network \citep{innocenti_thesis}.

This simple change adjusts certain aspects of the theory of parallel Newton methods. For example, in a boundary value problem, it is no longer required that there is a unique global minimizer, or that it results in a merit function with value $0$. Moreover, each parallel Newton step now requires not one but two parallel scans (one in the forward direction, one in the backward direction), which may enhance the appeal of smoothing-inspired approaches. 

Broadly speaking, expanding the richness of problems to which we apply parallel Newton methods will require a deeper usage of techniques from numerical analysis and possibly even further contributions to that field.

\subsubsection{Efficient implementation on parallel hardware}

As we discussed in \Cref{sec:pscan}, a fundamental ingredient of parallel Newton methods presented in this thesis is the parallel scan. 
However, there are a host of implementation details for using the parallel scan when programming on accelerated hardware like GPUs \citep{harris2007scan,gla,sarnthein2025blog}. For example, the presence of a general-purpose parallel scan is, as of the time of writing, a major difference between JAX \citep{jax2018github} and PyTorch \citep{paszke2019pytorch}, two leading Python libraries for deep learning. JAX has a general purpose parallel scan (\texttt{jax.lax.associative\_scan}) as a fundamental primitive, which allows for implementation of a wide range of parallel scans. For example, \texttt{dynamax}, a JAX library for probabilistic state space modeling \citep{linderman2025dynamax}, implements the parallel filtering and smoothing algorithms from \citet{parallel-kalman}. In contrast, PyTorch currently has only \texttt{torch.cumsum}, which is the parallel scan where the binary associative operator is addition\footnote{Although \citet{heinsen2023parallelization} shows that clever uses of \texttt{torch.cumsum} can parallelize scalar/diagonal LDSs, of the type that are used in quasi-DEER.}, and \texttt{torch.cumprod} (for scalar multiplication). This difference is why we implement the experiments in this dissertation in JAX. 

This lack of a general purpose parallel scan in PyTorch has also led to the custom development of highly-optimized, hardware-aware custom CUDA kernels for parallel scans \citep{sarnthein2025blog}. These custom parallel scans appear most prominently in Mamba \citep{mamba}, a leading SSM for language modeling, and ParaRNN \citep{danieli2025pararnn}, which applies parallel Newton iterations to 7B parameter nonlinear RNNs to achieve strong language modeling performance. There also exist useful implementations of parallel scans for scalar/diagonal LDSs in PyTorch such as \cite{proger}, which we used to implement quasi-Newton iterations in PyTorch in this repo: \url{https://github.com/lindermanlab/elk-torch}. 
Further improvements of the implementations of parallel scans will directly improve the performance of parallel Newton methods.

Moreover, in practice, when parallelizing over long sequences ($T \gg D$), the memory cost is often dominated by the size of intermediate state representations and the need to unroll computations over multiple fixed-point iterations. Chunking (dividing the sequence into smaller windows) and truncation (limiting the number of fixed-point iterations) are useful strategies to reduce memory usage in these settings \citep{dao2022flashattention, shih2023parallel, selvam2024selfrefining, geiping2025scaling, pmcmc}.

 \paragraph{Numerical Stability and Low Precision} A particularly important area for improvement of parallel Newton methods is their numerical stability and, in particular, their ability to handle lower precision. In particular, LDS matrices with spectral norm close to or greater than one can cause numerical instabilities in the parallel scan operation \citep{gonzalez2024scalable, gonzalez2025predictability}. This is especially critical in high-precision tasks or long sequences, and practitioners should monitor for numerical divergence or the accumulation of floating-point error.
 
 In practice, it has been extremely difficult to get parallel Newton methods to work reliably with lower precision than \texttt{float32}. Unfortunately, the tensor cores of modern GPUs are optimized to work best with lower precision (achieving much higher FLOPs per second in lower precision) \citep{markidis2018nvidia, micikevicius2018mixed}. Therefore, improving the robustness of parallel Newton methods (algorithmically or in their implementation) in lower precision is very important for deployment of these methods---especially as quantization and low-precision become increasingly important in industrial AI \citep{dettmers2022llm, gholami2022survey}.

 \paragraph{Fundamentally Different Approaches} Finally, we must be open to radically different and possibly transformational approaches to parallelizing over the sequence length. 
 For example, the parallel Newton methods presented in this thesis are predicated on the ease of parallelizing linear dynamical systems with a parallel scan, and the difficulty of directly parallelizing a nonlinear dynamical system with a parallel scan. However, as discussed in \Cref{ssc:nonlin_pscan}, composition of functions is inherently a binary associative operator---it is difficulties around intermediate storage that prevent us from directly using a parallel scan to parallelize nSSMs. We should be open to the existence of ingenious intermediate representations of the compositions of nonlinear functions that could remain expressive enough for a broad range of applications. There may also be useful connections to Koopman operator theory \citep{koopman1931hamiltonian, mezic2005spectral, williams2015data} that could allow us to (at least approximately) parallelize nonlinear dynamical systems in a constant number of iterations, even when the dynamics are marginal or unpredictable.

 We should even be open to eschewing the parallel scan entirely! For example, \citet{tang2024accelerating} builds up a structured matrix $\mathbf{G}$ that is an approximation to $\J^{-1}$. Thus, each application of their parallel Newton steps (called ParaTAA, where TAA stands for "Triangular Anderson Acceleration") is simply matrix multiplication by $\mathbf{G}$. Another approach that could eschew parallel scans is to use conjugate gradient methods \citep{ConjugateGradients} to evaluate the solve $\J^{-1} \mathbf{r}$. These approaches based on direct matrix multiplication could get around the $\mathcal{O}(\log T)$ depth of the parallel scan---as could approaches that truncate a work-inefficient parallel scan early.

 In short, there is a multitude of innovation yet to be discovered, both algorithmically and in the hardware-aware implementation of parallel Newton methods.

\subsection{Finding the best applications of parallel Newton methods}

Parallel Newton methods parallelize nonlinear SSMs over the sequence length. However, another (simpler) way to leverage parallel compute with SSMs is to evaluate many SSMs simultaneously. Moreover, parallel Newton methods as currently conceived simply evaluate and train SSMs---they do not change the underlying properties of the SSM. Thus, if the SSM has certain unfavorable properties itself, the parallel Newton method will not fix them.

For these reasons, it is important to consider the utility of parallel Newton methods for various SSMs (\Cref{tab:ssm_examples}). We highlight two important considerations.

\subsubsection{Latency vs. Throughput}

The benefit of parallel Newton methods is that they can decrease the \emph{latency} of evaluating a single nSSM over its sequence length. Sequential evaluation requires $\mathcal{O}(T)$ iterations to get from the start $s_0$ to the finish $s_T$.
In contrast, if the SSM is predictable, then parallel Newton methods take $\mathcal{O}((\log T)^2)$ iterations to evaluate the chain when there are $T$ processors, thus reducing the latency.

However, if we had simply launched $T$ sequential chains simultaneously on such a parallel machine, then each clock would generate $T$ new samples $s_t^{(b)}$, where $b$ indicates the batch id of these $T$ chains. 
In contrast, a parallel Newton method run on a single chain would use all of the $T$ processors, but would produce $T$ samples in $\mathcal{O}((\log T)^2)$ (one factor of $\log(T)$ for the parallel scan, and another factor of $\log(T)$ for the number of iterations needed for convergence). Therefore, batching sequential computation actually has better \emph{throughput} than parallel Newton methods by a factor of $\mathcal{O}((\log T)^2)$.

For this reason, \textbf{parallel Newton methods excel in settings where we care about latency and not throughput.}
Examples where latency is important include the training of nonlinear RNNs (where a forward pass must be completed before learning during the backwards pass can begin) and sampling from an MCMC chain (where there is an initial burn-in period at the beginning of the chain when the samples have not yet converged to the target distribution).
However, if throughput is more important for your application, then you will be better suited by employing sequential evaluation with large batch size.

\subsubsection{Expressive nonlinear RNNs}

Even in the setting of training nonlinear RNNs---where decreasing the latency of the forward pass is of vital importance---parallel Newton methods are only as good as the target they are evaluating.
In other words, DEER exactly evaluates the forward and backward passes of a nonlinear RNN---so if the underlying nonlinear RNN (say a GRU or LSTM) has undesirable properties, parallel Newton methods cannot fix those because they will replicate those undesirable properties as well. 

Stemming from their ability to simulate a Turing machine \citep{siegelmann1995computational}, RNNs have many desirable theoretical properties \emph{vis a vis} transformers, including an improved ability to track state \citep{merrill-sabharwal-2023-parallelism, merrill2024illusion, liu2025serial, siems2026learning} and the ability to express harder complexity classes \citep{merrill2026linear}. However, currently recurrent architectures struggle vs. transformers both during training and evaluation. During training, recurrent architectures continue to struggle with the problems of vanishing and exploding gradients \citep{hochreiter1991, lstm} and the curse of memory \citep{bengio1994learning, zucchet2024stability}. Because recurrent architectures arise from the repeated application of the same cell, small changes in the parameter can result in large changes in performance, resulting in a jagged loss landscape \citep{sutskever2013training, pascanu2013difficulty}. While research in improving gradient-based optimization of RNNs like BPTT \citep{werbos1990backpropagation} remains ongoing \citep{sofo}, the future of RNN training might even eschew backpropagation altogether \citep{chaubard2025scaling, eggroll}, perhaps one day unlocking more biologically plausible learning rules at scale \citep{williams1989learning, bellec2020solution}.

Moreover, recurrent architectures also struggle with memory-retrieval tasks during in-context learning vs. transformers \citep{arora2023zoology}. The hidden state of a transformer-its KV cache---scales linearly with the sequence length, while the hidden state of an RNN is constant size \citep{guo2025log}. Thus, the RNN enforces compression, at the cost of reducing its recall ability over long context. 

The predictability of an RNN---which we define and discuss in \Cref{sec:lle}---is an important concept for both the training and deployment of RNNs. Predictable RNNs will also enjoy stable backwards passes, thus mitigating any issues from exploding gradients. However, an overly contracting RNN will struggle with recall. Along these lines, works such as \citet{orvieto-resurrecting} suggest that the best performing RNNs will have LLE as close as possible to $0$ without becoming chaotic. Nonetheless, as demonstrated in \Cref{ch:scalable} and explained in \Cref{ch:predictability}, parallel Newton methods can struggle as we approach "the edge of stability" \citep{beggs2022cortex}. Therefore, while our theoretical work on predictability can help guide the design of nonlinear RNN architectures, fundamental work remains in the design, training and orchestration of RNNs towards the goal of achieving human-like intelligence.

% ********************************************************************
% Backmatter
%*******************************************************
\begin{appendices}
\crefalias{chapter}{appendix}   
\clearpage
\part{Appendix}
%********************************************************************
% Appendix
%*******************************************************

\chapter{Global Convergence of Parallel Newton Methods}\label{appendix:global_convergence}

This appendix contains an extended discussion of the relationship between \Cref{prop:global_convergence} of this thesis (Proposition 1 of \citet{gonzalez2024scalable}) and Theorem 3.6 of \citet{tang2024accelerating}.
At its core, \citet{tang2024accelerating} contains the fundamental ideas for global convergence with quasi-Newton methods, and their empirical results show that they had a correct understanding of this global convergence.  However, to the best of my abilities to understand the notation of \citet{tang2024accelerating}, their Theorem 3.6 is both incorrect as stated and weaker than necessary.
To this end, in this section we discuss the different thrusts of Proposition 1 of \citet{gonzalez2024scalable} and Theorem 3.6 of \citet{tang2024accelerating}, and present a cleaned statement and proof of Theorem 3.6 of \citet{tang2024accelerating}.

\section{Comparison of the two results}

In terms of superficial differences, one aspect to keep in mind is that \citet{gonzalez2024scalable} focused on RNNs while \citet{tang2024accelerating} focused on diffusion models. Both are nSSMs and the goal of parallelization is identical. However, the direction of time is different in both papers: for RNNs, time goes from $0$ to $T$, while in sampling from diffusion models, the convention is often for time to go backwards. To keep uniformity in notation throughout this thesis, we will standardize on time going forwards for parallelization over the sequence length.

Both \citet{gonzalez2024scalable} and \citet{tang2024accelerating} also use quasi-Newton approaches towards this goal of parallelizing over the sequence length.
However, \citet{gonzalez2024scalable} approximates $\J$ with $\tilde{\J}$ and then uses a parallel scan to invert $\tilde{J}$. In contrast, \citet{tang2024accelerating} uses a form of Broyden's bad update, i.e. they approximate $\J^{-1}$ directly with a matrix $\mathbf{G}$. \citet{tang2024accelerating} have the very good insight that as long as $\mathbf{G}$ satisfies certain conditions, then global convergence of their quasi-Newton method (which they call ParaTAA) is guaranteed.

Lightly massaging the notation of \citet{tang2024accelerating} to be in the format of this thesis, their statement of their Theorem 3.6 is
\begin{quote}
    Consider a general update rule: In the $(i)$th iteration, the update is
    \begin{equation}\label{eq:tang_update}
        \mathbf{s}_{1:T}^{(i+1)} = \mathbf{s}^{(i)}_{1:T} - \mathbf{G}^{(i)} \mathbf{r}(\mathbf{s}_{1:T}^{(i)}),
    \end{equation}
    with $\mathbf{G}^{(i)}$ being any {\color{red} arbitrary} matrix. If for any $j$ where $\mathbf{r}_k^{(i)} = 0$ for $k < j$, the matrix $\mathbf{G}^{(i)}$ satisfies { \color{blue} $\mathbf{G}^{(i)}[:jD, :jD] = \mathbf{I}_{jD}$ \color{blue}}, then the update rule will converge within $T$ steps.
\end{quote}
We put in red the part of this statement that is overly strong (i.e. rendering the statement incorrect); and in blue the part that is weaker than necessary. 
Again, note the effect of time reversal, and the fact that \citet{tang2024accelerating} defines their residual function to be the negative of our definition in \cref{eq:r_def}.
In \Cref{fig:tang_diagram}, we illustrate how the conditions placed on $\mathbf{G}$ in \citet{tang2024accelerating} interact with the update rule in \cref{eq:tang_update}. 
\begin{figure}[H]
    \centering
    \includegraphics[width=0.95\linewidth]{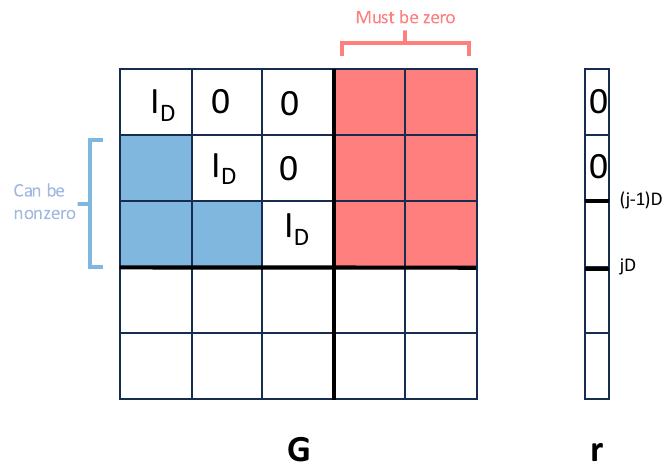}
    \caption{\textbf{Illustration of Theorem 3.6 of \citet{tang2024accelerating}}. In this illustration, $j=3$. The portion shaded red must be zero for the proof by induction to work, showing that $\mathbf{G}$ cannot be an { \color{red} arbitrary} matrix. The portion shaded blue can be nonzero, and the proof by induction will still hold.}
    \label{fig:tang_diagram}
\end{figure}
In particular, in \Cref{fig:tang_diagram}, we see that the blue-shaded blocks can be zero because they are always multiplied against residual entries that are zero. If the red shaded blocks are not zero, however, there is in general no guarantee that they will be multiplied against non-zero entries, and so can in general undo the causal filling in effect of this family of induction proofs. Since throughout their paper \citet{tang2024accelerating} consider $\mathbf{G}$ that are lower-triangular, this point is very minor.
Nonetheless, for clarity we reiterate that $\mathbf{G}$ cannot be an { \color{red} arbitrary} matrix for the proof by induction to hold (though lower triangular would certainly suffice).

Let us show a simple and concrete counterexample showing that if $\mathbf{G}$ is an { \color{red} arbitrary} matrix, then global convergence of \cref{eq:tang_update} is \emph{not} guaranteed. Let $f(s) = 2s$, and consider $s_0 = 2$. Then, if $T=2$, it follows that $s_1^{\star} = 4$ and  $s_2^{\star} = 8$. Consider the initialization $s_1^{(0)} = s_2^{(0)} = 2$, so that $r_1^{(0)} = -2$ and $r_2^{(0)} = -2$. Let the matrix $\mathbf{G}^{(i)}$ be determined by the following rule:
\begin{itemize}
    \item If $r_1^{(i)} = r_2^{(i)} = 0$, then $\mathbf{G}$ is the identity matrix
    \item otherwise, \begin{equation*}
        \mathbf{G} = \begin{pmatrix}
            1 & 1 \\
            0 & 1
        \end{pmatrix}.
    \end{equation*}
\end{itemize}
This rule satisfies the conditions of Theorem 3.6 of \citet{tang2024accelerating}.
Applying \cref{eq:tang_update}, the first update gives $s_1^{(1)} = 6, s_2^{(1)} = 4$ and the second update gives $s_1^{(2)} = 12, s_2^{(2)} = 12$. Therefore, we see that we do not get convergence in $T=2$ iterations.

For an example showing that $\mathbf{G}^{(i)}[:jD, :jD]$ { \color{blue} need not be the identity matrix for global convergence to hold}, simply consider $\mathbf{G}^{(i)} = \mathbf{J}^{-1}$, as shown in \cref{eq:J_inv}. While $\mathbf{J}^{-1}$ is block lower triangular, it is not exclusively the identity matrix in its blocks of the form $\mathbf{J}^{-1}[:jD, :jD]$.

\section{Corrected version of Theorem 3.6 of Tang et al.}

For the purpose of maintaining the literature, we now present a corrected version of Theorem 3.6 of \citet{tang2024accelerating}:
\begin{theorem}\label{thm:tang_correct}
    Consider a general quasi-Newton update of the form in \cref{eq:tang_update}.
    Assume $\mathbf{G}^{(i)}$ is {\color{red} lower triangular} and satisfies 
    \begin{equation*}
        \color{blue} \mathbf{G}^{(i)}[iD:(i+1)D-1, iD:(i+1)D-1] = I_D
    \end{equation*}
    Then the update rule will converge to $\mathbf{s}^\star$ within $T$ steps.
\end{theorem}
\begin{proof}
        By induction. 
\begin{itemize}
    \item Induction hypothesis: Assume that at iteration $(i)$, we have that $s_t^{(i)} = s_t^\star$ for all $t \leq i$. 
    \item Base case: at iteration $0$, $s_0 = s_0^\star$ by construction.
    \item Induction step: We also have $r_t^{(i)} = 0$ for all $t \leq i$.  So $\mathbf{G}^{(i)} \mathbf{r}^{(i)}_{1:T}$ has the first $i$ blocks equal to zero (i.e. $iD$ entries). Moreover, because the $(i+1)$st block of $\mathbf{G}^{(i)}$ is identity, it follows that the $(i+1)$st block entry of $\mathbf{G}^{(i)} \mathbf{r}^{(i)}_{1:T}$ is equal to $s_{i+1}^{(i)} - f_{i+1}(s_{i}^\star)$, so that $s_{i+1}^{(i+1)} = s_{i+1}^{*}$. Thus, we have shown that assuming the induction hypothesis at iteration $(i)$ leads to the induction hypothesis holding at iteration $(i+1).$
\end{itemize}
\end{proof}
Note that all quasi-DEER updates (i.e. of the form shown in \Cref{tab:fxd_pt_sum}) satisfy the assumption of  \Cref{thm:tang_correct}, as $\tilde{\J}$ is lower triangular and has all identities on its block diagonal. Thus, \Cref{thm:tang_correct} is a generalization of \Cref{prop:global_convergence}. \Cref{prop:global_convergence} discusses only approximations to the dynamics Jacobians $\{ \A_t \}$, while \Cref{thm:tang_correct} allows for approximations to the inverse of the Jacobian of the residual function, i.e. $\J^{-1}$.

\chapter{Predictability and Conditioning}\label{appendix:pl_from_lyapunov_appendix}
This appendix provides the proof of \Cref{theorem:PL-LLE} and an extended discussion of its assumptions and implications.

\section{Theorem statement and proof}

\begin{theorem*}[\Cref{theorem:PL-LLE}]
Assume that the LLE regularity condition from \cref{eq:LLE_regularity} holds. Then if $\lambda \neq 0$ the PL constant $\mu$ of the merit function in \eqref{eq:PL} satisfies
\begin{equation}\label{eq:recap_mu_LLE}
     \dfrac{1}{a} \cdot \dfrac{e^{\lambda} - 1}{e^{\lambda T} - 1} \ \leq \ \sqrt{\mu} \ \leq \ \min\left( \frac{1}{b} \cdot \frac{1}{e^{\lambda (T-1)}}, 1 \right).
\end{equation}
If $\lambda=0$, then the bounds are instead
\begin{equation*}
     \frac{1}{a T}\ \leq \ \sqrt{\mu} \ \leq \ \min\left( \frac{1}{b} \sqrt{\frac{2 D}{T +1}}, 1\right).
\end{equation*}
\end{theorem*}
\begin{proof}

Notice that the residual function Jacobian $\b J$  can be written as the difference of the identity and a $T$-nilpotent matrix $\b N$, as 
\[\b J = \b I_{TD} - \b N \quad \text{with} \quad \b N^T = \b 0_{TD} \]
Because $\b N$ is nilpotent, 
the Neumann series for $\b J^{-1}$ is a finite sum:
\begin{equation}
 \b J^{-1} = (\b I_{TD} - \b N)^{-1} 
\;=\;
\sum_{k=0}^{T-1} \,\b N^k.
\end{equation}
Straightforward linear algebra also shows that the norms of the powers of this nilpotent matrix are bounded, which enables one to upper bound the inverse of the Jacobian
\begin{equation}\label{eq:final_upper_bound}
    \|\b N^k \|_2 \leq a \, e^{\lambda k} \quad \text{and therefore} \quad \|\b J^{-1} \|_2 \ \leq \ \sum_{k=0}^{T-1} \,\b \| \b N^k \|_2 \ \leq \ \sum_{k=0}^{T-1} a \, e^{\lambda k} = a \dfrac{1-e^{\lambda T}}{1-e^{\lambda}}.
\end{equation}
The powers of $\b N$ are closely related to the dynamics of the nonlinear state space model. We provide a dynamical interpretation in \Cref{sec:ill_conditioned}.

To lower bound $\| \mathbf{J}^{-1} \|_2$, we observe that by the SVD, a property of the spectral norm is that
\begin{equation}\label{eq:spectral_sup}
    \|\mathbf{J}^{-1}\|_2 
    = \sup_{\substack{\|x\|_2 = 1 \\ \|y\|_2 = 1}} 
      \, x^{\top}\mathbf{J}^{-1}y.
\end{equation}
We pick two unit vectors $u$ and $v$, both in $\mathbb{R}^{TD}$, that are zero everywhere other than where they need to be to pull out the bottom-left block of $\mathbf{J}^{-1}$ (i.e., the only non-zero block in $\mathbf{N}^{T-1}$, which is equal to $\A_t \A_{T-1} \hdots \A_2$). Doing so, we get
\begin{equation*}
    u^T \mathbf{J}^{-1} v = \Tilde{u}^T (\A_t \A_{T-1} \hdots \A_2) \Tilde{v},
\end{equation*}
where $\Tilde{u}$ and $\Tilde{v}$ are unit vectors in $\mathbb{R}^D$, and are equal to the nonzero entries of $u$ and $v$.

Note, therefore, that because of \cref{eq:spectral_sup}, it follows that
\begin{equation}\label{eq:first_lower_bound}
    \Tilde{u}^T \left(\A_t \A_{T-1} \hdots \A_2\right) \Tilde{v} \ \leq \ \| \mathbf{J}^{-1} \|_2,
\end{equation}
i.e. we also have a \textbf{lower bound} on $\| \mathbf{J}^{-1} \|_2$.

Furthermore, choosing $\Tilde{u}$ and $\Tilde{v}$ to make
\begin{equation*}
    \Tilde{u}^T \left(\A_t \A_{T-1} \hdots \A_2\right) \Tilde{v} = \| \A_t \A_{T-1} \hdots \A_2 \|_2,
\end{equation*}
we can plug in this choice of $\Tilde{u}$ and $\Tilde{v}$ into \cref{eq:first_lower_bound}, to obtain
\begin{equation*}
    \| \A_t \A_{T-1} \hdots \A_2 \|_2 \leq \| \mathbf{J}^{-1} \|_2.
\end{equation*}
Applying the regularity conditions \eqref{eq:LLE_regularity} for $k = T-1$ and $t=2$ we obtain
\begin{equation}\label{eq:final_lower_bound}
    b \ e^{\lambda (T-1) } \leq \| \mathbf{J}^{-1} \|_2.
\end{equation}

Because
\[ \lambda_{\min}\left(\b J \b J^\top\right) \ = \ \frac{1}{\|\b J^{-1} \|_2^2},  \]
the result for $\lambda \neq 0$ follows by applying \cref{eq:final_upper_bound} and \cref{eq:final_lower_bound} at all $\mathbf{s}^{(i)}$ along the optimization trajectory.

Note that any choice of $\Tilde{u}$ and $\Tilde{v}$ results in a lower bound, i.e. we could also have targeted the block identity matrices. So, it also follows that $1 \leq \| \mathbf{J}^{-1} \|_2$, and so
\begin{equation*}
    \max\left( b \ e^{\lambda (T-1) }  , 1\right) \leq \| \mathbf{J}^{-1} \|_2.
\end{equation*}

Finally, let us conclude by considering the case $\lambda = 0$. In this setting, the lower bound on $\sqrt{\mu}$ follows from L'H\^opital's rule. For the upper bound, we again must lower bound $\| \mathbf{J}^{-1} \|_2$. 
To do so, we leverage the relationship between spectral and Frobenius norms, namely that for an $n \times n$ matrix $M$,
\begin{equation}\label{eq:spec_frob}
    \frac{\| M \|_F}{\sqrt{n}} \leq \| M \|_2 \leq \| M \|_F.
\end{equation}

We can find the squared Frobenius norm, i.e. $\| \mathbf{J}^{-1} \|_F^2$, which is the sum of the squares of all of the entries. The squared Frobenius norm factors over the block structure of the matrix, i.e. $\| \mathbf{J}^{-1} \|_F^2$ is the sum of the squared Frobenius norms of the blocks. We know that each block has spectral norm lower bounded by $b$, so each block also has Frobenius norm lower bounded by $b$. Therefore, summing up over all of the blocks, it follows that
\begin{equation*}
    b^2 \frac{T (T+1)}{2} \leq \| \mathbf{J}^{-1} \|_F^2
\end{equation*}
and
\begin{equation*}
    \| \mathbf{J}^{-1} \|_F \leq \sqrt{TD} \| \mathbf{J}^{-1} \|_2.
\end{equation*}
Putting these equations together, it follows that
\begin{equation*}
    b \sqrt{\frac{T (T+1)}{2}} \leq \sqrt{T D} \| \mathbf{J}^{-1} \|_2
\end{equation*}
or
\begin{equation*}
    b \sqrt{\frac{T+1}{2 D}} \leq \| \mathbf{J}^{-1} \|_2,
\end{equation*}
and so the upper bound on $\sqrt{\mu}$ when $\lambda=0$ follows from taking reciprocals.
\end{proof}

The above proof sheds light on how many dynamical system properties fall out of the structure of $\mathbf{J}(\mathbf{s})$, which we now discuss further. 

\section{Discussion of why small singular values leads to ill-conditioning}\label{sec:ill_conditioned}

Recall that our goal is to find a lower bound on the smallest singular value of $\mathbf{J}(\mathbf{s})$, which we denote by $\sigma_{\mathrm{min}}(\mathbf{J}(\mathbf{s}))$. This quantity controls the difficulty of optimizing $\mathcal{L}$. For example, the Gauss-Newton update is given by $\mathbf{J}(\mathbf{s})^{-1} \mathbf{r}(\mathbf{s})$. Recall that
\begin{align*}
    \sigma_{\mathrm{max}}\left(\mathbf{J}(\mathbf{s})^{-1}\right) & = \nicefrac{1}{\sigma_{\mathrm{min}}\left(\mathbf{J}(\mathbf{s})\right)} \\
    & = \| \mathbf{J}(\mathbf{s})^{-1} \|_2.
\end{align*}
Recall that an interpretation of the spectral norm $\| \mathbf{J}(\mathbf{s}) \|_2$ is how much multiplication by $\mathbf{J}(\mathbf{s})$ can increase the length of a vector. 
Therefore, we see that very small values of $\sigma_{\mathrm{min}}(\mathbf{J}(\mathbf{s}))$ result in large values of  $\| \mathbf{J}(\mathbf{s})^{-1} \|_2$, which means that $ \| \mathbf{J}(\mathbf{s})^{-1} \mathbf{r}(\mathbf{s}) \|_2$ can become extremely large as well, and small perturbations in $\mathbf{r}$ can lead to very different Gauss-Newton updates (i.e. the problem is ill-conditioned, cf. \citet{NocedalWright} Appendix A.1).

Furthermore, we observe that in the $\lambda > 0$ (unpredictable) setting and the large $T$ limit, the upper and lower bounds in \eqref{eq:recap_mu_LLE} are tight, as they are both $\mathcal{O}(e^{\lambda (T-1)})$. Thus, the upper and lower bounds together ensure that unpredictable dynamics will suffer from degrading conditioning.

In contrast, in the $\lambda < 0$ (predictable) setting, the lower bound on $\sqrt{\mu}$ converges to $\frac{1 - e^{\lambda}}{a}$, which is bounded away from zero and \emph{independent of the sequence length}. Thus, in predictable dynamics, there is a lower bound on $\sigma_{\min}(\mathbf{J})$ or, equivalently, an upper bound on $\sigma_{\max}(\mathbf{J}^{-1})$.

\section{The dynamical interpretation of the inverse Jacobian}

As shown in the above proof, 
\begin{equation*}
 \b J(\mathbf{s})^{-1} = (\b I_{TD} - \b N(\mathbf{s}))^{-1} 
\;=\;
\sum_{k=0}^{T-1} \,\b N(\mathbf{s})^k.
\end{equation*}

It is worth noting explicitly that
\begin{align}\label{eq:N_definition}
    \mathbf{N}(\mathbf{s}) = \begin{pmatrix}
        0 & 0 & \hdots & 0 & 0\\
        A_2 & 0 & \hdots & 0 & 0\\
        \vdots & \vdots & \ddots & \vdots & \vdots \\
        0 & 0 & \hdots & 0 & 0 \\
        0 & 0 & \hdots & \A_T & 0 \\
    \end{pmatrix} \quad \text{where} \quad \A_t &\coloneq \frac{\partial f_t}{\partial s_{t-1}}(s_{t-1}),
\end{align}
i.e. $\mathbf{N}(\mathbf{s})$ collects the Jacobians of the dynamics function along the first lower diagonal.
Each matrix power $\mathbf{N}^k$ therefore collects length $k$ products along the $k$th lower diagonal.
Thus, multiplication by $ \b J(\mathbf{s})^{-1} = 
\sum_{k=0}^{T-1} \,\b N(\mathbf{s})^k$ recovers running forward a linearized form of the dynamics, which is one of the core insights of DeepPCR and DEER \citep{deeppcr, deer2024}.

Concretely, in the setting where $T=4$, we have
\begin{align*}
\mathbf{N}^{0} & = 
    \begin{pmatrix}
    I_D & 0 & 0 & 0 \\
    0 & I_D & 0 & 0 \\
    0 & 0 & I_D & 0 \\
    0 & 0 & 0 & I_D
    \end{pmatrix} \\
    \mathbf{N} & = 
    \begin{pmatrix}
    0 & 0 & 0 & 0 \\
    \A_2 & 0 & 0 & 0 \\
    0 & \A_3 & 0 & 0 \\
    0 & 0 & \A_4 & 0
    \end{pmatrix} \\
    \mathbf{N}^2 & = \begin{pmatrix}
    0 & 0 & 0 & 0 \\
    0 & 0 & 0 & 0 \\
    \A_3 \A_2 & 0 & 0 & 0 \\
    0 & \A_4 \A_3 & 0 & 0
    \end{pmatrix} \\
    \mathbf{N}^3 & = \begin{pmatrix}
    0 & 0 & 0 & 0 \\
    0 & 0 & 0 & 0 \\
    0 & 0 & 0 & 0 \\
    \A_4 \A_3 \A_2 & 0 & 0 & 0
    \end{pmatrix} \\
    \mathbf{J}^{-1} & = \begin{pmatrix}
    I_D & 0 & 0 & 0 \\
    \A_2 & I_D & 0 & 0 \\
    \A_3 \A_2 & \A_3 & I_D & 0 \\
    \A_4 \A_3 \A_2 & \A_4 \A_3 & \A_4 & I_D
    \end{pmatrix} \\
\end{align*}

\subsection{Connection to semiseparable matrices and Mamba2}

Having depicted the structure of $\mathbf{J}^{-1}$, we note the connection between $\mathbf{J}^{-1}$ in this paper and the attention or sequence mixer matrix $M$ in \citet{mamba2}, which introduced the Mamba2 architecture (see equation 6 or Figure 2 of \citet{mamba2} for the form of $M$, and compare with $\mathbf{J}^{-1}$ above).

Mamba2 is a deep learning sequence modeling architecture. Its sequence mixer in each layer has at its core a linear dynamical system. \citet{mamba2} observe that while a linear dynamical system (LDS) can be evaluated recurrently (sequentially) or in parallel (for example, with a parallel scan), it can also be evaluated multiplying the inputs to the LDS by the matrix $M$. Since each DEER iteration is also a linear dynamical system, with the transition matrices given by $\{ \A_t \}_{t=2}^T$, it follows that $M$ in \citet{mamba2} and $\mathbf{J}^{-1}$ in our paper are the same object, and so results about these objects from these two papers transfer.

In particular, we observe that, in the language from \citet{mamba2}, the $\mathbf{J}^{-1}$ we consider in this paper is \emph{$D$-semiseparable} (see Definition 3.1 from \citet{mamba2}). Thus, any efficient, hardware-aware algorithms and implementations developed for $D$-semiseparable matrices could also be applied to accelerating each iteration of DEER, though we note that \citet{mamba2} focus on the 1-semiseparable setting, which they call a \emph{state space dual} or \emph{SSD} layer. In any case, using these connections to accelerate each iteration of DEER and related parallel Newton algorithms from a systems implementation perspective would be an interesting direction for future work.

\section{Framing based on global bounds}

We chose to prove Theorem \ref{theorem:PL-LLE} using condition \eqref{eq:LLE_regularity} in order to highlight the natural connection between the smallest singular value of $\mathbf{J}$ and system stability (as measured by its LLE). However, an assumption with a different framing would be to impose a uniform bound on the spectral norm of the dynamics Jacobian over the entire state space:
\begin{equation}\label{eq:bounded_spectral_norm}
    \sup_{s \in \mathbb{R}^D} \| \A(s) \|_2 \leq \rho .
\end{equation}
For $\rho < 1$, this assumption corresponds to global contraction of the dynamics \citep{lohmiller1998contraction}.

If we replace the LLE regularity condition \eqref{eq:LLE_regularity} with the global spectral norm bound \eqref{eq:bounded_spectral_norm} in the proof of Theorem \ref{theorem:PL-LLE}, we obtain that the PL constant is bounded away from zero, i.e.
\begin{equation*}
    \frac{1}{a} \cdot \frac{\rho - 1}{\rho^T - 1} \leq 
    \sqrt{\inf_{\mathbf{s} \in \mathbb{R}^{TD}} \sigma^2_{\min}(\mathbf{J}(\mathbf{s}))}.
\end{equation*}
In particular, if the dynamics are contracting everywhere (i.e., $\rho < 1$), the condition \eqref{eq:bounded_spectral_norm} guarantees good conditioning of $\mathbf{J}$ throughout the entire state space.

\section{Discussion of the LLE regularity conditions}

The LLE regularity conditions in \cref{eq:LLE_regularity} highlight the more natural "average case" behavior experienced along actual trajectories $\mathbf{s} \in \mathbb{R}^{TD}$. 
This "average case" behavior is highlighted, for example, by our experiments with the two-well system (cf. \Cref{ssc:2well}, where even though a global upper bound on $\| \A_t(s_t) \|_2$ over all of state space would be greater than $1$ (i.e., there are unstable regions of state space), we observe fast convergence of DEER because the system as a whole has negative LLE (its trajectories are stable on average). 

We also note the pleasing relationship the LLE regularity conditions have with the definition of the LLE given in \cref{eq:LLE_definition}.
Note that in the LLE regularity conditions in \cref{eq:LLE_regularity}, the variable $k$ denotes the sequence length under consideration. Taking logs and dividing by $k$, we therefore obtain
\begin{equation*}
    \frac{\log b}{k} + \lambda \leq \frac{1}{k} \log\left( \left\| \A_{t+k -1} \A_{t+k-2} \cdots \A_t \right\| \right) \leq \frac{\log a}{k} + \lambda.
\end{equation*}

Therefore, as $k \to T$, and as $T \to \infty$ (i.e., we consider longer and longer sequences), we observe that the finite-time estimates of the LLE converge to the true LLE $\lambda$.

We observe that as $\mathbf{s}^{(i)}$ approaches the true solution $\mathbf{s}^{*}$, the regularity conditions in \cref{eq:LLE_regularity} become increasingly reasonable. Since any successful optimization trajectory must eventually enter a neighborhood of $\mathbf{s}^{*}$, it is natural to expect these conditions to hold there. In fact, rather than requiring the regularity conditions over all of state space or along the entire optimization trajectory, one could alternatively assume that they hold within a neighborhood of $\mathbf{s}^{*}$, and prove a corresponding version of \Cref{theorem:PL-LLE}.

We now do so, using the additional assumption that $\mathbf{J}$ is $L$-Lipschitz.

\begin{theorem}\label{theorem:local_PL_LLE}
    If $\mathbf{J}$ is $L$-Lipschitz, then there exists a ball of radius $R$ around the solution $\b s^*$, denoted $B(\b s^*,R)$, such that
    \begin{equation*}
        \forall \b s \ \in \ B(\b s^*,R) \qquad 
        | \sigma_{\min}(\mathbf{J}(\mathbf{s})) - \sigma_{\min}(\mathbf{J}(\mathbf{s^*})) | \ \leq \ LR
    \end{equation*}
\end{theorem}
\begin{proof}
    The argument parallels the proof of Theorem 2 in \citet{liu2022loss}.

A fact stemming from the reverse triangle inequality is that for any two matrices $\b A$ and $\b B$,
\[
\sigma_{\min}(\b A) \;\geq\; \sigma_{\min}(\b B) - \|\b A - \b B\|\,.
\]

Applying this with $\b A = \b J(\b s)$ and $\b B = \b J(\b s^*)$, we obtain
\[
\sigma_{\min}(\b J(\b s)) \;\geq\; \sigma_{\min}(\b J(\b s^*)) - \|\b J(\b s) - \b J(\b s^*)\|\,.
\]

If the Jacobian $\b J(\cdot)$ is $L$-Lipschitz, then
\[
\|\b J(\b s) - \b J(\b s^*)\| \;\leq\; L \|\b s - \b s^*\|\,.
\]

Combining, we get
\[
\sigma_{\min}(\b J(\b s)) 
\;\geq\; \sigma_{\min}(\b J(\b s^*)) - L \|\b s - \b s^*\|\,
\]
and 
\[
\sigma_{\min}(\b J(\b s^*)) 
\;\geq\; \sigma_{\min}(\b J(\b s)) - L \|\b s - \b s^*\|\,,
\]
which gives
\begin{equation*}
    \sigma_{\min}(\b J(\b s^*)) - L \|\b s - \b s^*\| \leq \sigma_{\min} (\mathbf{J}(\mathbf{s})) \leq \sigma_{\min}(\b J(\b s^*)) + L \|\b s - \b s^*\|. 
\end{equation*}
Ensuring that $\|\b s - \b s^*\| \leq R$ completes the proof.
\end{proof}
A consequence of \Cref{theorem:local_PL_LLE} is that if the system is unpredictable, then there exists a finite ball around $\mathbf{s}^*$ where the conditioning of the merit function landscape is provably bad. 

As a concrete example, suppose that $\sigma_{\min}(\b J(\b s^*)) = \epsilon$ and $L = 1$. Then \textit{at best}, the PL constant of the loss function inside the ball $B(\b s^*,R)$ is $\epsilon + R$. If $\epsilon$ is small (bad conditioning) then $R$ can be chosen such that the PL constant inside the ball $B(\b s^*,R)$ is also small. 

\section{Controlling the maximum singular value}

In our proof of \Cref{theorem:PL-LLE}, we proved upper and lower bounds for $\sigma_{\min}(\mathbf{J}(\mathbf{s}))$ that depended on the sequence length $T$. We can also prove upper and lower bounds for $\sigma_{\max}(\mathbf{J}(\mathbf{s}))$, but these do not depend on the sequence length.

Assuming condition \eqref{eq:bounded_spectral_norm}, an upper bound on $\sigma_{\mathrm{max}}(\mathbf{J})$ is straightforward to compute via the triangle inequality, 
\begin{align*}
    \sigma_{\max}(\mathbf{J}) & = \| \mathbf{J} \|_2 \\
    & = \| \mathbf{I} - \mathbf{N} \|_2 \\
    & \leq 1 + \| \mathbf{N} \|_2.
\end{align*}
Recalling the definition of $\mathbf{N}$ in \eqref{eq:N_definition}, we observe that it is composed of $\{ \A_t \}$ along its lower block diagonal, and so we have
\begin{align*}
    \| \mathbf{N}(\mathbf{s}) \|_2 & = \sup_t \| \A_t(s_t) \| \\
    \sup_{\mathbf{s} \in \mathbb{R}^{TD}} \| \mathbf{N} (\b s) \|_2 & = \sup_{s \in \mathbb{R}^D} \| \A(s) \| \\
\end{align*}
Elaborating, for a particular choice of trajectory $\mathbf{s} \in \mathbb{R}^{TD}$, $\| \mathbf{N}(\mathbf{s}) \|_2$ is controlled by the maximum spectral norm of the Jacobians $\A_t(s_t)$ along this trajectory.
Analogously, $\sup_{\mathbf{s} \in \mathbb{R}^{TD}} \| \mathbf{N}(\mathbf{s}) \|_2$---i.e., the supremum of the spectral norm of $\mathbf{N}(\mathbf{s})$ over all possible trajectories $\mathbf{s} \in \mathbb{R}^{TD}$, i.e. the optimization space---is upper bounded by $\sup_{s \in \mathbb{R}^D } \| J(s) \|_2$, i.e. the supremum of the spectral norm of the system Jacobians over the state space $\mathbb{R}^D$.

Thus, it follows that
\begin{equation}\label{eq:ub_sigma_max}
    \sigma_{\max}(\mathbf{J}) \leq 1 + \rho.
\end{equation}
Importantly, the upper bound on $\sigma_{\max}(\mathbf{J})$ does not scale with the sequence length $T$.

To obtain the lower bound on $\sigma_{\max}(\mathbf{J})$, we notice that it has all ones along its main diagonal, and so simply by using the unit vector $\mathbf{e}_1$, we obtain
\begin{equation}\label{eq:lb_sigma_max}
    \mathbf{e}_1^{\top} \mathbf{J} \mathbf{e}_1 = 1 \leq \sigma_{\max}(\mathbf{J}).
\end{equation}

\section{Condition number of the Jacobian}

Note that the condition number $\kappa$ of a matrix is defined as the ratio of its maximum and minimum singular values, i.e.
\begin{equation*}
    \kappa(\mathbf{J}) = \frac{\sigma_{\max}(\mathbf{J})}{\sigma_{\min}(\mathbf{J})}.
\end{equation*}

However, because our bounds in \cref{eq:ub_sigma_max} and \cref{eq:lb_sigma_max} on $\sigma_{\max}(\mathbf{J})$ do not scale with the sequence length $T$, it follows that the scaling with $T$ of an upper bound on $\kappa(\mathbf{J})$---the conditioning of the optimization problem---is controlled solely by the bounds on $\sigma_{\min}(\mathbf{J})$ that we provided in Theorem \ref{theorem:PL-LLE}.
The importance of studying how the conditioning scales with $T$ stems from the fact that we would like to understand if there are regimes---particularly involving large sequence lengths and parallel computers---where parallel evaluation can be faster than sequential evaluation.

\chapter{Discussion of parallel chord methods}\label{app:parallel_chord}

\citet{ortega1970iterative} discuss at length iterative methods for solving arbitrary systems of nonlinear equations $\mathbf{F}(\mathbf{x}) = 0$ using iterations of the form
\begin{equation}\label{eq:parr_chord}
    \mathbf{x}^{(i+1)} = \mathbf{x}^{(i)} - \tilde{\J}(\mathbf{x}^{(i)})^{-1} \mathbf{F}(\mathbf{x}^{(i)})
\end{equation}
for some matrix $\tilde{\J}(\mathbf{x}^{(i)})$.
In general, $\tilde{\J}$ can be a function of the current iterate $\mathbf{x}^{(i)}$ or a fixed and constant matrix.
Newton's method corresponds to
\begin{equation*}
    \tilde{\J}(\mathbf{x}^{(i)}) = \J(\mathbf{x}^{(i)} \coloneq \dfrac{\partial \mathbf{F}}{\partial \mathbf{x}}(\mathbf{x}^{(i)}).
\end{equation*}
When $\tilde{\J}$ is fixed and constant, \cite{ortega1970iterative} describe the resulting family of fixed-point iterations as \emph{parallel-chord methods}. However, we will use this term for \emph{all} iterative methods with updates of the form in \cref{eq:parr_chord}, which includes both Newton and Picard iterations.

The term "parallel" in this context does not have to do with applying a parallel scan over the sequence length (which is the focus of this thesis). Instead, "parallel" in "parallel-chord methods" refers to the way in which Newton's method finds the zero of a function by making a guess for the zero, and then forming a chord that is parallel to the function at the current guess (\Cref{fig:pchord}). In one-dimension, the linearization is a line (a chord), while in higher-dimensions the linearization is in general a hyperplane. In Newton's method, the chord/hyperplane is tangent to the function at the current guess, while for other parallel-chord methods the approximate linearization is in general not tangent.

The equation $\mathbf{F}(\mathbf{x}) = \mathbf{0}$ is a fully general way to represent a system of nonlinear equations. However, in this paper, we focus on parallelizing Markovian state space models, as discussed in \Cref{ch:background}.

In their treatment of Picard iterations, \citet{ortega1970iterative} consider a more general formulation than that presented in \citet{shih2023parallel} or in \cref{eq:picard_shih}. Instead, similar to the definition presented in Appendix C.2.3 of \citet{gu2021combining}, \citet{ortega1970iterative} define Picard iterations in the setting where we have removed a linear component of $\mathbf{F}$, namely we have written
\begin{equation}\label{eq:picard_refactor}
    \mathbf{F} (\mathbf{s}) =: \tilde{\J} \mathbf{s} - \mathbf{G}(\mathbf{s}),
\end{equation}
for some constant, nonsingular matrix $\tilde{\J}$ and nonlinear function $\mathbf{G}(\cdot)$. Note that such redefinition of $\mathbf{F}(\cdot)$ in terms of $\tilde{\J}$ and $\mathbf{G}(\cdot)$ is always possible and not uniquely determined. After making such a redefinition, \citet{ortega1970iterative} define a Picard iteration as an update of the form
\begin{equation}\label{eq:picard_or}
    \mathbf{s}^{(i+1)} = \tilde{\J}^{-1} \mathbf{G}(\mathbf{s}^{(i)}).
\end{equation}
However, by multiplying both sides of \cref{eq:picard_refactor} by $\tilde{\J}^{-1}$, it follows that 
\begin{equation*}
    \tilde{\J}^{-1} \mathbf{G}(\mathbf{s}^{(i)}) = \mathbf{s}^{(i)} - \tilde{\J}^{-1} \mathbf{F}(\mathbf{s}^{(i)}),
\end{equation*}
showing that the Picard iterations as defined in \cref{eq:picard_or} fit into the parallel-chord framework set out in \cref{eq:parr_chord}.
Note that Picard iterations as defined by \citet{shih2023parallel} or in \cref{eq:picard_shih} of this paper also fit into the framework of \cref{eq:picard_refactor}: in the context of evaluating discretized ODEs, the residual becomes
\begin{equation*}
    F_{t+1}(\mathbf{s}) = x_{t+1} - x_t - \epsilon g_t(s_t).
\end{equation*}
Thus, in the context of \cref{eq:picard_refactor}, we have that the resulting $G_t(\mathbf{s}) = \epsilon g_{t-1}(x_{t-1})$, while the resulting $\tilde{\J}$ operator is given by \cref{eq:A_P}.

When we plug this $\tilde{\J}$ into \cref{eq:parr_chord} and simplify, we obtain the linear dynamical system in the "Picard" row of \Cref{tab:fxd_pt_sum}. In general, the fixed-point methods of the common form given by \cref{eq:qdeer_lds} all give rise to $\tilde{\J} \in \mathbb{R}^{TD \times TD}$ matrices of the form shown in \cref{eq:bold_A}.

Thus, \citet{ortega1970iterative} unites Newton and Picard iterations for the general root finding problem $F(\mathbf{s}) = 0$ under the umbrella of parallel-chord methods, which are iterative updates of the form of \cref{eq:parr_chord}. The framework we provide in \Cref{tab:fxd_pt_sum} can be understood as a specialization of parallel-chord methods for the particular problem of sequential evaluation discussed in \cref{eq:ssm}. Nonetheless, we focus on how in the specific problem of sequential evaluation, which is of great interest in many areas of machine learning, a wide variety of fixed-point methods become iterative application of LDSs, allowing them to be parallelized over the sequence length with an associative scan. This important perspective about parallelizability, which is of great interest in machine learning, is not discussed in \citet{ortega1970iterative} because they are considering a more general problem. 

\citet{ortega1970iterative} also discuss in their Chapters 7 and 10 how the closeness of the "parallel chord" (in general and in higher dimensions, the "approximating hyperplane") to the true linearization of the function (Newton's method) affects the number of iterations needed for the parallel-chord method to converge.
This analysis is directly analogous to our study of the effect of $\left \| \tilde{\J}(\mathbf{s}_{1:T}) - \mathbf{J}(\mathbf{s}_{1:T}) \right \|_2$ on the rate of convergence of fixed-point methods, see \Cref{lem:diff}.
In particular, in Chapter 10 of \cite{ortega1970iterative}, they consider the rates of convergence of fixed-point methods with updates taking the form of
\begin{equation}\label{eq:one_step_stationary}
    \mathbf{s}^{(i+1)} = \mathbf{U}(\mathbf{s}^{(i)}),
\end{equation}
for some function $\mathbf{U}(\cdot)$. \citet{ortega1970iterative} use the name \emph{one-step stationary methods} for such fixed-point methods with updates of the form \cref{eq:one_step_stationary}.

For parallel-chord methods of the form given in \cref{eq:parr_chord}, it follows that
\begin{equation}\label{eq:u_pchord}
    \mathbf{U}(\mathbf{s}^{(i)}) = \mathbf{s}^{(i)} - \tilde{\J}(\mathbf{s}^{(i)})^{-1} \mathbf{F}(\mathbf{s}^{(i)}).
\end{equation}

In particular, in their Chapters 7 and 10, \cite{ortega1970iterative} introduce and study $\sigma(\mathbf{U}, \mathbf{F}, \mathbf{s}^{\star})$, which determines the rate of convergence of iterative methods with updates of the form given by \cref{eq:one_step_stationary} to the solution $\mathbf{s}^{\star}$ of $\mathbf{F}(\mathbf{s})=\mathbf{0}$. They define $\sigma$ as
\begin{equation}\label{eq:sigma_or}
    \sigma(\mathbf{U}, \mathbf{F}, \mathbf{s}^{\star}) := \rho\left( \dfrac{\partial \mathbf{U}}{\partial \mathbf{s}}(\mathbf{s}^{\star}) \right),
\end{equation}
where $\rho(M)$ denotes the spectral radius of a matrix $M$.

In the context of parallel-chord methods where $\mathbf{U}(\cdot)$ is given by \cref{eq:u_pchord}, it follows that
\begin{equation*}
    \dfrac{\partial \mathbf{U}}{\partial \mathbf{s}}(\mathbf{s}^{\star}) = \mathbf{I} - \tilde{\J}(\mathbf{s}^{\star})^{-1} \dfrac{\partial \mathbf{F}}{\partial \mathbf{s}}(\mathbf{s}^{\star}),
\end{equation*}
because $\mathbf{F}(\mathbf{s}^{\star}) = 0$.
Thus it follows that if  $\tilde{\J} = \nicefrac{\partial \mathbf{F}}{\partial \mathbf{s}}(\mathbf{s}^{\star})$, then $\sigma=0$. Thus, lower values of $\sigma$ indicates that $\tilde{\J}$ is good approximation of the Jacobian matrix $\nicefrac{\partial \mathbf{F}}{\partial \mathbf{s}}$ evaluated at the zero $\mathbf{s}^{\star}$ of $\mathbf{F}$, while higher values of $\sigma$ indicate that $\tilde{\J}$ is a poor approximation for $\nicefrac{\partial \mathbf{F}}{\partial \mathbf{s}}$. \cite{ortega1970iterative} then use $\sigma$ in their Chapter 10 (in particular, their Theorem 10.1.4) to prove linear rates of convergence\footnote{where the rate is given by $\sigma$} for one-step stationary methods within a neighborhood of the solution $\mathbf{s}^{\star}$.

Thus, a takeaway from \cite{ortega1970iterative} (as paraphrased from \citet{Gasilov1981ParallelChord}) is that the closer $\tilde{\J}$ is to $\nicefrac{\partial \mathbf{F}}{\partial \mathbf{s}}$, the fewer iterations are needed for convergence to $\mathbf{s}^{\star}$. This takeaway is extremely similar to our guidance, though we specialize to the particular system of equations given by \cref{eq:r_def} that results from the goal of rolling out the Markov process given by \cref{eq:ssm}.

However, in the setting we consider in this paper---using fixed-point iterations of the form \cref{eq:qdeer_lds} to solve nonlinear equations of the form \cref{eq:r_def}---Theorem 10.1.4 of \citet{ortega1970iterative} is actually \emph{trivial}. By "trivial," we mean that it does not distinguish between the convergence rates of any of the fixed-point iterations we focus on in this paper.

To make this point more precisely, we review\footnote{We follow the presentation of Chapter 9 of \citet{ortega1970iterative}, in particular Definition 9.2.1.} the notion of \emph{root-convergence}, more commonly known as \emph{$R$-convergence}.
\begin{definition}[$R$-convergence]\label{def:r_convergence}
    Let $\mathcal{A}$ be a fixed-point operator with fixed-point $\mathbf{s}^{\star}$.
    Let $C(\mathcal{A}, \mathbf{s}^{\star})$ be the set of all sequences generated by $\mathcal{A}$ which converge to $\mathbf{s}^{\star}$.
    Then the $R_1$-factors of $\mathcal{A}$ at $\mathbf{s}^{\star}$ are given by
    \begin{equation}\label{eq:r_convergence}
    R_1(\mathcal{A}, \mathbf{s}^{\star})
    := \sup \left\{
      \limsup_{i\to\infty}\|\mathbf{s}^{(i)}-\mathbf{s}^{\star}\|^{1/i}
      \,\middle|\,
      \{\mathbf{s}^{(i)}\}_{i\ge 0}\in C(\mathcal{A},\mathbf{s}^{\star})
    \right\}.
    \end{equation}
\end{definition}
Intuitively, $R_1(\mathcal{A}, \mathbf{s}^{\star})$ gives the rate of linear convergence of a fixed-point operator $\mathcal{A}$ to its fixed-point $\mathbf{s}^{\star}$.
Theorem 10.1.4 of \citet{ortega1970iterative} implies that if $\mathcal{A}$ is a one-step stationary method with update given by $\mathbf{U}(\cdot)$, then $R_1(\mathcal{A}, \mathbf{s}^{\star}) = \sigma(\mathbf{U}, \mathbf{F}, \mathbf{s}^{\star})$. Therefore, if $\sigma > 0$, then $\sigma$ is the rate of $R$-linear convergence of $\mathcal{A}$ to $\mathbf{s}^{\star}$, while if $\sigma=0$, we say that $\mathcal{A}$ converges \emph{$R$-superlinearly}.
However, it is important to note that these definitions are \emph{asymptotic} in nature.

The fixed-point iterations considered in this paper, i.e. following the common form \cref{eq:qdeer_lds}, all have $\sigma=0$, and therefore can be said to converge $R$-superlinearly.
\begin{proposition}\label{prop:or_trivial}
    Let $\mathbf{F}(\mathbf{s}) = \mathbf{0}$ be a nonlinear equation of the form \cref{eq:r_def} with solution $\mathbf{s}^{\star}$. Let $\mathcal{A}$ be a parallel-chord method with fixed-point $\mathbf{s}^{\star}$. Then
    \begin{equation*}
        \sigma\left(\mathbf{U}, \mathbf{F}, \mathbf{s}^{\star}\right) = 0.
    \end{equation*}
\end{proposition}
\begin{proof}
    Both $\nicefrac{\partial \mathbf{F}}{\partial \mathbf{s}}(\mathbf{s}^{\star})$ and $\tilde{\J}(\mathbf{s}^{\star})$ are lower-triangular matrices with all $D \times D$ identity matrices on their main block-diagonal. In particular, $\tilde{\J}^{-1}$ is also a lower-triangular matrix with all $D \times D$ identity matrices on its main block-diagonal. Consequently, the product $\tilde{\J}^{-1} \dfrac{\partial \mathbf{F}}{\partial \mathbf{s}}$ is also a lower-triangular matrix with all $D \times D$ identity matrices on its main block-diagonal. As a result, $\mathbf{I} - \tilde{\J}^{-1} \dfrac{\partial \mathbf{F}}{\partial \mathbf{s}}$ is a lower-triangular matrix with all zeros on its main block-diagonal, and so has all its eigenvalues equal to $0$. Consequently, its spectral radius is equal to zero.
\end{proof}
It may seem counterintuitive that even Jacobi iterations technically enjoy $R$-superlinear convergence in the context of parallelizing Markov processes. However, this seemingly strange result stems from the asymptotic nature of \Cref{def:r_convergence} of $R$-convergence, and the fact that Proposition 1 of \citet{gonzalez2024scalable} guarantees that all fixed-point iterations of the form given by \cref{eq:qdeer_lds} will converge to $\mathbf{s}^{\star}$ in a finite number of iterations ($T$, to be exact). Therefore, for any LDS fixed-point scheme, we always have $\lim_{i \to \infty} \| \mathbf{s}^{(i)} - \mathbf{s}^{\star} \|=0$.

However, in both Proposition 4 of \citet{parasolver25} and \Cref{prop:ConvRates} of this paper, we effectively get around this difficulty by considering the \emph{spectral norm} instead of the \emph{spectral radius}.
The spectral norm always bounds the spectral radius, and so by focusing on spectral radius, \citet{ortega1970iterative} could get tighter bounds (faster rates of convergence). However, in our setting the spectral radius cannot distinguish between any of the fixed-point methods, and so we instead use the looser bound provided by the spectral norm, which can distinguish between the different fixed-point methods.
Note that the core entities are effectively the same, as $\gamma$ defined in \cref{eq:lin_rate} is equal to $\| \nicefrac{\partial \mathbf{U}}{\partial \mathbf{s}}(\mathbf{s}^{\star}) \|_2$.

Finally, again, because all of our fixed-point methods converge in at most $T$ iterations, asymptotic notions of linear convergence are not suitable to fully capture the behavior of these fixed point methods. 
For this reason, we use empirical case studies in \Cref{sec:tasks} to show that efficacy of the intuition, inspired by \Cref{prop:ConvRates}, that the closeness of $\tilde{A}_t$ to $\A_t$ impacts the number of iterations needed for $\mathcal{A}$ to converge. This empirical approach also highlights how the increased computational cost of higher-order fixed-point methods affects wall-clock time on GPUs.
\end{appendices}
%********************************************************************
% Other Stuff in the Back
%*******************************************************
\clearpage%********************************************************************
% Bibliography
%*******************************************************
% work-around to have small caps also here in the headline
\manualmark
\markboth{\spacedlowsmallcaps{\bibname}}{\spacedlowsmallcaps{\bibname}} % work-around to have small caps also
%\phantomsection 
\refstepcounter{dummy}
\addtocontents{toc}{\protect\vspace{\beforebibskip}} % to have the bib a bit from the rest in the toc
\addcontentsline{toc}{chapter}{\tocEntry{\bibname}}
\label{app:bibliography}
\printbibliography

\end{document}